\documentclass{book}
\usepackage[utf8]{inputenc}
\usepackage[margin=1in,left=1in]{geometry}
\usepackage{setspace}
\usepackage{graphicx}
\usepackage{amssymb, amsmath, amsthm, graphics}
\usepackage{mathabx,epsfig}
\usepackage[mathscr]{euscript}
\usepackage{tikz}
\usepackage{pigpen}
\usepackage{bm}
\usetikzlibrary{calc}
\usetikzlibrary{matrix}
\usepackage{latexsym}
\usepackage[all,knot]{xy}

\usepackage{color}

\usepackage{bbold}

\input xy
\xyoption{all}
\usepackage{euscript}
\usepackage{multirow}
\usepackage{etex, pictexwd,dcpic}
\usetikzlibrary{positioning}
\usetikzlibrary{shapes.geometric}
\usetikzlibrary{shapes.misc}
\usetikzlibrary{calc}
\usetikzlibrary{positioning}

\newcommand{\po}{\ar@{}[dr]|{\text{\pigpenfont R}}}
\newcommand{\pb}{\ar@{}[dr]|{\text{\pigpenfont J}}}

\usepackage{pgflibraryarrows}
\usepackage{pgflibrarysnakes}

\usetikzlibrary{trees} 
\usetikzlibrary[trees] 

\newtheorem{theorem}{Theorem}
\theoremstyle{definition}
\newtheorem{definition}{Definition}[section]
\newtheorem{notation}{Notation}[section]
\newtheorem{lemma}{Lemma}[section]
\newtheorem{corollary}{Corollary}[section]
\newtheorem{proposition}{Proposition}[section]

\newtheorem{remark}{Remark}[section]

\newtheorem{example}{Example}[section]

\numberwithin{equation}{section}

\def\undercrossing{\xy 
(-4,-5);(4,5)**\crv{(0,0)}
?>(.9)*\dir{>}?>(.5)*{\color{white}\bullet},
(4,-5);(-4,5)**\crv{(0,0)}
?>(.9)*\dir{>}
\endxy}

\def\overcrossing{\xy 
(4,-5);(-4,5)**\crv{(0,0)}
?>(.9)*\dir{>}?>(.5)*{\color{white}\bullet},
(-4,-5);(4,5)**\crv{(0,0)}
?>(.9)*\dir{>}
\endxy}

\def\noncrossing{\xy 
(-3,-5);(-3,5)**\crv{(-1,0)}
?>(.9)*\dir{>},
(3,-5);(3,5)**\crv{(1,0)}
?>(.9)*\dir{>}
\endxy}

\def\ijnoncrossing{\xy 
(-3,-5);(-3,5)**\crv{(-1,0)}
?>(.9)*\dir{>},
(3,-5);(3,5)**\crv{(1,0)}
?>(.9)*\dir{>}
,(-4,-5)*{\scriptstyle{1}}
,(4,-5)*{\scriptstyle{1}}
\endxy}

\def\unknot{
{\xy
(0,-5);(0,-5)**\crv{(5,-5)&(5,5)&(-5,5)&(-5,-5)}
?>(.3)*\dir{>}
\endxy}
}

\def\iunknot{
{\xy
(0,-5);(0,-5)**\crv{(5,-5)&(5,5)&(-5,5)&(-5,-5)}
?>(.3)*\dir{>}
,(-4,-5)*{\scriptstyle{1}}
\endxy}
}

\def\ijundercrossing{\xy 
(-4,-5);(4,5)**\crv{(0,0)}
?>(.9)*\dir{>}?>(.5)*{\color{white}\bullet},
(4,-5);(-4,5)**\crv{(0,0)}
?>(.9)*\dir{>}
,(-5,-5)*{\scriptstyle{1}}
,(5,-5)*{\scriptstyle{1}}
\endxy}

\def\ijovercrossing{\xy 
(4,-5);(-4,5)**\crv{(0,0)}
?>(.9)*\dir{>}?>(.5)*{\color{white}\bullet},
(-4,-5);(4,5)**\crv{(0,0)}
?>(.9)*\dir{>}
,(-5,-5)*{\scriptstyle{1}}
,(5,-5)*{\scriptstyle{1}}
\endxy}

\def\reidemeisteroneequal{
\xy
(2,5);(0,0.1)**\crv{(2,1)&(-2,-4)&(-4,0)&(-2,4)}?>(0)*\dir{<}
,(2,-5);(0.65,-1)**\crv{(2,-4)}
\endxy\quad=
\quad
\xy
(0,-5);(0,5)**\dir{-}?>(1)*\dir{>}
\endxy}

\def\reidemeisteroneideal{
R_1=\quad \xy
(2,5);(0,0.1)**\crv{(2,1)&(-2,-4)&(-4,0)&(-2,4)}?>(0)*\dir{<}
,(2,-5);(0.65,-1)**\crv{(2,-4)}
\endxy\quad-
\quad
\xy
(0,-5);(0,5)**\dir{-}?>(1)*\dir{>}
\endxy}

\def\reidemeistertwoequal{
\xy
(3,-7);(3,7)**\crv{(3,-5)&(3,-4)&(-2,0)&(3,4)&(3,5)}?>(.99)*\dir{>}
,(-3,-7);(-0.3,-2)**\crv{(-3,-5)&(-3,-4)}
,(-3,7);(-0.3,2)**\crv{(-3,5)&(-3,4)}?>(0)*\dir{<}
,(.6,-1.3);(0.6,1.4)**\crv{(3,0)}
\endxy\quad = 
\quad
\xy
(-3,-7);(-3,7)**\dir{-}?>(1)*\dir{>}
,(3,-7);(3,7)**\dir{-}?>(1)*\dir{>}
\endxy}

\def\reidemeistertwoideal{
R_{2a}=\qquad
\xy
(3,-7);(3,7)**\crv{(3,-5)&(3,-4)&(-2,0)&(3,4)&(3,5)}?>(.99)*\dir{>}
,(-3,-7);(-0.3,-2)**\crv{(-3,-5)&(-3,-4)}
,(-3,7);(-0.3,2)**\crv{(-3,5)&(-3,4)}?>(0)*\dir{<}
,(.6,-1.3);(0.6,1.4)**\crv{(3,0)}
\endxy\quad - 
\quad
\xy
(-3,-7);(-3,7)**\dir{-}?>(1)*\dir{>}
,(3,-7);(3,7)**\dir{-}?>(1)*\dir{>}
\endxy}

\def\reidemeistertwobequal{
 \xy
(-5,10);(5,10)**\crv{(-5,2)&(0,-6)&(5,2)}
?>(0)*\dir{<},
?>(.71)*{\color{white}\bullet},
?>(.30)*{\color{white}\bullet}
,(-5,-10);(5,-10)**\crv{(-5,-2)&(0,6)&(5,-2)}
?>(1)*\dir{>}
\endxy\quad
=\quad
\xy
(-5,-10);(5,-10)**\crv{(-5,-2)&(0,-2)&(5,-2)}
?>(1)*\dir{>}
,(-5,10);(5,10)**\crv{(-5,2)&(0,2)&(5,2)}
?>(0)*\dir{<}
\endxy\quad
}

\def\reidemeistertwobideal{
R_{2b}=\qquad \xy
(-5,10);(5,10)**\crv{(-5,2)&(0,-6)&(5,2)}
?>(0)*\dir{<},
?>(.71)*{\color{white}\bullet},
?>(.30)*{\color{white}\bullet}
,(-5,-10);(5,-10)**\crv{(-5,-2)&(0,6)&(5,-2)}
?>(1)*\dir{>}
\endxy\quad
-\quad
\xy
(-5,-10);(5,-10)**\crv{(-5,-2)&(0,-2)&(5,-2)}
?>(1)*\dir{>}
,(-5,10);(5,10)**\crv{(-5,2)&(0,2)&(5,2)}
?>(0)*\dir{<}
\endxy\quad
}

\def\reidemeisterthreeequal{
\xy
(-9,-7);(9,7)**\crv{(-9,-5)&(-9,-3)&(9,3)&(9,5)}?>(.99)*\dir{>}
,(9,-7);(1,-.3)**\crv{(9,-5)&(9,-3)}
,(-1,.3);(-5.3,2)**\crv{(-3,1)}
,(-6.7,2.8);(-9,7)**\crv{(-9,4)&(-9,5)}?>(.99)*\dir{>}
,(0,-7);(-5.3,-2.8)**\crv{(0,-5)&(0,-4)}
,(0,7);(-6.1,2.4)**\crv{(0,6)&(0,5)&(0,4)&(-6,2.8)}?>(0)*\dir{<}
,(-6.5,-2.2);(-6.1,2.4)**\crv{(-9,0)}
\endxy
\quad = 
\quad
\xy
(-9,-7);(9,7)**\crv{(-9,-5)&(-9,-3)&(9,3)&(9,5)}?>(.99)*\dir{>}
,(5.6,-2);(1,-.3)**\crv{(3,-1)}
,(9,-7);(6.7,-2.8)**\crv{(9,-5)&(9,-4)}
,(-1,.3);(-9,7)**\crv{(-9,3)&(-9,5)}?>(.99)*\dir{>}
,(0,-7);(6,-2.6)**\crv{(0,-5)&(0,-4)&(6,-3)}
,(0,7);(5.3,2.8)**\crv{(0,6)&(0,5)&(0,4)}?>(0)*\dir{<}
,(6,-2.6);(6.5,2.2)**\crv{(9,0)}
\endxy
}

\def\reidemeisterthreeideal{
R_3=\quad \xy
(-9,-7);(9,7)**\crv{(-9,-5)&(-9,-3)&(9,3)&(9,5)}?>(.99)*\dir{>}
,(9,-7);(1,-.3)**\crv{(9,-5)&(9,-3)}
,(-1,.3);(-5.3,2)**\crv{(-3,1)}
,(-6.7,2.8);(-9,7)**\crv{(-9,4)&(-9,5)}?>(.99)*\dir{>}
,(0,-7);(-5.3,-2.8)**\crv{(0,-5)&(0,-4)}
,(0,7);(-6.1,2.4)**\crv{(0,6)&(0,5)&(0,4)&(-6,2.8)}?>(0)*\dir{<}
,(-6.5,-2.2);(-6.1,2.4)**\crv{(-9,0)}
\endxy
\quad - 
\quad
\xy
(-9,-7);(9,7)**\crv{(-9,-5)&(-9,-3)&(9,3)&(9,5)}?>(.99)*\dir{>}
,(5.6,-2);(1,-.3)**\crv{(3,-1)}
,(9,-7);(6.7,-2.8)**\crv{(9,-5)&(9,-4)}
,(-1,.3);(-9,7)**\crv{(-9,3)&(-9,5)}?>(.99)*\dir{>}
,(0,-7);(6,-2.6)**\crv{(0,-5)&(0,-4)&(6,-3)}
,(0,7);(5.3,2.8)**\crv{(0,6)&(0,5)&(0,4)}?>(0)*\dir{<}
,(6,-2.6);(6.5,2.2)**\crv{(9,0)}
\endxy
}

\def\ev{
{\xy
(-5,-5);(5,-5)**\crv{(-5,5)&(0,5)&(5,5)}
?>(1)*\dir{>}
,(-4,-5)*{\scriptstyle{1}},(6,-5)*{\scriptstyle{1}}
\endxy}
}

\def\evop{
{\xy
(-5,-5);(5,-5)**\crv{(-5,5)&(0,5)&(5,5)}
?>(0)*\dir{<}
,(-4,-5)*{\scriptstyle{1}},(6,-5)*{\scriptstyle{1}}
\endxy}
}

\def\coev{
{\xy
(-5,5);(5,5)**\crv{(-5,-5)&(0,-5)&(5,-5)}
?>(0)*\dir{<}
,(-4,5)*{\scriptstyle{1}},(6,5)*{\scriptstyle{1}}
\endxy}
}

\def\coevop{
{\xy
(-5,5);(5,5)**\crv{(-5,-5)&(0,-5)&(5,-5)}
?>(1)*\dir{>}
,(-4,5)*{\scriptstyle{1}},(6,5)*{\scriptstyle{1}}
\endxy}
}

\def\esse{
{\xy
(-5,-10);(-5,10)**\crv{(-1,-6)&(0,-5)&(0,-2)&(0,2)&(0,5)&(-1,6)}
?>(.55)*\dir{>}?>(.1)*\dir{>}?>(.95)*\dir{>}
,(5,-10);(5,10)**\crv{(1,-6)&(0,-5)&(0,-2)&(0,2)&(0,5)&(1,6)}
?>(.1)*\dir{>}?>(.95)*\dir{>}
,(-7,-10)*{\scriptstyle{1}}
,(7,-10)*{\scriptstyle{1}}
,(-7,10)*{\scriptstyle{1}}
,(7,10)*{\scriptstyle{1}}
,(2,0)*{\scriptstyle{2}}
\endxy}\,
}

\def\essedual{
{\xy
(-5,-10);(-5,10)**\crv{(-1,-6)&(0,-5)&(0,-2)&(0,2)&(0,5)&(-1,6)}
?>(.5)*\dir{<}?>(.01)*\dir{<}?>(.95)*\dir{<}
,(5,-10);(5,10)**\crv{(1,-6)&(0,-5)&(0,-2)&(0,2)&(0,5)&(1,6)}
?>(.01)*\dir{<}?>(.95)*\dir{<}
,(-7,-10)*{\scriptstyle{1}}
,(7,-10)*{\scriptstyle{1}}
,(-7,10)*{\scriptstyle{1}}
,(7,10)*{\scriptstyle{1}}
,(2,0)*{\scriptstyle{2}}
\endxy}\,
}

\def\ti{
{\xy
(-5,-10);(-5,10)**\crv{(-1,-6)&(0,-5)&(0,-2)&(0,2)&(0,5)&(-1,6)}
?>(.55)*\dir{>}?>(.1)*\dir{>}?>(.95)*\dir{>}
,(5,-10);(5,10)**\crv{(1,-6)&(0,-5)&(0,-2)&(0,2)&(0,5)&(1,6)}
?>(.1)*\dir{>}?>(.95)*\dir{>}
,(0,-10);(0,10)**\dir{-}?>(.15)*\dir{>}?>(.95)*\dir{>}
,(-7,-10)*{\scriptstyle{1}}
,(7,-10)*{\scriptstyle{1}}
,(-7,10)*{\scriptstyle{1}}
,(7,10)*{\scriptstyle{1}}
,(-1,-11)*{\scriptstyle{1}}
,(-1,11)*{\scriptstyle{1}}
,(2,0)*{\scriptstyle{3}}
\endxy}
}

\def\essecircesse{
{\xy
(0,-10);(0,10)**\crv{(0,-7)&(0,-6)&(0,-5)&(-4,0)&(0,5)&(0,6)&(0,7)}
?>(.1)*\dir{>}?>(.99)*\dir{>}?>(.65)*\dir{>}
,(0,-10);(0,10)**\crv{(0,-7)&(0,-6)&(0,-5)&(4,0)&(0,5)&(0,6)&(0,7)}
?>(.65)*\dir{>}
,(0,10);(-5,14)**\crv{(0,10)&(0,11)&(-1,12)}?>(.95)*\dir{>}
,(0,10);(5,14)**\crv{(0,10)&(0,11)&(1,12)}?>(.95)*\dir{>}
,(0,-10);(-5,-14)**\crv{(0,-10)&(0,-11)&(-1,-12)}?>(.85)*\dir{<}
,(0,-10);(5,-14)**\crv{(0,-10)&(0,-11)&(1,-12)}?>(.85)*\dir{<}
,(-7,-15)*{\scriptstyle{1}}
,(7,-15)*{\scriptstyle{1}}
,(-7,15)*{\scriptstyle{1}}
,(7,15)*{\scriptstyle{1}}
,(-4,0)*{\scriptstyle{1}}
,(4,0)*{\scriptstyle{1}}
,(1.5,-9)*{\scriptstyle{2}}
,(1.5,9)*{\scriptstyle{2}}
\endxy}
}

\def\idone{
{\xy
(0,-10);(0,10)**\dir{-}?>(.5)*\dir{>}
,(-2,-10)*{\scriptstyle{1}}
,(-2,10)*{\scriptstyle{1}}
\endxy}
}

\def\idonedual{
{\xy
(0,-10);(0,10)**\dir{-}?>(.5)*\dir{<}
,(-2,-10)*{\scriptstyle{1}}
,(-2,10)*{\scriptstyle{1}}
\endxy}
}

\def\idoner{
{\xy
(0,-10);(0,10)**\dir{-}?>(.5)*\dir{>}
,(2,-10)*{\scriptstyle{1}}
,(2,10)*{\scriptstyle{1}}
\endxy}
}

\def\longidoner{
{\xy
(0,-15);(0,15)**\dir{-}?>(.5)*\dir{>}
,(2,-15)*{\scriptstyle{1}}
,(2,15)*{\scriptstyle{1}}
\endxy}
}

\def\idtwo{
{\xy
(-5,-10);(-5,10)**\dir{-}?>(.5)*\dir{>}
,(5,-10);(5,10)**\dir{-}?>(.5)*\dir{>}
,(-7,-10)*{\scriptstyle{1}}
,(7,-10)*{\scriptstyle{1}}
,(-7,10)*{\scriptstyle{1}}
,(7,10)*{\scriptstyle{1}}
\endxy}
}

\def\idtwodual{
{\xy
(-5,-10);(-5,10)**\dir{-}?>(.5)*\dir{<}
,(5,-10);(5,10)**\dir{-}?>(.5)*\dir{<}
,(-7,-10)*{\scriptstyle{1}}
,(7,-10)*{\scriptstyle{1}}
,(-7,10)*{\scriptstyle{1}}
,(7,10)*{\scriptstyle{1}}
\endxy}
}

\def\idthree{
{\xy
(-5,-10);(-5,10)**\dir{-}?>(.5)*\dir{>}
,(0,-10);(0,10)**\dir{-}?>(.5)*\dir{>}
,(5,-10);(5,10)**\dir{-}?>(.5)*\dir{>}
,(-7,-10)*{\scriptstyle{1}}
,(7,-10)*{\scriptstyle{1}}
,(-7,10)*{\scriptstyle{1}}
,(7,10)*{\scriptstyle{1}}
,(-2,-10)*{\scriptstyle{1}}
,(-2,10)*{\scriptstyle{1}}
\endxy}
}

\def\splus{\xy 
(4,-2.5);(-4,2.5)**\crv{(4,-2.25)&(4,-2.1)&(0,0)&(-4,2.1)&(-4,2.25)}
?>(.7)*\dir{>}?>(.5)*{\color{white}\bullet},
(-4,-2.5);(4,2.5)**\crv{(-4,-2.25)&(-4,-2.1)&(0,0)&(4,2.1) & (4,2.25)}
?>(.7)*\dir{>}
\endxy}

\newcommand{\splusk}[1]{\left.{\raisebox{-3.5pt}{\xy
(4,-7.5);(-4,-2.5)**\crv{(4,-7.25)&(4,-7.1)&(0,-5)&(-4,-2.9)& (-4,-2.75)}
?>(.5)*{\color{white}\bullet},
,(-4,-7.5);(4,-2.5)**\crv{(-4,-7.25)&(-4,-7.1)&(0,-5)&(4,-2.9)& (4,-2.75)}
,(4,-2.5);(-4,2.5)**\crv{(4,-2.25)&(4,-2.1)&(0,0)&(-4,2.1)& (-4,2.25)}
?>(.7)*\dir{>}?>(.5)*{\color{white}\bullet},
(-4,-2.5);(4,2.5)**\crv{(-4,-2.25)&(-4,-2.1)&(0,0)&(4,2.1) &(4,2.25)}
?>(.7)*\dir{>}
,(4,7);(-4,12)**\crv{(4,7.25)&(4,7.35)&(0,9.5)&(-4,11.6)& (-4,11.75)}
?>(.7)*\dir{>}?>(.5)*{\color{white}\bullet},
,(-4,7);(4,12)**\crv{(-4,7.25)&(-4,7.35)&(0,9.5)&(4,11.6)& (4,11.75)}
?>(.7)*\dir{>}
,(-4,5.5)*{\vdots}
,(4,5.5)*{\vdots}
\endxy}}\,\,\right\}\text{\small{${#1}$ \rm crossings}}}

\newcommand{\trsplusk}[1]{\left.{\raisebox{-3.5pt}{\xy
(4,-7.5);(-4,-2.5)**\crv{(0,-5)&(-4,-2.9)& (-4,-2.75)}
?>(.2)*{\color{white}\bullet},
,(-4,-7.5);(4,-2.5)**\crv{(0,-5)&(4,-2.9)& (4,-2.75)}
,(4,-2.5);(-4,2.5)**\crv{(4,-2.25)&(4,-2.1)&(0,0)&(-4,2.1)& (-4,2.25)}
?>(.7)*\dir{>}?>(.5)*{\color{white}\bullet},
(-4,-2.5);(4,2.5)**\crv{(-4,-2.25)&(-4,-2.1)&(0,0)&(4,2.1) &(4,2.25)}
?>(.7)*\dir{>}
,(4,7);(-4,12)**\crv{(4,7.25)&(4,7.35)&(0,9.5)}
?>(.7)*\dir{>}?>(.79)*{\color{white}\bullet},
,(-4,7);(4,12)**\crv{(-4,7.25)&(-4,7.35)&(0,9.5)}
?>(.7)*\dir{>}
,(-4,5.5)*{\vdots}
,(4,5.5)*{\vdots}
,(-4,12);(-4,-7.5)**\crv{(-5,12.4)&(-6,12.8)&(-9,1.25)&(-6,-8.3)&(-5,-7.9)}
,(4,12);(4,-7.5)**\crv{(5,12.4)&(6,12.8)&(9,1.25)&(6,-8.3)&(5,-7.9)}
\endxy}}\,\,\right\}\text{\small{${#1}$ \rm crossings}}}

\newcommand{\twistedunknot}{{\raisebox{7pt}{
\xy
(4,-7.5);(-4,2.5)**\crv{(3.5,-7)&(2,-4)&(0,0)}
?>(.6)*\dir{>}?>(.76)*{\color{white}\bullet},
(-4,-7.5);(4,2.5)**\crv{(-3.5,-7)&(-2,-4)&(0,0)}
?>(.6)*\dir{>},(-4,2.5);(-4,-7.5)**\crv{(-5,2.9)&(-6,3.3)&(-9,-2.5)&(-6,-8.3)&(-5,-7.9)}
,(4,2.5);(4,-7.5)**\crv{(5,2.9)&(6,3.3)&(9,-2.5)&(6,-8.3)&(5,-7.9)}
\endxy}}}

\newcommand{\hopf}{{\raisebox{7pt}{\xy
(4,-7.5);(-4,-2.5)**\crv{(0,-5)&(-4,-2.9)& (-4,-2.75)}
?>(.2)*{\color{white}\bullet},
,(-4,-7.5);(4,-2.5)**\crv{(0,-5)&(4,-2.9)& (4,-2.75)}
,(4,-2.5);(-4,2.5)**\crv{(4,-2.25)&(4,-2.1)&(0,0)}
?>(.7)*\dir{>}?>(.79)*{\color{white}\bullet},
(-4,-2.5);(4,2.5)**\crv{(-4,-2.25)&(-4,-2.1)&(0,0)}
?>(.7)*\dir{>}
,(-4,2.5);(-4,-7.5)**\crv{(-5,2.9)&(-6,3.3)&(-9,-2.5)&(-6,-8.3)&(-5,-7.9)}
,(4,2.5);(4,-7.5)**\crv{(5,2.9)&(6,3.3)&(9,-2.5)&(6,-8.3)&(5,-7.9)}
\endxy}}}

\newcommand{\trefoil}{{\raisebox{1pt}{\xy
(4,-7.5);(-4,-2.5)**\crv{(0,-5)&(-4,-2.9)& (-4,-2.75)}
?>(.2)*{\color{white}\bullet},
,(-4,-7.5);(4,-2.5)**\crv{(0,-5)&(4,-2.9)& (4,-2.75)}
,(4,-2.5);(-4,2.5)**\crv{(4,-2.25)&(4,-2.1)&(0,0)&(-4,2.1)& (-4,2.25)}
?>(.7)*\dir{>}?>(.5)*{\color{white}\bullet},
(-4,-2.5);(4,2.5)**\crv{(-4,-2.25)&(-4,-2.1)&(0,0)&(4,2.1) &(4,2.25)}
?>(.7)*\dir{>}
,(4,2.5);(-4,7.5)**\crv{(4,2.75)&(4,2.85)&(0,5)}
?>(.79)*{\color{white}\bullet},
,(-4,2.5);(4,7.5)**\crv{(-4,2.75)&(-4,2.85)&(0,5)}
,(-4,7.5);(-4,-7.5)**\crv{(-5,7.9)&(-6,8.3)&(-9,0)&(-6,-8.3)&(-5,-7.9)}
,(4,7.5);(4,-7.5)**\crv{(5,7.9)&(6,8.3)&(9,0)&(6,-8.3)&(5,-7.9)}
\endxy}}}

\def\tresse{
{\xy
(-5,-10);(-5,10)**\crv{(-1,-6)&(0,-5)&(0,-2)&(0,2)&(0,5)&(-1,6)}
?>(.55)*\dir{>}?>(.1)*\dir{>}?>(.95)*\dir{>}
,(5,-10);(5,10)**\crv{(1,-6)&(0,-5)&(0,-2)&(0,2)&(0,5)&(1,6)}
?>(.1)*\dir{>}?>(.95)*\dir{>}
,(-5,10);(-5,-10)**\crv{(-6,10.5)&(-7,10.8)&(-10,0)&(-7,-10.8)&(-6,-10.5)}
,(5,10);(5,-10)**\crv{(6,10.5)&(7,10.8)&(10,0)&(7,-10.8)&(6,-10.5)}
,(-8,0)*{\scriptstyle{1}}
,(8,0)*{\scriptstyle{1}}
,(2,0)*{\scriptstyle{2}}
\endxy}\,
}

\def\twocircles{
{\xy
(0,-10);(0,-10)**\crv{(5,-10)&(5,10)&(-5,10)&(-5,-10)}
?>(.3)*\dir{>}
,(11,-10);(11,-10)**\crv{(16,-10)&(16,10)&(6,10)&(6,-10)}
?>(0.7)*\dir{<}
,(7.2,-9)*{\scriptstyle{1}}
,(4,-9)*{\scriptstyle{1}}
\endxy}}

\def\zorro{
{\xy
(-3,0);(-3,10)**\crv{(-1,2)&(0,3)&(0,4)&(0,6)&(0,7)&(-1,8)}
?>(.55)*\dir{>}
?>(.95)*\dir{>}
,(3,0);(3,10)**\crv{(1,2)&(0,3)&(0,4)&(0,6)&(0,7)&(1,8)}
?>(.95)*\dir{>}
,(3,-10);(3,0)**\crv{(5,-8)&(6,-7)&(6,-6)&(6,-4)&(6,-3)&(5,-2)}
?>(.55)*\dir{>}?>(.1)*\dir{>}?>(.99)*\dir{>}
,(9,-10);(9,0)**\crv{(7,-8)&(6,-7)&(6,-6)&(6,-4)&(6,-3)&(7,-2)}
?>(.1)*\dir{>}
,(9,0);(10,10)**\crv{(10,1)&(10,3)&(10,4)&(10,6)&(10,7)&(10,8)}
?>(.5)*\dir{>}
,(-3,0);(-4,-10)**\crv{(-4,-1)&(-4,-3)&(-4,-4)&(-4,-6)&(-4,-7)&(-4,-8)}
?>(.5)*\dir{<}
,(-5,-10)*{\scriptstyle{1}}
,(10,-10)*{\scriptstyle{1}}
,(2,-10)*{\scriptstyle{1}}
,(4,10)*{\scriptstyle{1}}
,(-4,10)*{\scriptstyle{1}}
,(11,10)*{\scriptstyle{1}}
,(1,0)*{\scriptstyle{1}}
,(4.5,-5)*{\scriptstyle{2}}
,(1.5,5)*{\scriptstyle{2}}
\endxy}\,
}

\def\zorrob{
{\xy
(3,0);(3,10)**\crv{(1,2)&(0,3)&(0,4)&(0,6)&(0,7)&(1,8)}
?>(.55)*\dir{>}
?>(.95)*\dir{>}
,(-3,0);(-3,10)**\crv{(-1,2)&(0,3)&(0,4)&(0,6)&(0,7)&(-1,8)}
?>(.95)*\dir{>}
,(-3,-10);(-3,0)**\crv{(-5,-8)&(-6,-7)&(-6,-6)&(-6,-4)&(-6,-3)&(-5,-2)}
?>(.55)*\dir{>}?>(.1)*\dir{>}?>(.99)*\dir{>}
,(-9,-10);(-9,0)**\crv{(-7,-8)&(-6,-7)&(-6,-6)&(-6,-4)&(-6,-3)&(-7,-2)}
?>(.1)*\dir{>}
,(-9,0);(-10,10)**\crv{(-10,1)&(-10,3)&(-10,4)&(-10,6)&(-10,7)&(-10,8)}
?>(.5)*\dir{>}
,(3,0);(4,-10)**\crv{(4,-1)&(4,-3)&(4,-4)&(4,-6)&(4,-7)&(4,-8)}
?>(.5)*\dir{<}
,(5,-10)*{\scriptstyle{1}}
,(-10,-10)*{\scriptstyle{1}}
,(-2,-10)*{\scriptstyle{1}}
,(-4,10)*{\scriptstyle{1}}
,(4,10)*{\scriptstyle{1}}
,(-11,10)*{\scriptstyle{1}}
,(-1,0)*{\scriptstyle{1}}
,(-4.5,-5)*{\scriptstyle{2}}
,(-1.5,5)*{\scriptstyle{2}}
\endxy}\,
}

\def\coso{
{\xy
(0,-10);(0,10)**\crv{(0,-7)&(0,-6)&(0,-5)&(-4,0)&(0,5)&(0,6)&(0,7)}
?>(.1)*\dir{>}?>(.99)*\dir{>}?>(.65)*\dir{>}
,(0,-10);(0,10)**\crv{(0,-7)&(0,-6)&(0,-5)&(4,0)&(0,5)&(0,6)&(0,7)}
,(0,10);(-5,14)**\crv{(0,10)&(0,11)&(-1,12)}?>(.95)*\dir{>}
,(0,10);(5,14)**\crv{(0,10)&(0,11)&(1,12)}?>(.95)*\dir{>}
,(0,-10);(-5,-14)**\crv{(0,-10)&(0,-11)&(-1,-12)}?>(.85)*\dir{<}
,(0,-10);(5,-14)**\crv{(0,-10)&(0,-11)&(1,-12)}?>(.85)*\dir{<}
,(10,-14);(2,-2.2)**\crv{(10,-7)&(10,-6)&(4,-4)&(1,-4)}
?>(.1)*\dir{>}?>(.65)*\dir{>}
,(10,14);(2,2.2)**\crv{(10,7)&(10,6)&(4,4)&(1,4)}
?>(.1)*\dir{<}?>(.65)*\dir{<}
,(-7,-15)*{\scriptstyle{1}}
,(7,-15)*{\scriptstyle{1}}
,(-7,15)*{\scriptstyle{1}}
,(7,15)*{\scriptstyle{1}}
,(-4,0)*{\scriptstyle{1}}
,(4,0)*{\scriptstyle{2}}
,(1.5,-9)*{\scriptstyle{2}}
,(1.5,9)*{\scriptstyle{2}}
,(11,-15)*{\scriptstyle{1}}
,(11,15)*{\scriptstyle{1}}
,(1.5,-5)*{\scriptstyle{1}}
,(1.5,5)*{\scriptstyle{1}}
\endxy}
}

\def\cosob{
{\xy
(0,-10);(0,10)**\crv{(0,-7)&(0,-6)&(0,-5)&(4,0)&(0,5)&(0,6)&(0,7)}
?>(.1)*\dir{>}?>(.99)*\dir{>}?>(.65)*\dir{>}
,(0,-10);(0,10)**\crv{(0,-7)&(0,-6)&(0,-5)&(-4,0)&(0,5)&(0,6)&(0,7)}
,(0,10);(-5,14)**\crv{(0,10)&(0,11)&(-1,12)}?>(.95)*\dir{>}
,(0,10);(5,14)**\crv{(0,10)&(0,11)&(1,12)}?>(.95)*\dir{>}
,(0,-10);(5,-14)**\crv{(0,-10)&(0,-11)&(1,-12)}?>(.85)*\dir{<}
,(0,-10);(-5,-14)**\crv{(0,-10)&(0,-11)&(-1,-12)}?>(.85)*\dir{<}
,(-10,-14);(-2,-2.2)**\crv{(-10,-7)&(-10,-6)&(-4,-4)&(-1,-4)}
?>(.1)*\dir{>}?>(.65)*\dir{>}
,(-10,14);(-2,2.2)**\crv{(-10,7)&(-10,6)&(-4,4)&(-1,4)}
?>(.1)*\dir{<}?>(.65)*\dir{<}
,(7,-15)*{\scriptstyle{1}}
,(-7,-15)*{\scriptstyle{1}}
,(7,15)*{\scriptstyle{1}}
,(-7,15)*{\scriptstyle{1}}
,(4,0)*{\scriptstyle{1}}
,(-4,0)*{\scriptstyle{2}}
,(-1.5,-9)*{\scriptstyle{2}}
,(-1.5,9)*{\scriptstyle{2}}
,(-11,-15)*{\scriptstyle{1}}
,(-11,15)*{\scriptstyle{1}}
,(-1.5,-5)*{\scriptstyle{1}}
,(-1.5,5)*{\scriptstyle{1}}
\endxy}
}

\begin{document}
\thispagestyle{empty}
\centerline{Universit\"at Wien}
\vskip 1cm
  \centerline{DISSERTATION / DOCTORAL THESIS}
\vskip 2 cm
\centerline{Titel der Dissertation /Title of the Doctoral Thesis}
\vskip 1 cm
\centerline{\bf\Large ``A generators and relations derivation of}
\centerline{\bf\Large Khovanov homology of
2-strand braid links''}
\vskip 1 cm
\centerline{verfasst von / submitted by} 
\vskip 1 cm
\centerline{Omid Hurson, BSc MSc}
\vskip 1 cm
\centerline{angestrebter akademischer Grad / in partial fulfilment of the requirements}
\centerline{ for the degree of
Doktor der Naturwissenschaften (Dr. rer. nat.)}
\vskip 2cm
Wien, 2018 / Vienna 2018
\vskip .5 cm
Studienkennzahl lt. Studienblatt /

degree programme code as it appears on the student record sheet:
\hskip 2.7cm A 796 605 405
\vskip .5 cm
Dissertationsgebiet lt. Studienblatt /

field of study as it appears on the student record sheet: \hskip 4.7cm Mathematik
\vskip .5 cm
Betreut von /

Supervisor:\hskip 2.3cm
Univ.-Prof. Domenico Fiorenza, PhD; 
Univ.-Prof. Nils Carqueville, PhD
\vfill
\eject

\thispagestyle{empty}

\noindent {\large\bf{Abstract}}
\vskip .5 cm
{\vbox{\hsize 12 cm
In the first part of the Thesis, we reformulate the Murakami-Ohtsuki-Yamada 
state-sum description of the level $n$ Jones
polynomial of an oriented link in terms of a suitable braided
monoidal category whose morphisms are $\mathbb{Q}[q,q^{-1}]$-linear combinations of
oriented trivalent planar graphs, and give a corresponding
description for the HOMFLY-PT polynomial. 
\par 
In the second part, we
extend this construction and express the Khovanov-Rozansky
homology of an oriented link in terms of a combinatorially defined
category whose morphisms are equivalence classes of formal
complexes of (formal direct sums of shifted) oriented 
plane graphs. By working combinatorially, one avoids many of the
computational difficulties involved in the matrix factorization
computations of the original Khovanov-Rozansky formulation: one
systematically uses combinatorial relations satisfied by these
matrix factorizations to simplify the computation at a level that is
easily handled. By using this technique, we are able to provide a
computation of the level $n$ Khovanov-Rozansky invariant of the 2-strand braid link with $k$ crossings, for arbitrary $n$ and $k$, confirming
and extending previous results and conjectural predictions by
Anokhina-Morozov, Carqueville-Murfet, Dolotin-Morozov, Gukov-Iqbal-Kozcaz-Vafa, and Rasmussen.}}

\eject

\tableofcontents

\chapter{Introduction}
A classical topic in low dimensional topology is the study of knots, i.e., of  topological embeddings
\[
k\colon S^1\hookrightarrow \mathbb{R}^3
\]
and, more generally, of links, i.e., of topological embeddings 
\[
l\colon S^1\sqcup S^1\sqcup\cdots \sqcup S^1\hookrightarrow \mathbb{R}^3,
\]
where $\sqcup$ denotes the disjoint union. One considers two knots or two links to be equivalent if one can continuously deform one in the other. More formally, two knots $k_0,k_1\colon S^1\hookrightarrow \mathbb{R}^3$ will be considered equivalent if there is an isotopy joining them, i.e., if there is a continuous family of topological embeddings $k_t\colon S^1\hookrightarrow \mathbb{R}^3$ with $t$ ranging from $0$ to $1$ which at $t=0$ gives $k_0$ and at $t=1$ gives $k_1$ (and the same for links). So what one is ultimately interested in is the set of knots (or links) up to equivalence. Moreover, as $S^1$ is orientable, one can take care or discard this orientation, and so consider equivalence classes of oriented or unoriented knots and links. In this thesis we will be concerned with the first class, i.e., we will only consider oriented knots and links.
\par
A very convenient way is to visualize knots and links by ``drawing them in the plane'', i.e., by considering the images of regular projections: linear projections $\pi\colon \mathbb{R}^3\to \mathbb{R}^2$ such that $\pi\circ k\colon S^1\to \mathbb{R}^2$ is an immersion with only double crossings (and similarly for links). These images are called knot and link diagrams. Here are a few examples:
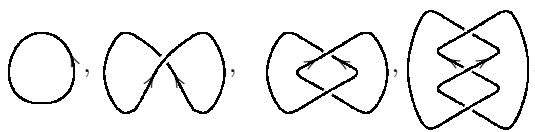
\[
\unknot\,,\twistedunknot\, ,\quad \hopf\,,\, \trefoil\, ,
\]
where the arrows denote the orientations on the knot or on the link components (in the example above all the diagrams represent knots, except the third one, which represents a 2-components link).
\par
This leads to the natural question: if two knots or links are isotopy equivalent, what can we say about their knot diagrams? can we have different diagrams for ``the same'' knot? The answer to the second question is easily seen to be: Yes! For instance, both 
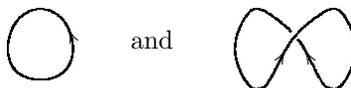
\[
\unknot
\qquad \text{and}\qquad
\twistedunknot
\]
are diagrams for the ``unknot''. So the real question is: how are two diagrams for the same knot related? The answer comes from Reidemeister's theorem: let $D_1$ and $D_2$ be plane diagrams for links $l_1$ and $l_2$. Then $l_1$ is in the same isotopy class as $l_2$ if and only if we can get from $D_1$ to $D_2$ with a sequence of
\begin{itemize}
\item plane isotopies
\item combinatorial moves called Reidemeister moves
\end{itemize}
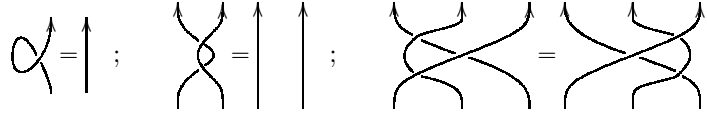
\[
\xy
(2,5);(0,0.1)**\crv{(2,1)&(-2,-4)&(-4,0)&(-2,4)}?>(0)*\dir{<}
,(2,-5);(0.65,-1)**\crv{(2,-4)}
\endxy=
\xy
(0,-5);(0,5)**\dir{-}?>(1)*\dir{>}
\endxy\quad; \qquad 
 \xy
(3,-7);(3,7)**\crv{(3,-5)&(3,-4)&(-2,0)&(3,4)&(3,5)}?>(.99)*\dir{>}
,(-3,-7);(-0.3,-2)**\crv{(-3,-5)&(-3,-4)}
,(-3,7);(-0.3,2)**\crv{(-3,5)&(-3,4)}?>(0)*\dir{<}
,(.6,-1.3);(0.6,1.4)**\crv{(3,0)}
\endxy =
\xy
(-3,-7);(-3,7)**\dir{-}?>(1)*\dir{>}
,(3,-7);(3,7)**\dir{-}?>(1)*\dir{>}
\endxy
\quad; \qquad
\xy
(-9,-7);(9,7)**\crv{(-9,-5)&(-9,-3)&(9,3)&(9,5)}?>(.99)*\dir{>}
,(9,-7);(1,-.3)**\crv{(9,-5)&(9,-3)}
,(-1,.3);(-5.3,2)**\crv{(-3,1)}
,(-6.7,2.8);(-9,7)**\crv{(-9,4)&(-9,5)}?>(.99)*\dir{>}
,(0,-7);(-5.3,-2.8)**\crv{(0,-5)&(0,-4)}
,(0,7);(-6.1,2.4)**\crv{(0,6)&(0,5)&(0,4)&(-6,2.8)}?>(0)*\dir{<}
,(-6.5,-2.2);(-6.1,2.4)**\crv{(-9,0)}
\endxy
=
\xy
(-9,-7);(9,7)**\crv{(-9,-5)&(-9,-3)&(9,3)&(9,5)}?>(.99)*\dir{>}
,(5.6,-2);(1,-.3)**\crv{(3,-1)}
,(9,-7);(6.7,-2.8)**\crv{(9,-5)&(9,-4)}
,(-1,.3);(-9,7)**\crv{(-9,3)&(-9,5)}?>(.99)*\dir{>}
,(0,-7);(6,-2.6)**\crv{(0,-5)&(0,-4)&(6,-3)}
,(0,7);(5.3,2.8)**\crv{(0,6)&(0,5)&(0,4)}?>(0)*\dir{<}
,(6,-2.6);(6.5,2.2)**\crv{(9,0)}
\endxy
 \]
 \vskip .5 cm
 \noindent
By Reidemeister's theorem, the  set of equivalence classes we are interested in, i.e.,  
 \[
 \mathcal{L}=\{\text{oriented links}\}/\{\text{isotopies}\}
 \]
 is the same as
 \[
 \{\text{oriented link diagrams in  $\mathbb{R}^2$} \}/\{\text{plane isotopies+Reidemeister moves}\}
 \]
 \vskip .7 cm
 
\noindent It would be nice if we could label any link in such a way that isotopic links get the same label (and so links with different labels are surely not isotopic).  This leads to the definition of link invariant: a function $f\colon \mathcal{L}\to S$, where $S$ is some set of labels.  In interesting examples, $S$ will be more structured, e.g., it will have a group or a ring structure. A remarkable fact is that, by Reidemeister's theorem recalled above, a link invariant is the same thing as a function
 \[
 f\colon  \{\text{oriented link diagrams in $\mathbb{R}^2$}\}/\{\text{plane isotopies}\} \to S
 \]
 which is invariant with respect to the Reidemeister moves. A celebrated example is the level $n$ Jones polynomial
 \[
 Jones_n\colon \{\text{oriented link diagrams in $\mathbb{R}^2$}\}/\{\text{plane isotopies}\} \to \mathbb{Q}[q,q^{-1}]
 \]
 defined by the rules:
 \begin{enumerate}
 \item $Jones_n(\unknot)=[n]_q:=\frac{q^n-q^{-n}}{q-q^{-1}}=q^{1-n}+q^{3-n}+\cdots+q^{n-3}+q^{n-1}$
 \item $Jones_n(L_1 \sqcup L_2)=Jones_n(L_1)\, Jones_n(L_2)$
 \item $Jones_n$ satisfies the skein relation
 \[
 q^{-n}Jones_n(\overcrossing)-q^nJones_n(\undercrossing)=(q^{-1}-q)Jones_n(\noncrossing)
 \]
 \end{enumerate}
\vskip .7 cm
\noindent 
The Jones polynomial is actually an instance of a general construction of links invariant out of rigid braided monoidal categories: Reshetikhin and Turaev \cite{RT1,RT2} show that choosing an object $X$ in a rigid braided monoidal category $\mathcal{C}$ uniquely defines a braided monoidal functor
\[
Z_X\colon \{\text{orient. tangle diagrams in  $\mathbb{R}^2$}\}/\{\text{isotopies + Reidemeister moves}\} \to \mathcal{C}
\]
with
\[
Z_X(\uparrow)=\{\mathrm{id}_X\colon X\to X\}.
\]
On the left hand side, equivalence classes of oriented tangle diagrams are a rigid braided monoidal category with the tensor product given by juxtaposition of tangle diagrams, composition given by gluing a tangle diagram on the top of on another if their top and bottom endpoints match, while the identity, the evaluation, coevaluation and braidings for the object $\uparrow$ are given by the tangle diagrams
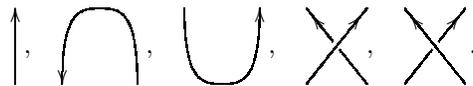
\[
{\xy
(0,-5);(0,5)**\dir{-}?>(1)*\dir{>}
\endxy}\,\,
,\quad {\xy
(-5,-5);(5,-5)**\crv{(-5,5)&(0,5)&(5,5)}
?>(0)*\dir{<}
\endxy}\,\,, \quad
{\xy
(-5,5);(5,5)**\crv{(-5,-5)&(0,-5)&(5,-5)}
?>(1)*\dir{>}
\endxy}
\,\,, \quad \overcrossing,\quad \undercrossing\,.
\] 
On morphisms, $Z_X$ acts by interpreting a tangle diagram as
a string diagram in $\mathcal{C}$, i.e., by the rules
\[
Z_X\left(\, \xy
(-5,-5);(5,-5)**\crv{(-5,5)&(0,5)&(5,5)}
?>(0)*\dir{<},
\endxy\, 
\right) = ev_X\colon X^\vee\otimes X \to \mathbf{1}_{\mathcal{C}},
\]
where $X^\vee$ is the dual of $X$ and $\mathbf{1}_\mathcal{C}$ is the unit object of $\mathcal{C}$, or
\[
Z_X\left(\, \overcrossing\, 
\right) = \sigma^+_{X,X}\colon X\otimes X \to X\otimes X,
\]
etc. The level $n$ Jones polynomial is originally recovered as a particular case of this, by choosing $\mathcal{C}$ to be a certain braided monoidal subcategory of the category $\mathrm{Rep}(U_q(\mathfrak{sl}_n))$  of (finite dimensional) representations of the quantum group $U_q(\mathfrak{sl}_n)$, and as $X$ the fundamental representation.

\vskip .7 cm

This is a highly nontrivial algebraic structure, but inspired by the work of Murakami,  Ohtsuki and Yamada \cite{moy},  one can also use a very simple combinatorial category $\mathbf{MOY}$ based on trivalent planar graphs instead. Namely, Murakami,  Ohtsuki and Yamada are able to express the level $n$ Jones polynomial of a link as a sum over trivalent graphs (with coefficients in $\mathbb{Q}[q,q^{-1}]$):
\[
Jones_n(L)=\sum_{\Gamma\in \mathcal{G}_L} a_\Gamma \langle \Gamma\rangle,
\]
where $a_\Gamma\in \mathbb{Q}[q,q^{-1}]$, and $\mathcal{G}_L$ is a collection of planar trivalent graphs associated to the oriented link $L$ by the rule
\[
\Bigl\langle\overcrossing\Bigr\rangle= 
q^{n-1}\left(
\Bigl\langle\noncrossing\,\Bigr\rangle
- q\,\Bigl\langle
\scalebox{.6}{{\xy
(-5,-10);(-5,10)**\crv{(-1,-6)&(0,-5)&(0,-2)&(0,2)&(0,5)&(-1,6)}
?>(.55)*\dir{>}?>(.1)*\dir{>}?>(.95)*\dir{>}
,(5,-10);(5,10)**\crv{(1,-6)&(0,-5)&(0,-2)&(0,2)&(0,5)&(1,6)}
?>(.1)*\dir{>}?>(.95)*\dir{>}
\endxy}}\,\,\Bigr\rangle
\right)
\]
\[
\Bigl\langle\undercrossing\Bigr\rangle= 
q^{1-n}\left(
\Bigl\langle\noncrossing\,\Bigr\rangle
- q^{-1}\,\Bigl\langle
\scalebox{.6}{{\xy
(-5,-10);(-5,10)**\crv{(-1,-6)&(0,-5)&(0,-2)&(0,2)&(0,5)&(-1,6)}
?>(.55)*\dir{>}?>(.1)*\dir{>}?>(.95)*\dir{>}
,(5,-10);(5,10)**\crv{(1,-6)&(0,-5)&(0,-2)&(0,2)&(0,5)&(1,6)}
?>(.1)*\dir{>}?>(.95)*\dir{>}
\endxy}}\,\,\Bigr\rangle
\right)
\]
Notice how this implies the skein relation for the Jones polynomial. The bracket $\langle -\rangle$ appearing in the Murakami-Ohtsuki-Yamada formula is a $\mathbb{Q}[q,q^{-1}]
$-linear function
\[
\langle -\rangle\colon \mathbb{Q}[q,q^{-1}]\{\text{plane trivalent diagrams}\} \to \mathbb{Q}[q,q^{-1}]
\]
(where on the left we mean the $\mathbb{Q}[q,q^{-1}]$-module generated by plane trivalent diagrams)
satisfying certain combinatorial relations such as
\[
{\xy
(0,-10);(0,10)**\crv{(0,-7)&(0,-6)&(0,-5)&(-4,0)&(0,5)&(0,6)&(0,7)}
?>(.2)*\dir{>}?>(.99)*\dir{>}?>(.65)*\dir{>}
,(0,-10);(0,10)**\crv{(0,-7)&(0,-6)&(0,-5)&(4,0)&(0,5)&(0,6)&(0,7)}
?>(.65)*\dir{>}
\endxy}=
\left[2\right]_q{\xy
(0,-10);(0,10)**\dir{-}
?>(.95)*\dir{>}
\endxy}\quad;
\qquad\qquad\qquad
{\xy
(0,-10);(0,10)**\crv{(0,-7)&(0,-6)&(0,-5)&(0,0)&(0,5)&(0,6)&(0,7)}
?>(.2)*\dir{>}?>(.99)*\dir{>}?>(.55)*\dir{>}
,(0,-3);(0,3)**\crv{(0,-4)&(0,-5)&(4,-6)&(6,0)&(4,6)&(0,5)&(0,4)}
?>(.55)*\dir{<}
\endxy}=
\left[ n-1\right]_q{\xy
(0,-10);(0,10)**\dir{-}
?>(.95)*\dir{>}
\endxy}
\quad;
\]
now called MOY relations. The MOY bracket $\langle -\rangle$ is therefore a $\mathbb{Q}[q,q^{-1}]
$-linear function
\[
\langle -\rangle\colon \{\text{oriented planar  trivalent diagrams}\}/\{\text{MOY relations}\} \to \mathbb{Q}[q,q^{-1}]
\]
In Chapter \ref{chapter:moy} of this Thesis we interpret this construction in terms of braided tensor categories by introducing a combinatorially defined rigid braided monoidal category $\mathbf{MOY}$ with 
\vskip .3 cm
Objects($\mathbf{MOY}$)=finite sequences of upgoing and downgoing arrows labeled by the integer 1
\vskip .3cm
Morphisms($\mathbf{MOY}$)=$\mathbb{Q}[q,q^{-1}]$-linear combinations of oriented trivalent graphs connecting the source sequence of arrows to the target sequence of arrows, modulo the Murakami-Ohtsuki-Yamada relations. Edges of this graphs are labelled so that at each vertex the sum of the incoming labels equals the sum of the outgoing labels.
\vskip .3 cm
\noindent
For instance, (the equivalence class of)
\[
(2+q^4)\,
\esse\,\,+
(q^2-7q^{-3})
\essecircesse
\]
is a morphism from $\uparrow^1\,\uparrow^1$ to $\uparrow^1\,\uparrow^1$ in $\mathbf{MOY}$.
\vskip 0.5 cm
\noindent Composition in $\mathbf{MOY}$ is given by gluing one trivalent graph on top of the other (and then extending this by $\mathbb{Q}[q,q^{-1}]$-linearity); tensor product is given by juxtaposition (again, extended by $\mathbb{Q}[q,q^{-1}]$-linearity). The category $\mathbf{MOY}$ is rigid with the evaluations and coevaluations given by the diagrams
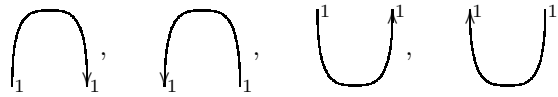
\[
\xy
(-5,-5);(5,-5)**\crv{(-5,5)&(0,5)&(5,5)}
?>(1)*\dir{>}
,(-4,-5)*{\scriptstyle{1}},(6,-5)*{\scriptstyle{1}}
\endxy,
\qquad
\xy
(-5,-5);(5,-5)**\crv{(-5,5)&(0,5)&(5,5)}
?>(0)*\dir{<}
,(-4,-5)*{\scriptstyle{1}},(6,-5)*{\scriptstyle{1}}
\endxy,
\qquad
\xy
(-5,5);(5,5)**\crv{(-5,-5)&(0,-5)&(5,-5)}
?>(1)*\dir{>}
,(-4,5)*{\scriptstyle{1}},(6,5)*{\scriptstyle{1}}
\endxy,
\qquad
\xy
(-5,5);(5,5)**\crv{(-5,-5)&(0,-5)&(5,-5)}
?>(0)*\dir{<}
,(-4,5)*{\scriptstyle{1}},(6,5)*{\scriptstyle{1}}
\endxy
\]
and it is braided with the braidings
\[
\sigma^+=q^{n-1}\left(\,\, \idtwo-q\esse\,\,\right),
\]
\[
\sigma^-= q^{1-n}\left(\,\, \idtwo-q^{-1}\esse\,\,\right)
\]
\vskip .7 cm
\noindent
In terms of the rigid braided monoidal category $\mathbf{MOY}$, the definition of  the MOY bracket $\langle -\rangle$ can be elegantly expressed through Reshetikhin-Turaev theorem: it is the unique braided monoidal functor
\[
\langle -\rangle\colon \mathbf{TD} \to \mathbf{MOY}
\]
with
\[
\left\langle\,\, {\xy
(0,-5);(0,5)**\dir{-}?>(.5)*\dir{>}
\endxy}\,\,\right\rangle =\,\,{\xy
(0,-5);(0,5)**\dir{-}?>(.5)*\dir{>}
\endxy}
\]
where
\[
\mathbf{TD}=\{\text{oriented tangle diagrams in  $\mathbb{R}^2$}\}/\{\text{isotopies + Reidemeister moves}\}.
\]
In terms of the MOY rules one easily computes the level $n$ Jones polynomial. For instance,
\begin{align*}
Jones_n\left(\, \hopf\,\right)&=\scalebox{.8}{$q^{2(n-1)}\left\langle
 \xy
(0,-5);(0,-5)**\crv{(5,-5)&(5,5)&(-5,5)&(-5,-5)}
?>(.3)*\dir{>}
\endxy
\,
\xy
(0,-5);(0,-5)**\crv{(5,-5)&(5,5)&(-5,5)&(-5,-5)}
?>(0.7)*\dir{<}
\endxy
\,
- \,2q\,
{\xy
(-5,-10);(-5,10)**\crv{(-7,-12)&(-9,0)&(-7,12)}
,(5,-10);(5,10)**\crv{(7,-12)&(9,0)&(7,12)}
,(-5,-10);(-5,10)**\crv{(-1,-6)&(0,-5)&(0,-2)&(0,2)&(0,5)&(-1,6)}
?>(.55)*\dir{>}?>(.1)*\dir{>}?>(.95)*\dir{>}
,(5,-10);(5,10)**\crv{(1,-6)&(0,-5)&(0,-2)&(0,2)&(0,5)&(1,6)}
?>(.1)*\dir{>}?>(.95)*\dir{>}
\endxy}\,\,+q^2\,
{\xy
(-5,-14);(-5,14)**\crv{(-7,-15)&(-9,0)&(-7,15)}
,(5,-14);(5,14)**\crv{(7,-15)&(9,0)&(7,15)}
,(0,-10);(0,10)**\crv{(0,-7)&(0,-6)&(0,-5)&(-4,0)&(0,5)&(0,6)&(0,7)}
?>(.2)*\dir{>}?>(.99)*\dir{>}?>(.65)*\dir{>}
,(0,-10);(0,10)**\crv{(0,-7)&(0,-6)&(0,-5)&(4,0)&(0,5)&(0,6)&(0,7)}
?>(.65)*\dir{>}
,(0,10);(-5,14)**\crv{(0,10)&(0,11)&(-1,12)}?>(.95)*\dir{>}
,(0,10);(5,14)**\crv{(0,10)&(0,11)&(1,12)}?>(.95)*\dir{>}
,(0,-10);(-5,-14)**\crv{(0,-10)&(0,-11)&(-1,-12)}
,(0,-10);(5,-14)**\crv{(0,-10)&(0,-11)&(1,-12)}
\endxy}\,\,\,
\right\rangle
$}\\
&=q^{2(n-1)}\left(q^{1-n}+q^3[n-1]_q\right)[n]_q
\end{align*}
Moreover, we show that a generalized version $\mathbf{MOY}_{\alpha,\zeta}$ of the category $\mathbf{MOY}$ can be constructed, leading to the HOMFLY-PT polynomial with parameters $\alpha$ and $z=\zeta^{-1}-\zeta$. The main result from Chapter \ref{chapter:moy} is then that the functor 
\[
Z_{\alpha,\zeta}\colon \mathbf{TD}\to \mathbf{MOY}_{\alpha,\zeta}
\]
identifies $\mathbf{MOY}_{\alpha,\zeta}$ with the universal rigid braided untwisted category $\mathcal{C}$ generated by one object whose braidings, evaluation and coevaluation morphism satisfy the skein relation
\[
\alpha\overcrossing-\alpha^{-1}\undercrossing=z\noncrossing
\]
as well as the normalization condition
\[
\unknot=\frac{\alpha-\alpha^{-1}}{z}\mathbf{1}_{\mathcal{C}}
\]
\vskip .7 cm
\noindent
A nontrivial refinement of the Jones polynomial has been obtained by Khovanov and Rozansky \cite{Khovanov-Rozansky}. Their link invariant is a polynomial in \emph{two} variables $q,t$ and their inverses, which specializes to the Jones polynomial when it is evaluated at $t=-1$. For instance
\[
KR_n\left(\,\hopf\,\right)=q^{2(n-1)}\left(q^{1-n}+t^2q^3[n-1]_q\right)[n]_q
\]
The Khovanov-Rozansky invariant is generally very hard to compute, as it does not satisfy a skein relation, and actually
 the original definition of the Khovanov-Rozansky is rather complicated: it involves complexes of matrix factorizations of certain Landau-Ginzburg potentials (polynomials in some set of variables with no constant or linear term, with finite dimensional Jacobian algebra). On the other hand, the construction of the combinatorial category $\mathbf{MOY}$ and the way it replaced the more apparently less trivial category $\mathrm{Rep}(U_q(\mathfrak{sl}_n))$ with a simple combinatorial one suggests the following question: is it possible to replace also this rather complicated algebraic category of Landau-Ginzburg potentials and matrix factorizations with a simple purely combinatorial category defined in terms of planar trivalent graphs?

\bigskip

The main result of this Thesis consists in giving a positive answer to this question. In Chapter \ref{chapter:KR} we define a braided monoidal category $\mathbf{KR}$ of formal complexes of (formal direct sums of shifted) MOY-type graphs, that encodes combinatorially all of the features of the category of complexes of matrix factorizations that are used by Khovanov and Rozansky in order to construct their link invariant. More in detail, in \cite{Khovanov-Rozansky}, Khovanov and Rozansky consider a monoidal category which could be called the ``homotopy category of chain complexes in the Landau-Ginzburg category'' and therefore denoted by the symbol $\mathcal{K}(\mathcal{LG})$. Objects of $\mathcal{K}(\mathcal{LG})$ are Landau-Ginzburg potentials, i.e., polynomials $W(\vec{x})$ in the variables $\vec{x}=(x_1,x_2,\dots)$ over the field $\mathbb{Q}$ (or an extension of it), such that the Jacobian ring $\mathbb{Q}[x_1,x_2,\dots]/(\partial_1W,\partial_2W,\dots)$ is finite dimensional. Morphisms between $W_1(\vec{x})$ and $W_2(\vec{y})$ are homotopy equivalence classes of complexes of matrix factorizations of $W_1(\vec{x})-W_2(\vec{y})$. The category $\mathcal{K}(\mathcal{LG})$ is rigid, and it is expected to be braided.\footnote{Unfortunately, we are not aware of any reference proving (or disproving) this fact. Our impression is that most experts are quite certain that $\mathcal{K}(\mathcal{LG})$ is indeed a braided monoidal category.}  Even without knowing for sure whether $\mathcal{K}(\mathcal{LG})$ is braided or not, Khovanov and Rozansky are nevertheless able to define a rigid monoidal functor
\[
KH_n\colon\mathbf{TD}\to \mathcal{K}(\mathcal{LG})
\]
with $KH_n(\uparrow)=x^{n+1}$, thus realizing $W(x)=x^{n+1}$ as a \emph{braided object} in $\mathcal{K}(\mathcal{LG})$. For a link $\Gamma$ the complex of matrix factorizations $KH_n(\Gamma)$ is a complex of matrix factorizations of the potential 0, and so it is simply a bicomplex of graded $\mathbb{Q}$-vector spaces. Up to homotopy, this is identified by the cohomology of its total complex. This is a finite dimensional  bigraded $\mathbb{Q}$-vector space, called the Khovanov-Rozansky cohomology of the link $\Gamma$. As the only invariant of a finite dimensional bigraded vector space is its graded dimension, the datum of the Khovanov-Rozansky cohomology of $\Gamma$ is equivalently the datum of a polynomial in $\mathbb{Z}[q,q^{-1},t,t^{-1}]$. This is precisely the Khovanov-Rozansky invariant mentioned above: $KR_n(\Gamma)=\dim_{q,t}KH_n(\Gamma)$.
\par 
As one could expect, the construction of $KH$ is in two steps: first, Khovanov and Rozansky define a rigid monoidal functor $KH_n\colon \mathbf{PTD}\to \mathcal{K}(\mathcal{LG})$ with $KH_n(\uparrow)=x^{n+1}$, where $\mathbf{PTD}$ denotes the rigid monoidal category of plane tangle diangrams modulo plane homotopies, and then they show that $KH$ is invariant with respect to the Reidemeister moves thus inducing a braided monoidal functor 
$KH_n\colon \mathbf{TD}\to \mathcal{K}(\mathcal{LG})$. The proof of this fact is quite nontrivial and  uses subtle properties of the complexes of matrix factorizations involved with the definition of $KH_n$. What we do in Chapter \ref{chapter:KR} is precisely to extract from Khovanov and Rozansky's proof all the properties of these complexes of matrix factorizations actually used, as well as additional properties demonstrated by Rasmussen in \cite{Rasmussen}, and change them into defining properties. By this, one tautologically has that the monoidal functor $KH$ factors as
 \begin{equation}\label{eq:krr1}\tag{$\ast$}
\xymatrix{
\mathbf{TD}\ar[r]^{Kh_n}\ar@/_2pc/[rr]_{KH_n} &\mathbf{KR} \ar[r]^{\mathrm{mf}} &\mathcal{K}(\mathcal{LG})
}
\end{equation}
for some monoidal functor $\mathrm{mf}\colon \mathbf{KR}\to \mathcal{K}(\mathcal{LG})$ with $\mathrm{mf}(\uparrow)=x^{n+1}$, where `mf' stands for `matrix factorizations'. Moreover $\mathrm{mf}$ is compatible with shifts and direct sums. This implies that it commutes with the tensor product by graded $\mathbb{Q}$-vector spaces (with a fixed basis).
\par
Unfortunately we are not able to prove whether $\mathrm{mf}$ is an embedding of categories. Clearly, it is injective on objects, but its faithfulness on morphisms is elusive: there could be nonequivalent formal complexes of MOY-type graphs leading to homotopy equivalent complexes of matrix factorizations. Actually, our opinion is that $\mathrm{mf}$ should not be expected to be faithful: it is likely that in $\mathcal{K}(\mathcal{LG})$ there are many more relations than those involved in the proof of the Reidemeister invariance of $KH_n$ by Khovanov and Rozansky and those evidentiated by Rasmussen.
Yet, if for a link $\Gamma$ we have
\[
Kh_n(\Gamma)=[V_\bullet]\otimes \emptyset,
\]
where $\emptyset$ is the empty diagram and  
\[
V_\bullet=\left(\cdots\to V_{i-1}\xrightarrow{0} V_i\xrightarrow{0} V_{i+1}\cdots \right),
\]
is some complex of graded $\mathbb{Q}$-vector spaces with trivial differentials,
then the commutativity of the diagram \ref{eq:krr1} implies that we have
\[
KH_n(\Gamma)=\oplus_i V_i,
\]
with the graded $\mathbb{Q}$-vector space $V_i$ in ``horizontal'' degree $i$. This precisely means that, for those links such that $Kh_n(\Gamma)=[V_\bullet]\otimes \emptyset$, the functor $Kh_n$ computes the Khovanov-Rozansky homology. This fact is extensively used in Chapter \ref{chapter:computation} to compute the level $n$ Khovanov-Rozansky polynomial for all 2-strands braid links with $k$ crossings (this is actually a knot when $k$ is odd and a 2-component link when $k$ is even). 
Namely, we obtain
\[
Kh_n\left(\,\trsplusk{2k+1}\,\right)=q^{(n-1)(2k+1)}\left(q^{1-n}[n]_q+(q^{-n}+tq^n)\frac{t^2q^4-t^{2k+2}q^{4k+4}}{1-t^2q^4}[n-1]_q \right)
\]
and
\[
Kh_n\left(\,\trsplusk{2k}\,\right)
=q^{(n-1)(2k)}\left(q^{1-n}[n]_q+(q^{-n}+tq^n)\frac{t^2q^4-t^{2k}q^{4k}}{1-t^2q^4}[n-1]_q+t^{2k}q^{4k-1}[n]_q[n-1]_q \right)
\]
 This perfectly matches (and so confirms and extends) the results obtained or predicted for this kind of knots and links by:
  
  \begin{itemize}
  \item Rasmussen for the odd $k$'s and $n\geq 4$ \cite{Rasmussen}; 
  
  \item Carqueville and Murfet for low values of $k$ and $n$ (computer assisted computation) \cite{Carqueville-Murfet};
  
  \item Gukov-Iqbal-Koz\c{c}az-Vafa for $k=2$ and all $n$ (string theory based prediction) \cite{Gukov};
  
  \item  Anokhina, Dolotin and Morozov for all $k$ and $n$ (algebraic combinatorics based prediction) \cite{Morozov1,Morozov2}.
\end{itemize}

\chapter{The tangles category}

\section{Ideals and quotients in tensor categories}
For $\mathbf{C}$ a category, we write $\mathbf{C}(a,b)$ or $\mathbf{C}_{a,b}$ for the set of morphisms from the object $a$ to the object $b$ in $\mathbf{C}$. Throughout this chapter, $R$ will be a fixed commutative ring with unit, which we will assume to be an integral domain.
\begin{definition}
We say that $\mathbf{C}$ is enriched in  $R$-modules\footnote{As $R$ is commutative, every $R$-module is naturally an $R$-$R$-bimodule.} if each hom-set $\mathbf{C}_{a,b}$ has the structure of $R$-module, and if these structures are compatible with the category structure of $\mathbf{C}$. 
\end{definition}
For instance, we require that the composition
\[
\mathbf{C}_{a,b}\times \mathbf{C}_{b,c}\to \mathbf{C}_{a,c}
\]
is $R$-bilinear, and so it is a morphism of $R$-modules
\[
\mathbf{C}_{a,b}\otimes_{R} \mathbf{C}_{b,c}\to \mathbf{C}_{a,c}.
\]
\begin{definition}
An \emph{ideal} $\mathbf{I}$ of a category $\mathbf{C}$ enriched in $R$-modules is a collection of sub-$R$-modules 
\[
\mathbf{I}_{a,b}\subseteq \mathbf{C}_{a,b},
\]
for each pair of objects $a$ and $b$ in $\mathbf{C}$ such that the composition induces morphisms
\[
\mathbf{I}_{a,b}\otimes_{R} \mathbf{C}_{b,c}\to \mathbf{I}_{a,c}
\]
and
\[
\mathbf{C}_{a,b}\otimes_{R} \mathbf{I}_{b,c}\to \mathbf{I}_{a,c}.
\]
\end{definition}
\begin{definition}
If $\mathbf{I}$ is an ideal of $\mathbf{C}$ we have a well defined category $\mathbf{C}/\mathbf{I}$, enriched in  $R$-modules, whose hom-spaces are
\[
(\mathbf{C}/\mathbf{I})_{a,b}=\mathbf{C}_{a,b}/\mathbf{I}_{a,b}.
\]
\end{definition}
\begin{remark}It is immediate to see that $\mathbf{C}$ is an ideal, and that the intersection of ideals is again an ideal. Therefore, given an arbitrary collection $\mathcal{I}$ of morphisms in $\mathbf{C}$ one can consider the ideal $\langle \mathcal{I}\rangle$ generated by $\mathcal{I}$: this will be the intersection of all the ideals of $\mathbf{C}$ containing $\mathcal{I}$.
\end{remark}

When $\mathbf{C}$ is a tensor category, we have in addition a tensor product and compatibility with the $R$-module enrichment is expressed by requiring that the tensor product of maps
\[
\otimes\colon \mathbf{C}_{a,b}\times \mathbf{C}_{c,d}\to \mathbf{C}_{a\otimes c,b\otimes d}
\]
is $A$-bilinear, thus giving a morphism of $R$-modules
\[
\otimes\colon \mathbf{C}_{a,b}\otimes_{R} \mathbf{C}_{c,d}\to \mathbf{C}_{a\otimes c,b\otimes d}.
\]
\begin{definition}
An ideal $\mathbf{I}$ of $\mathbf{C}$ will be called a \emph{tensor ideal} if the tensor product morphisms induce morphisms
\[
\otimes\colon \mathbf{I}_{a,b}\otimes_{R} \mathbf{C}_{c,d}\to \mathbf{I}_{a\otimes c,b\otimes d}.
\]
and
\[
\otimes\colon \mathbf{C}_{a,b}\otimes_{R} \mathbf{I}_{c,d}\to \mathbf{I}_{a\otimes c,b\otimes d}.
\]
\end{definition}
\begin{remark}
If $\mathbf{I}$ is a tensor ideal, then $\mathbf{C}/\mathbf{I}$ is naturally a tensor category, with tensor product induced by that of $\mathbf{C}$.
\end{remark}
\begin{remark}
Again, is immediate to see that $\mathbf{C}$ is a tensor ideal, and that the intersection of tensor ideals is again a tensor ideal. It follows that, given an arbitrary collection $\mathcal{I}$ of morphisms in $\mathbf{C}$ one can consider the tensor ideal $\langle \mathcal{I}^\otimes\rangle$ generated by $\mathcal{I}$: this will be the intersection of all the tensor ideals of $\mathbf{C}$ containing $\mathcal{I}$. Moreover, if $\mathbf{C}$ is rigid (i.e., one has a notion of duals for objects), then so is $\mathbf{C}/\mathbf{I}$, and if $\mathbf{C}$ is also ``balanced'', i.e., if each object is naturally isomorphic with its double dual in a way compatible with dualities and tensor product of objects, then so is also $\mathbf{C}/\mathbf{I}$.
\end{remark}
\vskip .5 cm

\section{The tangles category}

As an example of the general construction sketched in the previous section, we will now describe the category of $A$-linear combinations of tangles as a tensor category enriched in  $R$-modules. We start by considering the category  $\mathbf{PTD}$ of ($A$-linear combinations of) planar tangle diagrams.

\begin{definition}
The category $\mathbf{PTD}$ of ($A$-linear combinations of) planar tangle diagrams is the category enriched in $R$-modules whose
 objects are (possibly empty) finite sequences of upgoing or downgoing arrows. Morphisms between two objects $\vec{i}$ and $\vec{j}$ are given by $R$-linear combinations of plane tangle diagrams with oriented edges with incoming boundary data given by the sequence $\vec{i}$ and outgoing boundary data given by the sequence $\vec{j}$.
 \par
 The composition of morphisms is given by vertical concatenation of diagrams,
extended by $R$-bilinearity to $R$-linear combinations of diagrams.
\end{definition}

\begin{example} 
 For instance, an object in  $\mathbf{PTD}$ could be the following:
\[
\uparrow\,\, \downarrow \,\,\uparrow\,\, \uparrow
\]
\end{example}

\begin{example} 
  We draw the source object on the bottom and the target object on the top.  
For instance, if $R=\mathbb{Q}[q,q^{-1}]$,
\[
(q^{-2}+7q^3)\xy
(2,5);(0,0.1)**\crv{(2,1)&(-2,-4)&(-4,0)&(-2,4)}?>(0)*\dir{<}
,(2,-5);(0.65,-1)**\crv{(2,-4)}
,(-6,-5);(-3.7,-.1)**\crv{(-5,-1.4)}
,(-2.7,0.5);(-1,1)**\crv{(-1.7,1.2)}
,(0,.9);(0.5,.8)**\crv{(-.2,1)}
,(1.4,.6);(5,-5)**\crv{(3,0)}?>(1)*\dir{>}
\endxy
\]
is a morphism from $\uparrow\,\, \uparrow \,\,\downarrow$ to $\uparrow$. An example of composition is the following:
\[
q\,\xy
(-5,-5);(5,-5)**\crv{(-5,5)&(0,5)&(5,5)}
?>(1)*\dir{>},
\endxy
\circ \quad
(q^{-1}+q)\xy
(2,5);(0,0.1)**\crv{(2,1)&(-2,-4)&(-4,0)&(-2,4)}?>(0)*\dir{<}
,(2,-5);(0.65,-1)**\crv{(2,-4)}
,(-6,-5);(-3.7,-.1)**\crv{(-5,-1.4)}
,(-2.7,0.5);(-1,1)**\crv{(-1.7,1.2)}
,(0,.9);(0.5,.8)**\crv{(-.2,1)}
,(1.4,.6);(5,-5)**\crv{(3,0)}?>(1)*\dir{>}
,(7,5);(10,-5)**\crv{(7,0)}?>(1)*\dir{>}
\endxy
=(1+q^2)\xy
(2,5);(0,0.1)**\crv{(2,1)&(-2,-4)&(-4,0)&(-2,4)}
,(2,-5);(0.65,-1)**\crv{(2,-4)}
,(-6,-5);(-3.7,-.1)**\crv{(-5,-1.4)}
,(-2.7,0.5);(-1,1)**\crv{(-1.7,1.2)}
,(0,.9);(0.5,.8)**\crv{(-.2,1)}
,(1.4,.6);(5,-5)**\crv{(3,0)}?>(1)*\dir{>}
,(7,5);(10,-5)**\crv{(8,-1)}?>(1)*\dir{>}
,(2,5);(7,5)**\crv{(2,10)&(7,10)}
\endxy
\]

\end{example}
\begin{notation}
The identity morphism is given by the planar link diagram with just straight vertical lines connecting the bottom boundary with the top boundary. For instance, the identity of $\uparrow\,\, \downarrow \,\,\uparrow\,\, \uparrow$ is
\[
\xy
(-6,-5);(-6,5)**\dir{-}
?>(1)*\dir{>},
(-2,-5);(-2,5)**\dir{-}
?>(0)*\dir{<},
(2,-5);(2,5)**\dir{-}
?>(1)*\dir{>},
(6,-5);(6,5)**\dir{-}
?>(1)*\dir{>},
\endxy
\]
\end{notation}

\begin{remark}
The category $\mathbf{PTD}$ has a natural monoidal structure, with tensor product of objects and morphisms given by horizontal juxtaposition. For instance, if $R=\mathbb{Q}[q,q^{-1}]$, we have:
\[
(q^{-2}-q)\,\xy
(-5,-5);(5,-5)**\crv{(-5,5)&(0,5)&(5,5)}
?>(1)*\dir{>}
\endxy
\otimes \quad
(1+2q)\xy
(2,5);(0,0.1)**\crv{(2,1)&(-2,-4)&(-4,0)&(-2,4)}?>(0)*\dir{<}
,(2,-5);(0.65,-1)**\crv{(2,-4)}
,(-6,-5);(-3.7,-.1)**\crv{(-5,-1.4)}
,(-2.7,0.5);(-1,1)**\crv{(-1.7,1.2)}
,(0,.9);(0.5,.8)**\crv{(-.2,1)}
,(1.4,.6);(5,-5)**\crv{(3,0)}?>(1)*\dir{>}
,(7,5);(10,-5)**\crv{(7,0)}?>(1)*\dir{>}
\endxy
=(q^{-2}+2q^{-1}-q-2q^2)\,\xy
(-20,-5);(-10,-5)**\crv{(-20,5)&(-15,5)&(-10,5)}
?>(1)*\dir{>}
,(2,5);(0,0.1)**\crv{(2,1)&(-2,-4)&(-4,0)&(-2,4)}?>(0)*\dir{<}
,(2,-5);(0.65,-1)**\crv{(2,-4)}
,(-6,-5);(-3.7,-.1)**\crv{(-5,-1.4)}
,(-2.7,0.5);(-1,1)**\crv{(-1.7,1.2)}
,(0,.9);(0.5,.8)**\crv{(-.2,1)}
,(1.4,.6);(5,-5)**\crv{(3,0)}?>(1)*\dir{>}
,(7,5);(10,-5)**\crv{(7,0)}?>(1)*\dir{>}
\endxy
\]
The unit element is the empty sequence.
\end{remark}
\begin{remark}
 It is manifest by the definition of the objects and of the tensor product that $\mathbf{PTD}$ is monoidally generated by the objects $\uparrow$ and  $\downarrow$.
\end{remark}

\begin{remark} The monoidal category $\mathbf{PTD}$  also enjoys a natural structure of rigid monoidal category: the left and right dual of $\uparrow$ is $\downarrow$
and vice versa. The duality morphisms are given as follows:
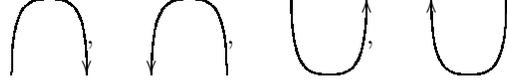
\[
\xy
(-5,-5);(5,-5)**\crv{(-5,5)&(0,5)&(5,5)}
?>(1)*\dir{>}
\endxy,
\qquad
\xy
(-5,-5);(5,-5)**\crv{(-5,5)&(0,5)&(5,5)}
?>(0)*\dir{<}
\endxy,
\qquad
\xy
(-5,5);(5,5)**\crv{(-5,-5)&(0,-5)&(5,-5)}
?>(1)*\dir{>}
\endxy,
\qquad
\xy
(-5,5);(5,5)**\crv{(-5,-5)&(0,-5)&(5,-5)}
?>(0)*\dir{<}
\endxy
\]
Since $\mathbf{PTD}$ is monoidally generated by $\{\uparrow,\downarrow\}$, these duality morphisms suffice to define duality morphisms for arbitrary objects in $\mathbf{PTD}$. In particular we see that each object in $\mathbf{PTD}$ is naturally identified with its double dual. This natural identification $\delta_{\vec{i}}=\mathrm{id}_{\vec{i}}\colon \vec{i}\xrightarrow{\sim} \vec{i}^{\vee\vee}=\vec{i}$ is compatible with dualities and with the tensor product of objects, i.e., it satisfies 
\[
\delta_{\vec{i}\otimes \vec{j}}=\delta_{\vec{i}}\otimes \delta_{\vec{j}}; \qquad \delta_{\emptyset}=\mathrm{id}_{\emptyset};\qquad \delta_{\vec{i}^\vee}=(\delta^\vee_{\vec{i}})^{-1}.
\]
Therefore, $\mathbf{PTD}$ is a ``balanced'' rigid category.
\end{remark}
\vskip .5 cm

\begin{definition}
Let $\mathcal{R}$ be the following set of morphisms in $\mathbf{PTD}$:
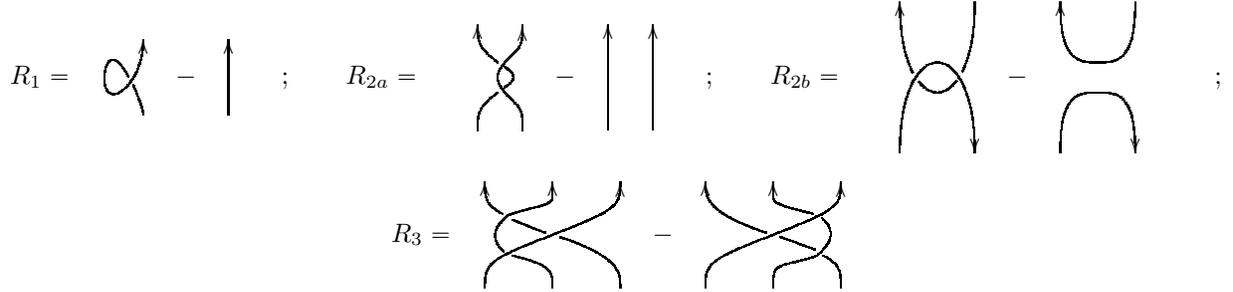
\[
\reidemeisteroneideal
\qquad;\qquad 
\reidemeistertwoideal\qquad;\qquad 
\reidemeistertwobideal\qquad;
\]
\[
\reidemeisterthreeideal
\]
as well as those obtained by reversing the orientations on the tangle diagrams or in the ambient $\mathbb{R}^2$, and let $\mathbf{I}_{\mathcal{R}}=\langle\mathcal{R}^\otimes \rangle$ be the tensor ideal generated by $\mathcal{R}$ in $\mathbf{PTD}$. Finally, let $\mathbf{TD}$ be the monoidal category defined by
\[
\mathbf{TD}=\mathbf{PTD}/{\mathbf{I}_{\mathcal{R}}}.
\]
We call $\mathbf{TD}$ the category of ($R$-linear combinations of) tangle diagrams.
\end{definition}
By the above discussion, we have the following.
\begin{proposition}
 The category $\mathbf{TD}$ is a balanced rigid category. 
 \end{proposition}

\begin{remark} As we are going to show in the next section, the balanced rigid category $\mathbf{TD}$ enjoys also an additional property: it is \emph{braided}. Therefore it is a \emph{ribbon} category in the terminology of Reshetikhin and Turaev \cite{RT1,RT2}. As in every ribbon category, the balancing isomorphisms $\delta{\vec{i}}$ can be encoded into suitable \emph{twist} morphisms $\theta_{\vec{i}}\colon \vec{i}\xrightarrow{\sim}\vec{i}$. We are going to see that these morphisms are trivial, i.e. $
\theta_{\vec{i}}=\mathrm{id}_{\vec{i}}$, for the ribbon category $\mathbf{TD}$.
\end{remark}

\subsection{The tangles category as a ribbon category}

We now want to define a braiding on $\mathbf{TD}$. Since $\mathbf{TD}$ is monoidally generated by the object $\uparrow$ together with its dual, all we have to do is to define the braiding for this object, i.e., we have to define a morphism
\[
\sigma\colon \left(\uparrow\,\, \uparrow \right) \to \left(\uparrow\,\, \uparrow\right)
\]
such that
\begin{itemize}
\item $\sigma$ is an isomorphism;
\item the braid relation 
\[
(\sigma \otimes \mathrm{id}_{\uparrow})\circ(\mathrm{Id}_{\uparrow}\otimes \sigma)\circ (\sigma\otimes \mathrm{id}_{\uparrow})=(\mathrm{id}_{\uparrow}\otimes \sigma)\circ (\sigma\otimes \mathrm{id}_{\uparrow})\circ(\mathrm{id}_{\uparrow}\otimes \sigma)
\]
is satisfied.
\end{itemize}
\begin{remark}
Equivalently, this can be expressed by requiring that there exists a pair of morphisms
\[
\sigma^+\colon \left(\uparrow\,\, \uparrow\right) \to \left(\uparrow\,\, \uparrow\right)
\,;\qquad
\sigma^-\colon \left(\uparrow\,\, \uparrow\right) \to \left(\uparrow\,\, \uparrow\right)
\]
such that
\begin{equation}\tag{R2a}
\sigma^+\circ\sigma^{-}=\sigma^-\circ\sigma^{+}=\mathrm{id}_{\uparrow\,\uparrow}
\end{equation}
\begin{equation}\tag{R3}
(\sigma^+ \otimes \mathrm{id}_{\uparrow})\circ(\mathrm{id}_{\uparrow}\otimes \sigma^+)\circ (\sigma^+\otimes \mathrm{id}_{\uparrow})=(\mathrm{id}_{\uparrow}\otimes \sigma^+)\circ (\sigma^+\otimes \mathrm{id}_{\uparrow})\circ(\mathrm{id}_{\uparrow}\otimes \sigma^+)
\end{equation}
Namely, $\sigma^+$ and $\sigma^-$ can be defined by $\sigma^+=\sigma$ and $\sigma^-=\sigma^{-1}$ and vice versa these equations define $\sigma$ from the pair $(\sigma^+,\sigma^-)$. 
\end{remark}
\begin{remark}
As $\mathbf{TD}$ is a rigid category, it is natural to require that it is \emph{rigid braided}. This requires the following compatibility between the braiding $\sigma$ and the evaluation and coevaluation morphisms:
\begin{equation}\tag{R2b}
(\mathrm{id}_{\uparrow}\otimes \mathrm{ev}_{\uparrow}\otimes\mathrm{id}_{\downarrow})\circ (\sigma^+\otimes (\sigma^-)^\vee) \circ (\mathrm{id}_{\uparrow}\otimes \mathrm{coev}\otimes\mathrm{id}_\downarrow)= \mathrm{coev}_{\uparrow} \circ \mathrm{ev}_{\uparrow}.
\end{equation}
\end{remark}

\begin{definition}
We set
\[
\sigma^+= \ijovercrossing
\]
and
\[
\sigma^-= \ijundercrossing
\]
\end{definition}
Then we have the following.
\begin{proposition}\label{prop:tangles-are-braided}
The monoidal category $\mathbf{TD}$ is rigid braided. 
\end{proposition}
\begin{proof}
This is immediate, as the relations (R2a), (R2b) and (R3) precisely correspond to the second (of type (a) and (b))
\[
\reidemeistertwoequal\qquad'\qquad \reidemeistertwobequal
\]
and to the third 
\[
\reidemeisterthreeequal
\]
Reidemeister relation, respectively, and these are satisfied in $\mathbf{TD}$ as we have obtained it from $\mathbf{PTD}$ by quotienting out the ideal generated by Reidemeister relations.
\end{proof}

\begin{definition}
We say that a ribbon category is \emph{untwisted} if the canonical twists are trivial. 
\end{definition}
We have the following refinement of Proposition \ref{prop:tangles-are-braided}.
\begin{proposition}
The rigid braided monoidal category $\mathbf{TD}$ is untwisted.
\end{proposition}
\begin{proof}
For $\mathbf{TD}$ the triviality of the twists amounts to the following equation:
\begin{equation}
\tag{R1}
(\mathrm{ev}_{\uparrow}\otimes \mathrm{id}_{\uparrow})\circ (\mathrm{id}_{\downarrow}\otimes \sigma^+)\circ (\mathrm{coev}_{\downarrow}\otimes \mathrm{id}_{\uparrow})=
=\mathrm{id}_{\uparrow}.
\end{equation}
This is satisfied as for $\mathbf{TD}$ the relation (R1) is precisely given by the first Reidemeister relation 
\[
\reidemeisteroneequal\quad .
\]
\end{proof}
We conclude this section by recalling the main property of the category $\mathbf{TD}$: its universality.
\begin{theorem}[Reshetikhin-Turaev, \cite{RT1,RT2}]\label{reshetikin-turaev}
The category $\mathbf{TD}$ is the \emph{universal} untwisted rigid braided category on one object. That is, if $\mathbf{C}$ is an untwisted rigid braided category (enriched in $R$-modules) and $X$ is an object of $\mathbf{C}$, then there exists and it is unique a rigid braided monoidal functor $Z_X\colon \mathbf{TD}\to \mathbf{C}$ with $Z(\uparrow)=X$.
\end{theorem}
\begin{remark}
The fact that $Z_X$ is a rigid braided monoidal functor means that evaluations, coevaluations and braidings in $\mathbf{TD}$ are mapped to evaluations, coevaluations and braidings involving the object $X$ and its dual $X^\ast$ in $\mathbf{C}$. 
\end{remark}
\begin{remark}
Since $\mathbf{TD}$ is a quotient of $\mathbf{PTD}$, the category $\mathbf{PTD}$ is ``more free'' than $\mathbf{TD}$: to give a monoidal functor from $\mathbf{PTD}$ to a balanced rigid monoidal category $\mathbf{C}$ is equivalent to giving an object $X$ of $\mathcal{C}$ together with pairs of morphism (not necessarily isomorphisms) $\phi^+,\phi^-\colon X\otimes X\to X\otimes X$. The morphisms $\phi^+,\phi^-$ are required to satisfy no relation: they can be chosen in a completely arbitrary way.
\end{remark}

\subsection{Links invariants from quotients of $\mathbf{TD}$}
In the language of the previous section, the celebrated result by Reidemeister on the relation between oriented\footnote{All tangles in this Thesis will be oriented. So we will often just write `tangle' to mean `oriented tangle'.} tangles in $\mathbb{R}^3$ (with endpoints on parallel planes $\mathbb{R}^2\times\{0\}$ and $\mathbb{R}^2\times\{1\}$) up to ambient isotopies (fixing the endpoints) and planar tangle diagrams can be stated by saying that drawing an oriented  tangle in $\mathbb{R}^3$ as a planar tangle diagram defines an injective map
\[
\{\text{tangles in $\mathbb{R}^3$}\}/\text{isotopy}\, \hookrightarrow \mathbf{TD}.
\]
This map is the \emph{tautological} oriented tangle invariant: every oriented tangle is mapped to itself, seen as a equivalence class of tangle diagrams. To get something more interesting, let us add more relations to quotient by, in addition to Reidemeister moves. 
\par
For \emph{every} ideal $I$ of $\mathbf{TD}$ we can consider the ideal $\mathbf{I}_{\mathcal{R}}+I$ and therefore we can cosider the quotient $\mathbf{PTD}/(\mathbf{I}_{\mathcal{R}}+I)\cong \mathbf{TD}/((\mathbf{I}_{\mathcal{R}}+I)/\mathbf{I}_{\mathcal{R}})$. This comes with a projection map 
\[
\mathbf{PTD}/\mathbf{I}_{\mathcal{R}}\cong  \mathbf{TD}\xrightarrow{\pi_I} \mathbf{TD}/((\mathbf{I}_{\mathcal{R}}+I)/\mathbf{I}_{\mathcal{R}})\cong \mathbf{PTD}_n/(\mathbf{I}_{\mathcal{R}}+I)
\] 
which is a braided monoidal functor. As a consequence we get the following.
\begin{remark}\label{rem:phi-i}
 For every ideal $I$ of $\mathbf{PTD}$ we get a (not necessarily injective) map $\phi_I$ as the composition
\[
\phi_I\colon \{\text{tangles in $\mathbb{R}^3$}\}/\text{isotopy} \hookrightarrow \mathbf{TD}\xrightarrow{\pi_I} \mathbf{PTD}_n/(\mathbf{I}_{\mathcal{R}}+I).
\]
\end{remark}
Let us focus our attention on the quotient map 
\[
\xymatrix{
\mathbf{PTD}_{\emptyset,\emptyset}\ar[r]\ar@/^2pc/[rr]^{[\,\,]}& (\mathbf{PLD}_n/\mathbf{I}_\mathcal{R})_{\emptyset,\emptyset}\ar[r]^-{\pi_I}& (\mathbf{PTD}/(\mathbf{I}_{\mathcal{R}}+I))_{\emptyset,\emptyset}
}
\]
 on \emph{closed} tangle diagrams, i.e., on (oriented) \emph{links}.
 \begin{remark}
As $\mathbf{PTD}_{\emptyset,\emptyset}=\mathrm{End}_{\mathbf{PTD}}(\emptyset)$, we have that $\mathbf{PTD}_{\emptyset,\emptyset}$ is an algebra. In particular, it is a $R$-algebra with unit element given by the empty diagram, so we have a canonical homomorphism of $R$-algebras
\[
R\to \mathbf{PTD}_{\emptyset,\emptyset}
\]
mapping the element 1 of $R$ to the empty link $\emptyset$ (as a morphism from the empty sequence $\emptyset$ to itself). As a consequence, each ideal $I$ of  $\mathbf{PTD}$ determines an homomorphism of $A$-algebras as the composition
\[
\rho_I\colon R\to \mathbf{PTD}_{\emptyset,\emptyset}\to \mathbf{PTD}/(\mathbf{I}_{\mathcal{R}}+I)_{\emptyset,\emptyset},
\]
and so an injective homomorphism of $R$-algebras
\[
\tilde{\rho}_I\colon R/\ker \rho_I\to \mathbf{PTD}/(\mathbf{I}_{\mathcal{R}}+I)_{\emptyset,\emptyset}.
\]
\end{remark}
\begin{lemma}\label{lemma:enriched}
The category $\mathbf{PTD}/(\mathbf{I}_{\mathcal{R}}+I)$ is enriched in $R/\ker \rho_I$-modules.
\end{lemma}
\begin{proof}
We need to show that for any $\Gamma\in \mathbf{PTD}_{\vec{i},\vec{j}}$ and every $a\in \ker \rho_I$, we have $a[\Gamma]=0$. As $\Gamma\cong \emptyset \otimes \Gamma$, we have $a[\Gamma]=a([\emptyset] \otimes [\Gamma])=(a [\emptyset])\otimes \Gamma=0$, as $a [\emptyset]=0$ by definition of $\ker \rho_I$.
\end{proof}
The case of interest for us will be when the relations coming from the Reidemeister moves and the ideal $I$ allow to reduce every link to a scalar multiple of the empty link. This is usually achieved by a set of relations that allow first to reduce each link to an $R$-linear combination of unknotted disjoint copies of $S^1$ in the plane, and then two disjoint copies of unknotted $S^1$'s to a scalar multiple of a single copy of an unknotted  $S^1$, so that interatively one ends up with a scalar multiple of a single copy of an unknotted $S^1$.  Finally one needs a further ``normalization'' move that identifies an ``unknot'' with a scalar multiple of the empty link.  
\par
As we are going to see, this is precisely what happens with the so called skein relations. Other examples of this construction are Kaufmann brackets relations \cite{Kauffman}.  More generally,
when this happens, the morphism $\rho_I$ is surjective and so $\tilde{\rho}_I$ is an isomorphism. The following is immediate.

\begin{lemma}\label{little-lemma}
Assume $\rho_I\colon R\to \mathbf{PTD}/(\mathbf{I}_{\mathcal{R}}+I)_{\emptyset,\emptyset}$ is surjective, and let $\tilde{p}_I\colon \mathbf{PTD}/(\mathbf{I}_{\mathcal{R}}+I)_{\emptyset,\emptyset}\to R/\ker \rho_I$  be the inverse of $\tilde{\rho}_I$. Then we have a canonical morphism of $A$-algebras
\[
\xymatrix{
\mathbf{PTD}_{\emptyset,\emptyset}\ar@/^2.5pc/[rr]^{p_I}\ar[r]^-{[\,\,]}&  \mathbf{PTD}/(\mathbf{I}_{\mathcal{R}}+I)_{\emptyset,\emptyset} \ar[r]^-{\tilde{p}_I}_-{\sim} &A/\ker\rho_I
}
\]
mapping a closed tangle diagram $\Gamma$ to the unique element $p_I(\Gamma)$ in $A/\ker\rho_I$ such that $\tilde{\rho}_I(p_I(\Gamma))=[\Gamma]$, where $[\Gamma]$ is the equivalence class of $\Gamma$ in $\mathbf{PTD}_{\emptyset,\emptyset}$.
\end{lemma}

\begin{remark}
The morphisms in Lemma \ref{little-lemma} fit into a commutative diagram
\[
\xymatrix{
R\ar[r]\ar[d]&\mathbf{PTD}_{\emptyset,\emptyset}\ar[dl]_{p_I}\ar[r]\ar@/^2pc/[rr]^{[\,\,]}& {\mathbf{PTD}/\mathbf{I}_{\mathcal{R}}}_{\emptyset,\emptyset}\ar[r]^-{\pi_I}& \mathbf{PTD}/(\mathbf{I}_{\mathcal{R}}+I)_{\emptyset,\emptyset}\ar@/_.2pc/[dlll]^{\tilde{p}_I}\\
R/\ker\rho_I\ar@/_1.5pc/[urrr]_{\tilde{\rho}_I}
}.
\]
 In other words,  $p_I(\Gamma)$ in $R/\ker\rho_I$ is the image of in $R/\ker\rho_I$ of an element $\lambda_\Gamma$ in $R$ such that 
\[
[\Gamma]=\lambda_\Gamma\,[\emptyset].
\]
Namely, as we are assuming $\rho_I$ to be surjective, for every link diagram $\Gamma$ there exists $\lambda_\Gamma$ in $R$ such that $[\Gamma]=\rho_I(\lambda_\Gamma)=\lambda_\Gamma\rho_I(1)=\lambda_\Gamma[\emptyset]$. Let $\tilde{\lambda}_\Gamma$ be the image of $\lambda_\Gamma$ in $A/\ker\rho_I$. Then $\tilde{\rho}_I(\tilde{\lambda}_\Gamma)=[\Gamma]=\tilde{\rho}_I(p_I(\Gamma))$, and so $p_I(\Gamma)=\tilde{\lambda}_\Gamma$ as  $\tilde{\rho}_I$ is an isomorphism.
\end{remark}
\begin{remark}
The equation $[\Gamma]=\lambda_\Gamma\,[\emptyset]$ uniquely determines $\lambda_\Gamma$ up to a term in $\ker\rho_I$. Indeed, if $\lambda_\Gamma\,[\emptyset]=\lambda_\Gamma'\,[\emptyset]$, then $\rho_I(\lambda_\Gamma-\lambda_\Gamma')=(\lambda_\Gamma-\lambda_\Gamma')\,[\emptyset]=[\Gamma]-[\Gamma]=0$.
\end{remark}
\vskip .7 cm
In the above hypothesis, the map $\phi_I$ of Remark \ref{rem:phi-i} is naturally identified with a map
\[
\phi_I\colon \{\text{links in $\mathbb{R}^3$}\}/\text{isotopy} \to R/\ker\rho_I,
\]
and so to get a nontrivial link invariant we only need to be able to choose an ideal $I$ such that $\ker\rho_I\neq R$. If we have been able to do so, then up to changing $R$ with $R/\ker \rho_I$, by Lemma \ref{lemma:enriched} we may assume that $\ker\rho_I=\{0\}$.
This leads us to the following.
\begin{proposition}\label{prop:link-invariant}
The datum of a multiplicative link invariant 
\[
\{\text{links in $\mathbb{R}^3$}\}/\text{isotopy}\to R
\]
is equivalent to the datum of an ideal $I$ of $\mathbf{PTD}$ such that 
\[\mathbf{PTD}/(\mathbf{I}_{\mathcal{R}}+I)_{\emptyset,\emptyset}\cong R.
\]
\end{proposition}
\begin{proof}
We have seen above as the datum of an ideal $I$ as in the assumptions of the statement induces a links invariant 
\[
\phi_I\colon \{\text{links in $\mathbb{R}^3$}\}/\text{isotopy} \to R
\]
which, by construction is multiplicative, i.e., is such that 
\[
\phi_I(\Gamma_1\sqcup \Gamma_2)=\phi_I(\Gamma_1)\,\phi_I(\Gamma_2).
\]

Vice versa, as isotopy classes of links are $R$-linear generators for $\mathbf{TD}_{\emptyset,\emptyset}$ any multiplicative $R$-valued links invariant is induced by an $R$-algebra homomorphism $p:\mathbf{PTD}_{\emptyset,\emptyset}\to R$ factoring through $\mathbf{TD}_{\emptyset,\emptyset}$, i.e., such that $p({\mathbf{I}_{\mathcal{R}}}_{\emptyset,\emptyset})=0$.  Let $I$ be an ideal of $\mathbf{PTD}$ such that $\ker(p)=(\mathbf{I}_{\mathcal{R}}+I)_{\emptyset,\emptyset}$. Then we have a commuative diagram
\[
\xymatrix{
R\ar[dr]^{\mathrm{id}}\ar[d]\ar@/_3pc/[dd]_{\rho^{}_{I_p}}\\
\mathbf{PTD}_{\emptyset,\emptyset}\ar[r]^{p}\ar[d]&R\\
\mathbf{PTD}_n/(\mathbf{I}_\mathcal{R}+I)_{\emptyset,\emptyset} \ar[ru]_{\tilde{p}},
}
\]
where $\tilde{p}$ is the injective homomorphism induced by the fact that $p$ factors through the quotient. Namely, since the morphism $R\to R$ given by the composition $R\to \mathbf{PTD}_{\emptyset,\emptyset}\xrightarrow{p}R$ in the diagram is a homomorphism of $R$-algebras with unit, it is necessarily the identity. This in particular implies that $\rho^{}_{I_p}$ is injective and that $\tilde{p}$ is surjective. Since $\tilde{p}$ is also injective, this shows that both $\rho^{}_{I_p}$ and $\tilde{p}$ are $R$-algebra isomorphisms, and they are inverse each other.
\end{proof}
\begin{remark}
By definition of the quotient $\mathbf{PTD}/(\mathbf{I}_{\mathcal{R}}+I)$, to give an $R$-algebra homomorphism 
\[
p\colon \mathbf{PTD}/(\mathbf{I}_{\mathcal{R}}+I)_{\emptyset,\emptyset} \to R
\]
is equivalent to giving an $R$-algebra homomorphism 
\[
p\colon \mathbf{PTD}_{\emptyset,\emptyset} \to R
\]
such that
\begin{itemize}
\item $p(\mathbf{I}_{\mathcal{R}})=0$, i.e., $p$ is invariant by the Reidemeister moves;
\item $p(I)=0$, i.e., $p$ is invariant by the relations in the ideal $I$.
\end{itemize}
\end{remark}

\subsection{The skein relation of the HOMFLY polynomial}
A  skein relation defined on the ring $R$ is a local relation on $R$-linear combination of oriented planar tangle diagrams saying that one can trade an overcrossing for an undercrossing by paying a price in the form of an error term which contains no crossings, and vice versa. More formally, we have the following.
\begin{definition}
A  skein relation defined on the ring $R$ is a linear relation in $\mathbf{PTD}_{\uparrow\uparrow,\uparrow\uparrow}$ of the form
\begin{equation}\label{eq:skein1}
\overcrossing=a\,\undercrossing+b\,\noncrossing
\end{equation}
with $a\in R^\times$, and $b\in R$, where $R^\times\subseteq R$ is the group of invertible elements of $R$. 
\end{definition}
\begin{remark}
The condition $a\in R^\times$ is not always included in the definition of skein relation, which is sometimes more generally given  without requiring $a$ to be invertible. We imposed invertibility of $a$ as this is what is needed to get the other way round relation
\begin{equation}\label{eq:skein2}
\undercrossing=a^{-1}\,\overcrossing+a^{-1}b\, \noncrossing.
\end{equation}
Moreover, if $b\in R^\times$ we can also write
\begin{equation}\label{eq:skein3}
\noncrossing=b^{-1}\,\overcrossing-ab^{-1}\,\undercrossing.
\end{equation}
\end{remark}
The first two skein relations (\ref{eq:skein1}-\ref{eq:skein2}) and Reidemeister relations allow one to reduce any closed link diagram to a $R$-linear combination of disjoint copies of unknot diagrams. Moreover, the third skein relation (\ref{eq:skein3}) shows that two disjoint unknot diagrams are equivalent (under Reidemeister moves) to a certain $R$-scalar multiple of a single unknot, and so inductively that an arbitrary $R$-linear combination of disjoint copies of unknot diagrams is ultimately equivalent to some scalar multiple of a single unknot diagram. Then, adding a single relation of the form
\begin{equation}\label{eq:normalization}
\unknot=c\emptyset
\end{equation}
for some $c\in R$ (a normalization condition), we see that every closed link diagram is equivalent to a scalar multiple of the empty link diagram. Clearly if $c=0$ then this scalar multiple is zero, so in order to get a nontrivial result, let us assume $c\neq 0$. The above discussion together with the results of the previous section amounts to a proof of the following.
\begin{proposition}\label{prop:homfly}
Let us denote by $\mathbf{I}_{a,b,c}$ the ideal of $\mathbf{PTD}$ defined by the Reidemeister moves and by the skein relation $(\ref{eq:skein1})$ with both $a$ and $b$ in $R^\times$. There exists and it is unique a multiplicative link invariant
 link invariant
 \[
\phi_{a,b,c}\colon \{\text{oriented links in $\mathbb{R}^3$}\}/\text{isotopy} \to R/\ker\rho_{\mathbf{I}_{a,b,c}},
\]
induced by a morphism
\[
p_{a,b,c}\colon\mathbf{PTD}_{\emptyset,\emptyset}\to R
\]
 with
\begin{itemize} 
\item \[
p_{a,b,c}(\Gamma)=[0]
\]
for any $\Gamma$ in $\mathbf{PTD}_{\emptyset,\emptyset}$ which is in the tensor ideal generated by the element $\overcrossing-a\,\undercrossing-b\,\noncrossing$;
\item 
\[ p_{a,b,c}\left(\,\unknot\,\right)=[c];
\]
\item
\[
 p_{a,b,c}(\emptyset)=[1].
 \]
 \end{itemize}
 \end{proposition}
 \begin{definition}[The HOMFLY-PT polynomial, \cite{homfly}]
 The multiplicative links invariant $\phi_{a,b,c}$ of Proposition \ref{prop:homfly} is called the HOMFLY-PT polynomial with parameters $a,b,c$.
 \end{definition}
 
\begin{example}[The Jones polynomial]\label{ex:jones} For $R=\mathbb{Q}[q,q^{-1}, (q-q^{-1})^{-1}]$, the algebra of Laurent polynomials with $\mathbb{Q}$-coefficients in the variable $q$ where we have additionally inverted the element $q-q^{-1}$, the choice $a=q^{-4}$, $b=q^{-1}-q^{-3}$ and $c=q+q^{-1}$ leads to the celebrated Jones polynomial \cite{Jones1,Jones2}. This generalizes to the level $n$ Jones polynomial for any $n\geq 2$, the case of the classical Jones polynomial corresponding to $n=2$, as follows. Notice how the somehow bizarre skein relation of the Jones polynomial, i.e., 
 \[
 \overcrossing=q^{-4}\,\undercrossing+(q^{-1}-q^{-3})\,\noncrossing
 \]
 can be more symmetrically written as
 \[
(q-q^{-1}) \noncrossing=q^2\,\overcrossing-q^{-2}\,\undercrossing
 \]
 and that the choice $c=q+q^{-1}$ can be rewritten as $c=\frac{q^2-q^{-2}}{q-q^{-1}}$. The generalization to the level $n$ Jones polynomial is then given by the skein relation
  \[
(q-q^{-1}) \noncrossing=q^n\,\overcrossing-q^{-n}\,\undercrossing
 \]
 (corresponding to the choice $a=q^{-2n}$ and $b=q^{1-n}-q^{-1-n}$) with the normalization condition
 \[
 c=\frac{q^n-q^{-n}}{q-q^{-1}}.
 \]
It is customary to call the polynomial
\[
\frac{q^n-q^{-n}}{q-q^{-1}}=q^{1-n}+q^{2-n}+\cdots + q^{n-2}+q^{n-1}
\] a $q$-integer and to denote it by the symbol $[n]_q$. That is, for the level $n$ Jones polynomial we have
\[
Jones_n\left(\,\unknot\,\right)=[n]_q.
\]

 \end{example}
\begin{remark}\label{rem:jones-little}
If $\Gamma$ is a link, the level $n$ Jones polynomial $Jones_n(\Gamma)$ is actually a Laurent polynomial in the variable $q$, i.e., $Jones_n$ actually takes its values in the subalgebra $\mathbb{Q}[q,q^{-1}]$ of $\mathbb{Q}[q,q^{-1}, (q-q^{-1})^{-1}]$. Namely, the two skein relations
  \[
 \overcrossing=q^{-4}\,\undercrossing+(q^{-1}-q^{-3})\,\noncrossing
 \]
 and
  \[
 \undercrossing=q^{4}\,\overcrossing-(q^{3}-q)\,\noncrossing
 \]
 and Reidemeister relations allow one to reduce any closed link diagram to a $\mathbb{Q}[q,q^{-1}]$-linear combination of disjoint copies of unknot diagrams. Each of these unknot diagrams is mapped to $[n]_q$, which is again an element in $\mathbb{Q}[q,q^{-1}]$. By multiplicativity the disjoint union of unknots is mapped to a power of $[n]_q$, and so ultimately we get a $\mathbb{Q}[q,q^{-1}]$-linear combination of powers of $[n]_q$, which is an element in $\mathbb{Q}[q,q^{-1}]$.
\end{remark}

The skein relation imposes a constraint on $c$ in order for $\ker\rho_{\mathbf{I}_{a,b,c}}$ to be zero, and so for the HOMFLY-PT polynomial $\phi_{a,b,c}$ to be actually a links invariant
 \[
\phi_{a,b,c}\colon \{\text{oriented links in $\mathbb{R}^3$}\}/\text{isotopy} \to R.
\]
This is the content of the following
\begin{lemma}\label{lemma:abc}
Assume $\ker\rho_{\mathbf{I}_{a,b,c}}=0$. Then $c=(1-a)b^{-1}$.
\end{lemma}
\begin{proof}
From
\[
b\noncrossing=\,\overcrossing-a\,\undercrossing
\]
we see that $b$ times the disjoint union of two unknots is equivalent to $(1-a)$ times a single unknot:
\[
b\, \xy
(0,-5);(0,-5)**\crv{(5,-5)&(5,5)&(-5,5)&(-5,-5)}
?>(.3)*\dir{>}
\endxy
\,
\xy
(0,-5);(0,-5)**\crv{(5,-5)&(5,5)&(-5,5)&(-5,-5)}
?>(0.7)*\dir{<}
\endxy
\,=\,
\xy
(-.7,.7);(.7,-.7)**\crv{(-10,10)&(-12,0)&(-10,-10)&(10,10)&(12,0)&(10,-10)}
?>(.8)*\dir{>}
\endxy
-a\,\xy
(-.7,-.7);(.7,.7)**\crv{(-10,-10)&(-12,0)&(-10,10)&(10,-10)&(12,0)&(10,10)}
?>(.785)*\dir{<}
\endxy
\,=\,
(1-a)\,\unknot\,.
\] 
By applying $p_{a,b,c}\colon\mathbf{PTD}_{\emptyset,\emptyset}\to A$ to both sides, we get the equation $
[bc^2-c(1-a))]=0$
in $A/\ker\rho_{\mathbf{I}_{a,b,c}}$. As we are assuming $\ker\rho_{\mathbf{I}_{a,b,c}}=0$, this gives 
\[
c(bc-(1-a))=0.
\]
If $R$ is an integral domain this gives $c=0$ or $bc=(1-a)$. Since we are assuming $c\neq 0$, we are left with the single equation $bc=(1-a)$. As $b$ is invertible in $R$, this gives $c=(1-a)b^{-1}$.
\end{proof}
\begin{remark}
In an integral domain the equation $bc=(1-a)$ uniquely determines $c$, unless $b=0$. More precisely, if $b\neq0$ there will be at most a unique $c$ satisfying the equation $bc=(1-a)$. This happens precisely when $(1-a)$ is a multiple of $b$, and in that case $c$ is 
determined as $c=\frac{1-a}{b}$. This observation allows in some cases to relax the condition on the invertibility of $b$, requiring only that $b$ divides $(1-a)$. This is what happens, for instance, when one wants to define the Jones polynomial directly as taking values in the algebra $\mathbb{Q}[q,q^{-1}]$ rather than in the larger algebra $\mathbb{Q}[q,q^{-1},(q-q^{-1})^{-1}]$, see Remark \ref{rem:jones-little}.
\end{remark}
\begin{example}
In the Jones polynomial from Example \ref{ex:jones}, we had $a=q^{-4}$, $b=q^{-1}-q^{-3}$ and $c=q+q^{-1}$, and indeed $q+q^{-1}=(1-q^{-4})/(q^{-1}-q^{-3})$. More generally, for the level $n$ Jones polynomial we had $a=q^{-2n}$, $b=q^{1-n}-q^{-1-n}$ and $c=[n]_q$, and indeed $(1-q^{-2n})/(q^{1-n}-q^{-1-n})=(q^n-q^{-n})/(q^-q^{-1})$.
\end{example} 
\begin{remark}
A remarkable fact is that, when $A$ is an integral domain, the necessary condition $c=(1-a)/b$ from Lemma \ref{lemma:abc} is also sufficient to have a trivial kernel $\ker\rho_{\mathbf{I}_{a,b,c}}$ and so to have a multiplicative link invariant 
\[
\phi_{a,b,c}\colon \{\text{oriented links in $\mathbb{R}^3$}\}/\text{isotopy} \to A.
\]
In  the original paper \cite{homfly} defining the HOMFLY-PT invariant, one finds outlines of four distinct proofs of the sufficiency of the condition $c=(1-a)/b$. We are going to provide a fifth one through MOY calculus in the following chapter.
\end{remark}
\begin{remark}
In the original definition of the HOMFLY-PT polynomial one does not have the normalization $\phi_{a,b,c}(\unknot)=c$ but rather the normalization $P_{a,b,c}(\unknot)=1$. In other words one has $P_{a,b,c}=c^{-1}\phi_{a,b,c}$. As a consequence of this, the HOMFLY-PT polynomial $P_{a,b,c}$ is not multiplicative on split unions of closed links but rather satisfies
\[
P_{a,b,c}(\Gamma_1\otimes \Gamma_2)=c\,P_{a,b,c}(\Gamma_1)P_{a,b,c}(\Gamma_2).
\] 
Written this way, the relation is meaningful even in the case $c=0$. In this case one has a link invariant which is \emph{not} multiplicative (and which can not be made multiplicative by a rescaling), but which is identically zero on splittable closed links. The most well-known example is probably the Alexander-Conway polynomial, corresponding to the choice $a=1$ and $b=q-q^{-1}$.
\end{remark}
\begin{remark}
Up to passing to a quadratic extension of $A$, we may assume there exists an element $\ell$ in $R$ such that $\ell^{-2}=-a$. Setting $m=-b\ell$, the skein relation can then be rewritten as
\[
\ell \,\overcrossing+\ell^{-1} \,\undercrossing +m \,\noncrossing=0,
\]
which is the form commonly encountered in the literature. Expressed in the variables $\ell,m$, the compatibility condition for the normalization becomes 
\[
c=-\frac{\ell+\ell^{-1}}{m} 
\]
and so on split unions one has the relation
\[
P_{\ell,m}(\Gamma_1\otimes \Gamma_2)=-\frac{\ell+\ell^{-1}}{m}\,P_{\ell,m}(\Gamma_1)P_{\ell,m}(\Gamma_2).
\] 
for the HOMFLY-PT polynomial with paramteres $\ell,m$. Another customary convention is to denote by $\alpha$ an element of $R$ such that $\alpha^{-2}=a$ (again, one can always assume that such an element exists up to passing to a quadratic extension of $R$) and to set $z=\alpha b$, so that the skein relation reads
\begin{equation}\label{eq:skein-homfly}
\alpha\,\overcrossing-\alpha^{-1}\,\undercrossing=z\,\noncrossing
\end{equation}
and the normalization condition is $c=(\alpha-\alpha^{-1})/z$. For instance, for the Jones polynomial one has $\alpha=q^2$ and $z=q-q^{-1}$. More generally, for the level $n$ Jones polynomial one has $\alpha=q^n$ and $z=q-q^{-1}$, so that the skein relation (\ref{eq:skein-homfly}) is 
\[
q^n\,\overcrossing-q^{-n}\,\undercrossing=(q-q^{-1})\,\noncrossing
\]
and
$c=\frac{q^n-q^{-n}}{q-q^{-1}}$
(see Example \ref{ex:jones}).
\end{remark}
It is convenient to summarize all of the above discussion in a definition
\begin{definition}\label{def:homfly-cat}
For $\alpha,z\in R^\times$ the $\mathbf{HOMFLY\text{-}PT}_{\alpha,z}$ category is the rigid balanced untwisted braided category enriched in $R$-modules defined by
\[
\mathbf{HOMFLY\text{-}PT}_{\alpha,z}=\mathbf{TD}\bigg/\left\langle \alpha \,\overcrossing+\alpha^{-1} \,\undercrossing -z \,\noncrossing\,\, ,\,\, \unknot-\frac{\alpha-\alpha^{-1}}{z}\emptyset\right\rangle,
\]
where $\langle \mathcal{S}\rangle$ denotes the tensor ideal generated by the set of morphisms $\mathcal{S}$. For $R=\mathbb{Q}[q,q^{-1},(q-q^{-1})^{-1}]$ and any $n\geq 2$, we set
\[
\mathbf{Jones}_n=\mathbf{HOMFLY\text{-}PT}_{q^n,q-q^{-1}}.
\]
\end{definition}
\begin{remark}
The above definitions, together with Proposition \ref{prop:link-invariant} imply that the HOMFLY-PT polynomial $\phi_{\alpha,z}$ is defined by the isomorphism
\[
{\mathbf{HOMFLY\text{-}PT}_{\alpha,z}}_{\emptyset,\emptyset}\cong R
\]
via the equation
\[
[\Gamma]=\phi_{\alpha,z}(\Gamma)\, [\emptyset],
\]
where $[\Gamma]$ denotes the equivalence class of the link $\Gamma$ in  ${\mathbf{HOMFLY\text{-}PT}_{\alpha,z}}_{\emptyset,\emptyset}$. In particular, the level $n$ Jones polynomial is defined by the equation 
\[
[\Gamma]=Jones_n(\Gamma)\, [\emptyset],
\]
where $[\Gamma]$ denotes the equivalence class of the link $\Gamma$ in  ${\mathbf{Jones}_{n}}_{\emptyset,\emptyset}$
\end{remark}

\chapter{The MOY category}\label{chapter:moy}

As a second example of the construction of a monoidal category as a quotient, we will now describe a tensor category enriched in $\mathbb{Q}[q,q^{-1}]$-modules, where $\mathbb{Q}[q,q^{-1}]$ is the $\mathbb{Q}$-algebra of Laurent polynomials in the variable $q$ which is implicit in the work by Murakami,  Ohtsuki and Yamada on the Jones polynomial \cite{moy}. For this reason we are going to call this monoidal category the MOY category. With respect to \cite{moy}, we are going to emphasize the language of braided rigid tensor categories and to prove a universal property of the MOY category.

\section{The category of oriented planar trivalent diagrams}
 We start by considering the category  $\mathbf{TrPD}$ of trivalent planar diagrams. Its objects are the same as for the category $\mathbf{PTD}$ of planar tangle diagrams, but now we decorate each source or target arrow by the integer 1. While, this is a trivial decoration, as all objects get a unique decoration given by a set of 1's, labelling the source and target arrows in this way is useful in view of the definition of the morphisms in  $\mathbf{TrPD}$. Here is a formal definition.

\begin{definition}
The category $\mathbf{TrPD}$ of ($\mathbb{Q}[q,q^{-1}]$-linear combinations of) oriented trivalent planar diagrams is the category enriched in $\mathbb{Q}[q,q^{-1}]$-modules whose
 objects are (possibly empty) finite sequences of upgoing or downgoing arrows. Morphisms between two objects $\vec{i}$ and $\vec{j}$ are given by $\mathbb{Q}[q,q^{-1}]$-linear combinations of plane trivalent graphs with oriented edges labeled by positive integers in the set $\{1,2\}$ and with incoming boundary data given by the sequence $\vec{i}$ and outgoing boundary data given by the sequence $\vec{j}$. At each vertex exactly two occurring edges are incoming and one outgoing or vice versa; moreover the sum of the incoming labels is required to be equal to the sum of the outgoing labels.
 \par
 The composition of morphisms is given by vertical concatenation of diagrams,
extended by $\mathbb{Q}[q,q^{-1}]$-bilinearity to $\mathbb{Q}[q,q^{-1}]$-linear combinations of diagrams.
\end{definition}

 \begin{example}
The following is a morphism from $\uparrow^1,\uparrow^1$ to $\uparrow^1,\uparrow^1$ in $\mathbf{TrPD}$:
\[
(5q^{-12}+11q^4)\idtwo-(7q^{-3}+3q^2+q^3)
\esse+(q^{-8}-3q^4)\essecircesse.
\]
\end{example}

\begin{example}
The following is a composition of morphisms in $\mathbf{TrPD}$:
\[
(1+q)\esse \circ (1-q)\esse = (1-q^2)\essecircesse.
\]
\end{example}

\begin{notation}
The identity morphism is given by the planar link diagram with just straight vertical lines connecting the bottom boundary with the top boundary. For instance, the identity of $\uparrow^1\,\, \downarrow^1 \,\,\uparrow^1\,\, \uparrow^1$ is
\[
\xy
(-6,-5);(-6,5)**\dir{-}
?>(1)*\dir{>},
(-2,-5);(-2,5)**\dir{-}
?>(0)*\dir{<},
(2,-5);(2,5)**\dir{-}
?>(1)*\dir{>},
(6,-5);(6,5)**\dir{-}
?>(1)*\dir{>},
(-4.7,-5)*{\scriptstyle{1}},
(-.7,-5)*{\scriptstyle{1}},
(3.3,-5)*{\scriptstyle{1}},
(7.3,-5)*{\scriptstyle{1}}
\endxy
\]
\end{notation}
\begin{remark}
The category $\mathbf{TrPD}$ has a natural monoidal structure, with tensor product of objects and morphisms given by horizontal juxtaposition. For instance,
\[
(q^{-2}-q)\,\idtwo
\otimes \quad
(1+2q)\esse
=(q^{-2}+2q^{-1}-q-2q^2)\,\idtwo\quad \esse
\]
The unit element is the empty sequence.
 It is manifest by the definition of the objects and of the tensor product that $\mathbf{TrPD}$ is monoidally generated by the objects $\uparrow^1$ and  $\downarrow^1$.
\end{remark}

\begin{remark} The monoidal category $\mathbf{TrPD}$  also enjoys a natural structure of rigid monoidal category,  with ``the same'' duality morphisms as $\mathbf{PTD}$. More precisely, the duality morphisms on the generating objects are:
\[
\xy
(-5,-5);(5,-5)**\crv{(-5,5)&(0,5)&(5,5)}
?>(1)*\dir{>}
,(-4,-5)*{\scriptstyle{1}},(6,-5)*{\scriptstyle{1}}
\endxy,
\qquad
\xy
(-5,-5);(5,-5)**\crv{(-5,5)&(0,5)&(5,5)}
?>(0)*\dir{<}
,(-4,-5)*{\scriptstyle{1}},(6,-5)*{\scriptstyle{1}}
\endxy,
\qquad
\xy
(-5,5);(5,5)**\crv{(-5,-5)&(0,-5)&(5,-5)}
?>(1)*\dir{>}
,(-4,5)*{\scriptstyle{1}},(6,5)*{\scriptstyle{1}}
\endxy,
\qquad
\xy
(-5,5);(5,5)**\crv{(-5,-5)&(0,-5)&(5,-5)}
?>(0)*\dir{<}
,(-4,5)*{\scriptstyle{1}},(6,5)*{\scriptstyle{1}}
\endxy
\]
These make $\mathbf{TrPD}$  a balanced rigid category. 
\end{remark}

\begin{notation}
We write $\mathrm{id}_k$ for the identity morphism
\[
\mathrm{id}_k\colon \underbrace{\mathrm{id}_{\uparrow^1}\otimes \mathrm{id}_{\uparrow^1}\otimes \cdots \otimes \mathrm{id}_{\uparrow^1}}_{k\text{ times}}\xrightarrow{\sim} \underbrace{\mathrm{id}_{\uparrow^1}\otimes \mathrm{id}_{\uparrow^1}\otimes \cdots \otimes \mathrm{id}_{\uparrow^1}}_{k\text{ times}}.
\]
Notice that we have
\[
\mathrm{id}_{k_1+k_2}=\mathrm{id}_{k_1}\otimes \mathrm{id}_{k_2}.
\]
Also we write $\mathrm{ev}_k$ and $\mathrm{coev}_k$ for the evaluation and coevaluation morphisms for the object $ \underbrace{\uparrow^1\otimes \uparrow^1\otimes \cdots \otimes \uparrow^1}_{k\text{ times}}$ and write simply  $\mathrm{ev}$ and $\mathrm{coev}$ for $\mathrm{ev}_1$ and $\mathrm{coev}_1$, respectively. Finally, we write $\mathrm{ev}_{-k}$ and $\mathrm{coev}_{-k}$ for the evaluation and coevaluation morphisms for the object $ \underbrace{\downarrow_1\otimes \downarrow_1\otimes \cdots \otimes \downarrow_1}_{k\text{ times}}$.
\end{notation}

\begin{remark}
One easily sees that morphism in $\mathbf{TrPD}$ are generated via tensor product and composition by identities, evaluations, coevaluations and by the morphism 
\[
S\colon  \uparrow^1\otimes \uparrow^1 \xrightarrow{} \uparrow^1\otimes \uparrow^1
\]
given by
\[
S={\xy
(-5,-10);(-5,10)**\crv{(-1,-6)&(0,-5)&(0,-2)&(0,2)&(0,5)&(-1,6)}
?>(.55)*\dir{>}?>(.1)*\dir{>}?>(.95)*\dir{>}
,(5,-10);(5,10)**\crv{(1,-6)&(0,-5)&(0,-2)&(0,2)&(0,5)&(1,6)}
?>(.1)*\dir{>}?>(.95)*\dir{>}
,(-7,-10)*{\scriptstyle{1}}
,(7,-10)*{\scriptstyle{1}}
,(-7,10)*{\scriptstyle{1}}
,(7,10)*{\scriptstyle{1}}
,(2,0)*{\scriptstyle{2}}
\endxy}
\]
together with their duals. 
\end{remark}
The above remark can be restated and emphasized as follows.
\begin{proposition}
The category $\mathbf{TrPD}$ is the free $\mathbb{Q}[q,q^{-1}]$-linear rigid monoidal category generated by a single object endowed with a distinguished endomorphism of the tensor product of this generating object with itself: for any  $\mathbb{Q}[q,q^{-1}]$-linear rigid monoidal category $\mathcal{C}$, the datum of a rigid $\mathbb{Q}[q,q^{-1}]$-linear monoidal functor $Z\colon \mathbf{TrPD}\to \mathcal{C}$ is equivalent to the choice of an object $X$ of $\mathcal{C}$ together with the choice of a morphism $\varphi\colon X\otimes X\to X\otimes X$.
\end{proposition}

\begin{notation}
For any nonnegative integer $j$, the \emph{$q$-integer} $[j]_q$ is the element of $\mathbb{Q}[q,q^{-1}]$ defined by
\[
[j]+q=\frac{q^j-q^{-j}}{q-q^{-1}}=q^{j-1}+q^{j-3}+\cdots +q^{3-j}+q^{1-j}
\]
The \emph{$q$-factorial} $[j]_q!$ is the element of $\mathbb{Q}[q,q^{-1}]$ defined by
\[
[j]_q!=[j]_q[j-1]_q\cdots[1]_q; \qquad [0]_q!=1.
\]
Finally, for any two nonengative integers $j$ and $k$, the \emph{$q$-binomial} $\left[\begin{matrix} j \\ k \end{matrix}\right]_q$  is the element of $\mathbb{Q}[q,q^{-1}]$ defined by
\[
\left[\begin{matrix} j \\ k \end{matrix}\right]_q=\frac{[j]_q!}{[k]_q![j-k]_q!}.
\]
\end{notation}

\begin{definition}
For a fixed nonnegative integer $n\geq 2$, let $\mathbf{I}_{MOY}$ be the tensor ideal in $\mathbf{TrPD}$ generated by the following relations\footnote{By a little abuse of notation, we say that a relation $a=b$ is a generator for an ideal $I$ to mean that the element $a-b$ is a generator for that ideal.} (and their duals):
\begin{enumerate}
\item $\mathrm{ev}\circ \mathrm{coev} = [n]_q$,
i.e.,
\begin{equation}\label{moy-1}\tag{moy-1}
\iunknot=[n]_q\, \emptyset
\end{equation}
\item $S\circ S=(q+q^{-1})S$,
i.e.,
\begin{equation}\label{moy-2}\tag{moy-2}
\essecircesse=(q+q^{-1})\esse
\end{equation}
\item $(\mathrm{ev}\otimes \mathrm{id})\circ (\mathrm{id}^\vee \otimes S)\circ (\mathrm{coev}\otimes \mathrm{id})=[n-1]\mathrm{id}$, 
i.e.,
\begin{equation}\label{moy-3}\tag{moy-3}
{\xy
(0,-10);(0,10)**\crv{(0,-7)&(0,-6)&(0,-5)&(0,0)&(0,5)&(0,6)&(0,7)}
?>(.2)*\dir{>}?>(.99)*\dir{>}?>(.55)*\dir{>}
,(0,-3);(0,3)**\crv{(0,-4)&(0,-5)&(-4,-6)&(-6,0)&(-4,6)&(0,5)&(0,4)}
?>(.55)*\dir{<}
,(2,-11)*{\scriptstyle{1}}
,(2,11)*{\scriptstyle{1}}
,(2,0)*{\scriptstyle{2}}
,(-6,0)*{\scriptstyle{1}}
\endxy}\,\,\,=
[ n-1]_q\,\,{\xy
(0,-10);(0,10)**\dir{-}
?>(.95)*\dir{>}
,(2,-11)*{\scriptstyle{1}}
,(2,11)*{\scriptstyle{1}}
\endxy}
\end{equation}

\item $(\mathrm{id}\otimes \mathrm{ev}\otimes\mathrm{id}^\vee)\circ (S\otimes S^\vee) \circ (\mathrm{id}\otimes \mathrm{coev}\otimes\mathrm{id}^\vee)=\mathrm{coev}\circ\mathrm{ev}+ [n-2]_q\, \mathrm{id}\otimes \mathrm{id}^\vee$,
i.e.,
\begin{equation}\label{moy-4}\tag{moy-4}
{\xy
(-10,-10);(-10,10)**\crv{(-6,-6)&(-5,-5)&(-5,-2)&(-5,-2)&(-5,5)&(-6,6)}
?>(.05)*\dir{>}?>(.99)*\dir{>}?>(.55)*\dir{>}
,(10,-10);(10,10)**\crv{(6,-6)&(5,-5)&(5,-2)&(5,2)&(5,5)&(6,6)}
?>(.03)*\dir{<}?>(.92)*\dir{<}?>(.45)*\dir{<}
,(-5,0);(5,0)**\crv{(-5,-2)&(-5,-3)&(-4,-5)&(-2,-6)&(2,-6)&(4,-5)&(5,-3)&(5,-2)}?>(.5)*\dir{<}
,(-5,0);(5,0)**\crv{(-5,2)&(-5,3)&(-4,5)&(-2,6)&(2,6)&(4,5)&(5,3)&(5,2)}?>(.5)*\dir{>}
,(-11,-10)*{\scriptstyle{1}}
,(-11,10)*{\scriptstyle{1}}
,(11,-10)*{\scriptstyle{1}}
,(11,10)*{\scriptstyle{1}}
,(-7,0)*{\scriptstyle{2}}
,(0,-8)*{\scriptstyle{1}}
,(6.5,0)*{\scriptstyle{2}}
,(0,8)*{\scriptstyle{1}}
\endxy}
\quad
=
\xy
(-5,-10);(5,-10)**\crv{(-5,-2)&(0,-2)&(5,-2)}
?>(1)*\dir{>},
(-6,-10)*{\scriptstyle{1}},(6,-10)*{\scriptstyle{1}}
,(-5,10);(5,10)**\crv{(-5,2)&(0,2)&(5,2)}
?>(0)*\dir{<},
(-6,10)*{\scriptstyle{1}},(6,10)*{\scriptstyle{1}}
\endxy
+\, [n-2]_q
{\xy
(-5,-10);(-5,10)**\dir{-}?>(.5)*\dir{>}
,(5,-10);(5,10)**\dir{-}?>(.5)*\dir{<}
,(-7,-10)*{\scriptstyle{1}}
,(7,-10)*{\scriptstyle{1}}
,(-7,10)*{\scriptstyle{1}}
,(7,10)*{\scriptstyle{1}}
\endxy}
\end{equation}

\item $(S\otimes\mathrm{id})\circ (\mathrm{id}\otimes S) \circ (S\otimes\mathrm{id})+(\mathrm{id}\otimes S)=(\mathrm{id}\otimes S) \circ (S\otimes\mathrm{id})\circ (\mathrm{id}\otimes S)+(S\otimes\mathrm{id})$,
i.e.,
\begin{equation}\label{moy-5}\tag{moy-5}
\coso\,\,+\,\,\idone\quad \esse = \cosob\,\,+\,\, \esse\quad\idoner
\end{equation}
\end{enumerate}
The \emph{MOY category} $\mathbf{MOY}$  the monoidal category enriched in $\mathbb{Q}[q,q^{-1},(q-q^{-1})^{-1}]$-modules defined by
\[
\mathbf{MOY}_n=\mathbf{TrPD}/{\mathbf{I}_{MOY}}.
\]
\end{definition}

By the discussion in the previous chapter, $\mathbf{MOY}$ is a balanced rigid category. As we are going to show in the next section, the balanced rigid category $\mathbf{MOY}$ enjoys also an additional property: it is \emph{braided}. Therefore it is a \emph{ribbon} category in the terminology of Reshetikhin and Turaev.

\begin{remark}
Notice that $[2]_q=q^{-1}+q$, so that we could have written (\ref{moy-2}) as
\[
S\circ S=[2]_q\,S.
\]
We don't do this since, as we are going to see in the next section, the coefficient in (\ref{moy-2}) generalizes in a different way with respect to the coefficients in (\ref{moy-1}),  (\ref{moy-3}) and (\ref{moy-4}) as we move to a more general coefficient ring than $\mathbb{Q}[q,q^{-1}]$.
\end{remark}

\section{The MOY category as a braided category}

\begin{definition}\label{def:braidings}
The morphism 
\[
\sigma^+,\sigma^-\colon \uparrow^1\otimes \uparrow^1 \to \uparrow^1\otimes \uparrow^1
\]
in $\mathbf{MOY}$ are defined as (the equivalence classes of)
\[
\sigma^+=q^{n-1}\idtwo-q^n\esse = q^{n-1}\left(\,\, \idtwo-q\esse\,\,\right),
\]
\[
\sigma^-=q^{1-n}\idtwo-q^{-n}\esse = q^{1-n}\left(\,\, \idtwo-q^{-1}\esse\,\,\right),
\]
\end{definition}

\begin{proposition}\label{prop:moy-is-braided}
The category $\mathbf{MOY}$ is a rigid braided and untwisted monoidal category with the braidings defined by $\sigma^+$ and $\sigma^-$ from Definition \ref{def:braidings}.
\end{proposition}
\begin{proof}
We begin by checking the first Reidemeister move. It is the identity
\[
\xy
(2,5);(0,0.1)**\crv{(2,1)&(-2,-4)&(-4,0)&(-2,4)}?>(0)*\dir{<}
,(2,-5);(0.65,-1)**\crv{(2,-4)}
,(1,-5)*{\scriptstyle{1}}
\endxy\quad=
\quad
\xy
(0,-5);(0,5)**\dir{-}?>(1)*\dir{>}
,(-1,-5)*{\scriptstyle{1}}
\endxy
\]
i.e., the identity 
\[
(\mathrm{ev}\otimes \mathrm{id})\circ (\mathrm{id}^\vee\otimes \sigma^+)\circ (\mathrm{coev}\otimes \mathrm{id})=\mathrm{id}
\]
(and similarly with $\sigma^+$ replaced by $\sigma^-$). By definition of the braidings, evaluations and coevaluations in $\mathbf{MOY}$ this corresponds to the identity
\[
\left(\evop\,\idone\,\,\right)\circ\left(\,\,\idonedual\,\,\otimes \left(
q^{n-1}\idtwo-q^n\esse\right)\right)\circ\left(\coevop\,\,\idone\right)=\idone
\]
We compute
\begin{align*}
&\left(\evop\,\idone\,\,\right)\circ\left(\,\,\idonedual\,\,\otimes \left(
q^{n-1}\idtwo-q^n\esse\right)\right)\circ\left(\coevop\,\,\idone\right)\\
&\qquad=\left(\evop\,\idone\,\,\right)\circ\left(q^{n-1} \,\,\idonedual\,\qquad\idtwo -q^n\idonedual\quad\esse\right)\circ\left(\coevop\,\,\idone\right)\\
&\qquad=
 q^{n-1}\left(
\,\,\iunknot \quad {\xy
(0,-10);(0,10)**\dir{-}
?>(.95)*\dir{>}
,(2,-11)*{\scriptstyle{1}}
,(2,11)*{\scriptstyle{1}}
\endxy} \,\,-q\,\,
{\xy
(0,-10);(0,10)**\crv{(0,-7)&(0,-6)&(0,-5)&(0,0)&(0,5)&(0,6)&(0,7)}
?>(.2)*\dir{>}?>(.99)*\dir{>}?>(.55)*\dir{>}
,(0,-3);(0,3)**\crv{(0,-4)&(0,-5)&(-4,-6)&(-6,0)&(-4,6)&(0,5)&(0,4)}
?>(.55)*\dir{<}
,(2,-11)*{\scriptstyle{1}}
,(2,11)*{\scriptstyle{1}}
,(-2,0)*{\scriptstyle{2}}
,(-6,0)*{\scriptstyle{1}}
\endxy}
\right)=
q^{n-1}\left(\, [n]_q\idone\,\,-q[n-1]_q\,\,\idone\,\,\right)\\
&\qquad=(q^{n-1}[n]_q-q^n[n-1]_q)\,\,\idone
\,\quad=\quad\,\idone\,
\end{align*}
where we have used relations (\ref{moy-1}) and (\ref{moy-3}), and the identity
\[
q^{n-1}[n]_q-q^n[n-1]_q=\frac{q^{n-1}(q^n-q^{-n})-q^n(q^{n-1}-q^{1-n})}{q-q^{-1}}=\frac{-q^{-1}+q}{q-q^{-1}}=1.
\]
When $\sigma^+$ is replaced by $\sigma^-$ we need to prove
\[
(\mathrm{ev}\otimes \mathrm{id})\circ (\mathrm{id}^\vee\otimes \sigma^-)\circ (\mathrm{coev}\otimes \mathrm{id})=\mathrm{id},
\]
i.e, by reasoning as above,
\[
 q^{1-n}\left(
\,\,\iunknot \quad {\xy
(0,-10);(0,10)**\dir{-}
?>(.95)*\dir{>}
,(2,-11)*{\scriptstyle{1}}
,(2,11)*{\scriptstyle{1}}
\endxy} \,\,-q^{-1}\,\,
{\xy
(0,-10);(0,10)**\crv{(0,-7)&(0,-6)&(0,-5)&(0,0)&(0,5)&(0,6)&(0,7)}
?>(.2)*\dir{>}?>(.99)*\dir{>}?>(.55)*\dir{>}
,(0,-3);(0,3)**\crv{(0,-4)&(0,-5)&(-4,-6)&(-6,0)&(-4,6)&(0,5)&(0,4)}
?>(.55)*\dir{<}
,(2,-11)*{\scriptstyle{1}}
,(2,11)*{\scriptstyle{1}}
,(-2,0)*{\scriptstyle{2}}
,(-6,0)*{\scriptstyle{1}}
\endxy}
\right)=\,\,\idone
\]
We have
\begin{align*}
 q^{1-n}\left(
\,\,\iunknot \quad {\xy
(0,-10);(0,10)**\dir{-}
?>(.95)*\dir{>}
,(2,-11)*{\scriptstyle{1}}
,(2,11)*{\scriptstyle{1}}
\endxy} \,\,-q^{-1}\,\,
{\xy
(0,-10);(0,10)**\crv{(0,-7)&(0,-6)&(0,-5)&(0,0)&(0,5)&(0,6)&(0,7)}
?>(.2)*\dir{>}?>(.99)*\dir{>}?>(.55)*\dir{>}
,(0,-3);(0,3)**\crv{(0,-4)&(0,-5)&(-4,-6)&(-6,0)&(-4,6)&(0,5)&(0,4)}
?>(.55)*\dir{<}
,(2,-11)*{\scriptstyle{1}}
,(2,11)*{\scriptstyle{1}}
,(-2,0)*{\scriptstyle{2}}
,(-6,0)*{\scriptstyle{1}}
\endxy}
\right)&=
q^{1-n}\left(\, [n]_q\idone\,\,-q^{-1}[n-1]_q\,\,\idone\,\,\right)\\
&=(q^{1-n}[n]_q-q^{-n}[n-1]_q)\,\,\idone
\,=\,\idone\,
\end{align*}
where again we used relations (\ref{moy-1}) and (\ref{moy-3}), and the  now the identity
\[
q^{1-n}[n]_q-q^{-n}[n-1]_q=\frac{q^{1-n}(q^n-q^{-n})-q^{-n}(q^{n-1}-q^{1-n})}{q-q^{-1}}=\frac{q-q^{-1}}{q-q^{-1}}=1.
\]
\vskip .7 cm 

Next, we check the second Reidemeister move of the first kind. This is the identity
\[
\xy
(3,-7);(3,7)**\crv{(3,-5)&(3,-4)&(-2,0)&(3,4)&(3,5)}?>(.99)*\dir{>}
,(-3,-7);(-0.3,-2)**\crv{(-3,-5)&(-3,-4)}
,(-3,7);(-0.3,2)**\crv{(-3,5)&(-3,4)}?>(0)*\dir{<}
,(.6,-1.3);(0.6,1.4)**\crv{(3,0)}
,(-4,-7.5)*{\scriptstyle{1}}
,(4,-7.5)*{\scriptstyle{1}}
\endxy\quad = 
\quad
\xy
(-3,-7);(-3,7)**\dir{-}?>(1)*\dir{>}
,(3,-7);(3,7)**\dir{-}?>(1)*\dir{>}
,(-4,-7.5)*{\scriptstyle{1}}
,(4,-7.5)*{\scriptstyle{1}}
\endxy
\]
and so the identity 
\[
\sigma^+\circ \sigma^-= \mathrm{id}_2
\]
(as well as with the role of $\sigma^+$ and $\sigma^-$ interchanged. In $\mathbf{MOY}$ this is the identity
\[
\left(q^{n-1}\idtwo-q^n\esse\right)\circ\left(q^{1-n}\idtwo-q^{-n}\esse\right)=\idtwo
\]
\vskip .7 cm

We compute
\begin{align*}
&\left(q^{n-1}\idtwo-q^n\esse\right)\circ \left(q^{1-n}\idtwo-q^{-n}\esse\right)\\
&\qquad=\idtwo- q^{-1}\esse-q \esse+\essecircesse\\
&\qquad=\idtwo- (q^{-1}+q) \esse+(q+q^{-1})\esse\\
&\qquad=\idtwo,
\end{align*}
where we used equation (\ref{moy-2}).

Next we come to the second Reidemeister move of the second kind is the identity
\[
\xy
(-5,10);(5,10)**\crv{(-5,2)&(0,-6)&(5,2)}
?>(0)*\dir{<},
?>(.71)*{\color{white}\bullet},
?>(.30)*{\color{white}\bullet}
,(-6,10)*{\scriptstyle{1}},(6,10)*{\scriptstyle{1}}
,(-5,-10);(5,-10)**\crv{(-5,-2)&(0,6)&(5,-2)}
?>(1)*\dir{>}
,(-6,-10)*{\scriptstyle{1}},(6,-10)*{\scriptstyle{1}}
\endxy\,\,
=
\,\,\xy
(-5,-10);(5,-10)**\crv{(-5,-2)&(0,-2)&(5,-2)}
?>(1)*\dir{>}
,(-6,-10)*{\scriptstyle{1}},(6,-10)*{\scriptstyle{1}}
,(-5,10);(5,10)**\crv{(-5,2)&(0,2)&(5,2)}
?>(0)*\dir{<}
,(-6,10)*{\scriptstyle{1}},(6,10)*{\scriptstyle{1}}
\endxy
\]
and so the identity 
\[
(\mathrm{id}\otimes \mathrm{ev}\otimes\mathrm{id}^\vee)\circ (\sigma^+\otimes (\sigma^-)^\vee) \circ (\mathrm{id}\otimes \mathrm{coev}\otimes\mathrm{id}^\vee)= \mathrm{coev} \circ \mathrm{ev}.
\]
In $\mathbf{MOY}$ this is the identity
\begin{align*}
\left(\idone\quad \ev\,\,\idonedual\right)&\circ\left(\left(q^{n-1}\idtwo-q^n\esse\right)\otimes\right.\\
& \left.\left(q^{1-n}\idtwodual-q^{-n}\essedual\right)\right)
\circ \left(\idone\quad \coev\,\, \idonedual\right)\\
&=
\,\,{\xy
(-5,-10);(5,-10)**\crv{(-5,-2)&(0,-2)&(5,-2)}
?>(1)*\dir{>}
,(-6,-10)*{\scriptstyle{1}},(6,-10)*{\scriptstyle{1}}
,(-5,10);(5,10)**\crv{(-5,2)&(0,2)&(5,2)}
?>(0)*\dir{<}
,(-6,10)*{\scriptstyle{1}},(6,10)*{\scriptstyle{1}}
\endxy}
\end{align*}
By using equations (\ref{moy-1}), (\ref{moy-2}) and (\ref{moy-4}), we compute
\begin{align*}
\left(\idone\quad \ev\,\,\idonedual\right)&\circ\left(\left(q^{n-1}\idtwo-q^n\esse\right)\otimes\right.\\
& \left.\left(q^{1-n}\idtwodual-q^{-n}\essedual\right)\right)
\circ \left(\idone\quad \coev\,\, \idonedual\right)\\
&=
\left(\idone\quad \ev\,\,\idonedual\right)\circ\left(\idtwo\,\quad\idtwodual-q^{-1}\idtwo\,\,\essedual\right.\\
& \left.-q\esse\,\,\idtwodual+\esse\,\,\essedual\right)
\circ \left(\idone\quad \coev\,\, \idonedual\right)\\
\end{align*}

\begin{align*}
\phantom{mmm}&
{\xy
(-10,-10);(-10,10)**\dir{-}?>(.5)*\dir{>}
,(-12,-11)*{\scriptstyle{1}}
,(-12,11)*{\scriptstyle{1}}
,(-2,-5);(-2,-5)**\crv{(3,-5)&(3,5)&(-7,5)&(-7,-5)}
?>(.3)*\dir{<}
,(3.5,1)*{\scriptstyle{1}}
,(6,-10);(6,10)**\dir{-}?>(.5)*\dir{<}
,(8,-11)*{\scriptstyle{1}}
,(8,11)*{\scriptstyle{1}}
\endxy}
\,
-q^{-1}
{\xy
(-10,-10);(-10,10)**\dir{-}?>(.5)*\dir{>}
,(-12,-11)*{\scriptstyle{1}}
,(-12,11)*{\scriptstyle{1}}
,(0,10);(0,-10)**\crv{(0,7)&(0,6)&(0,5)&(0,0)&(0,-5)&(0,-6)&(0,-7)}
?>(.2)*\dir{>}?>(.99)*\dir{>}?>(.55)*\dir{>}
,(0,3);(0,-3)**\crv{(0,4)&(0,5)&(-4,6)&(-6,0)&(-4,-6)&(0,-5)&(0,-4)}
?>(.55)*\dir{<}
,(2,-11)*{\scriptstyle{1}}
,(2,11)*{\scriptstyle{1}}
,(2,0)*{\scriptstyle{2}}
,(-6,0)*{\scriptstyle{1}}
\endxy}\,-\,
q\,\,\,
{\xy
(0,-10);(0,10)**\crv{(0,-7)&(0,-6)&(0,-5)&(0,0)&(0,5)&(0,6)&(0,7)}
?>(.2)*\dir{>}?>(.99)*\dir{>}?>(.55)*\dir{>}
,(0,-3);(0,3)**\crv{(0,-4)&(0,-5)&(4,-6)&(6,0)&(4,6)&(0,5)&(0,4)}
?>(.55)*\dir{<}
,(2,-11)*{\scriptstyle{1}}
,(2,11)*{\scriptstyle{1}}
,(1.5,0)*{\scriptstyle{2}}
,(7,0)*{\scriptstyle{1}}
,(10,-10);(10,10)**\dir{-}?>(.5)*\dir{<}
,(12,-11)*{\scriptstyle{1}}
,(12,11)*{\scriptstyle{1}}
\endxy}
\,\,+\,\, 
{\xy
(-10,-10);(-10,10)**\crv{(-6,-6)&(-5,-5)&(-5,-2)&(-5,-2)&(-5,5)&(-6,6)}
?>(.05)*\dir{>}?>(.99)*\dir{>}?>(.55)*\dir{>}
,(10,-10);(10,10)**\crv{(6,-6)&(5,-5)&(5,-2)&(5,2)&(5,5)&(6,6)}
?>(.03)*\dir{<}?>(.92)*\dir{<}?>(.45)*\dir{<}
,(-5,0);(5,0)**\crv{(-5,-2)&(-5,-3)&(-4,-5)&(-2,-6)&(2,-6)&(4,-5)&(5,-3)&(5,-2)}?>(.5)*\dir{<}
,(-5,0);(5,0)**\crv{(-5,2)&(-5,3)&(-4,5)&(-2,6)&(2,6)&(4,5)&(5,3)&(5,2)}?>(.5)*\dir{>}
,(-11,-10)*{\scriptstyle{1}}
,(-11,10)*{\scriptstyle{1}}
,(11,-10)*{\scriptstyle{1}}
,(11,10)*{\scriptstyle{1}}
,(-7,0)*{\scriptstyle{2}}
,(0,-8)*{\scriptstyle{1}}
,(6.5,0)*{\scriptstyle{2}}
,(0,8)*{\scriptstyle{1}}
\endxy}\\
&=
[n]_q\idone\,\qquad\idonedual-(q+q^{-1})[n-1]_q+\idone\,\qquad\idonedual
\,\,+\,\,
{\xy
(-5,-10);(5,-10)**\crv{(-5,-2)&(0,-2)&(5,-2)}
?>(1)*\dir{>},
(-6,-10)*{\scriptstyle{1}},(6,-10)*{\scriptstyle{1}}
,(-5,10);(5,10)**\crv{(-5,2)&(0,2)&(5,2)}
?>(0)*\dir{<},
(-6,10)*{\scriptstyle{1}},(6,10)*{\scriptstyle{1}}
\endxy}
\,\,+\, [n-2]_q
{\xy
(-5,-10);(-5,10)**\dir{-}?>(.5)*\dir{>}
,(5,-10);(5,10)**\dir{-}?>(.5)*\dir{<}
,(-7,-10)*{\scriptstyle{1}}
,(7,-10)*{\scriptstyle{1}}
,(-7,10)*{\scriptstyle{1}}
,(7,10)*{\scriptstyle{1}}
\endxy}\\
&=\left([n]_q-(q+q^{-1})[n-1]_q+[n-2]_q\right)\idone\,\qquad\idonedual\,\,+\, {\xy
(-5,-10);(5,-10)**\crv{(-5,-2)&(0,-2)&(5,-2)}
?>(1)*\dir{>},
(-6,-10)*{\scriptstyle{1}},(6,-10)*{\scriptstyle{1}}
,(-5,10);(5,10)**\crv{(-5,2)&(0,2)&(5,2)}
?>(0)*\dir{<},
(-6,10)*{\scriptstyle{1}},(6,10)*{\scriptstyle{1}}
\endxy}\\
&=
{\xy
(-5,-10);(5,-10)**\crv{(-5,-2)&(0,-2)&(5,-2)}
?>(1)*\dir{>},
(-6,-10)*{\scriptstyle{1}},(6,-10)*{\scriptstyle{1}}
,(-5,10);(5,10)**\crv{(-5,2)&(0,2)&(5,2)}
?>(0)*\dir{<},
(-6,10)*{\scriptstyle{1}},(6,10)*{\scriptstyle{1}}
\endxy},
\end{align*}
by the identities
 \[
 [n]_q+[n-2]_q=\frac{q^n+q^{n-2}-q^{2-n}-q^{-n}}{q-q^{-1}}
 \]
 and
 \[
 (q +q^{-1})[n-1]_q= (q +q^{-1})\frac{q^{n-1}-q^{1-n}}{q-q^{-1}}=\frac{q^{n}-q^{2-n}+q^{n-2}-q^{-n}}{q-q^{-1}},
 \]
 giving
 \[
 [n]_q -(q +q^{-1})[n-1]_q+ [n-2]_q=0.
 \]
\vskip .7 cm

 Finally, we check the third Reidemeister move. This is the identity
 
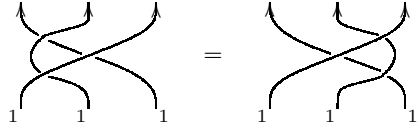
\[
\xy
(-9,-7);(9,7)**\crv{(-9,-5)&(-9,-3)&(9,3)&(9,5)}?>(.99)*\dir{>}
,(9,-7);(1,-.3)**\crv{(9,-5)&(9,-3)}
,(-1,.3);(-5.3,2)**\crv{(-3,1)}
,(-6.7,2.8);(-9,7)**\crv{(-9,4)&(-9,5)}?>(.99)*\dir{>}
,(0,-7);(-5.3,-2.8)**\crv{(0,-5)&(0,-4)}
,(0,7);(-6.1,2.4)**\crv{(0,6)&(0,5)&(0,4)&(-6,2.8)}?>(0)*\dir{<}
,(-6.5,-2.2);(-6.1,2.4)**\crv{(-9,0)}
,(-10,-8)*{\scriptstyle{1}}
,(-1,-8)*{\scriptstyle{1}}
,(10,-8)*{\scriptstyle{1}}
\endxy
\quad =
\quad
\xy
(-9,-7);(9,7)**\crv{(-9,-5)&(-9,-3)&(9,3)&(9,5)}?>(.99)*\dir{>}
,(5.6,-2);(1,-.3)**\crv{(3,-1)}
,(9,-7);(6.7,-2.8)**\crv{(9,-5)&(9,-4)}
,(-1,.3);(-9,7)**\crv{(-9,3)&(-9,5)}?>(.99)*\dir{>}
,(0,-7);(6,-2.6)**\crv{(0,-5)&(0,-4)&(6,-3)}
,(0,7);(5.3,2.8)**\crv{(0,6)&(0,5)&(0,4)}?>(0)*\dir{<}
,(6,-2.6);(6.5,2.2)**\crv{(9,0)}
,(-10,-8)*{\scriptstyle{1}}
,(-1,-8)*{\scriptstyle{1}}
,(10,-8)*{\scriptstyle{1}}
\endxy
\]
i.e., the identity 
\begin{equation}\label{eq:reidemeister3-verify}
(\sigma^+\otimes\mathrm{id})\circ (\mathrm{id}\otimes \sigma^+) \circ (\sigma^+\otimes\mathrm{id})= (\mathrm{id}\otimes \sigma^+) \circ (\sigma^+\otimes\mathrm{id})\circ (\mathrm{id}\otimes \sigma^+) 
\end{equation}
(as well as with the role of $\sigma^+$ and $\sigma^-$ interchanged). In $\mathbf{MOY}$, this is the identity
\begin{align*}
&\left(\left(q^{n-1}\idtwo-q^n\esse\right)\otimes \,\,\idone\,\,\right)\circ \left(\,\,\idone\,\,\otimes \left(q^{n-1}\idtwo-q^n\esse\right)\right) \\
&\qquad\qquad\qquad\qquad
\circ \left(\left(q^{n-1}\idtwo-q^n\esse\right)\otimes\,\,\idone\,\,\right)\\
&\qquad\qquad= \left(\,\,\idone\,\,\otimes \left(q^{n-1}\idtwo-q^n\esse\right)\right) \circ \left(\left(q^{n-1}\idtwo-q^n\esse\right)\otimes\,\,\idone\,\,\right)\\
&\qquad\qquad\qquad\qquad\circ \left(\,\,\idone\,\,\otimes \left(q^{n-1}\idtwo-q^n\esse\right)\right). 
\end{align*}

We compute the left hand side of equation (\ref{eq:reidemeister3-verify}) first. We have
\begin{align*}
&\left(\left(q^{n-1}\idtwo-q^n\esse\right)\otimes \,\,\idone\,\,\right)\circ \left(\,\,\idone\,\,\otimes \left(q^{n-1}\idtwo-q^n\esse\right)\right) \\
&\qquad\qquad\qquad\qquad
\circ \left(\left(q^{n-1}\idtwo-q^n\esse\right)\otimes\,\,\idone\,\,\right)\\
\end{align*}
\begin{align*}
&\qquad\qquad=
\left(q^{n-1}\idtwo\qquad\idone\,\,-q^n\esse \,\,\idone\,\,\right)\circ \left(q^{n-1}\,\idone\,\qquad\idtwo-q^n\,\idone\,\,\esse\right)\\
&\qquad\qquad\qquad\qquad
\circ \left(q^{n-1}\idtwo\qquad\idone\,\,-q^n\esse \,\,\idone\,\,\right)\\
&\qquad\qquad=
q^{3n-3}\idtwo\qquad\idone\,\, -q^{3n-2}\,\idone\,\,\esse\,\,-2q^{3n-2}\, \esse\,\idone\,\,\\
&\qquad\qquad\qquad\qquad+
q^{3n-1}\zorro+q^{3n-1}\zorrob+q^{3n-1}\essecircesse\,\idone\,\,-q^{3n}\coso\\
&\qquad\qquad=
q^{3n-3}\idtwo\qquad\idone\,\, -q^{3n-2}\,\idone\,\,\esse\,\,-2q^{3n-2}\, \esse\,\idone\,\,\\
&\qquad\qquad\qquad\qquad+
q^{3n-1}\zorro+q^{3n-1}\zorrob+q^{3n-1}(q+q^{-1})\esse\,\idone\,\,-q^{3n}\coso
\end{align*}
\begin{align*}
&\qquad\qquad=
q^{3n-3}\idtwo\qquad\idone\,\, -q^{3n-2}\,\idone\,\,\esse\,\,-q^{3n-2}\, \esse\,\idone\,\,\\
&\qquad\qquad\qquad\qquad+
q^{3n-1}\zorro+q^{3n-1}\zorrob+q^{3n}\left(\esse\,\idone\,\,-\coso\right)
\end{align*}
Computing in the same way the right hand side of equation (\ref{eq:reidemeister3-verify}), we find that it equals
\begin{align*}
&q^{3n-3}\idtwo\qquad\idone\,\, -q^{3n-2}\, \esse\,\idone\,\,-q^{3n-2}\,\idone\,\,\esse\,\,\\
&\qquad\qquad\qquad\qquad+
q^{3n-1}\zorrob+q^{3n-1}\zorro+q^{3n}\left(\,\idone\,\esse\,-\cosob\right)
\end{align*}
So, in order for the left and the right hand side of equation (\ref{eq:reidemeister3-verify}) to be equal, we need
\[
\esse\,\idone\,\,-\coso=\,\idone\,\esse\,-\cosob
\]
and this is precisely the content of equation (\ref{moy-5}).
\end{proof}

\subsection{The level $n$ Jones polynomial via the MOY category}\label{sec:Jones-via-MOY}

As $\mathbf{MOY}$ is a rigid braided untwisted monoidal category enriched in $\mathbb{Q}[q,q^{-1},(q-q^{-1})^{-1}]$-modules it receives a distinguished morphism from $\mathbf{TD}$ (enriched in $\mathbb{Q}[q,q^{-1},(q-q^{-1})^{-1}]$-modules), by Theorem \ref{reshetikin-turaev}. More precisely, there exists a unique rigid braided monoidal functor
\[
Z_{MOY}\colon \mathbf{TD}\to \mathbf{MOY}
\] 
such that $Z(\uparrow)=\uparrow^1$.

\begin{proposition}
The functor $Z_{MOY}$ factors through $\mathbf{Jones}_n$. 
\end{proposition}
\begin{proof}
By definition of $\mathbf{Jones}_n$ (Definition \ref{def:homfly-cat}) we only need to prove that the braidings, the evaluation and the coevaluation of $\uparrow^1$ in $\mathbf{MOY}$ satisfy the skein relation
\[
q^{-n}\ijovercrossing - q^n\ijundercrossing = (q^{-1}-q)\ijnoncrossing
\]
as well as the normalization condition
\[
\iunknot=[n]_q\emptyset.
\]
The normalization condition is ensured by equation (\ref{moy-1}). Concerning the skein relation, we have
\begin{align*}
q^{-n}\ijovercrossing &- q^n\ijundercrossing = \\
&=q^{-n}\left(q^{n-1}\idtwo-q^n\esse \right)-q^n\left(
q^{1-n}\idtwo-q^{-n}\esse \right)\\
&=(q^{-1}-q)\idtwo\,\,.
\end{align*}
\end{proof}
\begin{corollary}\label{cor:jones}
The level $n$ Jones polynomial of a link $\Gamma$ is computed via the MOY category by the equation
\[
Z_{MOY}(\Gamma)=Jones_n(\Gamma)\, \emptyset.
\]
\end{corollary}
\begin{remark}
The skein relation used in the above proof actually is obtained from the skein relation in Definition \ref{def:homfly-cat} by the change of variables $q\leftrightarrow q^{-1}$. This is due to the fact that different conventions in the definition of the Jones polynomial, precisely related by the involution  $q\leftrightarrow q^{-1}$, are adopted in the literature.
\end{remark}

\subsection{The Jones polynomial of the 2-strand $k$-fold braid link}

As an example, we compute the Jones polynomial of the 2-strand $k$-fold braid link via the $\mathbf{MOY}$ category. By definition, such a link is the closure of a $k$-fold braiding: 
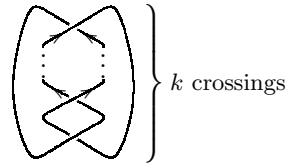
\[
\trsplusk{k}
\]
The heart of the computation consist in computing the $k$-fold braiding $(\sigma^+)^k\colon \uparrow^1\otimes \uparrow^1\to \uparrow^1\otimes \uparrow^1$ in $\mathbf{MOY}$.
Before computing $(\sigma^+)^k$ for an arbitrary $k$, let us compute $(\sigma^+)^2$ as a warm up. 
\begin{lemma}\label{lemma:k=2}
We have
\[
(\sigma^+)^2=q^{2(n-1)}\left(
\,\,\idtwo\,\,
- q(1-q^2)\,
\esse\,\,
\right)
\]
\end{lemma}
\begin{proof}
As
\[
\sigma^+=q^{n-1}\left(\,\,\idtwo\,-q \esse\right), 
\]
we have 
\begin{align*}
(\sigma^+)^2&=q^{2n-2}\left(\,\,\idtwo\,-q\esse\,\right)\circ \left(\,\,\idtwo\,-q\esse\,\right)\\
&=q^{2n-2}\left(\,\,\idtwo\,\,-2q\esse+q^2\essecircesse\,\right)\\
&=q^{2n-2}\left(\,\,\idtwo\,\,-2q\esse+q^2(q^{-1}+q)\esse\,\right)\\
&=q^{2(n-1)}\left(\,\,\idtwo\,\,-q(1-q^2)\esse\,\right).
\end{align*}
\end{proof}

\begin{lemma}\label{lemma:jones-k-braided}
For any $k\geq 2$ one has
\begin{align*}
(\sigma^+)^k&=q^{k(n-1)}\left(
\,\,\idtwo\,
- q(1-q^2+q^4-q^6+\cdots-(-1)^{k-1}q^{2(k-1)})\,
\esse\,\,
\right)\\
&=q^{k(n-1)}\left(
\,\,\idtwo\,
- q\frac{1-(-q^2)^k}{1+q^2}\,
\esse\,
\right).
\end{align*}
\end{lemma}
\begin{proof}
The proof is by induction on $k$. The base of the induction is the case $k=2$ proved in Lemma \ref{lemma:k=2} above. Now assume the statement is true for $k=k_0$ and let us prove it for $k_0+1$. We have
\begin{align*}
(\sigma^+)^{k_0+1}&=(\sigma^+)^{k_0}\circ \sigma^+\\
&=q^{k_0(n-1)}q^{n-1}\left(\,\,\idtwo\,\,- q\frac{1-(-q^2)^{k_0}}{1+q^2}\esse\,\right))\circ\left(\,\,\idtwo\,\,-q\esse\,\right)\\
&=q^{(k_0+1)(n-1)}\left(\,\,\idtwo\,-q\esse-q\frac{1-(-q^2)^{k_0}}{1+q^2}\esse+q^2\frac{1-(-q^2)^{k_0}}{1+q^2}\essecircesse\,\right)\\
&=q^{(k_0+1)(n-1)}\left(\,\,\idtwo\,-q\esse-q\frac{1-(-q^2)^{k_0}}{1+q^2}\esse+q^2(q^{-1}+q)\frac{1-(-q^2)^{k_0}}{1+q^2}\esse\,\right)\\
&=q^{(k_0+1)(n-1)}\left(\,\,\idtwo\,-q\esse+q^3\frac{1-(-q^2)^{k_0}}{1+q^2}\esse\,\right)\\
&=q^{(k_0+1)(n-1)}\left(\,\,\idtwo\,+\frac{-q-q^3+q^3-q^3(-q^2)^{k_0}}{1+q^2}\esse\,\right)\\
&=q^{(k_0+1)(n-1)}\left(\,\,\idtwo\,-q\frac{1-(-q^2)^{k_0+1}}{1+q^2}\esse\,\right)
\end{align*}
\end{proof}

\begin{lemma}\label{lemma:tresse}
We have
\[
\tresse=[n]_q[n-1]_q\,\emptyset
\]
\end{lemma}
\begin{proof}
By equations (\ref{moy-3}) and (\ref{moy-1}), we have
\begin{align*}
\tresse&=
[n-1]_q \iunknot\\
&=
[n-1]_q[n]_q\,\emptyset
\end{align*}
\end{proof}

As an immediate corollary, we recover the known expression for the Jones polynomial of the 2-strand $k$-fold braid link \cite{fw}.
\begin{corollary}
For any $k\geq 2$, the Jones polynomial of the 2-strand $k$-fold braid link is
\[
q^{k(n-1)}[n]_q\left( \frac{q^2-(-q^2)^k}{1+q^2}[n]_q+q^{1-n}\frac{1-(-q^2)^k}{1+q^2}\right)
\]
\end{corollary}
\begin{proof}
By Lemma \ref{lemma:jones-k-braided} ad Lemma \ref{lemma:tresse}, we have
\begin{align*}
Z_{MOY}\left(\trsplusk{k}\right)&=
q^{k(n-1)}\left(
\,\,\twocircles\,
- q\frac{1-(-q^2)^k}{1+q^2}\,
\tresse\,
\right)\\
&=q^{k(n-1)}\left(([n]_q)^2 \,\emptyset - q\frac{1-(-q^2)^k}{1+q^2}[n]_q[n-1]_q\,\emptyset\right)\\
&=q^{k(n-1)}[n]_q\left([n]_q  - q\frac{1-(-q^2)^k}{1+q^2}[n-1]_q\right)\,\emptyset\\
&=q^{k(n-1)}[n]_q\left([n]_q  - \frac{1-(-q^2)^k}{1+q^2}([n]_q-q^{1-n})\right)\,\emptyset\\
&=q^{k(n-1)}[n]_q\left( \frac{q^2-(-q^2)^k}{1+q^2}[n]_q+q^{1-n}\frac{1-(-q^2)^k}{1+q^2}\right)\,\emptyset
\end{align*}
where we used equation (\ref{moy-1}). The conclusion now follows from Corollary \ref{cor:jones}.
\end{proof}

\section{A MOY category for the HOMFLY-PT polynomial}
Looking at the proof of the $\mathbf{MOY}$ category being braided one may notice that the ring $\mathbb{Q}[q,q^{-1},(q-q^{-1})^{-1}]$ actually plays very little role. What really did matter in the proof have been the algebraic identities
\[
q^{n-1}[n]-q^n[n-1]=1
\]
\[
q^{1-n}[n]-q^{-n}[n-1]=1
\]
\[
[n] -(q +q^{-1})[n-1]+ [n-2]=0.
\]
This paves the way to the following generalization. 
\begin{notation}
Let $R$ be a commutative ring, and let $\alpha$ and $z$ be elements in $R$ such that $\alpha,z\in R^\times$, and let $\zeta\in R^\times$ be such that $\zeta^{-1}-\zeta=z$. We can always assume that such an element exists, up to passing to a quadratic extension of $R$. Finally, let $n$ be a fixed positive integer. For any nonengative integer $k$, the symbol $[k]_{\alpha,\zeta}$ will denote the element
\[
[k]_{\alpha,\zeta}=\frac{\alpha^{-1} (-\zeta)^{k-n}-\alpha(-\zeta)^{n-k}}{\zeta-\zeta^{-1}}.
\]
of $R$.
\end{notation}
\begin{lemma}
The following identities hold:
\begin{align}
\label{prima-eq}
\zeta^{-1}[n]_{\alpha,\zeta}+[n-1]_{\alpha,\zeta}&=\alpha\\
\label{seconda-eq}
\zeta[n]_{\alpha,\zeta}+[n-1]_{\alpha,\zeta}&=\alpha^{-1}\\
\label{terza-eq}
[n]_{\alpha,\zeta}+(\zeta +\zeta^{-1})[n-1]_{\alpha,\zeta}+ [n-2]_{\alpha,\zeta}&=0.
\end{align}
\end{lemma}
\begin{proof}
We compute
\[
\zeta^{-1}[n]_{\alpha,\zeta}+[n-1]_{\alpha,\zeta}=\frac{\alpha^{-1} \zeta^{-1}-\alpha\zeta^{-1}}{\zeta-\zeta^{-1}}-\frac{\alpha^{-1} \zeta^{-1}-\alpha\zeta}{\zeta-\zeta^{-1}}=\alpha
\]
\[
\zeta[n]_{\alpha,\zeta}+[n-1]_{\alpha,\zeta}=\frac{\alpha^{-1} \zeta -\alpha\zeta}{\zeta-\zeta^{-1}}-\frac{\alpha^{-1} \zeta^{-1}-\alpha\zeta}{\zeta-\zeta^{-1}}=\alpha^{-1}
\]
and
\begin{align*}
[n]_{\alpha,\zeta} &+(\zeta +\zeta^{-1})[n-1]_{\alpha,\zeta}+ [n-2]_{\alpha,\zeta}=\frac{\alpha^{-1} -\alpha}{\zeta-\zeta^{-1}}-\zeta\frac{\alpha^{-1} \zeta^{-1}-\alpha\zeta}{\zeta-\zeta^{-1}}\\
&-\zeta^{-1}\frac{\alpha^{-1} \zeta^{-1}-\alpha\zeta}{\zeta-\zeta^{-1}}+\frac{\alpha^{-1} \zeta^{-2}-\alpha\zeta^{2}}{\zeta-\zeta^{-1}}\\
&=\frac{\alpha^{-1} -\alpha}{\zeta-\zeta^{-1}}-\frac{\alpha^{-1} -\alpha\zeta^2}{\zeta-\zeta^{-1}}
-\frac{\alpha^{-1} \zeta^{-2}-\alpha}{\zeta-\zeta^{-1}}+\frac{\alpha^{-1} \zeta^{-2}-\alpha\zeta^{2}}{\zeta-\zeta^{-1}}\\
&=0.
\end{align*}
\end{proof}

The above relations suggest that it should be possible to define a generalization $\mathbf{MOY}_{\alpha,\zeta}$ of the category $\mathbf{MOY}$, related to the HOMFLY-PT polynomial with parameters $\alpha,z$ the same way $\mathbf{MOY}$ is related to the level $n$ Jones polynomial. That it should be possible to define such a rigid braided monoidal untwisted tensor category can be read through the lines of the seminal paper \cite{moy}, on which this chapter is mainly based. However details are not provided in \cite{moy}, so we give a detailed construction here. The construction closely follows the construction of the category $\mathbf{MOY}$ in the previous sections, so the reader can safely skip all the proofs: they are only provided for the sake of completeness.

\begin{definition}
Let $R$, $\alpha$, $z$ and $\zeta$ as above.
For a fixed nonnegative integer $n\geq 2$, let $\mathbf{I}_{MOY_{\alpha,\zeta}}$ be the tensor ideal in $\mathbf{TrPD}$ generated by the following relations (and their duals):
\begin{enumerate}
\item $\mathrm{ev}\circ \mathrm{coev} = [n]_{\alpha,\zeta}$,
i.e.,
\begin{equation}\label{moy-1az}\tag{moy-1${}_{\alpha,\zeta}$}
\iunknot=[n]_{\alpha,\zeta}\, \emptyset
\end{equation}
\item $S\circ S=-(\zeta+\zeta^{-1})S$,
i.e.,
\begin{equation}\label{moy-2az}\tag{moy-2${}_{\alpha,\zeta}$}
\essecircesse=-(\zeta+\zeta^{-1})\esse
\end{equation}
\item $(\mathrm{ev}\otimes \mathrm{id})\circ (\mathrm{id}^\vee \otimes S)\circ (\mathrm{coev}\otimes \mathrm{id})=[n-1]_{\alpha,\zeta}\,\mathrm{id}$, 
i.e.,
\begin{equation}\label{moy-3az}\tag{moy-3${}_{\alpha,\zeta}$}
{\xy
(0,-10);(0,10)**\crv{(0,-7)&(0,-6)&(0,-5)&(0,0)&(0,5)&(0,6)&(0,7)}
?>(.2)*\dir{>}?>(.99)*\dir{>}?>(.55)*\dir{>}
,(0,-3);(0,3)**\crv{(0,-4)&(0,-5)&(-4,-6)&(-6,0)&(-4,6)&(0,5)&(0,4)}
?>(.55)*\dir{<}
,(2,-11)*{\scriptstyle{1}}
,(2,11)*{\scriptstyle{1}}
,(2,0)*{\scriptstyle{2}}
,(-6,0)*{\scriptstyle{1}}
\endxy}\,\,\,=
[ n-1]_{\alpha,\zeta}\,\,{\xy
(0,-10);(0,10)**\dir{-}
?>(.95)*\dir{>}
,(2,-11)*{\scriptstyle{1}}
,(2,11)*{\scriptstyle{1}}
\endxy}
\end{equation}

\item $(\mathrm{id}\otimes \mathrm{ev}\otimes\mathrm{id}^\vee)\circ (S\otimes S^\vee) \circ (\mathrm{id}\otimes \mathrm{coev}\otimes\mathrm{id}^\vee)=\mathrm{coev}\circ\mathrm{ev}+ [n-2]_{\alpha,\zeta}\, \mathrm{id}\otimes \mathrm{id}^\vee$,
i.e.,
\begin{equation}\label{moy-4az}\tag{moy-4${}_{\alpha,\zeta}$}
{\xy
(-10,-10);(-10,10)**\crv{(-6,-6)&(-5,-5)&(-5,-2)&(-5,-2)&(-5,5)&(-6,6)}
?>(.05)*\dir{>}?>(.99)*\dir{>}?>(.55)*\dir{>}
,(10,-10);(10,10)**\crv{(6,-6)&(5,-5)&(5,-2)&(5,2)&(5,5)&(6,6)}
?>(.03)*\dir{<}?>(.92)*\dir{<}?>(.45)*\dir{<}
,(-5,0);(5,0)**\crv{(-5,-2)&(-5,-3)&(-4,-5)&(-2,-6)&(2,-6)&(4,-5)&(5,-3)&(5,-2)}?>(.5)*\dir{<}
,(-5,0);(5,0)**\crv{(-5,2)&(-5,3)&(-4,5)&(-2,6)&(2,6)&(4,5)&(5,3)&(5,2)}?>(.5)*\dir{>}
,(-11,-10)*{\scriptstyle{1}}
,(-11,10)*{\scriptstyle{1}}
,(11,-10)*{\scriptstyle{1}}
,(11,10)*{\scriptstyle{1}}
,(-7,0)*{\scriptstyle{2}}
,(0,-8)*{\scriptstyle{1}}
,(6.5,0)*{\scriptstyle{2}}
,(0,8)*{\scriptstyle{1}}
\endxy}
\quad
=
\xy
(-5,-10);(5,-10)**\crv{(-5,-2)&(0,-2)&(5,-2)}
?>(1)*\dir{>},
(-6,-10)*{\scriptstyle{1}},(6,-10)*{\scriptstyle{1}}
,(-5,10);(5,10)**\crv{(-5,2)&(0,2)&(5,2)}
?>(0)*\dir{<},
(-6,10)*{\scriptstyle{1}},(6,10)*{\scriptstyle{1}}
\endxy
+\, [n-2]_{\alpha,\zeta}\,
{\xy
(-5,-10);(-5,10)**\dir{-}?>(.5)*\dir{>}
,(5,-10);(5,10)**\dir{-}?>(.5)*\dir{<}
,(-7,-10)*{\scriptstyle{1}}
,(7,-10)*{\scriptstyle{1}}
,(-7,10)*{\scriptstyle{1}}
,(7,10)*{\scriptstyle{1}}
\endxy}
\end{equation}

\item $(S\otimes\mathrm{id})\circ (\mathrm{id}\otimes S) \circ (S\otimes\mathrm{id})+(\mathrm{id}\otimes S)=(\mathrm{id}\otimes S) \circ (S\otimes\mathrm{id})\circ (\mathrm{id}\otimes S)+(S\otimes\mathrm{id})$,
i.e.,
\begin{equation}\label{moy-5az}\tag{moy-5${}_{\alpha,\zeta}$}
\coso\,\,+\,\,\idone\quad \esse = \cosob\,\,+\,\, \esse\quad\idoner
\end{equation}
\end{enumerate}
The \emph{HOMFLY-PT category} $\mathbf{MOY}_{\alpha,z}$  the monoidal category enriched in $R$-modules defined by
\[
\mathbf{MOY}_{\alpha,\zeta}=\mathbf{TrPD}/{\mathbf{I}_{MOY_{\alpha,\zeta}}}.
\]
\end{definition}

\begin{remark}
Notice that 
\[
[2]_{\alpha,\zeta}=\frac{\alpha^{-1} (-\zeta)^{2-n}-\alpha(-\zeta)^{n-2}}{\zeta-\zeta^{-1}}\neq -(\zeta+\zeta^{-1}).
\]
\end{remark}

\begin{remark}\label{homfly-to-jones}
When $R=\mathbb{Q}[\alpha,\alpha^{-1},\zeta,\zeta^{-1},(\zeta-\zeta^{-1})^{-1}]$, the morphism of rings $R\to \mathbb{Q}[q,q^{-1},(q-q^{-1})^{-1}]$ given by $\alpha\mapsto -q^{-n}$ and $\zeta\mapsto -q$ induces a morphism $\mathbf{MOY}_{\alpha,\zeta}\to \mathbf{MOY}$. This way one recovers that the level $n$ Jones polynomial is a particular case of the HOMFLY polynomial. Notice that the image of $[k]_{\alpha,\zeta}$ under this homomorphism is the element $[k]$ in $\mathbb{Q}[q,q^{-1}]\subseteq \mathbb{Q}[q,q^{-1},(q-q^{-1})^{-1}]$, as
\[
\frac{(-q^n) (-(-q))^{k-n}-(-q^n)^{-1}(-(-q))^{n-k}}{(-q)-(-q)^{-1}}=\frac{-q^n q^{k-n}+q^{-n}q^{n-k}}{-q+q^{-1}}=\frac{q^k-q^{-k}}{q-q^{-1}}.
\]
\end{remark}

By complete analogy with $\mathbf{MOY}$, we have the following

\begin{proposition}
The category $\mathbf{MOY}_{\alpha,\zeta}$ is a rigid braided and untwisted monoidal category with the braidings 
\[
\sigma^+,\sigma^-\colon \uparrow^1\otimes \uparrow^1 \to \uparrow^1\otimes \uparrow^1
\]
defined by
\[
\sigma^+=\alpha^{-1}\left(\zeta^{-1}\,\,\idtwo\,\,+\esse\,\right)
\]
and
\[
\sigma^-=\alpha\left(\zeta\,\,\idtwo\,\,+\esse\,\right)
\]
\end{proposition}
\begin{proof}
As it is a quotient of $\mathbf{TrPD}$, the category $\mathbf{MOY}_{\alpha,\zeta}$ is rigid monoidal. So we only need checking it is braided untwisted, i.e., that Reidemeister relations are satisfied. This is done by verbatim repeating the proof of Proposition \ref{prop:moy-is-braided}. Here we provide a fully detailed proof, to show how we correctly generalized the quantum integers $[k]_q$ to the elements $[k]_{\alpha,\zeta}$ and the relations (\ref{moy-1}-\ref{moy-5}) to the  relations (\ref{moy-1az}-\ref{moy-5az}).
\par
By definition of the braidings, evaluations and coevaluations in $\mathbf{MOY}_{\alpha,\zeta}$ invariance with respect to the first Reidemeister move for an overcrossing corresponds to the identity
\[
\alpha^{-1}\left(\evop\,\idone\,\,\right)\circ\left(\,\,\idonedual\,\,\otimes \left(
\zeta^{-1}\,\,\idtwo+\esse\right)\right)\circ\left(\coevop\,\,\idone\right)=\idone
\]
We compute
\begin{align*}
&\alpha^{-1}\left(\evop\,\idone\,\,\right)\circ\left(\,\,\idonedual\,\,\otimes \left(
\zeta^{-1}\,\,\idtwo+\esse\right)\right)\circ\left(\coevop\,\,\idone\right)\\
&\qquad=
 \alpha^{-1}\left(
\zeta^{-1}\,\,\iunknot \quad {\xy
(0,-10);(0,10)**\dir{-}
?>(.95)*\dir{>}
,(2,-11)*{\scriptstyle{1}}
,(2,11)*{\scriptstyle{1}}
\endxy} \,\,+\,\,
{\xy
(0,-10);(0,10)**\crv{(0,-7)&(0,-6)&(0,-5)&(0,0)&(0,5)&(0,6)&(0,7)}
?>(.2)*\dir{>}?>(.99)*\dir{>}?>(.55)*\dir{>}
,(0,-3);(0,3)**\crv{(0,-4)&(0,-5)&(-4,-6)&(-6,0)&(-4,6)&(0,5)&(0,4)}
?>(.55)*\dir{<}
,(2,-11)*{\scriptstyle{1}}
,(2,11)*{\scriptstyle{1}}
,(-2,0)*{\scriptstyle{2}}
,(-6,0)*{\scriptstyle{1}}
\endxy}
\right)=
\alpha^{-1}\left(\, \zeta^{-1}[n]_{\alpha,\zeta}\idone\,\,+[n-1]_{\alpha,\zeta}\,\,\idone\,\,\right)=\quad\,\idone\,
\end{align*}
thanks to relations (\ref{moy-1az}), (\ref{moy-3az}), and (\ref{prima-eq}).
For an undercrossing, we need to prove
\[
 \alpha\left(\zeta
\,\,\iunknot \quad {\xy
(0,-10);(0,10)**\dir{-}
?>(.95)*\dir{>}
,(2,-11)*{\scriptstyle{1}}
,(2,11)*{\scriptstyle{1}}
\endxy} \,\,+\,\,
{\xy
(0,-10);(0,10)**\crv{(0,-7)&(0,-6)&(0,-5)&(0,0)&(0,5)&(0,6)&(0,7)}
?>(.2)*\dir{>}?>(.99)*\dir{>}?>(.55)*\dir{>}
,(0,-3);(0,3)**\crv{(0,-4)&(0,-5)&(-4,-6)&(-6,0)&(-4,6)&(0,5)&(0,4)}
?>(.55)*\dir{<}
,(2,-11)*{\scriptstyle{1}}
,(2,11)*{\scriptstyle{1}}
,(-2,0)*{\scriptstyle{2}}
,(-6,0)*{\scriptstyle{1}}
\endxy}
\right)=\,\,\idone
\]
By relations (\ref{moy-1az}) and (\ref{moy-3az}), the left hand side is
\[
\alpha\left(\, \zeta [n]_{\alpha,\zeta}\idone\,\,+[n-1]_{\alpha,\zeta}\,\,\idone\,\,\right)=\alpha(\zeta[n]_{\alpha,\zeta}+[n-1]_{\alpha,\zeta})\,\,\idone\,
\]
and the conclusion follows from equation (\ref{seconda-eq}). 
Checking the second Reidemeister move of the first kind in $\mathbf{MOY}_{\alpha,z}$ amounts to proving the identity
\[
\left(\zeta^{-1}\idtwo+\esse\right)\circ\left(\zeta\idtwo+\esse\right)=\idtwo
\]
By (\ref{moy-2az}), the left hand side of the above equation is
\[
\idtwo+(\zeta^{-1}+\zeta) \esse+\essecircesse=\idtwo.
\]
Next, the invariants with respect to the second Reidemeister move of the second kind in $\mathbf{MOY}_{\alpha,\zeta}$ is the identity
\begin{align*}
\left(\idone\quad \ev\,\,\idonedual\right)&\circ\left(\left(\zeta^{-1}\idtwo+\esse\right)\otimes\right.\\
& \left.\left(\zeta\idtwodual+\essedual\right)\right)
\circ \left(\idone\quad \coev\,\, \idonedual\right)\\
&=
\,\,{\xy
(-5,-10);(5,-10)**\crv{(-5,-2)&(0,-2)&(5,-2)}
?>(1)*\dir{>}
,(-6,-10)*{\scriptstyle{1}},(6,-10)*{\scriptstyle{1}}
,(-5,10);(5,10)**\crv{(-5,2)&(0,2)&(5,2)}
?>(0)*\dir{<}
,(-6,10)*{\scriptstyle{1}},(6,10)*{\scriptstyle{1}}
\endxy}
\end{align*}
The left hand side of the above identity is
\[
{\xy
(-10,-10);(-10,10)**\dir{-}?>(.5)*\dir{>}
,(-12,-11)*{\scriptstyle{1}}
,(-12,11)*{\scriptstyle{1}}
,(-2,-5);(-2,-5)**\crv{(3,-5)&(3,5)&(-7,5)&(-7,-5)}
?>(.3)*\dir{<}
,(3.5,1)*{\scriptstyle{1}}
,(6,-10);(6,10)**\dir{-}?>(.5)*\dir{<}
,(8,-11)*{\scriptstyle{1}}
,(8,11)*{\scriptstyle{1}}
\endxy}
\,
+\zeta
{\xy
(-10,-10);(-10,10)**\dir{-}?>(.5)*\dir{>}
,(-12,-11)*{\scriptstyle{1}}
,(-12,11)*{\scriptstyle{1}}
,(0,10);(0,-10)**\crv{(0,7)&(0,6)&(0,5)&(0,0)&(0,-5)&(0,-6)&(0,-7)}
?>(.2)*\dir{>}?>(.99)*\dir{>}?>(.55)*\dir{>}
,(0,3);(0,-3)**\crv{(0,4)&(0,5)&(-4,6)&(-6,0)&(-4,-6)&(0,-5)&(0,-4)}
?>(.55)*\dir{<}
,(2,-11)*{\scriptstyle{1}}
,(2,11)*{\scriptstyle{1}}
,(2,0)*{\scriptstyle{2}}
,(-6,0)*{\scriptstyle{1}}
\endxy}\,+\zeta^{-1}\,\,\,
{\xy
(0,-10);(0,10)**\crv{(0,-7)&(0,-6)&(0,-5)&(0,0)&(0,5)&(0,6)&(0,7)}
?>(.2)*\dir{>}?>(.99)*\dir{>}?>(.55)*\dir{>}
,(0,-3);(0,3)**\crv{(0,-4)&(0,-5)&(4,-6)&(6,0)&(4,6)&(0,5)&(0,4)}
?>(.55)*\dir{<}
,(2,-11)*{\scriptstyle{1}}
,(2,11)*{\scriptstyle{1}}
,(1.5,0)*{\scriptstyle{2}}
,(7,0)*{\scriptstyle{1}}
,(10,-10);(10,10)**\dir{-}?>(.5)*\dir{<}
,(12,-11)*{\scriptstyle{1}}
,(12,11)*{\scriptstyle{1}}
\endxy}
\,\,+\,\, 
{\xy
(-10,-10);(-10,10)**\crv{(-6,-6)&(-5,-5)&(-5,-2)&(-5,-2)&(-5,5)&(-6,6)}
?>(.05)*\dir{>}?>(.99)*\dir{>}?>(.55)*\dir{>}
,(10,-10);(10,10)**\crv{(6,-6)&(5,-5)&(5,-2)&(5,2)&(5,5)&(6,6)}
?>(.03)*\dir{<}?>(.92)*\dir{<}?>(.45)*\dir{<}
,(-5,0);(5,0)**\crv{(-5,-2)&(-5,-3)&(-4,-5)&(-2,-6)&(2,-6)&(4,-5)&(5,-3)&(5,-2)}?>(.5)*\dir{<}
,(-5,0);(5,0)**\crv{(-5,2)&(-5,3)&(-4,5)&(-2,6)&(2,6)&(4,5)&(5,3)&(5,2)}?>(.5)*\dir{>}
,(-11,-10)*{\scriptstyle{1}}
,(-11,10)*{\scriptstyle{1}}
,(11,-10)*{\scriptstyle{1}}
,(11,10)*{\scriptstyle{1}}
,(-7,0)*{\scriptstyle{2}}
,(0,-8)*{\scriptstyle{1}}
,(6.5,0)*{\scriptstyle{2}}
,(0,8)*{\scriptstyle{1}}
\endxy}\,.
\]
By equations (\ref{moy-1az}), (\ref{moy-2az}) and (\ref{moy-4az}), this is
\[
\left([n]_{\alpha,\zeta}+(\zeta+\zeta^{-1})[n-1]_{\alpha,\zeta}
+[n-2]_{\alpha,\zeta}\right)
\idone\,\qquad\idonedual
\,\,+\,\,{\xy
(-5,-10);(5,-10)**\crv{(-5,-2)&(0,-2)&(5,-2)}
?>(1)*\dir{>},
(-6,-10)*{\scriptstyle{1}},(6,-10)*{\scriptstyle{1}}
,(-5,10);(5,10)**\crv{(-5,2)&(0,2)&(5,2)}
?>(0)*\dir{<},
(-6,10)*{\scriptstyle{1}},(6,10)*{\scriptstyle{1}}
\endxy}\quad
=
\quad
{\xy
(-5,-10);(5,-10)**\crv{(-5,-2)&(0,-2)&(5,-2)}
?>(1)*\dir{>},
(-6,-10)*{\scriptstyle{1}},(6,-10)*{\scriptstyle{1}}
,(-5,10);(5,10)**\crv{(-5,2)&(0,2)&(5,2)}
?>(0)*\dir{<},
(-6,10)*{\scriptstyle{1}},(6,10)*{\scriptstyle{1}}
\endxy},
\]
by (\ref{terza-eq}). Finally, the invariance with respect to the third Reidemeister move in $\mathbf{MOY}_{\alpha,\zeta}$ amounts to the identity
\begin{align*}
&\left(\left(\zeta^{-1}\idtwo+\esse\right)\otimes \,\,\idone\,\,\right)\circ \left(\,\,\idone\,\,\otimes \left(\zeta^{-1}\idtwo+\esse\right)\right) \\
&\qquad\qquad\qquad\qquad
\circ \left(\left(\zeta^{-1}\idtwo+\esse\right)\otimes\,\,\idone\,\,\right)\\
&\qquad\qquad= \left(\,\,\idone\,\,\otimes \left(\zeta^{-1}\idtwo+\esse\right)\right) \circ \left(\left(\zeta^{-1}\idtwo+\esse\right)\otimes\,\,\idone\,\,\right)\\
&\qquad\qquad\qquad\qquad\circ \left(\,\,\idone\,\,\otimes \left(\zeta^{-1}\idtwo+\esse\right)\right). 
\end{align*}
By (\ref{moy-2az}), the left hand side is
\begin{align*}
&\zeta^{-3}\idtwo\qquad\idone\,\, +\zeta^{-2}\,\idone\,\,\esse\,\,+2\zeta^{-2}\, \esse\,\idone\,\,\\
&\qquad\qquad\qquad\qquad+
\zeta^{-1}\zorro+\zeta^{-1}\zorrob+\zeta^{-1}\essecircesse\,\idone\,\,+\coso\\
\end{align*}
\begin{align*}
&\qquad\qquad=
\zeta^{-3}\idtwo\qquad\idone\,\, +\zeta^{-2}\,\idone\,\,\esse\,\,+\zeta^{-2}\, \esse\,\idone\,\,\\
&\qquad\qquad\qquad\qquad+
\zeta^{-1}\zorro+\zeta^{-1}\zorrob+\left(-\esse\,\idone\,\,+\coso\right)
\end{align*}
By computing in the same way the right hand side, we find that it equals
\begin{align*}
&\zeta^{-3}\idtwo\qquad\idone\,\, +\zeta^{-2}\, \esse\,\idone\,\,+\zeta^{-2}\,\idone\,\,\esse\,\,\\
&\qquad\qquad\qquad\qquad+
\zeta^{-1}\zorrob+\zeta^{-1}\zorro+\left(-\,\idone\,\esse\,+\cosob\right)\,,
\end{align*}
and we conclude by (\ref{moy-5az}).
\end{proof}

\subsection{The HOMFLY-PT polynomial via the MOY category}

As it is to be expected, we can perfectly parallel the construction from Section \ref{sec:Jones-via-MOY}.

\begin{proposition}
Let $R$ be a commutative ring, and let $\alpha$ and $z$ be elements in $R$ such that $\alpha,z\in R^\times$, and let $\zeta\in R^\times$ be such that $\zeta^{-1}-\zeta=z$. There exists a unique rigid braided monoidal functor
\[
Z_{\alpha,\zeta}\colon \mathbf{TD}\to \mathbf{MOY}_{\alpha,\zeta}
\] 
such that $Z_{\alpha,\zeta}(\uparrow)=\uparrow^1$. Moreover,
the functor $Z_{\alpha,\zeta}$ factors through $\mathbf{HOMFLY\text{-}PT}_{\alpha,z}$. 
\end{proposition}
\begin{proof}
The existence and uniqueness of $Z_{\alpha,\zeta}$ comes from Theorem \ref{reshetikin-turaev}. To conclude, 
by definition of $\mathbf{HOMFLY\text{-}PT}_{\alpha,z}$ (Definition \ref{def:homfly-cat}) we only need to prove that the braidings, the evaluation and the coevaluation of $\uparrow^1$ in $\mathbf{MOY}$ satisfy the skein relation
\[
\alpha\ijovercrossing -\alpha^{-1}\ijundercrossing = z\ijnoncrossing
\]
as well as the normalization condition
\[
\iunknot=\frac{\alpha-\alpha^{-1}}{z}\emptyset.
\]
The normalization condition is ensured by equation (\ref{moy-1az}), as
\[
[n]_{\alpha,\zeta}=\frac{\alpha^{-1}-\alpha}{\zeta-\zeta^{-1}}=\frac{\alpha-\alpha^{-1}}{z}.
\]
Concerning the skein relation, we have
\begin{align*}
\alpha\ijovercrossing &- \alpha^{-1}\ijundercrossing = \\
&=\left(\zeta^{-1}\idtwo+\esse \right)-\left(
\zeta\idtwo+\esse \right)\\
&=(\zeta^{-1}-\zeta)\idtwo\,= z\,\idtwo\,.
\end{align*}
\end{proof}
\begin{corollary}\label{cor:homfly}
The HOMFLY-PT polynomial with parameters $\alpha,z$ of a link $\Gamma$ is computed via the category $\mathbf{MOY}_{\alpha,\zeta}$ by the equation
\[
Z_{\alpha,\zeta}(\Gamma)=\phi_{\alpha,z}(\Gamma)\, \emptyset.
\]
\end{corollary}

\begin{example}\label{ex:homfly-braid}
Reasoning as we did for the Jones polynomial in Section \ref{sec:Jones-via-MOY}, we can compute the HOMFLY polynomial of the 2-strand $k$-fold braid link. Namely, in $\mathbf{MOY}_{\alpha,\zeta}$ we have
\begin{align*}
(\sigma^+)^k&=\alpha^{-k}\zeta^{-k}\left(
\quad 
{\xy
(-5,-10);(-5,10)**\dir{-}?>(.5)*\dir{>}
,(5,-10);(5,10)**\dir{-}?>(.5)*\dir{>}
\endxy}\,
+\zeta(1-\zeta^2+\zeta^4-\zeta^6+\cdots-(-1)^{k-1}\zeta^{2(k-1)})\,
{\xy
(-5,-10);(-5,10)**\crv{(-1,-6)&(0,-5)&(0,-2)&(0,2)&(0,5)&(-1,6)}
?>(.55)*\dir{>}?>(.1)*\dir{>}?>(.95)*\dir{>}
,(5,-10);(5,10)**\crv{(1,-6)&(0,-5)&(0,-2)&(0,2)&(0,5)&(1,6)}
?>(.1)*\dir{>}?>(.95)*\dir{>}
\endxy}\,\,\,
\right)\\
&=\alpha^{-k}\zeta^{-k}\left(
\quad 
{\xy
(-5,-10);(-5,10)**\dir{-}?>(.5)*\dir{>}
,(5,-10);(5,10)**\dir{-}?>(.5)*\dir{>}
\endxy}\,
+\zeta\frac{1-(-\zeta^2)^k}{1+\zeta^2}\,
{\xy
(-5,-10);(-5,10)**\crv{(-1,-6)&(0,-5)&(0,-2)&(0,2)&(0,5)&(-1,6)}
?>(.55)*\dir{>}?>(.1)*\dir{>}?>(.95)*\dir{>}
,(5,-10);(5,10)**\crv{(1,-6)&(0,-5)&(0,-2)&(0,2)&(0,5)&(1,6)}
?>(.1)*\dir{>}?>(.95)*\dir{>}
\endxy}\,\,\,
\right)
\end{align*}
for any $k\geq 2$. Therefore,
\begin{align*}
\phi_{\alpha,z}\left(\trsplusk{k}\right)&=\alpha^{-k}\zeta^{-k}\left([n]_{\alpha,\zeta}+\zeta\frac{1-(-\zeta^2)^k}{1+\zeta^2}\,[n-1]_{\alpha,\zeta}\right)[n]_{\alpha,\zeta}\\
&=\alpha^{-k}\zeta^{-k}\left([n]_{\alpha,\zeta}+\frac{1-(-\zeta^2)^k}{1+\zeta^2}\,(\alpha\zeta-[n]_{\alpha,\zeta})\right)[n]_{\alpha,\zeta}\\
&=\alpha^{-k}\zeta^{-k}\left(\frac{\zeta^2+(-\zeta^2)^{k}}{1+\zeta^2}[n]_{\alpha,\zeta}+\frac{1-(-\zeta^2)^k}{1+\zeta^2}\,\alpha\zeta
\right)[n]_{\alpha,\zeta}
\end{align*}
The right hand side seems to depend on $\zeta$ rather than on $z$, but it is actually not so. For instance, for $k=1$ we have
\[
\alpha^{-1}\zeta^{-1}\left(\alpha\zeta
\right)[n]_{\alpha,\zeta}=\frac{\alpha-\alpha^{-1}}{z}
\]
while for $k=2$ we have
\begin{align*}
\alpha^{-2}\zeta^{-2}\left(\frac{\zeta^2+\zeta^4}{1+\zeta^2}[n]_{\alpha,\zeta}+\frac{1-\zeta^4}{1+\zeta^2}\,\alpha\zeta
\right)[n]_{\alpha,\zeta}
&=
\alpha^{-2}\zeta^{-2}\left(\zeta^2[n]_{\alpha,\zeta}+(1-\zeta^2)\,\alpha\zeta
\right)[n]_{\alpha,\zeta}\\
&=\alpha^{-2}\zeta^{-2}\left(\zeta^2[n]_{\alpha,\zeta}+(\zeta z)\,\alpha\zeta
\right)[n]_{\alpha,\zeta}\\
&=\alpha^{-2}\left([n]_{\alpha,\zeta}+z\,\alpha
\right)[n]_{\alpha,\zeta}\\
&=\alpha^{-2}\left(\frac{\alpha-\alpha^{-1}}{z}+z\,\alpha
\right)\frac{\alpha-\alpha^{-1}}{z}
\end{align*}
To see that in general
\[
\alpha^{-k}\zeta^{-k}\left(\frac{\zeta^2+(-\zeta^2)^{k}}{1+\zeta^2}[n]_{\alpha,\zeta}+\frac{1-(-\zeta^2)^k}{1+\zeta^2}\,\alpha\zeta
\right)[n]_{\alpha,\zeta}
\]
is a function of $z$ rather than of $\zeta$, we show that it is invariant under the change of variable $\zeta\mapsto -\zeta^{-1}$. We have
\begin{align*}
\alpha^{-k}(-\zeta)^{k}&\left(\frac{\zeta^{-2}+(-\zeta^{-2})^{k}}{1+\zeta^{-2}}[n]_{\alpha,-\zeta^{-1}}-\frac{1-(-\zeta^{-2})^k}{1+\zeta^{-2}}\,\alpha\zeta^{-1}
\right)[n]_{\alpha,-\zeta^{-1}}\\
&=\alpha^{-k}(-\zeta)^{k}(-\zeta)^{-2k}\left((-\zeta)^{2k}\frac{1-(-\zeta^{-2})^{k-1}}{1+\zeta^{2}}[n]_{\alpha,-\zeta^{-1}}-(-\zeta)^{2k}\frac{1-(-\zeta^{-2})^k}{1+\zeta^{2}}\,\alpha\zeta
\right)[n]_{\alpha,-\zeta^{-1}}\\
&=\alpha^{-k}\zeta^{-k}\left(\frac{(-\zeta^2)^{k}+\zeta^2}{1+\zeta^{2}}[n]_{\alpha,\zeta}-\frac{(-\zeta^2)^{k}-1}{1+\zeta^{2}}\,\alpha\zeta
\right)[n]_{\alpha,\zeta},
\end{align*}
which is the desired identity.
\end{example}

 \subsection{The universality of the category $\mathbf{MOY}_{\alpha,\zeta}$}
 In the deriving the HOMFLY-PT polynomial from the relations in the category $\mathbf{MOY}_{\alpha,\zeta}$ we have in particular noticed how, given a morphism 
 \[S=\esse\colon \uparrow^1\otimes \uparrow^1 \to \uparrow^1\otimes \uparrow^1
 \] 
 satisfying relations (\ref{moy-1az}-\ref{moy-5az}) one gets a braiding satisfying the skein relation
 \begin{equation}\label{eq:skein-here}
\alpha\ijovercrossing -\alpha^{-1}\ijundercrossing = z\ijnoncrossing,
\end{equation}
as well as evaluation and coevaluation morphisms satisfying the normalization condition
 \begin{equation}\label{eq:normalization-here}
\iunknot=\frac{\alpha-\alpha^{-1}}{z}\emptyset.
\end{equation}
Actually also the converse is true, so that the a priori mysterious relations (\ref{moy-1az}-\ref{moy-5az}) are nothing but a way of encoding the axioms of a nontwisted braided rigid category generated by one object whose braidings satisfy the skein relation (\ref{eq:skein-here}) and whose evaluation and coevaluation morphisms satisfy the normalization condition (\ref{eq:normalization-here}). Namely, we have 
\begin{theorem}
The morphism
\[
Z_{\alpha,\zeta}\colon \mathbf{HOMFLY\text{-}PT}_{\alpha,z}\to \mathbf{MOY}_{\alpha,\zeta}
\]
is an isomorphism of rigid braided untwisted monoidal categories.
\end{theorem}
\begin{proof}
We define a rigid monoidal functor $\Xi_{\alpha,\zeta}\colon \mathbf{TrPD}\to \mathbf{PTD}$ in the obvious way on identities, evaluations and coevaluations, and setting
\[
\Xi_{\alpha,\zeta}\left(\,\esse\,\right)=\alpha\overcrossing-\zeta^{-1}\,\noncrossing
\] 
The functor $\Xi_{\alpha,\zeta}$ induces a functor
\[
\Xi_{\alpha,\zeta}\colon \mathbf{MOY}_{\alpha,\zeta}\to \mathbf{HOMFLY\text{-}PT}_{\alpha,z}.
\]
To see this, we need to check that $\Xi_{\alpha,\zeta}$ maps the relations (\ref{moy-1az}-\ref{moy-5az}) into the ideal in $\mathbf{PTD}$ generated by the Reidemeister moves, the skein relation (\ref{eq:skein-here}) and the normalization condition (\ref{eq:normalization-here}). Equivalently, that the relations  (\ref{moy-1az}-\ref{moy-5az}) go to zero via the composition
\[
\mathbf{TrPD}\xrightarrow{\Xi_{\alpha,\zeta}} \mathbf{PTD}\to \mathbf{HOMFLY\text{-}PT}_{\alpha,z}.
\]
For the relation (\ref{moy-1az}) this is almost tautologically verified. We have
\[
\Xi_{\alpha,\zeta}\left(\iunknot - [n]_{\alpha,\zeta}\emptyset\right)=\unknot - [n]_{\alpha,\zeta}\emptyset = 0
\]
by the normalization condition in $\mathbf{HOMFLY\text{-}PT}_{\alpha,z}$. For the relation (\ref{moy-2az}) we have
\begin{align*}
\Xi_{\alpha,\zeta}\left(\,\essecircesse+(\zeta+\zeta^{-1})\esse\,\right)&=
\left(\alpha\overcrossing-\zeta^{-1}\,\noncrossing\right)\circ \left(\alpha\overcrossing-\zeta^{-1}\,\noncrossing\right)\\
&\qquad+(\zeta+\zeta^{-1})\left(\alpha\overcrossing-\zeta^{-1}\,\noncrossing\right).
\end{align*} 
By the skein relation in $\mathbf{HOMFLY\text{-}PT}_{\alpha,z}$, the right hand side is
\begin{align*}
\left(\alpha\overcrossing-\zeta^{-1}\,\noncrossing\right)\circ &\left(\alpha^{-1}\undercrossing-\zeta\,\noncrossing\right)
+(\zeta+\zeta^{-1})\left(\alpha\overcrossing-\zeta^{-1}\,\noncrossing\right)\\
&=2\noncrossing-\alpha\zeta\overcrossing-\alpha^{-1}\zeta^{-1}\undercrossing
+\alpha\zeta\overcrossing+\alpha\zeta^{-1}\overcrossing-\noncrossing-\zeta^{-2}\noncrossing\\
&=\zeta^{-1}\left(\alpha\overcrossing-\alpha^{-1}\undercrossing
+(\zeta-\zeta^{-1})\noncrossing\,\,\right)\\
&=\zeta^{-1}\left(\alpha\overcrossing-\alpha^{-1}\undercrossing
-z\noncrossing\,\,\right)=0
\end{align*}
Next, we consider the relation (\ref{moy-3az}). We have
\begin{align*}
\Xi_{\alpha,\zeta}\left(\,{\xy
(0,-10);(0,10)**\crv{(0,-7)&(0,-6)&(0,-5)&(0,0)&(0,5)&(0,6)&(0,7)}
?>(.2)*\dir{>}?>(.99)*\dir{>}?>(.55)*\dir{>}
,(0,-3);(0,3)**\crv{(0,-4)&(0,-5)&(-4,-6)&(-6,0)&(-4,6)&(0,5)&(0,4)}
?>(.55)*\dir{<}
,(2,-11)*{\scriptstyle{1}}
,(2,11)*{\scriptstyle{1}}
,(2,0)*{\scriptstyle{2}}
,(-6,0)*{\scriptstyle{1}}
\endxy}\,-
[ n-1]_{\alpha,\zeta}\,\,{\xy
(0,-10);(0,10)**\dir{-}
?>(.95)*\dir{>}
,(2,-11)*{\scriptstyle{1}}
,(2,11)*{\scriptstyle{1}}
\endxy}\,\right)&=
\alpha\,\,
{\xy
(2,5);(0,0.1)**\crv{(2,1)&(-2,-4)&(-4,0)&(-2,4)}?>(0)*\dir{<}
,(2,-5);(0.65,-1)**\crv{(2,-4)}
\endxy}
-\zeta^{-1}\unknot\,\,{\xy 
(3,-5);(3,5)**\crv{(1,0)}
?>(.9)*\dir{>}
\endxy}
-[ n-1]_{\alpha,\zeta}\,\,{\xy
(0,-5);(0,5)**\dir{-}?>(1)*\dir{>}
\endxy}\\
&=(\alpha-\zeta^{-1}[n]_{\alpha,\zeta}-[ n-1]_{\alpha,\zeta}\,)\,\,{\xy
(0,-5);(0,5)**\dir{-}?>(1)*\dir{>}
\endxy}\\
&=0,
\end{align*}
by equation (\ref{prima-eq}). For relation (\ref{moy-4az}) we have
\begin{align*}
\Xi_{\alpha,\zeta}\left(\,{\xy
(-10,-10);(-10,10)**\crv{(-6,-6)&(-5,-5)&(-5,-2)&(-5,-2)&(-5,5)&(-6,6)}
?>(.05)*\dir{>}?>(.99)*\dir{>}?>(.55)*\dir{>}
,(10,-10);(10,10)**\crv{(6,-6)&(5,-5)&(5,-2)&(5,2)&(5,5)&(6,6)}
?>(.03)*\dir{<}?>(.92)*\dir{<}?>(.45)*\dir{<}
,(-5,0);(5,0)**\crv{(-5,-2)&(-5,-3)&(-4,-5)&(-2,-6)&(2,-6)&(4,-5)&(5,-3)&(5,-2)}?>(.5)*\dir{<}
,(-5,0);(5,0)**\crv{(-5,2)&(-5,3)&(-4,5)&(-2,6)&(2,6)&(4,5)&(5,3)&(5,2)}?>(.5)*\dir{>}
,(-11,-10)*{\scriptstyle{1}}
,(-11,10)*{\scriptstyle{1}}
,(11,-10)*{\scriptstyle{1}}
,(11,10)*{\scriptstyle{1}}
,(-7,0)*{\scriptstyle{2}}
,(0,-8)*{\scriptstyle{1}}
,(6.5,0)*{\scriptstyle{2}}
,(0,8)*{\scriptstyle{1}}
\endxy}
\,-\,
{\xy
(-5,-10);(5,-10)**\crv{(-5,-2)&(0,-2)&(5,-2)}
?>(1)*\dir{>},
(-6,-10)*{\scriptstyle{1}},(6,-10)*{\scriptstyle{1}}
,(-5,10);(5,10)**\crv{(-5,2)&(0,2)&(5,2)}
?>(0)*\dir{<},
(-6,10)*{\scriptstyle{1}},(6,10)*{\scriptstyle{1}}
\endxy}
-\, [n-2]_{\alpha,\zeta}\,
{\xy
(-5,-10);(-5,10)**\dir{-}?>(.5)*\dir{>}
,(5,-10);(5,10)**\dir{-}?>(.5)*\dir{<}
,(-7,-10)*{\scriptstyle{1}}
,(7,-10)*{\scriptstyle{1}}
,(-7,10)*{\scriptstyle{1}}
,(7,10)*{\scriptstyle{1}}
\endxy}\,\right)&\\
&\hskip-6cm=
{ \xy
(-5,10);(5,10)**\crv{(-5,2)&(0,-6)&(5,2)}
?>(0)*\dir{<},
?>(.71)*{\color{white}\bullet},
?>(.30)*{\color{white}\bullet}
,(-5,-10);(5,-10)**\crv{(-5,-2)&(0,6)&(5,-2)}
?>(1)*\dir{>}
\endxy}\,-\alpha\zeta\,
{\xy
(-2,5);(0,0.1)**\crv{(-2,1)&(2,-4)&(4,0)&(2,4)}?>(0)*\dir{<}?>(.21)*{\color{white}\bullet}
,(-2,-5);(0,0.1)**\crv{(-2,-4)}
\endxy}\,\,
{\xy 
(3,5);(3,-5)**\crv{(1,0)}
?>(.9)*\dir{>}
\endxy}\,-\alpha^{-1}\zeta^{-1}
\,
{\xy 
(-3,-5);(-3,5)**\crv{(1,0)}
?>(.9)*\dir{>}
\endxy}\,
\,{\xy
(2,-5);(0,-0.1)**\crv{(2,-1)&(-2,4)&(-4,0)&(-2,-4)}?>(0)*\dir{<}
,(2,5);(0.65,1)**\crv{(2,4)}
\endxy}\,\,
\,-\,
{\xy
(-5,-10);(5,-10)**\crv{(-5,-2)&(0,-2)&(5,-2)}
?>(1)*\dir{>},
,(-5,10);(5,10)**\crv{(-5,2)&(0,2)&(5,2)}
?>(0)*\dir{<}
\endxy}\,
+
{\xy 
(-3,-5);(-3,5)**\crv{(1,0)}
?>(.9)*\dir{>}
\endxy}\,{\xy
(0,-5);(0,-5)**\crv{(5,-5)&(5,5)&(-5,5)&(-5,-5)}
?>(.3)*\dir{<}
\endxy}\,
{\xy 
(3,5);(3,-5)**\crv{(1,0)}
?>(.9)*\dir{>}
\endxy}
-\, [n-2]_{\alpha,\zeta}\,
{\xy 
(-3,-5);(-3,5)**\crv{(1,0)}
?>(.9)*\dir{>}
\endxy}\,\quad
{\xy 
(3,5);(3,-5)**\crv{(1,0)}
?>(.9)*\dir{>}
\endxy}\,\\
&\hskip-6cm
=(-\alpha\zeta-\alpha^{-1}\zeta^{-1}+[n]_{\alpha,\zeta}-[n-2]_{\alpha,\zeta})\, {\xy 
(-3,-5);(-3,5)**\crv{(1,0)}
?>(.9)*\dir{>}
\endxy}\,\quad
{\xy 
(3,5);(3,-5)**\crv{(1,0)}
?>(.9)*\dir{>}
\endxy}\,\\
&\hskip-6cm
=(-[n]_{\alpha,\zeta}-\zeta[n-1]_{\alpha,\zeta}-[n]_{\alpha,\zeta}-\zeta^{-1}[n-1]_{\alpha,\zeta}+[n]_{\alpha,\zeta}-[n-2]_{\alpha,\zeta})\, {\xy 
(-3,-5);(-3,5)**\crv{(1,0)}
?>(.9)*\dir{>}
\endxy}\,\quad
{\xy 
(3,5);(3,-5)**\crv{(1,0)}
?>(.9)*\dir{>}
\endxy}\,\\
&\hskip-6cm=0,
\end{align*}
by equations (\ref{prima-eq}), (\ref{seconda-eq}) and (\ref{terza-eq}). Finally, for relation (\ref{moy-5az}) we have
\begin{align*}
\Xi_{\alpha,\zeta}&\left(\coso\,\,+\,\,\idone\quad \esse - \cosob\,\,-\,\, \esse\quad\idoner\right)\\
&=\alpha^3\left(\,\,{\xy
(-9,-7);(9,7)**\crv{(-9,-5)&(-9,-3)&(9,3)&(9,5)}?>(.99)*\dir{>}
,(9,-7);(1,-.3)**\crv{(9,-5)&(9,-3)}
,(-1,.3);(-5.3,2)**\crv{(-3,1)}
,(-6.7,2.8);(-9,7)**\crv{(-9,4)&(-9,5)}?>(.99)*\dir{>}
,(0,-7);(-5.3,-2.8)**\crv{(0,-5)&(0,-4)}
,(0,7);(-6.1,2.4)**\crv{(0,6)&(0,5)&(0,4)&(-6,2.8)}?>(0)*\dir{<}
,(-6.5,-2.2);(-6.1,2.4)**\crv{(-9,0)}
\endxy
\quad - 
\quad
\xy
(-9,-7);(9,7)**\crv{(-9,-5)&(-9,-3)&(9,3)&(9,5)}?>(.99)*\dir{>}
,(5.6,-2);(1,-.3)**\crv{(3,-1)}
,(9,-7);(6.7,-2.8)**\crv{(9,-5)&(9,-4)}
,(-1,.3);(-9,7)**\crv{(-9,3)&(-9,5)}?>(.99)*\dir{>}
,(0,-7);(6,-2.6)**\crv{(0,-5)&(0,-4)&(6,-3)}
,(0,7);(5.3,2.8)**\crv{(0,6)&(0,5)&(0,4)}?>(0)*\dir{<}
,(6,-2.6);(6.5,2.2)**\crv{(9,0)}
\endxy
}\,\,\right)
-\alpha^2\zeta^{-1}\left(\,\,
\raisebox{10pt}{{\xy
(4,-7.5);(-4,-2.5)**\crv{(0,-5)&(-4,-2.9)& (-4,-2.75)}
?>(.2)*{\color{white}\bullet},
,(-4,-7.5);(4,-2.5)**\crv{(0,-5)&(4,-2.9)& (4,-2.75)}
,(4,-2.5);(-4,2.5)**\crv{(4,-2.25)&(4,-2.1)&(0,0)}
?>(.7)*\dir{>}?>(.79)*{\color{white}\bullet},
(-4,-2.5);(4,2.5)**\crv{(-4,-2.25)&(-4,-2.1)&(0,0)}
?>(.7)*\dir{>}
\endxy}}\,\, \,\,\,{\xy 
(3,-5);(3,5)**\crv{(1,0)}
?>(.9)*\dir{>}
\endxy}\,\, -
\,\,{\xy 
(-3,-5);(-3,5)**\crv{(1,0)}
?>(.9)*\dir{>}
\endxy}\,\,\,\,
\raisebox{10pt}{{\xy
(4,-7.5);(-4,-2.5)**\crv{(0,-5)&(-4,-2.9)& (-4,-2.75)}
?>(.2)*{\color{white}\bullet},
,(-4,-7.5);(4,-2.5)**\crv{(0,-5)&(4,-2.9)& (4,-2.75)}
,(4,-2.5);(-4,2.5)**\crv{(4,-2.25)&(4,-2.1)&(0,0)}
?>(.7)*\dir{>}?>(.79)*{\color{white}\bullet},
(-4,-2.5);(4,2.5)**\crv{(-4,-2.25)&(-4,-2.1)&(0,0)}
?>(.7)*\dir{>}
\endxy}}
\,\,
\right)\\
&\qquad+\alpha(\zeta^{-2}-1)\left(
\,\, \overcrossing \,\,{\xy 
(3,-5);(3,5)**\crv{(1,0)}
?>(.9)*\dir{>}
\endxy}\,-\,
{\xy 
(-3,-5);(-3,5)**\crv{(1,0)}
?>(.9)*\dir{>}
\endxy}\,\overcrossing \,\,
\right)\\
&=
-\alpha^2\zeta^{-1}\left(\,\,
\raisebox{10pt}{{\xy
(4,-7.5);(-4,-2.5)**\crv{(0,-5)&(-4,-2.9)& (-4,-2.75)}
?>(.2)*{\color{white}\bullet},
,(-4,-7.5);(4,-2.5)**\crv{(0,-5)&(4,-2.9)& (4,-2.75)}
,(4,-2.5);(-4,2.5)**\crv{(4,-2.25)&(4,-2.1)&(0,0)}
?>(.7)*\dir{>}?>(.79)*{\color{white}\bullet},
(-4,-2.5);(4,2.5)**\crv{(-4,-2.25)&(-4,-2.1)&(0,0)}
?>(.7)*\dir{>}
\endxy}}\,\, \,\,\,{\xy 
(3,-5);(3,5)**\crv{(1,0)}
?>(.9)*\dir{>}
\endxy}\,\, -
\,\,{\xy 
(-3,-5);(-3,5)**\crv{(1,0)}
?>(.9)*\dir{>}
\endxy}\,\,\,\,
\raisebox{10pt}{{\xy
(4,-7.5);(-4,-2.5)**\crv{(0,-5)&(-4,-2.9)& (-4,-2.75)}
?>(.2)*{\color{white}\bullet},
,(-4,-7.5);(4,-2.5)**\crv{(0,-5)&(4,-2.9)& (4,-2.75)}
,(4,-2.5);(-4,2.5)**\crv{(4,-2.25)&(4,-2.1)&(0,0)}
?>(.7)*\dir{>}?>(.79)*{\color{white}\bullet},
(-4,-2.5);(4,2.5)**\crv{(-4,-2.25)&(-4,-2.1)&(0,0)}
?>(.7)*\dir{>}
\endxy}}
\,\,
\right)+\alpha(\zeta^{-2}-1)\left(
\,\, \overcrossing \,\,{\xy 
(3,-5);(3,5)**\crv{(1,0)}
?>(.9)*\dir{>}
\endxy}\,-\,
{\xy 
(-3,-5);(-3,5)**\crv{(1,0)}
?>(.9)*\dir{>}
\endxy}\,\overcrossing \,\,
\right).
\end{align*}
By the skein relation,
\[
\alpha^2\zeta^{-1}\left(\,\,\raisebox{10pt}{{\xy
(4,-7.5);(-4,-2.5)**\crv{(0,-5)&(-4,-2.9)& (-4,-2.75)}
?>(.2)*{\color{white}\bullet},
,(-4,-7.5);(4,-2.5)**\crv{(0,-5)&(4,-2.9)& (4,-2.75)}
,(4,-2.5);(-4,2.5)**\crv{(4,-2.25)&(4,-2.1)&(0,0)}
?>(.7)*\dir{>}?>(.79)*{\color{white}\bullet},
(-4,-2.5);(4,2.5)**\crv{(-4,-2.25)&(-4,-2.1)&(0,0)}
?>(.7)*\dir{>}
\endxy}}\,\,\right)=\zeta^{-1}\noncrossing + \alpha(\zeta^{-2}-1)\overcrossing\,\,,
\]
so that
\[
-\alpha^2\zeta^{-1}\left(\,\,
\raisebox{10pt}{{\xy
(4,-7.5);(-4,-2.5)**\crv{(0,-5)&(-4,-2.9)& (-4,-2.75)}
?>(.2)*{\color{white}\bullet},
,(-4,-7.5);(4,-2.5)**\crv{(0,-5)&(4,-2.9)& (4,-2.75)}
,(4,-2.5);(-4,2.5)**\crv{(4,-2.25)&(4,-2.1)&(0,0)}
?>(.7)*\dir{>}?>(.79)*{\color{white}\bullet},
(-4,-2.5);(4,2.5)**\crv{(-4,-2.25)&(-4,-2.1)&(0,0)}
?>(.7)*\dir{>}
\endxy}}\,\, \,\,\,{\xy 
(3,-5);(3,5)**\crv{(1,0)}
?>(.9)*\dir{>}
\endxy}\,\, -
\,\,{\xy 
(-3,-5);(-3,5)**\crv{(1,0)}
?>(.9)*\dir{>}
\endxy}\,\,\,\,
\raisebox{10pt}{{\xy
(4,-7.5);(-4,-2.5)**\crv{(0,-5)&(-4,-2.9)& (-4,-2.75)}
?>(.2)*{\color{white}\bullet},
,(-4,-7.5);(4,-2.5)**\crv{(0,-5)&(4,-2.9)& (4,-2.75)}
,(4,-2.5);(-4,2.5)**\crv{(4,-2.25)&(4,-2.1)&(0,0)}
?>(.7)*\dir{>}?>(.79)*{\color{white}\bullet},
(-4,-2.5);(4,2.5)**\crv{(-4,-2.25)&(-4,-2.1)&(0,0)}
?>(.7)*\dir{>}
\endxy}}
\,\,
\right)+\alpha(\zeta^{-2}-1)\left(
\,\, \overcrossing \,\,{\xy 
(3,-5);(3,5)**\crv{(1,0)}
?>(.9)*\dir{>}
\endxy}\,-\,
{\xy 
(-3,-5);(-3,5)**\crv{(1,0)}
?>(.9)*\dir{>}
\endxy}\,\overcrossing \,\,
\right)=0.
\]
This shows that $\Xi_{\alpha,\zeta}$ is a functor $\Xi_{\alpha,\zeta}\colon \mathbf{MOY}_{\alpha,\zeta}\to \mathbf{HOMFLY\text{-}PT}_{\alpha,z}$. It is now immediate to check that it is the inverse to $Z_{\alpha,\zeta}$. On the identities, on the coevaluations and on the evaluations this is obvious. On the additional generating morphisms $S$, $\sigma^+$ and $\sigma^-$ we have:
\begin{align*}
Z_{\alpha,\zeta}\left(\Xi_{\alpha,\zeta}\left(\esse\right)\right)&=Z_{\alpha,\zeta}\left(
\alpha\overcrossing-\zeta^{-1}\,\noncrossing\,\right)\\
&=\alpha(\alpha^{-1}\left(\zeta^{-1}\,\,\idtwo\,\,+\esse\,\right))
-\zeta^{-1}\,\idtwo\\
&=\esse
\end{align*}
\begin{align*}
\Xi_{\alpha,\zeta}\left(Z_{\alpha,\zeta}\left(\overcrossing\right)\right)&=\Xi_{\alpha,\zeta}\left(\alpha^{-1}\left(\zeta^{-1}\,\,\idtwo\,\,+\esse\,\right)\right)\\
&=\alpha^{-1}\left(\zeta^{-1}\noncrossing+\alpha\overcrossing-\zeta^{-1}\noncrossing\right)\\
&=\overcrossing
\end{align*}
and
\begin{align*}
\Xi_{\alpha,\zeta}\left(Z_{\alpha,\zeta}\left(\undercrossing\right)\right)&=\Xi_{\alpha,\zeta}\left(\alpha\left(\zeta\,\,\idtwo\,\,+\esse\,\right)\right)\\
&=\alpha\left(\zeta\noncrossing+\alpha\overcrossing-\zeta^{-1}\noncrossing\right)\\
&=\alpha\left(\alpha\overcrossing-z\noncrossing\right)\\
&=\alpha\left(\alpha^{-1}\undercrossing\right)\\
&=\undercrossing,
\end{align*}
by the skein relation in $\mathbf{HOMFLY}\text{-}\mathbf{PT}_{\alpha,z}$.
\end{proof}

\begin{corollary}\label{cor:moy-vs-jones}
The morphism
\[
Z\colon \mathbf{Jones}_{n}\to \mathbf{MOY}
\]
is an isomorphism of rigid braided untwisted monoidal categories.
\end{corollary}

\subsection{The Anokhina-Dolotin-Morozov conjecture}

It has been predicted in \cite{Gukov} and verified up to $n=11$ by Carqueville and Murfet \cite{Carqueville-Murfet} that the Poincar\'e polynomial of the Khovanov homology of Hopf link (the 2-strand braid link with 2 crossings), i.e., of the link
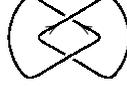
\[
\hopf
\]
is 
\[
\mathcal{P}_{H^\bullet_{Kh}(\mathrm{Hopf})}=q^{2(n-1)}\left(q^{1-n}+t^2q^3[n-1]_q\right)[n]_q,
\]
where as usual $[n]_q=\frac{q^n-q^{-n}}{q-q^{-1}}$. If we compute the HOMFLY-PT polynomial for the Hopf link with the parameters $(\alpha,z)=(q^{-n}t^{-1},q^{-1}t^{-1}-qt)$ and choose $\zeta=qt$, we find from Example \ref{ex:homfly-braid} 
\begin{align*}
\phi_{q^{-n}t, q^{-1}t^{-1}-qt}\left(\hopf\right)&=
q^{2(n-1)}\left([n]_{q^{-n}t,qt}+qt(1-q^2t^2)\,[n-1]_{q^{-n}t,qt}\right)[n]_{q^{-n}t,qt}\\
&=q^{2(n-1)}\left(q^{1-n}-t^{3}q^3\,[n-1]_{q^{-n}t^{-1},qt}\right)[n]_{q^{-n}t^{-1},qt},
\end{align*}
where we used the identity (\ref{prima-eq}):
\[
[n]_{q^{-n}t^{-1},qt}+qt[n-1]_{q^{-n}t^{-1},qt}=q^{1-n}.
\]
Now, this expression is almost the same as $\mathcal{P}_{H^\bullet_{Kh}(\mathrm{Hopf})}$ but not quite. Not only we have the symbols $[k]_{q^{-n}t^{-1},qt}$ in place of the quantum integers $[k]_q$, but we have a $-t^3$ in the second expression which was $+t^2$ in the first expression. Clearly the two coincide at $t=-1$. 
\vskip .5 cm

If we do not use the $\mathbf{MOY}_{q^{-n}t^{-1},qt}$ relations nor the original $\mathbf{MOY}$ relations, but the following sort of mix:
\begin{enumerate}
\item $\mathrm{ev}\circ \mathrm{coev} = [n]_{q^{-n}t^{-1},qt}$,
i.e.,
\begin{equation}\label{moy-1m}\tag{moy-1${}_M$}
\iunknot=[n]_{q^{-n}t^{-1},qt}\, \emptyset
\end{equation}
\item $S\circ S=(q-t^{-1}q^{-1})S$,
i.e.,
\begin{equation}\label{moy-2m}\tag{moy-2${}_{M}$}
\essecircesse=(q-t^{-1}q^{-1})\esse
\end{equation}
\item $(\mathrm{ev}\otimes \mathrm{id})\circ (\mathrm{id}^\vee \otimes S)\circ (\mathrm{coev}\otimes \mathrm{id})=[n-1]_{q^{-n}t^{-1},qt}\,\mathrm{id}$, 
i.e.,
\begin{equation}\label{moy-3m}\tag{moy-3${}_{M}$}
{\xy
(0,-10);(0,10)**\crv{(0,-7)&(0,-6)&(0,-5)&(0,0)&(0,5)&(0,6)&(0,7)}
?>(.2)*\dir{>}?>(.99)*\dir{>}?>(.55)*\dir{>}
,(0,-3);(0,3)**\crv{(0,-4)&(0,-5)&(-4,-6)&(-6,0)&(-4,6)&(0,5)&(0,4)}
?>(.55)*\dir{<}
,(2,-11)*{\scriptstyle{1}}
,(2,11)*{\scriptstyle{1}}
,(2,0)*{\scriptstyle{2}}
,(-6,0)*{\scriptstyle{1}}
\endxy}\,\,\,=
[ n-1]_{q^{-n}t^{-1},qt}\,\,{\xy
(0,-10);(0,10)**\dir{-}
?>(.95)*\dir{>}
,(2,-11)*{\scriptstyle{1}}
,(2,11)*{\scriptstyle{1}}
\endxy}
\end{equation}

\item $(\mathrm{id}\otimes \mathrm{ev}\otimes\mathrm{id}^\vee)\circ (S\otimes S^\vee) \circ (\mathrm{id}\otimes \mathrm{coev}\otimes\mathrm{id}^\vee)=\mathrm{coev}\circ\mathrm{ev}+ [n-2]_{q^{-n}t^{-1},qt}\, \mathrm{id}\otimes \mathrm{id}^\vee$,
i.e.,
\begin{equation}\label{moy-4m}\tag{moy-4${}_{M}$}
{\xy
(-10,-10);(-10,10)**\crv{(-6,-6)&(-5,-5)&(-5,-2)&(-5,-2)&(-5,5)&(-6,6)}
?>(.05)*\dir{>}?>(.99)*\dir{>}?>(.55)*\dir{>}
,(10,-10);(10,10)**\crv{(6,-6)&(5,-5)&(5,-2)&(5,2)&(5,5)&(6,6)}
?>(.03)*\dir{<}?>(.92)*\dir{<}?>(.45)*\dir{<}
,(-5,0);(5,0)**\crv{(-5,-2)&(-5,-3)&(-4,-5)&(-2,-6)&(2,-6)&(4,-5)&(5,-3)&(5,-2)}?>(.5)*\dir{<}
,(-5,0);(5,0)**\crv{(-5,2)&(-5,3)&(-4,5)&(-2,6)&(2,6)&(4,5)&(5,3)&(5,2)}?>(.5)*\dir{>}
,(-11,-10)*{\scriptstyle{1}}
,(-11,10)*{\scriptstyle{1}}
,(11,-10)*{\scriptstyle{1}}
,(11,10)*{\scriptstyle{1}}
,(-7,0)*{\scriptstyle{2}}
,(0,-8)*{\scriptstyle{1}}
,(6.5,0)*{\scriptstyle{2}}
,(0,8)*{\scriptstyle{1}}
\endxy}
\quad
=
\xy
(-5,-10);(5,-10)**\crv{(-5,-2)&(0,-2)&(5,-2)}
?>(1)*\dir{>},
(-6,-10)*{\scriptstyle{1}},(6,-10)*{\scriptstyle{1}}
,(-5,10);(5,10)**\crv{(-5,2)&(0,2)&(5,2)}
?>(0)*\dir{<},
(-6,10)*{\scriptstyle{1}},(6,10)*{\scriptstyle{1}}
\endxy
+\, [n-2]_{q^{-n}t^{-1},qt}\,
{\xy
(-5,-10);(-5,10)**\dir{-}?>(.5)*\dir{>}
,(5,-10);(5,10)**\dir{-}?>(.5)*\dir{<}
,(-7,-10)*{\scriptstyle{1}}
,(7,-10)*{\scriptstyle{1}}
,(-7,10)*{\scriptstyle{1}}
,(7,10)*{\scriptstyle{1}}
\endxy}
\end{equation}

\item $(S\otimes\mathrm{id})\circ (\mathrm{id}\otimes S) \circ (S\otimes\mathrm{id})+(\mathrm{id}\otimes S)=(\mathrm{id}\otimes S) \circ (S\otimes\mathrm{id})\circ (\mathrm{id}\otimes S)+(S\otimes\mathrm{id})$,
i.e.,
\begin{equation}\label{moy-5m}\tag{moy-5${}_{M}$}
\coso\,\,+\,\,\idone\quad \esse = \cosob\,\,+\,\, \esse\quad\idoner
\end{equation}
\end{enumerate}
and consider the same braidings as in $\mathbf{MOY}_{q^{-n}t^{-1},qt}$, i.e., 
\[
\sigma^+_{M}=q^nt\left(q^{-1}t^{-1}\,\idtwo\,+\esse\,\right); \qquad 
\sigma^-_{M}=q^{-n}t^{-1}\left(qt\,\idtwo\,+\esse\,\right),
\]
then
we don't have a braided monoidal category as, for instance,
\begin{align*}
\sigma^+_M\circ \sigma^-_M&=\left(q^{-1}t^{-1}\,\idtwo\,+\esse\,\right)\circ \left(qt\,\idtwo\,+\esse\,\right)\\
&=\,\idtwo+(q^{-1}t^{-1}+qt)  \esse+\essecircesse\\
&=\, \idtwo\,+q(1+t) \esse\\
&\neq\, \idtwo\,.
\end{align*}

Yet, 
\begin{align*}
(\sigma^+_M)^2&=q^{2n}t^2\left(q^{-1}t^{-1}\,\idtwo\,+\esse\,\right)\circ\left(q^{-1}t^{-1}\,\idtwo\,+\esse\,\right)\\
&=q^{2(n-1)}\left(\,\idtwo+2qt  \esse+ q^2t^2\essecircesse\,\right)\\
&=q^{2(n-1)}\left(\,\idtwo+2qt  \esse+ q^2t^2(q-t^{-1}q^{-1})\esse\,\right)\\
&=q^{2(n-1)}\left(\,\idtwo+qt  \esse+ q^3t^2\esse\,\right),
\end{align*}
and so if we use relations (\ref{moy-1m}-\ref{moy-5m}) to compute what the string diagrams calculus would associate to the Hopf link in the above presentation as a planar link diagram\footnote{as we are not working in a braided monoidal categoty, we only have a string diagrams calculus for planar tangle diagrams, not for tangles in $\mathbb{R}^3$.} we find
\begin{align*}
Z_M\left(\hopf\right)&= q^{2(n-1)}\left(\,\twocircles\, 
+qt  \tresse+ q^3t^2\tresse\,\right)\\
&=
q^{2(n-1)}\left(\, [n]_{q^{-n}t^{-1},qt} 
+qt  [n-1]_{q^{-n}t^{-1},qt}+ q^3t^2[n-1]_{q^{-n}t^{-1},qt}\,\right)[n]_{q^{-n}t^{-1},qt}\, \emptyset\\
&=
q^{2(n-1)}\left(\, q^{1-n}+ q^3t^2[n-1]_{q^{-n}t^{-1},qt}\,\right)[n]_{q^{-n}t^{-1},qt}\,  \emptyset,
\end{align*}
and now the coefficient of this expression formally reproduces the form for the correct expression for the Poincar\'e polynomial of the Khovanov homology of the Hopf link. 

This can be elegantly rephrased as follows. Consider the ring, that we will call the Morozov ring,
\[
\mathbb{M}=\mathbb{Q}[t,t^{-1},q,q^{-1},\mathbf{[n]},\mathbf{[n-1]},\dots \mathbf{[1]}],
\]
where $\mathbf{[n]},\mathbf{[n-1]},\dots \mathbf{[1]}$ are independent variables and 
its quotient
\[
\widetilde{\mathbb{M}}=\mathbb{M}/I_M,
\]
where $I_M$ is the ideal generated by a certain set of relations, including the following two:
\footnote{
It should be noticed that as soon as we localize $\widetilde{\mathbb{M}}$ at $qt-q^{-1}t^{-1}$, i.e., as soon as we declare $qt-q^{-1}t^{-1}$ to be invertible
the two relations $\mathbf{[n]}+qt\mathbf{[n-1]}=q^{1-n}$ and $\mathbf{[n]}+q^{-1}t^{-1}\mathbf{[n-1]}=q^{n-1}$
force $\mathbf{[n]}=[n]_{q^{-n}t^{-1},qt}$ and $\mathbf{[n-1]}=[n-1]_{q^{-n}t^{-1},qt}$. Similarly, one has relations in $I_M$ (that we are not expliciting) that force $
\mathbf{[k]}=[k]_{q^{-n}t^{-1},qt}$
as soon as $qt-q^{-1}t^{-1}$ is invertible.
}
\[
\mathbf{[n]}+qt\mathbf{[n-1]}=q^{1-n}
\]
\[
\mathbf{[n]}+q^{-1}t^{-1}\mathbf{[n-1]}=q^{n-1}.
\]
Then, on the monoidal category $\mathbf{TrPD}$ with coefficients in $\tilde{\mathbb{M}}$ introduce the MOY-type relations
\begin{enumerate}
\item $\mathrm{ev}\circ \mathrm{coev} = \mathbf{[n]}$,
i.e.,
\begin{equation}\label{moy-1M}\tag{moy-1${}_{\mathbb{M}}$}
\iunknot=\mathbf{[n]}\, \emptyset
\end{equation}
\item $S\circ S=(q-t^{-1}q^{-1})S$,
i.e.,
\begin{equation}\label{moy-2M}\tag{moy-2${}_{\mathbb{M}}$}
\essecircesse=(q-t^{-1}q^{-1})\esse
\end{equation}
\item $(\mathrm{ev}\otimes \mathrm{id})\circ (\mathrm{id}^\vee \otimes S)\circ (\mathrm{coev}\otimes \mathrm{id})=\mathbf{[n-1]}\,\mathrm{id}$, 
i.e.,
\begin{equation}\label{moy-3M}\tag{moy-3${}_{\mathbb{M}}$}
{\xy
(0,-10);(0,10)**\crv{(0,-7)&(0,-6)&(0,-5)&(0,0)&(0,5)&(0,6)&(0,7)}
?>(.2)*\dir{>}?>(.99)*\dir{>}?>(.55)*\dir{>}
,(0,-3);(0,3)**\crv{(0,-4)&(0,-5)&(-4,-6)&(-6,0)&(-4,6)&(0,5)&(0,4)}
?>(.55)*\dir{<}
,(2,-11)*{\scriptstyle{1}}
,(2,11)*{\scriptstyle{1}}
,(2,0)*{\scriptstyle{2}}
,(-6,0)*{\scriptstyle{1}}
\endxy}\,\,\,=
\mathbf{[n-1]}\,\,{\xy
(0,-10);(0,10)**\dir{-}
?>(.95)*\dir{>}
,(2,-11)*{\scriptstyle{1}}
,(2,11)*{\scriptstyle{1}}
\endxy}
\end{equation}

\item $(\mathrm{id}\otimes \mathrm{ev}\otimes\mathrm{id}^\vee)\circ (S\otimes S^\vee) \circ (\mathrm{id}\otimes \mathrm{coev}\otimes\mathrm{id}^\vee)=\mathrm{coev}\circ\mathrm{ev}+ \mathbf{[n-2]}\, \mathrm{id}\otimes \mathrm{id}^\vee$,
i.e.,
\begin{equation}\label{moy-4M}\tag{moy-4${}_{\mathbb{M}}$}
{\xy
(-10,-10);(-10,10)**\crv{(-6,-6)&(-5,-5)&(-5,-2)&(-5,-2)&(-5,5)&(-6,6)}
?>(.05)*\dir{>}?>(.99)*\dir{>}?>(.55)*\dir{>}
,(10,-10);(10,10)**\crv{(6,-6)&(5,-5)&(5,-2)&(5,2)&(5,5)&(6,6)}
?>(.03)*\dir{<}?>(.92)*\dir{<}?>(.45)*\dir{<}
,(-5,0);(5,0)**\crv{(-5,-2)&(-5,-3)&(-4,-5)&(-2,-6)&(2,-6)&(4,-5)&(5,-3)&(5,-2)}?>(.5)*\dir{<}
,(-5,0);(5,0)**\crv{(-5,2)&(-5,3)&(-4,5)&(-2,6)&(2,6)&(4,5)&(5,3)&(5,2)}?>(.5)*\dir{>}
,(-11,-10)*{\scriptstyle{1}}
,(-11,10)*{\scriptstyle{1}}
,(11,-10)*{\scriptstyle{1}}
,(11,10)*{\scriptstyle{1}}
,(-7,0)*{\scriptstyle{2}}
,(0,-8)*{\scriptstyle{1}}
,(6.5,0)*{\scriptstyle{2}}
,(0,8)*{\scriptstyle{1}}
\endxy}
\quad
=
\xy
(-5,-10);(5,-10)**\crv{(-5,-2)&(0,-2)&(5,-2)}
?>(1)*\dir{>},
(-6,-10)*{\scriptstyle{1}},(6,-10)*{\scriptstyle{1}}
,(-5,10);(5,10)**\crv{(-5,2)&(0,2)&(5,2)}
?>(0)*\dir{<},
(-6,10)*{\scriptstyle{1}},(6,10)*{\scriptstyle{1}}
\endxy
+\, \mathbf{[n-2]}\,
{\xy
(-5,-10);(-5,10)**\dir{-}?>(.5)*\dir{>}
,(5,-10);(5,10)**\dir{-}?>(.5)*\dir{<}
,(-7,-10)*{\scriptstyle{1}}
,(7,-10)*{\scriptstyle{1}}
,(-7,10)*{\scriptstyle{1}}
,(7,10)*{\scriptstyle{1}}
\endxy}
\end{equation}

\item $(S\otimes\mathrm{id})\circ (\mathrm{id}\otimes S) \circ (S\otimes\mathrm{id})+(\mathrm{id}\otimes S)=(\mathrm{id}\otimes S) \circ (S\otimes\mathrm{id})\circ (\mathrm{id}\otimes S)+(S\otimes\mathrm{id})$,
i.e.,
\begin{equation}\label{moy-5M}\tag{moy-5${}_{\mathbb{M}}$}
\coso\,\,+\,\,\idone\quad \esse = \cosob\,\,+\,\, \esse\quad\idoner
\end{equation}
\end{enumerate}
as well as the braidings
\[
\sigma^+_{\mathbb{M}}=q^nt\left(q^{-1}t^{-1}\,\idtwo\,+\esse\,\right); \qquad 
\sigma^-_{\mathbb{M}}=q^{-n}t^{-1}\left(qt\,\idtwo\,+\esse\,\right),
\]
Now, the above computation precisely tells that by quotienting $\mathbf{TrPD}$ by the relations (\ref{moy-1M}-\ref{moy-5M}) we get a rigid monoidal category $\mathbf{MOY}_{\tilde{\mathbb{M}}}$ which is not braided but such that the string diagram computation of the planar link
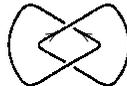
\[
\hopf
\]
presenting the Hopf link gives
\[
Z_{\mathbb{M}}\left(\hopf\right)=q^{2(n-1)}\left(
q^{1-n}+q^3t^2\mathbf{[n-1]}
\right)\mathbf{[n]} \,\emptyset.
\]
The coefficient on the right hand side is an element in the quotient ring $\tilde{\mathbb{M}}$. However, if we consider the representative element $\mathcal{P}_{\mathbb{M}}({\rm Hopf})\in \mathbb{M}$ given by
\[
\mathcal{P}_{\mathbb{M}}({\rm Hopf}):=q^{2(n-1)}\left(
q^{1-n}+q^3t^2\mathbf{[n-1]}
\right)\mathbf{[n]},
\] 
this has the following remarkable property: the morphism of rings $\mathbb{M}\to \mathbb{Q}[t,t^{-1},q,q^{-1}]$ given by $\mathbf{[k]}\mapsto [k]_q$ maps it to the Poincar\'e polynomial of the Khovanov homology of the Hopf link. Notice that the morphism $\mathbf{[k]}\mapsto [k]_q$ does \emph{not} induce a morphism of rings $\tilde{\mathbb{M}}\to \mathbb{Q}[t,t^{-1},q,q^{-1}]$.
\par
This considerations lead to the following conjecture, due to Anokhina, Dolotin and Morozov, who checked it in an enormous number of cases: \emph{for any link $\Gamma$ there exists an element $\mathcal{P}_{\mathbb{M}}(\Gamma)$ in $\mathbb{M}$ which is mapped to the Poincar\'e polynomial of the Khovanov homology of the link $\Gamma$ by the rings homomorphism $\mathbb{M}\to \mathbb{Q}[t,t^{-1},q,q^{-1}]$ given by $\mathbf{[k]}\mapsto [k]$, and that is a representative for the element in $\widetilde{\mathbb{M}}$ obtained by evaluating a planar link diagram representing $\Gamma$ according to the string diagrams calculus rules for $\mathbf{MOY}_{\tilde{\mathbb{M}}}$}.
\vskip .7 cm
Formulated this way\footnote{This formulation is our interpretation of a few of the constructions in \cite{Morozov1}-\cite{Morozov2}: any inaccuracy should be attributed to us and not to the authors of the articles we took inspiration from.} the conjecture seems to deal with black magic rather than with Mathematics: the rigid category $\mathbf{MOY}_{\tilde{\mathbb{M}}}$ is not  braided so there is no reasonable hope one can extract link invariants from it. And indeed, in the papers \cite{Morozov1}-\cite{Morozov2} there are a few points where different choices are possible and ``one has to do the right choice in order to get the correct result''. Yet, the principle inspiring it is a strong one: it should be possible to compute the Poincar\'e polynomial of the Khovanov homology of a link $\Gamma$ directly in terms of combinatorial manipulations of the MOY graphs associated with $\Gamma$, and avoiding a representation of these in terms of matrix factorizations as in the original Khovanov-Rozansky work \cite{Khovanov-Rozansky}. 
\vskip .7 cm
This is the point of view we take in this Thesis, arriving at a formalization that, beyond providing abstractly a link invariant\footnote{This is essentially a rewriting of the proofs in Khovanov-Rozansky's \cite{Khovanov-Rozansky}, changing into axioms the properties of the specific matrix factorizations used there.} is suitable for at least simple explicit computations. This way we are able to provide a generators and relations computation  of the  Poincar\'e polynomial of the Khovanov homology of the $k$-crossing 2-strand braid link. Up to our knowledge, such a computation was not previously available in the literature if not in low $n$ cases or in a form which is either conjectural or appealing to Physics arguments. Our result confirms all of these previously available computations or conjectures.
 

\chapter{The category $\mathbf{KR}$ of formal complexes of MOY graphs}\label{chapter:KR}

The aim of this chapter is to define a braided monoidal category $\mathbf{KR}$ generated by an object $\uparrow^1$ together with a braided monoidal functor 
\[
\mathcal{P}\colon \mathbf{KR}\to \mathbf{MOY}
\]
(the ``Euler characteristic'') factoring the Reshetikhin-Turaev universal morphism $Z\colon \mathbf{TD}\to \mathbf{MOY}$. As we have shown in Chapter \ref{chapter:moy} that the category $\mathbf{MOY}$ is isomorphic to the category $\mathbf{Jones}_n$, this construction provides a refinement of the level $n$ Jones polynomial. More precisely, writing 
\[
Kh_n\colon \mathbf{TD}\to  \mathbf{KR}
\]
for the Reshetikhin-Turaev universal morphism given by the choice of the generating object $\uparrow^1$ (the ``Khovanov-Rozansky invariant''), the factorization
\[
\xymatrix{
\mathbf{TD}\ar[r]^{Kh_n}\ar@/_2pc/[rr]_{Z} &\mathbf{KR} \ar[r]^{\mathcal{P}} &\mathbf{Jones}_n
}
\]
exhibits the level $n$ Jones polynomial of a link $\Gamma$ as the Euler characteristic of its Khovanov-Rozansky invariant.

\section{Formal direct sums of shifted MOY-type graphs}
Morphisms in the category $\mathbf{KR}$ will be certain formal complexes of formal direct sums of shifted planar directed graphs of MOY type. We will construct these in steps, starting with the graphs.
\begin{definition}
The collection of graphs of \emph{MOY type} is the collection of oriented planar diagrams with labelled edges obtained by extending the collection of generating morphisms in the category $\mathbf{MOY}$ with the distinguished morphism
\[
T\colon \uparrow^1\otimes \uparrow^1\otimes\uparrow^1 \to \uparrow^1\otimes \uparrow^1\otimes\uparrow^1
\]
given by
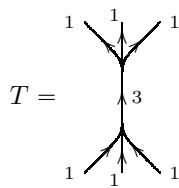
\[
T=\ti
\]
and its dual (i.e., the same diagram with reversed orientations).
\end{definition}
\begin{remark}
By saying that we extend the collection of morphism in $\mathbf{MOY}$ by $T$ in the above definition, we mean that all the planar graphs obtained by the generating morphisms in $\mathbf{MOY}$ and by copies of $T$ and $T^\vee$ by tensor product or composition is in our collection of MOY-type graphs. That is, a MOY-type graph is a planar directed diagram that can ``built'' out of the following elementary pieces:
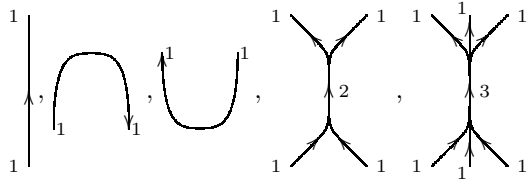
\[
\idone\,\,,\, \ev\,,\, \coev\,,\, \esse\,,\, \ti
\]
and their duals.
\end{remark}

\begin{remark}
The auxiliary graph $T$ is introduced to conveniently handle the invariance by the third Reidemeister move. Its introduction may seem a bit artificial and indeed in some sense it is. A more transparent treatment would be obtained by including objects $\uparrow^i$ for any $i\in\{1,2,\dots,n\}$ and consider MOY-type graphs with all possible labels on the edges in the set $\{1,\dots,n\}$ subject to the constrain that at each vertex the sum of the incoming labels equals the sum of outgoing labels. These more general graphs already appear in the original paper \cite{moy}, which then focuses on the particular graphs we used to define the morphisms in the category $\mathbf{MOY}$. The general case is investigated in \cite{wu}, where it is related to the $n$-coloured Jones polynomial and to a colored version of Khovanov-Rozansky homology. Following \cite{wu} one could indeed easily generalize the constructions of the category $\mathbf{MOY}$ to the construction of a category $\mathbf{MOY}_{\{1,2,\dots,n\}}$. We did not pursue this generalization here, as the corresponding generalization of $\mathbf{KR}$ seems presently to be out of reach, due to the difficulty of extracting a complete set of elementary morphisms and of relations between them from the complex of matrix factorizations associated to colored MOY-type graphs in \cite{wu}  
\end{remark}

\begin{definition} A \emph{shifted} MOY-type graph is an object of the form $\Gamma\{m\}$ where $\Gamma$ is a MOY-type graph and $m$ is an integer, which we will call the ``shift'' in accordance with the terminology in \cite{Khovanov-Rozansky}. A formal sum of shifted MOY-type graphs is an expression of the form
\[
\oplus_i \Gamma_i\{m_i\}
\]
where $i$ ranges over a finite set of indices, and all the MOY-type graphs in the above expression have the same source and the same target data (i.e., the same source and target sequences of $\uparrow^1$'s and $\downarrow^1$'s).
\end{definition}

\begin{remark}
In \cite{Khovanov-Rozansky} the shift numbers correspond to actual degree shifts in complexes of matrix factorizations. Here they are just labels. However we are going to compare the two shifts as we will compare the combinatorial category $\mathbf{KR}$ the the ``derived Ginzburg-Landau category'' of homotopy classes of complexes of matrix factorizations in Section \ref{sec:comparison}.
\end{remark}
\begin{notation}
The label $0$ will usually be omitted, i.e., we will simply write $\Gamma$ for $\Gamma\{0\}$. 
\end{notation}
\begin{notation}
For every nonnegative integer $m$, we will write $[m]\,\Gamma$ for the formal direct sum
\[
[m]\,\Gamma=\Gamma\{1-m\}\oplus \Gamma\{3-m\}\oplus \cdots \oplus \Gamma\{m-1\}.
\]
When $\Gamma$ is the empty diagram we will often omit it from the notation, i.e., we will simply write $[m]$ for $[m]\emptyset$. Also, we will write $\mathbb{Q}\{m\}$ for $\emptyset\{m\}$, so that
\[
[m]=\mathbb{Q}\{1-m\}\oplus \mathbb{Q}\{3-m\}\oplus \cdots \oplus \mathbb{Q}\{m-1\}.
\]
\end{notation}
\begin{remark}
We have a natural notion of direct sum of formal sums of shifted MOY-type graphs, as well as a tensor product of formal sums of MOY-type graphs by finite dimensional graded $\mathbb{Q}$-vector spaces with graded bases (i.e., with a given isomorphism of graded  $\mathbb{Q}$-vector spaces to a graded  $\mathbb{Q}$-vector space of the form
\[
\oplus_i \mathbb{Q}^{n_i}\{m_i\},
\]
where $\{m_i\}$ denotes an actual degree shift here. The tensor product rule is, clearly,
\[
(\mathbb{Q}^k\{m\})\otimes \Gamma\{l\}=\underbrace{\Gamma\{k+l\}\oplus \Gamma\{k+l\}\oplus\cdots \Gamma\{k+l\}}_{k\text{ times}},
\]
extended by bilinearity to formal direct sums of MOY-type graphs and direct sums of graded $\mathbb{Q}$-vector spaces.

\end{remark}

\begin{remark}
We can extend the composition and tensor product of MOY-type graphs (defined in the same way as the composition and tensor product of morphisms in $\mathbf{MOY}$) to shifthed MOY-type graphs by the rule
\[
\Gamma\{m\}\circ \Phi\{k\}=\Gamma\circ\Phi\{m+k\}; \qquad \Gamma\{m\}\otimes \Phi\{k\}=\Gamma\otimes\Phi\{m+k\}.
\]
As all MOY-type graphs in a formal direct sum have the same source sequence and the same target sequence of of $\uparrow^1$'s and $\downarrow^1$'s, we can extend these operatins by bilinearity and define a composition $M\circ N$ and a tensor product $M\otimes N$ of formal direct sums $M$ and $N$ of shifted MOY-type graphs. All the graphs appearing in  $M\circ N$ and in $M\otimes N$ will be shifted copies of MOY-type graphs all having the same source sequence and the same target sequence of of $\uparrow^1$'s and $\downarrow^1$'s; so both $M\circ N$ and  $M\otimes N$ will be again formal direct sums of shifted MOY-type graphs.
\end{remark}

\section{Morphisms between formal direct sums}
The next step consists in defining a set of morphisms between the formal direct sums of MOY-type graphs introduced in the previous section. As, by analogy with the $\mathbf{MOY}$ category we may think of formal direct sums of MOY-type graphs as morphisms between sequences of $\uparrow^1$'s and $\downarrow^1$'s, this step suggests we are here going towards the construction of a 2-category. This is indeed so. However, we will not be interested in constructing such a 2-category in full detail, as we are going to eventually decategorify it by taking suitable equivalence classes of 2-morphisms. Therefore, we will not need completely constructing the 2-category: we will directly introduce certain equivalence classes of complexes of formal sums of shifted MOY-type graphs which will be our 1-morphisms between sequences of $\uparrow^1$'s and $\downarrow^1$'s. These can be thought of as equivalence classes of 2-morphisms, although we actually do not really construct the 2-category. Before introducing complexes of (formal sums of shifted) MOY-type graphs, we need define what the arrows in these complexes (i.e., the morphisms between MOY-type graphs) would be. This is done in several steps.
\subsection{Elemementary morphisms}
In this section we introduce a set of generators for the morphisms in a category of MOY-type graphs. We will call these generators ``elementary morphism'', although at this stage they are not, of course, yet morphisms in a category. The idea is that the morphisms in the category we are going to describe will be ``built up'' starting from these ``elementary morphisms''. Sources and targets of these elementary morphisms will be formal directed sums of shifted  MOY-type graphs as defined in the previous section.
Finally $n$ will be a fixed integer with $n\geq 2$.

\begin{definition}\label{def:generating}
The set $\mathcal{E}$ of \emph{elementary morphisms} is the collection of arrows $\{1_\Gamma, i^u_\Gamma,i_{d;\Gamma}, \chi_0,\chi_1,\alpha,\gamma,\lambda, \epsilon, \mu, \varphi,\eta \}$, with $\Gamma$ ranging among all MOY graphs,
where
\begin{itemize}
\item $1_\Gamma$ is the identity arrow
\[
1_\Gamma\colon \Gamma \to \Gamma
\]
(we will often denote this arrow just by $1$ leaving $\Gamma$ be understood by the context)
\item $i^u_\Gamma$ and $i_{d;\Gamma}$ are two invertible arrows
\[
i^u_\Gamma\colon \Gamma \xrightarrow{\sim}  \mathrm{id}_{t(\Gamma)}\circ \Gamma; \qquad \qquad i_{d;\Gamma}\colon \Gamma \xrightarrow{\sim} \Gamma\circ \mathrm{id}_{s(\Gamma)}
\]
where $s(\Gamma)$ and $t(\Gamma)$ denote the source and the target of $\Gamma$, respectively.
(we will often denote these arrows just by $i^u$ and $i_d$, leaving $\Gamma$ be understood by the context)
\item $\chi_0$ is a distinguishes arrow
\[
\chi_0\colon \mathrm{id}_2\to S\{1\},
\]
i.e., in graphical notation,
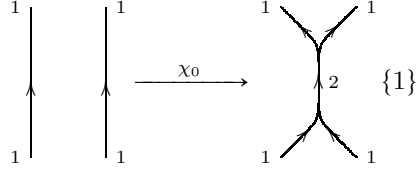
\[
\idtwo\xrightarrow{\phantom{mm}\chi_0\phantom{mm}} \,
\esse\{1\}
\]
\item $\chi_1$ is a distinguishes arrow
\[
\chi_1\colon S\{-1\}\to \mathrm{id}_2,
\]
i.e., in graphical notation,
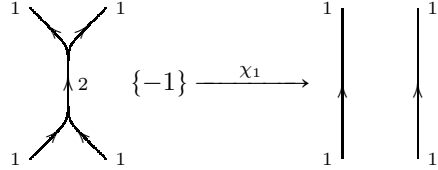
\[
\esse\{-1\}
\xrightarrow{\phantom{mm}\chi_1\phantom{mm}} \,
\idtwo\]

\item $\epsilon$ is a distinguished arrow
\[
\mathrm{id}_1\{-1\}\xrightarrow{\epsilon} \mathrm{id}_1\{1\},
\]
i.e., in graphical notation,
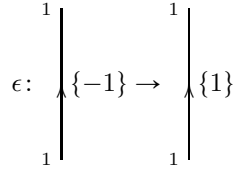
\[
\epsilon\colon \idone\,\,\{-1\}\to \idone\,\,\{1\}
\]

\item $\lambda$ is a distinguished invertible arrow
\[
\mathrm{ev}\circ \mathrm{coev}\xrightarrow[\sim]{\lambda} [n]=[n]\,\emptyset,
\]
i.e., in graphical notation,
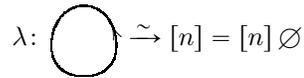
\[
\lambda\colon \unknot\xrightarrow{\sim} [n]=[n]\,\emptyset
\]

\item $\mu$ is a distinguished invertible arrow
\[
\mu\colon (\mathrm{ev}\otimes \mathrm{id}_1)\circ(\mathrm{id}_1\otimes S)\circ (\mathrm{coev}\otimes \mathrm{id}_1)\xrightarrow{\sim}
[n-1]\, \mathrm{id}_1,
\]
i.e., in graphical notation
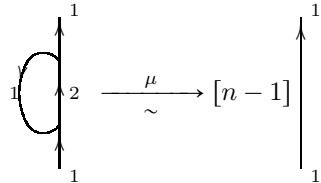
\[
{\xy
(0,-10);(0,10)**\crv{(0,-7)&(0,-6)&(0,-5)&(0,0)&(0,5)&(0,6)&(0,7)}
?>(.2)*\dir{>}?>(.99)*\dir{>}?>(.55)*\dir{>}
,(0,-3);(0,3)**\crv{(0,-4)&(0,-5)&(-4,-6)&(-6,0)&(-4,6)&(0,5)&(0,4)}
?>(.55)*\dir{<}
,(2,-11)*{\scriptstyle{1}}
,(2,11)*{\scriptstyle{1}}
,(2,0)*{\scriptstyle{2}}
,(-6,0)*{\scriptstyle{1}}
\endxy}\,\,\,\xrightarrow[\sim]{\phantom{mm}\mu\phantom{mm}}
[ n-1]\,\,{\xy
(0,-10);(0,10)**\dir{-}
?>(.95)*\dir{>}
,(2,-11)*{\scriptstyle{1}}
,(2,11)*{\scriptstyle{1}}
\endxy}
\]
\item $\varphi$ and $\psi$ are two a distinguished invertible arrows
\[
\varphi,\psi\colon S\circ S\xrightarrow{\sim} S\{-1\}\oplus S\{1\},
\]
i.e., in graphical notation
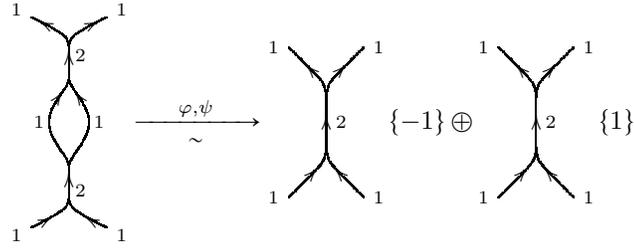
\[
\essecircesse\xrightarrow[\sim]{\phantom{mm}\varphi,\psi\phantom{mm}}
\esse\{-1\}
\oplus
\esse\{1\}
\]

\item $\alpha,\gamma$  are distinguished arrows
\[
\alpha,\gamma\colon S\{-1\}\to S\{1\},
\]
i.e., in graphical notation
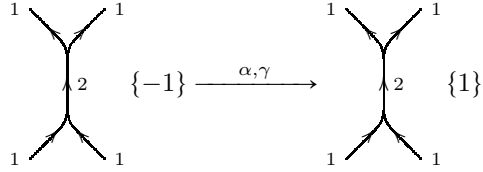
\[
\esse\{-1\}\xrightarrow{\phantom{mm}\alpha,\gamma\phantom{mm}}
\esse\{1\}
\]

\item $\nu$ is a distinguished invertible arrow
\[
\nu\colon (\mathrm{id}_1\otimes \mathrm{ev}\otimes\mathrm{id}_1^\vee)\circ (S\otimes S^\vee) \circ (\mathrm{id}_1\otimes \mathrm{coev}\otimes\mathrm{id}_1^\vee)
\xrightarrow{\sim} \mathrm{coev}\circ\mathrm{ev}\oplus [n-2]\otimes \mathrm{id}_1\otimes \mathrm{id}_1^\vee
\]
i.e., in graphical notation
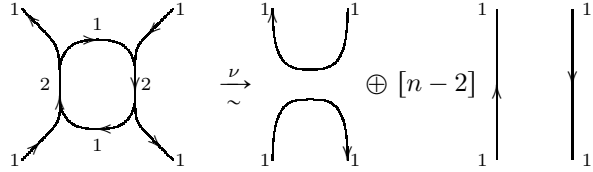
\[
{\xy
(-10,-10);(-10,10)**\crv{(-6,-6)&(-5,-5)&(-5,-2)&(-5,-2)&(-5,5)&(-6,6)}
?>(.05)*\dir{>}?>(.99)*\dir{>}?>(.55)*\dir{>}
,(10,-10);(10,10)**\crv{(6,-6)&(5,-5)&(5,-2)&(5,2)&(5,5)&(6,6)}
?>(.03)*\dir{<}?>(.92)*\dir{<}?>(.45)*\dir{<}
,(-5,0);(5,0)**\crv{(-5,-2)&(-5,-3)&(-4,-5)&(-2,-6)&(2,-6)&(4,-5)&(5,-3)&(5,-2)}?>(.5)*\dir{<}
,(-5,0);(5,0)**\crv{(-5,2)&(-5,3)&(-4,5)&(-2,6)&(2,6)&(4,5)&(5,3)&(5,2)}?>(.5)*\dir{>}
,(-11,-10)*{\scriptstyle{1}}
,(-11,10)*{\scriptstyle{1}}
,(11,-10)*{\scriptstyle{1}}
,(11,10)*{\scriptstyle{1}}
,(-7,0)*{\scriptstyle{2}}
,(0,-8)*{\scriptstyle{1}}
,(6.5,0)*{\scriptstyle{2}}
,(0,8)*{\scriptstyle{1}}
\endxy}
\quad
\xrightarrow[\sim]{\nu}
\xy
(-5,-10);(5,-10)**\crv{(-5,-2)&(0,-2)&(5,-2)}
?>(1)*\dir{>},
(-6,-10)*{\scriptstyle{1}},(6,-10)*{\scriptstyle{1}}
,(-5,10);(5,10)**\crv{(-5,2)&(0,2)&(5,2)}
?>(0)*\dir{<},
(-6,10)*{\scriptstyle{1}},(6,10)*{\scriptstyle{1}}
\endxy
\oplus\, [n-2]
{\xy
(-5,-10);(-5,10)**\dir{-}?>(.5)*\dir{>}
,(5,-10);(5,10)**\dir{-}?>(.5)*\dir{<}
,(-7,-10)*{\scriptstyle{1}}
,(7,-10)*{\scriptstyle{1}}
,(-7,10)*{\scriptstyle{1}}
,(7,10)*{\scriptstyle{1}}
\endxy}
\]
\item $\eta$ is a distinguished invertible arrow
\[
\eta\colon (S\otimes\mathrm{id}_1)\circ (\mathrm{id}_1\otimes S) \circ (S\otimes\mathrm{id}_1)\xrightarrow{\sim} (S\otimes \mathrm{id}_1)\oplus T
\]
i.e., in graphical notation
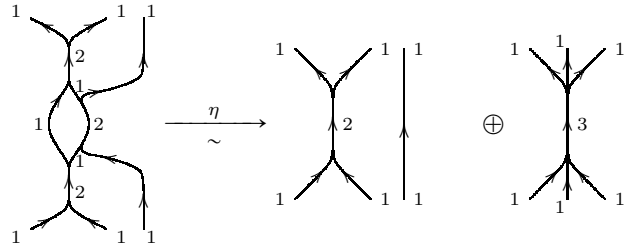
\[
\coso
\xrightarrow[\sim]{\phantom{mm}\eta\phantom{mm}}
\esse\,\,\idoner\qquad
\oplus \ti
\]
\end{itemize}
The set $\mathcal{E}$ also contains all the vertical and horizontal reflections of the above arrows, which we will denote by the same symbols. For instance, we also have in $\mathcal{E}$ the distinguished invertible arrow
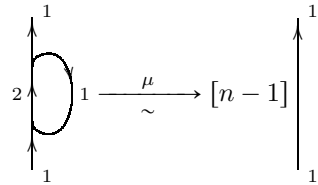
\[
{\xy
(0,-10);(0,10)**\crv{(0,-7)&(0,-6)&(0,-5)&(0,0)&(0,5)&(0,6)&(0,7)}
?>(.2)*\dir{>}?>(.99)*\dir{>}?>(.55)*\dir{>}
,(0,-3);(0,3)**\crv{(0,-4)&(0,-5)&(4,-6)&(6,0)&(4,6)&(0,5)&(0,4)}
?>(.55)*\dir{<}
,(2,-11)*{\scriptstyle{1}}
,(2,11)*{\scriptstyle{1}}
,(-2,0)*{\scriptstyle{2}}
,(7,0)*{\scriptstyle{1}}
\endxy}\xrightarrow[\sim]{\phantom{mm}\mu\phantom{mm}}
[ n-1]\,\,{\xy
(0,-10);(0,10)**\dir{-}
?>(.95)*\dir{>}
,(2,-11)*{\scriptstyle{1}}
,(2,11)*{\scriptstyle{1}}
\endxy}
\]
etc. 
\end{definition}

We will usually omit the morphisms $i^u$ and $i_d$ from the notation. So, for instance, we will simply write
\[
S\xrightarrow{\chi_0\circ 1} S\circ S\{1\}
\]
to mean the composition
\[
S\xrightarrow{i^u_S}\mathrm{id}_2\circ S\xrightarrow{\chi_0\circ 1}S\circ S\{1\}
\]
Graphically, this means that we simply write
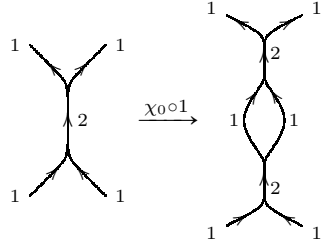
\[
\esse
\xrightarrow{\chi_0\circ 1} 
\essecircesse
\]
to actually mean the composition
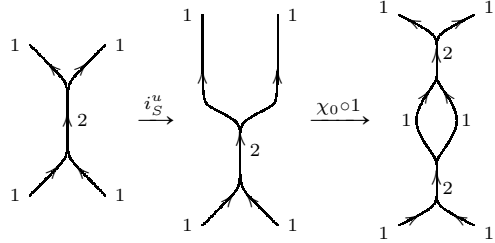
\[
\esse
\xrightarrow{i^u_S}
{\xy
(-5,-14);(-5,14)**\crv{(-1,-10)&(0,-9)&(0,-6)&(0,-2)&(0,0)&(-5,2)&(-5,3)}
?>(.55)*\dir{>}?>(.1)*\dir{>}?>(.95)*\dir{>}
,(5,-14);(5,14)**\crv{(1,-10)&(0,-9)&(0,-6)&(0,-2)&(0,0)&(5,2)&(5,3)}
?>(.1)*\dir{>}?>(.95)*\dir{>}
,(-7,-14)*{\scriptstyle{1}}
,(7,-14)*{\scriptstyle{1}}
,(-7,14)*{\scriptstyle{1}}
,(7,14)*{\scriptstyle{1}}
,(2,-4)*{\scriptstyle{2}}
\endxy}\,
\xrightarrow{\chi_0\circ 1} 
\essecircesse
\]
Also, when there are no ambiguities, we will also omit the identity morphisms, so for instance we will simply write
\[
\unknot\,\,\{-1\}\xrightarrow{\epsilon} \unknot\,\,\{1\}
\]
to mean the composition
\[
\unknot\,\,\{-1\} \xrightarrow[\sim]{1\circ i^u_{\mathrm{coev}}} {\xy
(0,-10);(0,-10)**\crv{(5,-10)&(5,10)&(-5,10)&(-5,-10)}
?>(.3)*\dir{>}\endxy}\,\,\{-1\} \xrightarrow{1\circ(1\otimes \epsilon)\circ 1} {\xy
(0,-10);(0,-10)**\crv{(5,-10)&(5,10)&(-5,10)&(-5,-10)}
?>(.3)*\dir{>}\endxy}\,\,\{1\}\xrightarrow[\sim]{1\circ (i^u_{\mathrm{coev}})^{-1}}\unknot\,\,\{1\}
\]
Finally, if $f\colon \Gamma_1\to \Gamma_2$ is an elementary morphism, we denote by the same symbol its ``shifted'' versions, i.e., we simply write $f\colon \Gamma_1\{m\}\to \Gamma_2\{m\}$ instead of the more precise $f\{m\}\colon \Gamma_1\{m\}\to \Gamma_2\{m\}$.

\begin{remark}
Notice that if $f\colon \Gamma_1\to \Gamma_2$ is a morphism between (formal direct sums of shifted) MOY-type graphs, then $\Gamma_1$ and $\Gamma_2$ have the same incoming and the same outgoing sequences of $\uparrow^1$'s and $\downarrow^1$'s. This is coherent with the idea we are (partially) constructing a 2-categorification of $\mathbf{MOY}$ as mentioned at the beginning of this chapter. 
\end{remark}

\subsection{Composite morphisms}
Starting from the elementary morphisms we can consider the ${\mathbb{Q}}$-vector space generated by the elements in  $\mathcal{E}$ by $\mathbb{Q}$-linear combinations, compositions and direct sums. More precisely, we give the following
\begin{definition}
The set $\mathcal{M}$ of morphisms between formal direct sums of MOY-type graphs is the smallest set with the following properties:
\begin{itemize}
\item $\mathcal{E}\subseteq \mathcal{M}$;

\item if $f,g\colon \Gamma_1\to \Gamma_2$ are $\mathcal{M}$, then for any $a,b\in \mathbb{Q}$, the element $a\,g+b\,g\colon \Gamma_1\to \Gamma_2$ is in $\mathcal{M}$ ($\mathcal{M}$ is a $\mathbb{Q}$-vector space);

\item if $f\colon \Gamma_1\to \Gamma_2$ and $g\colon \Gamma_2\to \Gamma_3$ are in $\mathcal{M}$, then $gf\colon \Gamma_1\to \Gamma_3$ is in $\mathcal{M}$ (horizontal composition);

\item if $f\colon \Gamma_1\to \Gamma_2$ and $g\colon \Psi_1\to \Psi_2$ are in $\mathcal{M}$, and the compositions $\Psi_i\circ\Gamma_i$ are defined,  then $g\circ f\colon \Psi_1\circ\Gamma_1\to \Psi_2\circ\Gamma_2$ is in $\mathcal{M}$ (vertical composition);

\item if $f\colon \Gamma_1\to \Gamma_2$ is in $\mathcal{M}$ and $m\in \mathbb{Z}$, then $f\{m\}\colon \Gamma_1\{m\}\to \Gamma_2\{m\}$ is in $\mathcal{M}$;

\item  if $f\colon \Gamma_1\to \Gamma_2$ and $g\colon \Psi_1\to \Psi_2$ are in $\mathcal{M}$, then $f\oplus g\colon \Gamma_1\oplus \Psi_1\to \Gamma_2\oplus \Psi_2$ is in $\mathcal{M}$.
\end{itemize}
In the above conditions, $\Gamma_i$ and $\Phi_i$ denote arbitary formal direct sums of shifted MOY-type graphs.
\end{definition}

\begin{remark}
By the above definition, the set $\mathcal{M}$ can be  recursively built from $\mathcal{E}$ by considering morphisms of increasing complexity. 
\end{remark}

\subsection{Relations between the elementary morphisms}\label{sec:relations}
As we want to think of $\mathcal{M}$ as the set of morphisms between formal direct sums of MOY-type graphs we need to impose associativity relations and identity element axiom on the horizontal composition in $\mathcal{M}$, i.e., we impose the relations
\[
(fg)h=f(gh); \qquad f\,1_\Gamma=1_\Phi\, f=f
\]
for any composable $f,g,h,1_\Gamma,1\Phi$ in $\mathcal{M}$. Also, as we want to think of the above construction as part of the construction of a 2-categorification of $\mathbf{MOY}$, we also impose associativity and identity element axiom for the vertical composition as well as the compatibility between vertical and horizontal composition of elements in $\mathcal{M}$.
\begin{remark}\label{rem:bimonoid}
Having imposed the associative, identity and compatibility relations, we are now justified in thinking of $\mathcal{M}$ as a set of morphisms bewtween formal direct sums of MOY-type graphs. If we do not want to think of the elements in $\mathcal{M}$ as morphisms, as we have not fully defined the category of which $\mathcal{M}$ would be the set of morphisms, we can say that $\mathcal{M}$ modulo the associative, identity and compatibility relations is a $\mathbb{Q}$-linear partial bi-monoid.
\end{remark}
In any case, even after having quotiented out the associative, identity and compatibility relations, $\mathcal{M}$ is too big for any meaningful use: it is just the set of morphisms freely generated by our collection of elementary morphisms. To turn it in something that can actually be used to define a rigid braided monoidal category, and so ultmately to define link invariants we impose additional relations on the elements in $\mathcal{M}$.
\begin{remark}
In view of the higher categorical inspiration of the construction, these relations can be thought of as 3-morphism, so in a sense we are going to construct the 1-category $\mathbf{KR}$ by providing sketches of a 3-category that in the end we decategorify twice.
\end{remark}  
As $\mathcal{M}$ is generated by $\mathcal{E}$, in order to impose additional relations on $\mathcal{M}$ we impose a set of additional relations on certain compositions of elementary morphisms. The whole set of the additional relations will the be the set generated by these elementary relations.
\vskip .7 cm
The set of additional relations between compositions of elementary morphisms that we impose is the following:
\begin{itemize}
\item If $\mathbf{0}$ denotes the empty direct sum, then the diagrams
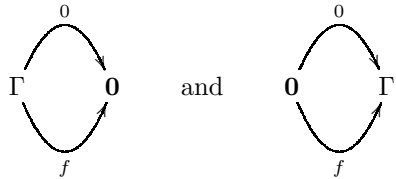
\[
\xymatrix{
\Gamma \ar@/_2pc/[r]_{f} \ar@/^2pc/[r]^{0}&\mathbf{0}
}
\qquad \text{and}\qquad
\xymatrix{
\mathbf{0} \ar@/_2pc/[r]_{f} \ar@/^2pc/[r]^{0}&\Gamma
}
\]
commute for any $f$ (i.e., the only morphisms to and from $\mathbf{0}$ are the zero morphism);

\item For any wo composable MOY garphs, $\Phi$ and $\Gamma$  the following diagram commutes:
\[
\xymatrix{
\Phi\circ\Gamma \ar@/_2pc/[r]_{i_{d;\Phi}\circ 1} \ar@/^2pc/[r]^{1_\Phi\circ i^u_\Gamma}&\Phi\circ\mathrm{id}\circ\Gamma.
}
\]

\item 
We have $\chi_1\chi_0=0$, i.e., the commutativity of the diagram
\[
\xymatrix{
 \mathrm{id}_2\{-1\}\ar[r]^-{\chi_0} \ar@/_2pc/[rr]_{0}& S\ar[r]^-{\chi_1}&  \mathrm{id}_2\{1\}
}
\]
i.e., graphically,
\[
\xymatrix{
\idtwo\{-1\}\ar[r]^-{\chi_0} \ar@/_4pc/[rr]_{0}& \esse\ar[r]^-{\chi_1}& \idtwo\{1\}
}
\]

\item We have a commutative diagram
\[
\xymatrix{
(\mathrm{ev}\circ \mathrm{coev})\{-1\} \ar[r]^-{1\circ (1\otimes \epsilon)\circ 1}\ar[d]_{\lambda}^{\wr} &
(\mathrm{ev}\circ \mathrm{coev}) \{1\}\ar[d]_{\lambda}^{\wr}\\
[n]\{-1\}\ar@{=}[d] & [n]\{1\}\ar@{=}[d]\\
\mathbb{Q}\{-n\}\oplus [n-1] 
\ar[r]^-{\tiny{\left(\begin{matrix}0&1\\0 & 0\end{matrix}\right)}}&
[n-1]\oplus\mathbb{Q}\{n\} 
}
\]
i.e., in graphical notation,
\begin{equation}\tag{moy$\epsilon$}\label{eq:moyepsilon}
\xymatrix{
*++++{\unknot\,\, \{-1\}} \ar[r]^-{\epsilon}\ar[d]_{\lambda}^{\wr} &
*++++{\phantom{mm}\unknot\,\, \{1\}}\ar[d]_{\lambda}^{\wr}\\
[n]\{-1\}\ar@{=}[d] & [n]\{1\}\ar@{=}[d]\\
\mathbb{Q}\{-n\}\oplus [n-1] 
\ar[r]^-{\tiny{\left(\begin{matrix}0&1\\0 & 0\end{matrix}\right)}}&
[n-1]\oplus\mathbb{Q}\{n\} 
}
\end{equation}

\item We have a commutative diagram
\[
\xymatrix{
(\mathrm{ev}\circ \mathrm{coev})\otimes \mathrm{id}_1\ar[r]^-{1_{\mathrm{id}_1}\otimes \chi_0}\ar[d]_{\lambda\otimes 1}^{\wr} &
(\mathrm{ev}\otimes \mathrm{id}_1)\circ(\mathrm{id}_1\otimes S)\circ (\mathrm{coev}\otimes \mathrm{id}_1)\{1\}\ar[dd]^{\mu\{1\}}_{\wr}\\
[n]\otimes \mathrm{id}_1\ar@{=}[d]&\\
\mathrm{id}_1\{1-n\}\oplus [n-1]\mathrm{id}_1\{1\}
\ar[r]^-{\tiny{\left(\begin{matrix}0&1\end{matrix}\right)}}& [n-1]\mathrm{id}_1\{1\}
}
\]
i.e., in graphical notation, a commutative diagram
\begin{equation}\tag{moy$\chi_0$}\label{eq:moychi0}
\xymatrix{
\qquad\unknot \quad {\xy
(0,-10);(0,10)**\dir{-}
?>(.95)*\dir{>}
,(2,-11)*{\scriptstyle{1}}
,(2,11)*{\scriptstyle{1}}
\endxy} \ar[r]^-{1_{\mathrm{id}_1}\otimes \chi_0}\ar[d]_-{\lambda\otimes 1}^-{\wr} &
{\xy
(0,-10);(0,10)**\crv{(0,-7)&(0,-6)&(0,-5)&(0,0)&(0,5)&(0,6)&(0,7)}
?>(.2)*\dir{>}?>(.99)*\dir{>}?>(.55)*\dir{>}
,(0,-3);(0,3)**\crv{(0,-4)&(0,-5)&(-4,-6)&(-6,0)&(-4,6)&(0,5)&(0,4)}
?>(.55)*\dir{<}
,(2,-11)*{\scriptstyle{1}}
,(2,11)*{\scriptstyle{1}}
,(-2,0)*{\scriptstyle{2}}
,(-6,0)*{\scriptstyle{1}}
\endxy}\{1\}\ar[d]^{\mu\{1\}}_{\wr}\\
{\xy
(0,-10);(0,10)**\dir{-}
?>(.95)*\dir{>}
,(2,-11)*{\scriptstyle{1}}
,(2,11)*{\scriptstyle{1}}
\endxy} \{1-n\}\oplus [n-1]\,{\xy
(0,-10);(0,10)**\dir{-}
?>(.95)*\dir{>}
,(2,-11)*{\scriptstyle{1}}
,(2,11)*{\scriptstyle{1}}
\endxy} \{1\}
\ar[r]^-{\tiny{\left(\begin{matrix}0&1\end{matrix}\right)}}& [n-1]\,{\xy
(0,-10);(0,10)**\dir{-}
?>(.95)*\dir{>}
,(2,-11)*{\scriptstyle{1}}
,(2,11)*{\scriptstyle{1}}
\endxy} \{1\}
}
\end{equation}

\item We have a commutative diagram
\[
\xymatrix{
(\mathrm{ev}\otimes \mathrm{id}_1)\circ(\mathrm{id}_1\otimes S)\circ (\mathrm{coev}\otimes \mathrm{id}_1)\{1\}\ar[dd]^{\mu\{-1\}}_{\wr}
\ar[r]^-{1_{\mathrm{id}_1}\otimes \chi_1}&(\mathrm{ev}\circ \mathrm{coev})\otimes \mathrm{id}_1\ar[d]_{\lambda\otimes 1}^{\wr} 
\\
&[n]\otimes \mathrm{id}_1\ar@{=}[d]\\
[n-1]\mathrm{id}_1\{1\}\ar[r]^-{\tiny{\left(\begin{matrix}0&1\end{matrix}\right)}}& \mathrm{id}_1\{1-n\}\oplus [n-1]\mathrm{id}_1\{1\}
 }
\]
i.e., in graphical notation, a commuatative diagram
\begin{equation}\tag{moy$\chi_1$}\label{eq:moychi1}
\xymatrix{
{\xy
(0,-10);(0,10)**\crv{(0,-7)&(0,-6)&(0,-5)&(0,0)&(0,5)&(0,6)&(0,7)}
?>(.2)*\dir{>}?>(.99)*\dir{>}?>(.55)*\dir{>}
,(0,-3);(0,3)**\crv{(0,-4)&(0,-5)&(-4,-6)&(-6,0)&(-4,6)&(0,5)&(0,4)}
?>(.55)*\dir{<}
,(2,-11)*{\scriptstyle{1}}
,(2,11)*{\scriptstyle{1}}
,(-2,0)*{\scriptstyle{2}}
,(-6,0)*{\scriptstyle{1}}
\endxy}\{-1\}\ar[d]^{\mu\{-1\}}_{\wr}
 \ar[r]^-{1_{\mathrm{id}_1}\otimes \chi_1}
&
\qquad\quad \unknot \quad {\xy
(0,-10);(0,10)**\dir{-}
?>(.95)*\dir{>}
,(2,-11)*{\scriptstyle{1}}
,(2,11)*{\scriptstyle{1}}
\endxy}\ar[d]_-{\lambda\otimes 1}^-{\wr} 
\\
[n-1]\,{\xy
(0,-10);(0,10)**\dir{-}
?>(.95)*\dir{>}
,(2,-11)*{\scriptstyle{1}}
,(2,11)*{\scriptstyle{1}}
\endxy} \{-1\}\ar[r]^-{\tiny{\left(\begin{matrix}1&0\end{matrix}\right)}}&
[n-1]\,{\xy
(0,-10);(0,10)**\dir{-}
?>(.95)*\dir{>}
,(2,-11)*{\scriptstyle{1}}
,(2,11)*{\scriptstyle{1}}
\endxy}\{-1\}\oplus {\xy
(0,-10);(0,10)**\dir{-}
?>(.95)*\dir{>}
,(2,-11)*{\scriptstyle{1}}
,(2,11)*{\scriptstyle{1}}
\endxy} \{n-1\} 
 }
\end{equation}

\item
We have a commutative diagram
\[
\xymatrix{
S\{-1\}\ar[r]^-{\chi_0\circ 1} \ar@/_2pc/[rr]_{\alpha}& S\circ S\ar[r]^-{1\circ \chi_1}& S\{1\}
}
\]
i.e., graphically,
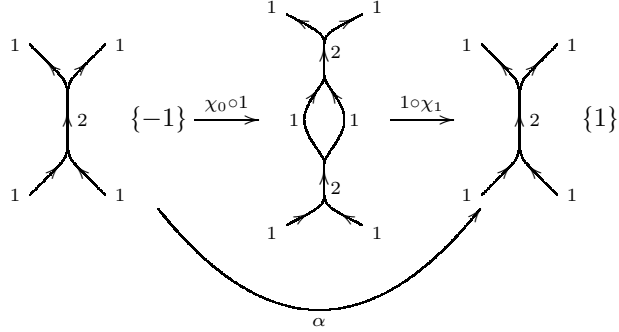
\[
\xymatrix{
\esse\{-1\}\ar[r]^-{\chi_0\circ 1} \ar@/_6pc/[rr]_{\alpha}& \essecircesse\ar[r]^-{1\circ \chi_1}& \esse\{1\}
}
\]
In other words, $\alpha=\chi_0\circ \chi_1$, i.e., $\alpha$ is the composition of $\chi_0$ and $\chi_1$. One could object that $\alpha$ is not elementary, as it is  a suitable compositions of $\chi_0$ and $\chi_1$. However, as it will show up several times in what follows, we find it convenient to think of $\alpha$ as an elementary morphism, which additionally enjoys a nice relationship with the other two elementary morphisms $\chi_0$ and $\chi_1$.

\item We have a commutative diagram
\[
\xymatrix{
(\mathrm{ev}\otimes \mathrm{id}_1)\circ(\mathrm{id}_1\otimes S)\circ (\mathrm{coev}\otimes \mathrm{id}_1)\{-1\}\ar[r]^-{1_{\mathrm{id}_1}\otimes \alpha}\ar[d]^{\mu\{1\}}_{\wr} &
(\mathrm{ev}\otimes \mathrm{id}_1)\circ(\mathrm{id}_1\otimes S)\circ (\mathrm{coev}\otimes \mathrm{id}_1)\{1\}\ar[d]^{\mu\{1\}}_{\wr}\\
[n-1]\mathrm{id}_1\{-1\}
\ar[r]^-{0}& [n-1]\mathrm{id}_1\{1\}
}
\]
i.e., in graphical notation, a commutative diagram
\begin{equation}\tag{moy$\chi_0$a}\label{eq:moichi0a}
\xymatrix{
\qquad{\xy
(0,-10);(0,10)**\crv{(0,-7)&(0,-6)&(0,-5)&(0,0)&(0,5)&(0,6)&(0,7)}
?>(.2)*\dir{>}?>(.99)*\dir{>}?>(.55)*\dir{>}
,(0,-3);(0,3)**\crv{(0,-4)&(0,-5)&(-4,-6)&(-6,0)&(-4,6)&(0,5)&(0,4)}
?>(.55)*\dir{<}
,(2,-11)*{\scriptstyle{1}}
,(2,11)*{\scriptstyle{1}}
,(-2,0)*{\scriptstyle{2}}
,(-6,0)*{\scriptstyle{1}}
\endxy}\{-1\} \ar[r]^-{1_{\mathrm{id}_1}\otimes \alpha}\ar[d]_-{\mu\{-1\}}^-{\wr} &
{\xy
(0,-10);(0,10)**\crv{(0,-7)&(0,-6)&(0,-5)&(0,0)&(0,5)&(0,6)&(0,7)}
?>(.2)*\dir{>}?>(.99)*\dir{>}?>(.55)*\dir{>}
,(0,-3);(0,3)**\crv{(0,-4)&(0,-5)&(-4,-6)&(-6,0)&(-4,6)&(0,5)&(0,4)}
?>(.55)*\dir{<}
,(2,-11)*{\scriptstyle{1}}
,(2,11)*{\scriptstyle{1}}
,(-2,0)*{\scriptstyle{2}}
,(-6,0)*{\scriptstyle{1}}
\endxy}\{1\}\ar[d]^{\mu\{1\}}_{\wr}\\
 [n-1]\,{\xy
(0,-10);(0,10)**\dir{-}
?>(.95)*\dir{>}
,(2,-11)*{\scriptstyle{1}}
,(2,11)*{\scriptstyle{1}}
\endxy} \{-1\}
\ar[r]^-{0}& [n-1]\,{\xy
(0,-10);(0,10)**\dir{-}
?>(.95)*\dir{>}
,(2,-11)*{\scriptstyle{1}}
,(2,11)*{\scriptstyle{1}}
\endxy} \{1\}
}
\end{equation}

\item We have a commutative diagram
\[
\xymatrix{
(\mathrm{ev}\otimes \mathrm{id}_1)\circ(\mathrm{id}_1\otimes S)\circ (\mathrm{coev}\otimes \mathrm{id}_1)\{-1\}\ar[r]^-{1_{\mathrm{id}_1}\otimes \gamma}\ar[d]^{\mu\{1\}}_{\wr} &
(\mathrm{ev}\otimes \mathrm{id}_1)\circ(\mathrm{id}_1\otimes S)\circ (\mathrm{coev}\otimes \mathrm{id}_1)\{1\}\ar[d]^{\mu\{1\}}_{\wr}\\
[n-1]\mathrm{id}_1\{-1\}
\ar[r]^-{\epsilon}& [n-1]\mathrm{id}_1\{1\}
}
\]
i.e., in graphical notation, a commutative diagram
\begin{equation}\tag{moy$\chi_0$b}\label{eq:moichi0b}
\xymatrix{
\qquad{\xy
(0,-10);(0,10)**\crv{(0,-7)&(0,-6)&(0,-5)&(0,0)&(0,5)&(0,6)&(0,7)}
?>(.2)*\dir{>}?>(.99)*\dir{>}?>(.55)*\dir{>}
,(0,-3);(0,3)**\crv{(0,-4)&(0,-5)&(-4,-6)&(-6,0)&(-4,6)&(0,5)&(0,4)}
?>(.55)*\dir{<}
,(2,-11)*{\scriptstyle{1}}
,(2,11)*{\scriptstyle{1}}
,(-2,0)*{\scriptstyle{2}}
,(-6,0)*{\scriptstyle{1}}
\endxy}\{-1\} \ar[r]^-{1_{\mathrm{id}_1}\otimes \gamma}\ar[d]_-{\mu\{-1\}}^-{\wr} &
{\xy
(0,-10);(0,10)**\crv{(0,-7)&(0,-6)&(0,-5)&(0,0)&(0,5)&(0,6)&(0,7)}
?>(.2)*\dir{>}?>(.99)*\dir{>}?>(.55)*\dir{>}
,(0,-3);(0,3)**\crv{(0,-4)&(0,-5)&(-4,-6)&(-6,0)&(-4,6)&(0,5)&(0,4)}
?>(.55)*\dir{<}
,(2,-11)*{\scriptstyle{1}}
,(2,11)*{\scriptstyle{1}}
,(-2,0)*{\scriptstyle{2}}
,(-6,0)*{\scriptstyle{1}}
\endxy}\{1\}\ar[d]^{\mu\{1\}}_{\wr}\\
 [n-1]\,{\xy
(0,-10);(0,10)**\dir{-}
?>(.95)*\dir{>}
,(2,-11)*{\scriptstyle{1}}
,(2,11)*{\scriptstyle{1}}
\endxy} \{-1\}
\ar[r]^-{\epsilon}& [n-1]\,{\xy
(0,-10);(0,10)**\dir{-}
?>(.95)*\dir{>}
,(2,-11)*{\scriptstyle{1}}
,(2,11)*{\scriptstyle{1}}
\endxy} \{1\}
}
\end{equation}

\item We have a commutative diagram
\[
\xymatrix{
S\{-1\}
\ar@/_1pc/[dr]_-{\left(\begin{smallmatrix}
1\\ 0
\end{smallmatrix}\right)} \ar[r]^{1\circ \chi_0}&
\quad
S\circ S
\ar^-{\varphi}[d]\\
 &
S\{-1\} \oplus S\{1\}
}
\]
i.e., in graphical notation, a commutative diagram
\begin{equation}\tag{moy2a}\label{eq:moy2a}
\xymatrix{
\esse\{-1\}
\ar@/_1pc/[dr]_-{\left(\begin{smallmatrix}
1\\ 0
\end{smallmatrix}\right)} \ar[r]^{1\circ \chi_0}&
\quad
\essecircesse
\ar^-{\varphi}[d]\\
 &
\esse\{-1\} \oplus \esse\{1\}
}\end{equation}

\item We have a commutative diagram
\[
\xymatrix{
S\{-1\}
\ar@/_1pc/[dr]_-{\left(\begin{smallmatrix}
1\\ \alpha
\end{smallmatrix}\right)} \ar[r]^{\chi_0\circ 1}&
\quad
S\circ S
\ar^-{\varphi}[d]\\
 &
S\{-1\} \oplus S\{1\}
}
\]
i.e., in graphical notation, a commutative diagram
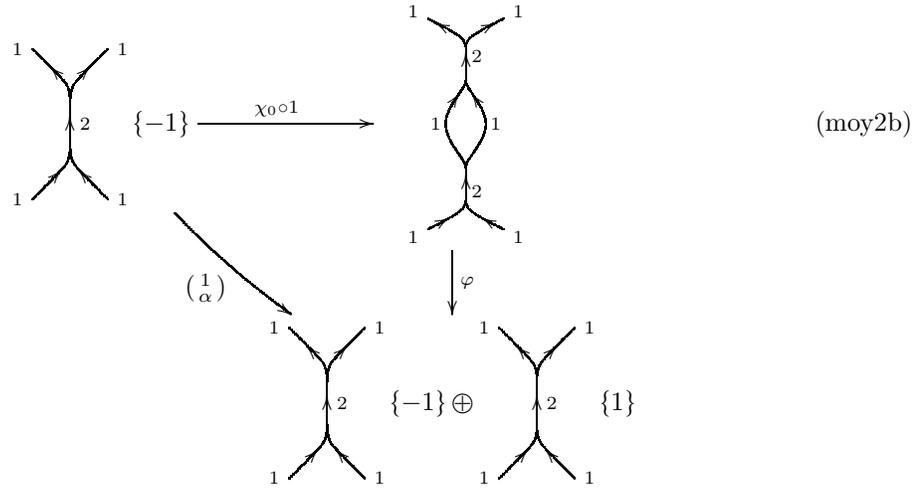
\begin{equation}\tag{moy2b}\label{eq:moy2b}
\xymatrix{
\esse\{-1\}
\ar@/_1pc/[dr]_-{\left(\begin{smallmatrix}
1\\ \alpha
\end{smallmatrix}\right)} \ar[r]^{\chi_0\circ 1}&
\quad
\essecircesse
\ar^-{\varphi}[d]\\
 &
\esse\{-1\} \oplus \esse\{1\}
}\end{equation}

\item We have a commutative diagram
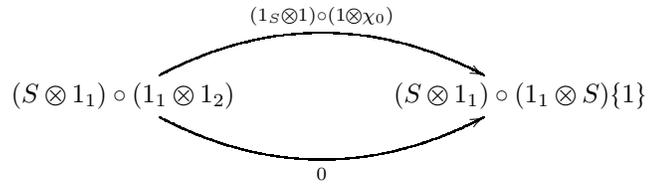
\[
\xymatrix{
(S\otimes 1_1)\circ(1_1\otimes 1_2)
 \ar@/_2pc/[rr]_{0} \ar@/^2pc/[rr]^{(1_S\otimes 1)\circ(1\otimes\chi_0 )}&&
(S\otimes 1_1)\circ (1_1\otimes S) \{1\}
 }
\]
i.e., we have
\[
(1_S\otimes 1)\circ(1\otimes\chi_0 )=0
\]
In graphical notation, this is the commutative diagram
\begin{equation}\tag{moyz}\label{eq:moyz}
\xymatrix{
\esse\,\,\idoner
 \ar@/_4pc/[rr]_{0} \ar@/^4pc/[rr]^{(1_S\otimes 1)\circ(1\otimes\chi_0 )}&&\zorro\{1\}
 }
\end{equation}

\item We have a commutative diagram
\[
\xymatrix{
S\{-1\}
\ar@/_1pc/[dr]_-{\left(\begin{smallmatrix}
1\\ \gamma
\end{smallmatrix}\right)} \ar[r]^{\chi_0\circ 1}&
\quad
S\circ S
\ar^-{\psi}[d]\\
 &
S\{-1\} \oplus S\{1\}
}
\]
i.e., in graphical notation, a commutative diagram
\begin{equation}\tag{moy2c}\label{eq:moy2c}
\xymatrix{
\esse\{-1\}
\ar@/_1pc/[dr]_-{\left(\begin{smallmatrix}
1\\ \gamma
\end{smallmatrix}\right)} \ar[r]^{\chi_0\circ 1}&
\quad
\essecircesse
\ar^-{\psi}[d]\\
 &
\esse\{-1\} \oplus \esse\{1\}
}\end{equation}

\item We have a commutative diagram
\[
\xymatrix{
S\circ S\ar[r]^-{\varphi}\ar[d]_{1\circ\alpha}& S\{-1\} \oplus S\{1\}\ar[d]^{\tiny{\left(\begin{matrix}0 & 1\\ 0 & 0\end{matrix}\right)}}\\
S\circ S\{2\}\ar[r]^-{\psi}&S\{1\} \oplus S\{3\}
}
\]
i.e., in graphical notation, a commutative diagram
\begin{equation}\tag{shift1}\label{eq:shift}
\xymatrix{
\essecircesse\ar[r]^-{\varphi}\ar[d]_{1\circ\alpha}& \esse\{-1\} \oplus \esse\{1\}\ar[d]^{\tiny{\left(\begin{matrix}0 & 1\\ 0 & 0\end{matrix}\right)}}\\
\essecircesse\{2\}\ar[r]^-{\psi}&\esse\{1\} \oplus \esse\{3\}
}
\end{equation}

\item We have a commutative diagram
\[
\xymatrix{
S\circ S\ar[r]^-{\psi}\ar[d]_{1\circ\gamma}& S\{-1\} \oplus S\{1\}\ar[d]^{\tiny{\left(\begin{matrix}0 & 1\\ 0 & 0\end{matrix}\right)}}\\
S\circ S\{2\}\ar[r]^-{\varphi}&S\{1\} \oplus S\{3\}
}
\]
i.e., in graphical notation, a commutative diagram
\begin{equation}\tag{shift2}\label{eq:shift2}
\xymatrix{
\essecircesse\ar[r]^-{\psi}\ar[d]_{1\circ\gamma}& \esse\{-1\} \oplus \esse\{1\}\ar[d]^{\tiny{\left(\begin{matrix}0 & 1\\ 0 & 0\end{matrix}\right)}}\\
\essecircesse\{2\}\ar[r]^-{\varphi}&\esse\{1\} \oplus \esse\{3\}
}
\end{equation}

\item We have a commutative diagram
\[
\xymatrix{S\circ S\ar[d]_{\varphi}\ar[r]^{1\circ\chi_1}&S\{1\}\\
S\{-1\} \oplus S\{1\}\ar@/_1pc/[ru]_-{\tiny{\left(\begin{matrix}0&1\end{matrix}\right)}}
}
\]
i.e., in graphical notation, a commuatative diagram
\begin{equation}\tag{moy2d}\label{eq:moy2d}
\xymatrix{
\essecircesse\ar[d]_{\varphi}\ar[r]^{1\circ\chi_1}&\esse\{1\}\\
\esse\{-1\} \oplus \esse\{1\}\ar@/_1pc/[ru]_-{\tiny{\left(\begin{matrix}0&1\end{matrix}\right)}}
}
\end{equation}

\item We have a commutative diagram
\[
\scalebox{.8}{
\xymatrix{
\mathrm{id}_1\otimes (\mathrm{ev}^\vee\otimes \mathrm{id}_{1^\vee})\circ(\mathrm{id}_{1^\vee}\otimes S^\vee)\circ (\mathrm{coev}^\vee\otimes \mathrm{id}_{1^\vee})\{-1\}
\ar[rr]^-{
 \chi_0\otimes 1}
\ar[d]_{1\otimes \mu}^{\wr}
&&
(\mathrm{id}_1\otimes \mathrm{ev}\otimes\mathrm{id}_1^\vee)\circ (S\otimes S^\vee) \circ (\mathrm{id}_1\otimes \mathrm{coev}\otimes\mathrm{id}_1^\vee)
\ar[d]_{\nu}^{\wr}
\\
[n-1]\mathrm{id}_1\otimes \mathrm{id}_{1^\vee}
\{-1\}\ar@{=}[d]
&&
\mathrm{coev}\circ\mathrm{ev}\oplus [n-2]\otimes \mathrm{id}_1\otimes \mathrm{id}_1^\vee\ar@{=}[d]
\\ 
(\mathbb{Q}\{1-n\}\oplus[n-2])\mathrm{id}_1\otimes \mathrm{id}_{1^\vee}
\ar[rr]^-{\left(
\begin{smallmatrix}
0 & 0\\
0 &1 \\
\end{smallmatrix}\right)}
&&
\mathrm{coev}\circ\mathrm{ev}\oplus [n-2]\otimes \mathrm{id}_1\otimes \mathrm{id}_1^\vee}
}
\]
i.e., in graphical notation
\begin{equation}
\label{eq:moynu1}\tag{moy$\nu$1}
\xymatrix{
{\xy
(-10,-10);(-10,10)**\dir{-}?>(.5)*\dir{>}
,(-12,-11)*{\scriptstyle{1}}
,(-12,11)*{\scriptstyle{1}}
,(0,10);(0,-10)**\crv{(0,7)&(0,6)&(0,5)&(0,0)&(0,-5)&(0,-6)&(0,-7)}
?>(.2)*\dir{>}?>(.99)*\dir{>}?>(.55)*\dir{>}
,(0,3);(0,-3)**\crv{(0,4)&(0,5)&(-4,6)&(-6,0)&(-4,-6)&(0,-5)&(0,-4)}
?>(.55)*\dir{<}
,(2,-11)*{\scriptstyle{1}}
,(2,11)*{\scriptstyle{1}}
,(2,0)*{\scriptstyle{2}}
,(-6,0)*{\scriptstyle{1}}
\endxy}\{-1\}
\ar[rr]^-{
 \chi_0\otimes 1}
\ar[dd]_{1\otimes \mu}^{\wr}
&&
{\xy
(-10,-10);(-10,10)**\crv{(-6,-6)&(-5,-5)&(-5,-2)&(-5,-2)&(-5,5)&(-6,6)}
?>(.05)*\dir{>}?>(.99)*\dir{>}?>(.55)*\dir{>}
,(10,-10);(10,10)**\crv{(6,-6)&(5,-5)&(5,-2)&(5,2)&(5,5)&(6,6)}
?>(.03)*\dir{<}?>(.92)*\dir{<}?>(.45)*\dir{<}
,(-5,0);(5,0)**\crv{(-5,-2)&(-5,-3)&(-4,-5)&(-2,-6)&(2,-6)&(4,-5)&(5,-3)&(5,-2)}?>(.5)*\dir{<}
,(-5,0);(5,0)**\crv{(-5,2)&(-5,3)&(-4,5)&(-2,6)&(2,6)&(4,5)&(5,3)&(5,2)}?>(.5)*\dir{>}
,(-11,-10)*{\scriptstyle{1}}
,(-11,10)*{\scriptstyle{1}}
,(11,-10)*{\scriptstyle{1}}
,(11,10)*{\scriptstyle{1}}
,(-7,0)*{\scriptstyle{2}}
,(0,-8)*{\scriptstyle{1}}
,(6.5,0)*{\scriptstyle{2}}
,(0,8)*{\scriptstyle{1}}
\endxy}
\ar[dd]_{\nu}^{\wr}
\\
\\
[n-1]{\xy
(-5,-10);(-5,10)**\dir{-}?>(.5)*\dir{>}
,(-7,-11)*{\scriptstyle{1}}
,(-7,11)*{\scriptstyle{1}}
,(5,-10);(5,10)**\dir{-}?>(.5)*\dir{<}
,(7,-11)*{\scriptstyle{1}}
,(7,11)*{\scriptstyle{1}}
\endxy}\{-1\}\ar@{=}[dd]
&&
\xy
(-5,-10);(5,-10)**\crv{(-5,-2)&(0,-2)&(5,-2)}
?>(1)*\dir{>},
(-6,-10)*{\scriptstyle{1}},(6,-10)*{\scriptstyle{1}}
,(-5,10);(5,10)**\crv{(-5,2)&(0,2)&(5,2)}
?>(0)*\dir{<},
(-6,10)*{\scriptstyle{1}},(6,10)*{\scriptstyle{1}}
\endxy
\oplus\, [n-2]
{\xy
(-5,-10);(-5,10)**\dir{-}?>(.5)*\dir{>}
,(5,-10);(5,10)**\dir{-}?>(.5)*\dir{<}
,(-7,-10)*{\scriptstyle{1}}
,(7,-10)*{\scriptstyle{1}}
,(-7,10)*{\scriptstyle{1}}
,(7,10)*{\scriptstyle{1}}
\endxy}
\ar@{=}[dd]
\\ \\
(\mathbb{Q}\{1-n\}\oplus[n-2]){\xy
(-5,-10);(-5,10)**\dir{-}?>(.5)*\dir{>}
,(-7,-11)*{\scriptstyle{1}}
,(-7,11)*{\scriptstyle{1}}
,(5,-10);(5,10)**\dir{-}?>(.5)*\dir{<}
,(7,-11)*{\scriptstyle{1}}
,(7,11)*{\scriptstyle{1}}
\endxy}
\ar[rr]^-{\left(
\begin{smallmatrix}
0 & 0\\
0 &1 \\
\end{smallmatrix}\right)}
&&
\xy
(-5,-10);(5,-10)**\crv{(-5,-2)&(0,-2)&(5,-2)}
?>(1)*\dir{>},
(-6,-10)*{\scriptstyle{1}},(6,-10)*{\scriptstyle{1}}
,(-5,10);(5,10)**\crv{(-5,2)&(0,2)&(5,2)}
?>(0)*\dir{<},
(-6,10)*{\scriptstyle{1}},(6,10)*{\scriptstyle{1}}
\endxy
\oplus\,  [n-2]
{\xy
(-5,-10);(-5,10)**\dir{-}?>(.5)*\dir{>}
,(5,-10);(5,10)**\dir{-}?>(.5)*\dir{<}
,(-7,-10)*{\scriptstyle{1}}
,(7,-10)*{\scriptstyle{1}}
,(-7,10)*{\scriptstyle{1}}
,(7,10)*{\scriptstyle{1}}
\endxy}
}
\end{equation}

\item We have a commutative diagram
\[
\scalebox{.8}{
\xymatrix{
(\mathrm{id}_1\otimes \mathrm{ev}\otimes\mathrm{id}_1^\vee)\circ (S\otimes S^\vee) \circ (\mathrm{id}_1\otimes \mathrm{coev}\otimes\mathrm{id}_1^\vee)
\ar[d]_{\nu}^{\wr}
\ar[rr]^-{
1\otimes  \chi_1}
&&
 (\mathrm{id}_{1}\otimes \mathrm{ev} )\circ(S\otimes \mathrm{id}_{1})\circ (\mathrm{id}_{1}\otimes \mathrm{coev})\otimes \mathrm{id}_{1^\vee}\{1\}
\ar[d]_{\mu\otimes 1}^{\wr}
\\
\mathrm{coev}\circ\mathrm{ev}\oplus [n-2]\otimes \mathrm{id}_1\otimes \mathrm{id}_1^\vee
\ar@{=}[d]
&&
[n-1]\mathrm{id}_1\otimes \mathrm{id}_{1^\vee}
\{-1\}
\ar@{=}[d]
\\
\mathrm{coev}\circ\mathrm{ev}\oplus [n-2]\otimes \mathrm{id}_1\otimes \mathrm{id}_1^\vee
\ar[rr]^-{\left(
\begin{smallmatrix}
0 & 1\\
0 &0 \\
\end{smallmatrix}\right)}
&&
([n-2]\oplus \mathbb{Q}\{n-1\})\mathrm{id}_1\otimes \mathrm{id}_{1^\vee}
}}
\]
i.e., in graphical notation
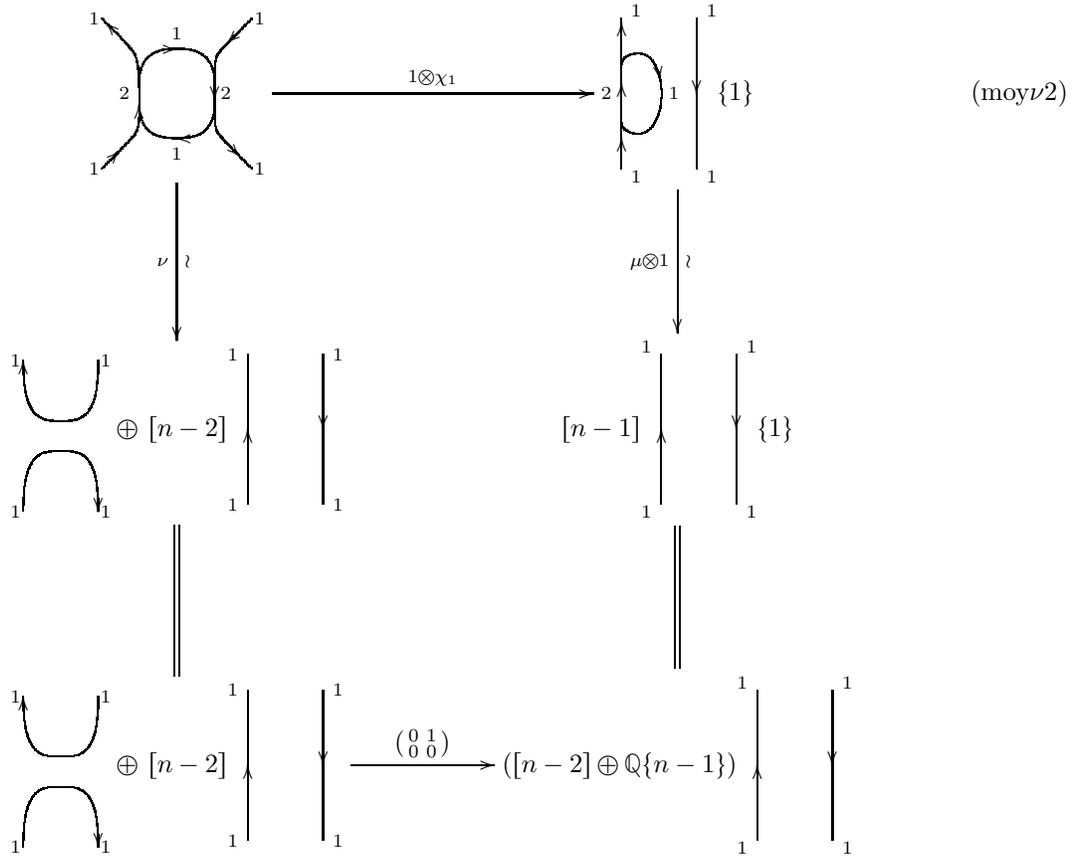
\begin{equation}
\label{eq:moynu2}\tag{moy$\nu$2}
\xymatrix{
{\xy
(-10,-10);(-10,10)**\crv{(-6,-6)&(-5,-5)&(-5,-2)&(-5,-2)&(-5,5)&(-6,6)}
?>(.05)*\dir{>}?>(.99)*\dir{>}?>(.55)*\dir{>}
,(10,-10);(10,10)**\crv{(6,-6)&(5,-5)&(5,-2)&(5,2)&(5,5)&(6,6)}
?>(.03)*\dir{<}?>(.92)*\dir{<}?>(.45)*\dir{<}
,(-5,0);(5,0)**\crv{(-5,-2)&(-5,-3)&(-4,-5)&(-2,-6)&(2,-6)&(4,-5)&(5,-3)&(5,-2)}?>(.5)*\dir{<}
,(-5,0);(5,0)**\crv{(-5,2)&(-5,3)&(-4,5)&(-2,6)&(2,6)&(4,5)&(5,3)&(5,2)}?>(.5)*\dir{>}
,(-11,-10)*{\scriptstyle{1}}
,(-11,10)*{\scriptstyle{1}}
,(11,-10)*{\scriptstyle{1}}
,(11,10)*{\scriptstyle{1}}
,(-7,0)*{\scriptstyle{2}}
,(0,-8)*{\scriptstyle{1}}
,(6.5,0)*{\scriptstyle{2}}
,(0,8)*{\scriptstyle{1}}
\endxy}
\ar[rr]^-{
1\otimes  \chi_1}
\ar[dd]_{\nu}^{\wr}
&&
{\xy
(0,-10);(0,10)**\crv{(0,-7)&(0,-6)&(0,-5)&(0,0)&(0,5)&(0,6)&(0,7)}
?>(.2)*\dir{>}?>(.99)*\dir{>}?>(.55)*\dir{>}
,(0,-3);(0,3)**\crv{(0,-4)&(0,-5)&(4,-6)&(6,0)&(4,6)&(0,5)&(0,4)}
?>(.55)*\dir{<}
,(2,-11)*{\scriptstyle{1}}
,(2,11)*{\scriptstyle{1}}
,(-2,0)*{\scriptstyle{2}}
,(7,0)*{\scriptstyle{1}}
,(10,-10);(10,10)**\dir{-}?>(.5)*\dir{<}
,(12,-11)*{\scriptstyle{1}}
,(12,11)*{\scriptstyle{1}}
\endxy}\{1\}
\ar[dd]_{\mu\otimes 1}^{\wr}
\\
\\
\xy
(-5,-10);(5,-10)**\crv{(-5,-2)&(0,-2)&(5,-2)}
?>(1)*\dir{>},
(-6,-10)*{\scriptstyle{1}},(6,-10)*{\scriptstyle{1}}
,(-5,10);(5,10)**\crv{(-5,2)&(0,2)&(5,2)}
?>(0)*\dir{<},
(-6,10)*{\scriptstyle{1}},(6,10)*{\scriptstyle{1}}
\endxy
\oplus\, [n-2]
{\xy
(-5,-10);(-5,10)**\dir{-}?>(.5)*\dir{>}
,(5,-10);(5,10)**\dir{-}?>(.5)*\dir{<}
,(-7,-10)*{\scriptstyle{1}}
,(7,-10)*{\scriptstyle{1}}
,(-7,10)*{\scriptstyle{1}}
,(7,10)*{\scriptstyle{1}}
\endxy}
\ar@{=}[dd]
&&
[n-1]{\xy
(-5,-10);(-5,10)**\dir{-}?>(.5)*\dir{>}
,(-7,-11)*{\scriptstyle{1}}
,(-7,11)*{\scriptstyle{1}}
,(5,-10);(5,10)**\dir{-}?>(.5)*\dir{<}
,(7,-11)*{\scriptstyle{1}}
,(7,11)*{\scriptstyle{1}}
\endxy}\{1\}\ar@{=}[dd]
\\ \\
\xy
(-5,-10);(5,-10)**\crv{(-5,-2)&(0,-2)&(5,-2)}
?>(1)*\dir{>},
(-6,-10)*{\scriptstyle{1}},(6,-10)*{\scriptstyle{1}}
,(-5,10);(5,10)**\crv{(-5,2)&(0,2)&(5,2)}
?>(0)*\dir{<},
(-6,10)*{\scriptstyle{1}},(6,10)*{\scriptstyle{1}}
\endxy
\oplus\,  [n-2]
{\xy
(-5,-10);(-5,10)**\dir{-}?>(.5)*\dir{>}
,(5,-10);(5,10)**\dir{-}?>(.5)*\dir{<}
,(-7,-10)*{\scriptstyle{1}}
,(7,-10)*{\scriptstyle{1}}
,(-7,10)*{\scriptstyle{1}}
,(7,10)*{\scriptstyle{1}}
\endxy}
\ar[rr]^-{\left(
\begin{smallmatrix}
0 & 1\\
0 &0 \\
\end{smallmatrix}\right)}
&&
([n-2]\oplus \mathbb{Q}\{n-1\}){\xy
(-5,-10);(-5,10)**\dir{-}?>(.5)*\dir{>}
,(-7,-11)*{\scriptstyle{1}}
,(-7,11)*{\scriptstyle{1}}
,(5,-10);(5,10)**\dir{-}?>(.5)*\dir{<}
,(7,-11)*{\scriptstyle{1}}
,(7,11)*{\scriptstyle{1}}
\endxy}
}
\end{equation}

\item We have a commutative diagram
\[
\xymatrix{
(S\otimes \mathrm{id}_1)\circ(\mathrm{id}_1\otimes S)\circ(\mathrm{id}_1\otimes\mathrm{id}_1\otimes \mathrm{id}_1)\{-1\} \ar[rrr]^-{(1\otimes 1)\circ (1\otimes 1)\circ (\chi_0\otimes 1)} \ar@/_1pc/[rrrd]_-{\tiny{\left(\begin{matrix}\chi_1&0\end{matrix}\right)}}
&&&(S\otimes \mathrm{id}_1)\circ(\mathrm{id}_1\otimes S)\circ(S\otimes \mathrm{id}_1)\ar[d]_-{\wr}^-\eta\\
&&& (S\otimes \mathrm{id}_1)\oplus T
}
\]
i.e., in graphical notation
\begin{equation}\tag{{moy$\eta$}}\label{moyeta}
\xymatrix{
\zorro \{-1\}\ar[r]^{\chi_0\otimes 1} \ar@/_2pc/[rd]_-{\tiny{\left(\begin{matrix}\chi_1&0\end{matrix}\right)}}
&\coso\ar[d]_-{\wr}^-\eta\\
& \esse\,\,\idoner\,\oplus\ti
}
\end{equation}
\end{itemize}

\subsection{Example: a commutative diagram}\label{sub:example}
As an illustrative example, we show that
\[
\alpha\chi_0=0
\]
by showing that
 the diagram
\[
\xymatrix{
\mathrm{id}_2\ar[r]^{\chi_0}\ar[d]_{\chi_0}& S\{1\}\ar[d]^{\alpha}\\
S\{1\}\ar[r]^-{0}&S\{3\}
}
\]
commutes. In graphical notation, and using the definition of $\alpha$ and the relation $\chi_1\chi_0=0$, this diagram is
\[
\xymatrix{
\idtwo\ar[rr]^{\chi_0}\ar[drr]^{\chi_0\circ\chi_0}\ar[ddr]_{\chi_0\circ\chi_0}\ar[dd]_{\chi_0}&& \esse\{1\}\ar[d]^-{\chi_0\circ1}\\
&& \essecircesse \{2\}\ar[d]^{1\circ\chi_1}\ar@{=}[dl]\\
\esse\{1\}\ar[r]^-{1\circ\chi_0}& \essecircesse\{2\}\ar[r]^{1\circ \chi_1 }&\esse\{3\}
}
\]
whose commutativity is manifest.

\subsection{Another example}
As a second example, we prove that the composition 
\[
\esse\,\,\idoner\,\{-2\}\xrightarrow{1\circ(\chi_0\otimes 1)\circ 1}\essecircesse\,\longidoner\{-1\}\xrightarrow{1\circ(1\otimes \chi_0)\circ 1}\coso
\]
is zero. Namely, we have a commutative diagram
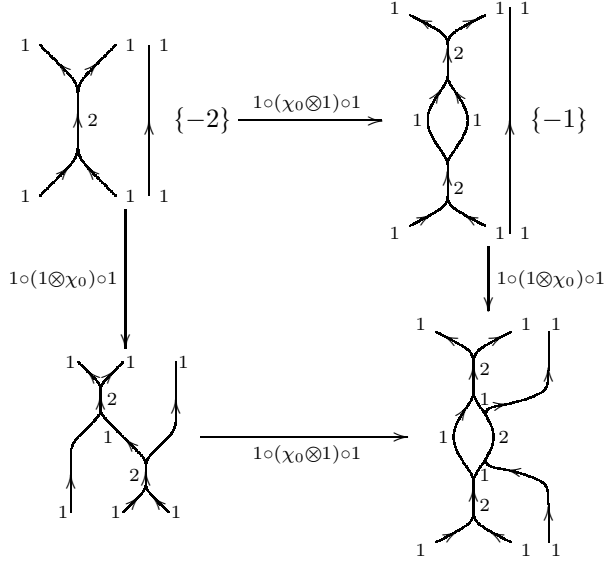
\[
\xymatrix{
\esse\,\,\idoner\,\{-2\}\ar[rr]^{1\circ(\chi_0\otimes 1)\circ 1}\ar[d]_{1\circ(1\otimes \chi_0)\circ 1} &&\essecircesse\,\longidoner\{-1\}\ar[d]^{1\circ(1\otimes \chi_0)\circ 1}\\
\zorro\ar[rr]_{1\circ(\chi_0\otimes 1)\circ1} && \coso
}
\]
where the left vertical arrow is zero by relation (??).

\section{Formal complexes of MOY graphs}
The next step in our construction of the category $\mathbf{KR}$ is the definition of formal complexes of (formal direct sums of shifted) MOY-type graphs and of a set of equivalence relations on them.

\begin{definition}\label{eq:formal-complex}
A bounded formal complex of (formal direct sums of shifted) MOY-type graphs is a sequence
\[
\cdots \xrightarrow{A_{i-1}} M_i\xrightarrow{A_i}M_{i+1}\xrightarrow{A_{i+1}}\cdots
\]
with $i\in \mathbb{Z}$, where
\begin{itemize}
\item the $M_i$ are formal direct sums of shifthed MOY-type graphs, all with the same input and the same output sequences of $\uparrow^1$'s and $\downarrow^1$'s;
\item $A_i\colon M_i\to M_{i+1}$ is am element in the quotient $\widetilde{\mathcal{M}}$ of the $\mathbb{Q}$-linear partial monoid\footnote{See Remark \ref{rem:bimonoid}} $\mathcal{M}$ with respect to the horizontal composition by the ideal generated by the relations in Section \ref{sec:relations};
\item one has $A_{i+1}\circ A_i=0$, for any $i\in \mathbb{Z}$; 
\item there exist two integers $i_{\min}$ and $i_{\max}$ such that $M_i=\mathbf{0}$ for any $i<i_{\min}$ ad any $i>i_{\max}$.
\end{itemize}
WE say that the component $M_i$ of $M_{\bullet}$ is in degree$i$.\par
A morphism $\varphi_\bullet\colon M_\bullet\to N_\bullet$ of formal complexes of (formal direct sums of shifthed) MOY-type graphs is a sequence of morphisms $\varphi_i\colon M_i\to N_i$ in $\widetilde{\mathcal{M}}$ such that the diagram
\[
\xymatrix{
\cdots\ar[r]& M_i\ar[d]_{\varphi_i}\ar[r]^{A_i}&M_{i+1}\ar[r]\ar[d]^{\varphi_{i+1}}&\cdots\\
\cdots\ar[r]& N_i\ar[r]^{B_i}&N_{i+1}\ar[r]&\cdots
}
\]
commutes (i.e., $\varphi_{i+1}\circ A_i=B_i\circ\varphi_i$ in $\widetilde{\mathcal{M}}$, for any $i\in \mathbb{Z}$).
\end{definition}

\begin{remark}
The equation $A_{i+1}\circ A_i=0$ in Definition \ref{eq:formal-complex} makes sense as $\mathcal{M}$ modulo the ideal generated by the relations in Section \ref{sec:relations} is a $\mathbb{Q}$-vector space.
\end{remark}

\begin{remark}
As all the MOY-type graphs appearing in a formal complex have the same input and the same output sequences of $\uparrow^1$'s and $\downarrow^1$'s it is meaningful to talk of the source sequence and of the target sequence of a formal complex of  (formal direct sums of shifted) MOY-type graphs. 
\end{remark}

\begin{remark}
As an obvious generalization one can consider formal bicomplexes  of  (formal direct sums of shifted) MOY-type graphs, and one has a totalization functor
\[
\mathrm{tot}\colon \mathbf{Bicomplexes}_{MOY}\to \mathbf{Complexes}_{MOY},
\]
defined as the usual totalization functor from bicomplexes of $R$-modules to complexes of $R$-modules for a ring $R$. Even more in general, one can consider the category $\mathbf{n\text{-}Complexes}_{MOY}$ of $n$-complexes of (formal direct sums of shifted) MOY-type graphs and the corresponding totalization functor
\[
\mathrm{tot}\colon \mathbf{n\text{-}Complexes}_{MOY}\to \mathbf{Complexes}_{MOY}.
\]
Notice how the fact that all the graphs appearing in an $n$-complex of (formal direct sums of shifted) MOY-type graphs have the  the same input and the same output sequences of $\uparrow^1$'s and $\downarrow^1$'s plays an essential role here.
\end{remark}

\begin{definition}
The \emph{tensor product} of two formal complexes of (formal direct sums of shifted) MOY-type graphs $M_\bullet$ and $N\bullet$ is the formal complex
\[
(M\otimes N)_\bullet := \mathrm{tot}\left(M_\bullet\otimes N_\bullet\right),
\]
where $M_\bullet\otimes N_\bullet$ is the formal bicomplex
\[
\xymatrix{
&\vdots&&\vdots\\
\cdots\ar[r]&M_i\otimes N_{j+1} \ar[u]\ar[rr]^{A_{i+1}\otimes 1_{N_{j+1}}} && M_{i+1}\otimes N_{j+1} \ar[u]\ar[r]&\cdots\\
\cdots\ar[r]&M_i\otimes N_j \ar[rr]^{A_i\otimes 1_{N_j}}\ar[u]^{1_{M_{i}}\otimes B_j}&& M_{i+1}\otimes N_j\ar[u]_{1_{M_{i+1}}\otimes B_j} \ar[r]&\cdots\\
&\vdots\ar[u]&&\vdots\ar[u]
}
\]
\end{definition}
If $M_\bullet$ and $N\bullet$ are two formal complexes of (formal direct sums of shifted) MOY-type graphs and the output sequence of $\uparrow^1$'s and $\downarrow^1$'s of $N_\bullet$ coincides with the input sequence of $M_\bullet$, then all the compositions $M_i\circ N_j$ are defined and one can give the following definition.
\begin{definition}
The \emph{composition} of two formal complexes of (formal direct sums of shifted) MOY-type graphs $M_\bullet$ and $N\bullet$ such that the output sequence of $\uparrow^1$'s and $\downarrow^1$'s of $N_\bullet$ coincides with the input sequence of $M_\bullet$ is the formal complex
\[
(M\circ N)_\bullet := \mathrm{tot}\left(M_\bullet\circ N_\bullet\right),
\]
where $M_\bullet \circ\ N_\bullet$ is the formal bicomplex
\[
\xymatrix{
&\vdots&&\vdots\\
\cdots\ar[r]&M_i\circ N_{j+1} \ar[u]\ar[rr]^{A_{i+1}\circ 1_{N_{j+1}}} && M_{i+1}\circ N_{j+1} \ar[u]\ar[r]&\cdots\\
\cdots\ar[r]&M_i\circ N_j \ar[rr]^{A_i\circ 1_{N_j}}\ar[u]^{1_{M_{i}}\circ B_j}&& M_{i+1}\circ N_j\ar[u]_{1_{M_{i+1}}\circ B_j} \ar[r]&\cdots\\
&\vdots\ar[u]&&\vdots\ar[u]
}
\]
\end{definition}

\vskip .7 cm
We introduce on (bounded) formal complexes a sum operation given by direct sum of complexes,
\[
M_\bullet\oplus N_\bullet= \left(
\cdots \xrightarrow{A_{i-1}\oplus B_{i-1}} M_i\oplus N_i\xrightarrow{A_i\oplus B_i}M_{i+1}\oplus N_{i+1}\xrightarrow{A_{i+1}\oplus B_{i+1}}\cdots\right)
\]
 as well as the following equivalence relations:
\begin{itemize}
\item (eq1) Isomorphic formal complexes are equivalent. That is, if we have a morphism of formal complexes
\[
\xymatrix{
\cdots\ar[r]& M_i\ar[d]_{\varphi_i}^{\wr}\ar[r]^{A_i}&M_{i+1}\ar[r]\ar[d]^{\varphi_{i+1}}_{\wr}&\cdots\\
\cdots\ar[r]& N_i\ar[r]^{B_i}&N_{i+1}\ar[r]&\cdots
}
\]
where all the $\varphi_i$'s are invertible (i.e., for each $i$ there exists an element $\psi_i\colon N_i\to M_i$ in $\widetilde{\mathcal{M}}$ such that $\varphi_i\circ\psi_i=1_{N_i}$ and $\psi_i\circ\varphi_i=1_{M_i}$ in $\widetilde{\mathcal{M}}$), then
\[
\cdots\to M_i\xrightarrow{A_i} M_{i+1}\to\cdots
\]
is equivalent to
\[
\cdots\to N_i\xrightarrow{B_i} N_{i+1}\to\cdots.
\]
\item (eq2)  A formal complex of the form
\[
\cdots \to M_{i-1}\xrightarrow{\left(
\begin{matrix}
\alpha_{i-1}\\A_{i-1}
\end{matrix}
\right)}
N\oplus M_i\xrightarrow{\left(
\begin{matrix}
1&\alpha_i\\0&A_i
\end{matrix}
\right)}
N\oplus M_{i+1}\xrightarrow{\left(
\begin{matrix}
0&A_{i+1}
\end{matrix}
\right)}
 M_{i+2}\to\cdots
\]
is equivalent to the formal complex
\[
\cdots\to  M_{i-1}\ \xrightarrow{A_{i-1}}
M_i\xrightarrow{A_i} M_{i+1}\xrightarrow{A_{i+1}} M_{i+2}
\to \cdots
\]
That is, we we ``simplify'' terms $\Gamma\xrightarrow{1}\Gamma$ if they appear as subcomplexes.

\item (eq3) A formal complex of the form
\[
\cdots \to M_{i-1}\xrightarrow{\left(
\begin{matrix}
A_{i-1}\\0
\end{matrix}
\right)}
 M_i \oplus N\xrightarrow{\left(
\begin{matrix}
A_i&\alpha_i\\0&1
\end{matrix}
\right)}
M_{i+1}\oplus N\xrightarrow{\left(
\begin{matrix}
A_{i+1}&\alpha_{i+1}
\end{matrix}
\right)}
 M_{i+2}\to\cdots
\]
is equivalent
to the formal complex
\[
\cdots\to  M_{i-1}\ \xrightarrow{A_{i-1}}
M_i\xrightarrow{A_i} M_{i+1}\xrightarrow{A_{i+1}} M_{i+2}
\to \cdots
\]
That is, we can ``simplify'' terms $\Gamma\xrightarrow{1}\Gamma$ if they appear as  quotient complexes.
\end{itemize}
\begin{notation}
We denote by $\mathbf{FC_{MOY}}$ the set of equivalence classes of formal complexes of (formal direct sums of shifted) MOY-type graphs, with respect to the equivalence relation generated by the equivalence relations (eq1)-(eq3) above. The equivalence class of the formal complex $M_\bullet$ will be denoted by the symbol $[M_\bullet]$.
\end{notation}
By the usual properties of direct sums and tensor product of complexes and totalization functors one immediately sees that the tensor product, the composition and the direct sum operations are compatible with the equivalence relations (eq1)-(eq3) and so induce corresponding operations on $\mathbf{FC_{MOY}}$. More formally, this is stated as follows.
\begin{proposition}\label{prop:operations}
The operations
\[
+,\otimes,\circ \colon  \mathbf{FC_{MOY}}\times \mathbf{FC_{MOY}} \to \mathbf{FC_{MOY}}
\]
defined by
\begin{align*}
[M_\bullet]+[N_\bullet]&=[M_\bullet\oplus N_\bullet]\\
[M_\bullet]\otimes[N_\bullet]&=[M_\bullet\otimes N_\bullet]\\
[M_\bullet]\circ[N_\bullet]&=[M_\bullet\circ N_\bullet]
\end{align*}
are well defined and satisfy all the expected associativity, commutativity and distributivity conditions. In particular $(\mathbf{FC_{MOY}},+,\otimes)$ is an abelian semiring.
\end{proposition}
\begin{remark}
The operation $([M_\bullet],[N_\bullet])\mapsto [M_\bullet]\circ[N_\bullet]$ in the above proposition is defined when the common target sequence of $\uparrow^1$'s and $\downarrow^1$'s of the MOY-type graphs appearing in $N_\bullet$ coincides with the source sequence MOY-type graphs appearing in $M_\bullet$. These sequences are well defined, i.e., they do not depend on the particular representatives $M_\bullet$ and $N_\bullet$ chosen for the equivalence classes $[M_\bullet]$ and $[N_\bullet]$.
\end{remark}
\begin{remark}
The abelian semigroup $(\mathbf{FC_{MOY}},+)$ carries a natural action of the semiring $\mathbb{Z}_{\geq 0}[q,q^{-1}]$ of Laurent polynomials in the variable $q$ with coefficients in the abelian semiring $\mathbb{Z}_{\geq 0}$ of nonnegative integers given by
\[
q^i\cdot [M_\bullet]=[M_\bullet\{i\}],
\]
extended by additivity. This makes $(\mathbf{FC_{MOY}},+)$ a \emph{semimodule} for the semiring $\mathbb{Z}_{\geq 0}[q,q^{-1}]$. Notice that the $\mathbb{Z}_{\geq 0}[q,q^{-1}]$-action preserves both the source and the target sequences of elements in  $\mathbf{FC_{MOY}}$.
\end{remark}

\section{The category $\mathbf{KR}$}
We are finally in position to define the category $\mathbf{KR}$ announced at the beginning of this section.
\begin{notation}
For any two sequences $\vec{i}$ and $\vec{j}$ of $\uparrow^1$'s and $\downarrow^1$'s, we write $\mathbf{FC_{MOY}}(\vec{i},\vec{j})$ for the subset of $\mathbf{FC_{MOY}}$ consisting of all equivalence classes of formal complexes of (formal direct sums of shifted) MOY-type graphs with input sequence $\vec{i}$ and output sequence $\vec{j}$.
\end{notation}

\begin{definition}
The monoidal category $\mathbf{KR}$  is defined as follows:
\begin{itemize}
\item objects in $\mathbf{KR}$ are finite sequences of $\uparrow^1$'s and $\downarrow^1$'s;
\item the set $\mathbf{KR}(\vec{i},\vec{j})$ of morphisms between two objects $\vec{i}$ and $\vec{j}$ is the $\mathbb{Z}_{\geq 0}[q,q^{-1}]$-semimodule $\mathbf{FC_{MOY}}(\vec{i},\vec{j})$.
\end{itemize}
Composition of morphisms is given by the composition
\[
\mathbf{FC_{MOY}}(\vec{j},\vec{k})\times \mathbf{FC_{MOY}}(\vec{i},\vec{j})\to \mathbf{FC_{MOY}}(\vec{i},\vec{k})
\]
of Proposition \ref{prop:operations}.\par
The identity morphism is given by the (equivalence class of the) formal complex consisting of the MOY-type diagram with just straight vertical lines connecting the bottom boundary with the top boundary in degree zero, and $\mathbf{0}$ in all the others degrees.\par
The tensor product of objects is given by concatenation. Tensor product of morphisms is given by the tensor product
\[
\mathbf{FC_{MOY}}(\vec{i}_1,\vec{j}_1)\times \mathbf{FC_{MOY}}(\vec{j}_1,\vec{j}_2)\to \mathbf{FC_{MOY}}((\vec{i}_1\otimes \vec{i}_2,\vec{j}_1\otimes \vec{j}_2)
\]
of Proposition \ref{prop:operations}. The unit object for the tensor product is the empty sequence.
\end{definition}
\begin{example}
The identity of $\uparrow^1\,\, \downarrow^1 \,\,\uparrow^1\,\, \uparrow^1$ is
\[
\left[
\cdots\to \mathbf{0}\to\mathbf{0}\to
\xy
(-6,-5);(-6,5)**\dir{-}
?>(1)*\dir{>},
(-2,-5);(-2,5)**\dir{-}
?>(0)*\dir{<},
(2,-5);(2,5)**\dir{-}
?>(1)*\dir{>},
(6,-5);(6,5)**\dir{-}
?>(1)*\dir{>},
(-4.7,-5)*{\scriptstyle{1}},
(-.7,-5)*{\scriptstyle{1}},
(3.3,-5)*{\scriptstyle{1}},
(7.3,-5)*{\scriptstyle{1}}
\endxy\,
\to \mathbf{0}\to\mathbf{0}\to\cdots\right]
\]
with $\xy
(-6,-5);(-6,5)**\dir{-}
?>(1)*\dir{>},
(-2,-5);(-2,5)**\dir{-}
?>(0)*\dir{<},
(2,-5);(2,5)**\dir{-}
?>(1)*\dir{>},
(6,-5);(6,5)**\dir{-}
?>(1)*\dir{>},
(-4.7,-5)*{\scriptstyle{1}},
(-.7,-5)*{\scriptstyle{1}},
(3.3,-5)*{\scriptstyle{1}},
(7.3,-5)*{\scriptstyle{1}}
\endxy$ in degree zero.
\end{example}
Next, we introduce on $\mathbf{KR}$ the structure of rigid category. As $\mathbf{KR}$ is monoidally generated by $\uparrow^1$ and $\downarrow^1$ it will suffice defining evaluation and coevaluation morphisms for these.
\begin{proposition}
The evaluation and coevaluation morphisms
\[
\mathrm{ev}_{\uparrow^1}=\left[
\cdots\to \mathbf{0}\to\mathbf{0}\to
\evop
\to \mathbf{0}\to\mathbf{0}\to\cdots\right]
\]
\[
\mathrm{ev}_{\downarrow^1}=\left[
\cdots\to \mathbf{0}\to\mathbf{0}\to
\ev
\to \mathbf{0}\to\mathbf{0}\to\cdots\right]
\]
\[
\mathrm{coev}_{\uparrow^1}=\left[
\cdots\to \mathbf{0}\to\mathbf{0}\to
\coev
\to \mathbf{0}\to\mathbf{0}\to\cdots\right]
\]
\[
\mathrm{coev}_{\downarrow^1}=\left[
\cdots\to \mathbf{0}\to\mathbf{0}\to
\coevop
\to \mathbf{0}\to\mathbf{0}\to\cdots\right],
\]
where the nonzero graph is in degree zero, endow $\mathbf{KR}$ with the structure of balanced rigid category and exhibit $\uparrow^1$ as the (left and right) dual of $\downarrow^1$ and vive versa.  
\end{proposition}
\begin{proof}
As the formal complexes defining the evaluation and coevaluation morphisms in $\mathbf{KR}$ are concentarted in degree zero, the verification of the Zorro moves is reduced to its verification in $\mathbf{MOY}$, where it is immediate. 
\end{proof}

\subsection{The braidings in $\mathbf{KR}$} 
Abstracting from \cite{Khovanov-Rozansky} we now define two morphisms
\[
\sigma^+,\sigma^-\colon \uparrow^1\,\,\uparrow^1\,\, \to\,\, \uparrow^1\,\,\uparrow^1
\]
in $\mathbf{KR}$ which satisfy equations (R1), (R2a), (R2b) and (R3). As $\mathbf{KR}$ is monoidally generated by $\uparrow^1$ and its dual $\downarrow^1$, this will endow $\mathbf{KR}$ with the structure of untwisted braided category. This is the content of the following
\begin{theorem}
The morphisms
\[
\sigma^+=\left[
\cdots\to \mathbf{0}\to\mathbf{0}\to
\idtwo\{n-1\}\xrightarrow{\chi_0} \esse\{n\}
\to \mathbf{0}\to\mathbf{0}\to\cdots\right]
\]
with $\idtwo\{n-1\}$ in degree zero, and
\[
\sigma^-=\left[
\cdots\to \mathbf{0}\to\mathbf{0}\to
\esse\{-n\}\xrightarrow{\chi_1} \idtwo\{1-n\}
\to \mathbf{0}\to\mathbf{0}\to\cdots\right]
\]
with $\idtwo\{1-n\}$ in degree zero, endow $\mathbf{KR}$ with the structure of untwisted braided category, and so there is a distinguished braided monoidal functor
\[
Kh_n\colon \mathbf{TD}\to \mathbf{KR}
\]
defined by the universal property of $\mathbf{TD}$ (Theorem \ref{reshetikin-turaev}).
\end{theorem}
\begin{proof}
As $\mathbf{KR}$ is a rigid balanced category, there is a unique rigid monoidal functor 
\[
Kh_n\colon \mathbf{PTD}\to \mathbf{KR}
\]
with $Kh_n(\uparrow)=\uparrow^1$. So what we want to prove is that $Kh_n$ factors through $\mathbf{TD}$. This is equivalent to showing that $Kh_n\colon \mathbf{PTD}\to \mathbf{KR}$ is invariant with respect to the Reidemeister moves. The proof of this fact, that we present below, is an formal version of the proof given by Khovanov and Rozansky in the context of complexes of matrix factorizations in Section ?? of \cite{Khovanov-Rozansky}. In the remainder of this proof we will display only the relevant part of (equivalence classes of) formal complexes, i.e., if, for instance we are dealing with a formal complex of the form
\[
[\cdots\mathbf{0}\to\mathbf{0}\to M\xrightarrow{A} N\to \mathbf{0}\to \mathbf{0}\to\cdots] 
\]
we will write it simply as
\[
[M\xrightarrow{A} N]
\]
and will specify the degree of $M$ or $N$. Also, we will write $\ijovercrossing$ and $\ijundercrossing$ for $\sigma^+$ and $\sigma^-$.

\subsubsection{Invariance by Reidemeister 1}
We want to prove that
\[
Kh_n\left[\xy
(2,5);(0,0.1)**\crv{(2,1)&(-2,-4)&(-4,0)&(-2,4)}?>(0)*\dir{<}
,(2,-5);(0.65,-1)**\crv{(2,-4)}
\endxy\,\right]\quad=
\quad
Kh_n\left(
\xy
(0,-5);(0,5)**\dir{-}?>(1)*\dir{>}
\endxy\,\,\right).
\]
(as well as with the overcrossing replaced by an undercrossing). As
\[
\ijovercrossing=\left[
{\xy
(-5,-10);(-5,10)**\dir{-}?>(.5)*\dir{>}
,(5,-10);(5,10)**\dir{-}?>(.5)*\dir{>}
,(-7,-10)*{\scriptstyle{1}}
,(7,-10)*{\scriptstyle{1}}
,(-7,10)*{\scriptstyle{1}}
,(7,10)*{\scriptstyle{1}}
\endxy}\,\{n-1\}\,
\xrightarrow{\phantom{mm}\chi_0\phantom{mm}} \,
\esse\{n\}\,\right]
\]
We have
\[
Kh_n\left(\xy
(2,5);(0,0.1)**\crv{(2,1)&(-2,-4)&(-4,0)&(-2,4)}?>(0)*\dir{<}
,(2,-5);(0.65,-1)**\crv{(2,-4)}
,(1,-5)*{\scriptstyle{1}}
\endxy\,\right)\quad= \left[
\qquad\unknot \quad {\xy
(0,-10);(0,10)**\dir{-}
?>(.95)*\dir{>}
,(2,-11)*{\scriptstyle{1}}
,(2,11)*{\scriptstyle{1}}
\endxy} \xrightarrow{1_{\mathrm{id}_1}\otimes \chi_0}
{\xy
(0,-10);(0,10)**\crv{(0,-7)&(0,-6)&(0,-5)&(0,0)&(0,5)&(0,6)&(0,7)}
?>(.2)*\dir{>}?>(.99)*\dir{>}?>(.55)*\dir{>}
,(0,-3);(0,3)**\crv{(0,-4)&(0,-5)&(-4,-6)&(-6,0)&(-4,6)&(0,5)&(0,4)}
?>(.55)*\dir{<}
,(2,-11)*{\scriptstyle{1}}
,(2,11)*{\scriptstyle{1}}
,(-2,0)*{\scriptstyle{2}}
,(-6,0)*{\scriptstyle{1}}
\endxy}\{1\}
\right]\{n-1\}
\]
Therefore, by the commutative diagram (\ref{eq:moychi0}) and the equivalence relation (eq1) we have
\[
Kh_n\left(\xy
(2,5);(0,0.1)**\crv{(2,1)&(-2,-4)&(-4,0)&(-2,4)}?>(0)*\dir{<}
,(2,-5);(0.65,-1)**\crv{(2,-4)}
\endxy\,\right)\quad=
\left[
{\xy
(0,-10);(0,10)**\dir{-}
?>(.95)*\dir{>}
,(2,-11)*{\scriptstyle{1}}
,(2,11)*{\scriptstyle{1}}
\endxy} \{1-n\}\oplus [n-1]\,{\xy
(0,-10);(0,10)**\dir{-}
?>(.95)*\dir{>}
,(2,-11)*{\scriptstyle{1}}
,(2,11)*{\scriptstyle{1}}
\endxy} \{1\}
\xrightarrow{\tiny{\left(\begin{matrix}0&1\end{matrix}\right)}}  [n-1]\,{\xy
(0,-10);(0,10)**\dir{-}
?>(.95)*\dir{>}
,(2,-11)*{\scriptstyle{1}}
,(2,11)*{\scriptstyle{1}}
\endxy} \{1\}
\right]\{n-1\}
\]
By the equivalence relation (eq3) we get
\[
Kh_n\left(\xy
(2,5);(0,0.1)**\crv{(2,1)&(-2,-4)&(-4,0)&(-2,4)}?>(0)*\dir{<}
,(2,-5);(0.65,-1)**\crv{(2,-4)}
\endxy\,\right)\quad=
\left[
{\xy
(0,-10);(0,10)**\dir{-}
?>(.95)*\dir{>}
,(2,-11)*{\scriptstyle{1}}
,(2,11)*{\scriptstyle{1}}
\endxy} \{1-n\}\to 0
\right]\{n-1\}
=\left[
{\xy
(0,-10);(0,10)**\dir{-}
?>(.95)*\dir{>}
,(2,-11)*{\scriptstyle{1}}
,(2,11)*{\scriptstyle{1}}
\endxy}
\right]
=Kh_n\left(\,
\xy
(0,-5);(0,5)**\dir{-}?>(1)*\dir{>}
\endxy\,\,\right).
\] 
The proof for the undercrossing is completely analogous. This proves that $Kh_n$ is invariant under the first Reidemeister move.

\subsection{Invariance by Reidemeister 2 (first kind)}
We want to prove that
\[
Kh_n\left(\,\,
\xy
(3,-7);(3,7)**\crv{(3,-5)&(3,-4)&(-2,0)&(3,4)&(3,5)}?>(.99)*\dir{>}
,(-3,-7);(-0.3,-2)**\crv{(-3,-5)&(-3,-4)}
,(-3,7);(-0.3,2)**\crv{(-3,5)&(-3,4)}?>(0)*\dir{<}
,(.6,-1.3);(0.6,1.4)**\crv{(3,0)}
\endxy\,\,\right)\quad =
\quad
Kh_n\left(\,\,
\xy
(-3,-7);(-3,7)**\dir{-}?>(1)*\dir{>}
,(3,-7);(3,7)**\dir{-}?>(1)*\dir{>}
\endxy\,\,
\right)
\]
(as well as with the overcrossings exchanged with undercrossings and vice versa). By definition of $Kh_n$, and recalling that
\[
\ijovercrossing=\left[
{\xy
(-5,-10);(-5,10)**\dir{-}?>(.5)*\dir{>}
,(5,-10);(5,10)**\dir{-}?>(.5)*\dir{>}
,(-7,-10)*{\scriptstyle{1}}
,(7,-10)*{\scriptstyle{1}}
,(-7,10)*{\scriptstyle{1}}
,(7,10)*{\scriptstyle{1}}
\endxy}\,\{n-1\}\,
\xrightarrow{\phantom{mm}\chi_0\phantom{mm}} \,
\esse\{n\}\,\right]
\]
and
\[
\ijundercrossing=\left[
\esse\{-n\}\,
\xrightarrow{\phantom{mm}\chi_1\phantom{mm}} \,
{\xy
(-5,-10);(-5,10)**\dir{-}?>(.5)*\dir{>}
,(5,-10);(5,10)**\dir{-}?>(.5)*\dir{>}
,(-7,-10)*{\scriptstyle{1}}
,(7,-10)*{\scriptstyle{1}}
,(-7,10)*{\scriptstyle{1}}
,(7,10)*{\scriptstyle{1}}
\endxy}\,\{1-n\}\,\right]
\]
we have
\[
Kh_n\left(\,\,
\xy
(3,-7);(3,7)**\crv{(3,-5)&(3,-4)&(-2,0)&(3,4)&(3,5)}?>(.99)*\dir{>}
,(-3,-7);(-0.3,-2)**\crv{(-3,-5)&(-3,-4)}
,(-3,7);(-0.3,2)**\crv{(-3,5)&(-3,4)}?>(0)*\dir{<}
,(.6,-1.3);(0.6,1.4)**\crv{(3,0)}
\endxy\,\,\right)\quad =[\mathrm{tot}\left(\raisebox{80 pt}{
\xymatrix{
{\xy
(-5,-10);(-5,10)**\crv{(-1,-6)&(0,-5)&(0,-2)&(0,2)&(0,5)&(-1,6)}
?>(.55)*\dir{>}?>(.1)*\dir{>}?>(.95)*\dir{>}
,(5,-10);(5,10)**\crv{(1,-6)&(0,-5)&(0,-2)&(0,2)&(0,5)&(1,6)}
?>(.1)*\dir{>}?>(.95)*\dir{>}
,(-7,-10)*{\scriptstyle{1}}
,(7,-10)*{\scriptstyle{1}}
,(-7,10)*{\scriptstyle{1}}
,(7,10)*{\scriptstyle{1}}
,(2,0)*{\scriptstyle{2}}
\endxy}\{-1\}\ar[rr]^{1\circ\chi_1}\ar[dd]^{\chi_0\circ 1}
&&
{\xy
(-5,-10);(-5,10)**\dir{-}?>(.5)*\dir{>}
,(5,-10);(5,10)**\dir{-}?>(.5)*\dir{>}
,(-7,-10)*{\scriptstyle{1}}
,(7,-10)*{\scriptstyle{1}}
,(-7,10)*{\scriptstyle{1}}
,(7,10)*{\scriptstyle{1}}
\endxy}\ar[dd]_{\chi_0\circ1}\\
\\
\essecircesse\ar[rr]_{1\circ \chi_1}&&{\xy
(-5,-10);(-5,10)**\crv{(-1,-6)&(0,-5)&(0,-2)&(0,2)&(0,5)&(-1,6)}
?>(.55)*\dir{>}?>(.1)*\dir{>}?>(.95)*\dir{>}
,(5,-10);(5,10)**\crv{(1,-6)&(0,-5)&(0,-2)&(0,2)&(0,5)&(1,6)}
?>(.1)*\dir{>}?>(.95)*\dir{>}
,(-7,-10)*{\scriptstyle{1}}
,(7,-10)*{\scriptstyle{1}}
,(-7,10)*{\scriptstyle{1}}
,(7,10)*{\scriptstyle{1}}
,(2,0)*{\scriptstyle{2}}
\endxy}\{1\}
}}
\right)]
\]
\[
\qquad\qquad=\left[
{\xy
(-5,-10);(-5,10)**\crv{(-1,-6)&(0,-5)&(0,-2)&(0,2)&(0,5)&(-1,6)}
?>(.55)*\dir{>}?>(.1)*\dir{>}?>(.95)*\dir{>}
,(5,-10);(5,10)**\crv{(1,-6)&(0,-5)&(0,-2)&(0,2)&(0,5)&(1,6)}
?>(.1)*\dir{>}?>(.95)*\dir{>}
,(-7,-10)*{\scriptstyle{1}}
,(7,-10)*{\scriptstyle{1}}
,(-7,10)*{\scriptstyle{1}}
,(7,10)*{\scriptstyle{1}}
,(2,0)*{\scriptstyle{2}}
\endxy}\{-1\}
\xrightarrow{\left(
\begin{smallmatrix}
1\circ \chi_1\\ \\
\chi_0\circ 1
\end{smallmatrix}\right)} 
{\xy
(-5,-10);(-5,10)**\dir{-}?>(.5)*\dir{>}
,(5,-10);(5,10)**\dir{-}?>(.5)*\dir{>}
,(-7,-10)*{\scriptstyle{1}}
,(7,-10)*{\scriptstyle{1}}
,(-7,10)*{\scriptstyle{1}}
,(7,10)*{\scriptstyle{1}}
\endxy}\oplus \essecircesse
\xrightarrow{\left(
\begin{smallmatrix}
-\chi_0\circ 1 &&
1\circ \chi_1 
\end{smallmatrix}\right)}{\xy
(-5,-10);(-5,10)**\crv{(-1,-6)&(0,-5)&(0,-2)&(0,2)&(0,5)&(-1,6)}
?>(.55)*\dir{>}?>(.1)*\dir{>}?>(.95)*\dir{>}
,(5,-10);(5,10)**\crv{(1,-6)&(0,-5)&(0,-2)&(0,2)&(0,5)&(1,6)}
?>(.1)*\dir{>}?>(.95)*\dir{>}
,(-7,-10)*{\scriptstyle{1}}
,(7,-10)*{\scriptstyle{1}}
,(-7,10)*{\scriptstyle{1}}
,(7,10)*{\scriptstyle{1}}
,(2,0)*{\scriptstyle{2}}
\endxy}\{1\}
\right],
\]
where the term ${\xy
(-5,-10);(-5,10)**\crv{(-1,-6)&(0,-5)&(0,-2)&(0,2)&(0,5)&(-1,6)}
?>(.55)*\dir{>}?>(.1)*\dir{>}?>(.95)*\dir{>}
,(5,-10);(5,10)**\crv{(1,-6)&(0,-5)&(0,-2)&(0,2)&(0,5)&(1,6)}
?>(.1)*\dir{>}?>(.95)*\dir{>}
,(-7,-10)*{\scriptstyle{1}}
,(7,-10)*{\scriptstyle{1}}
,(-7,10)*{\scriptstyle{1}}
,(7,10)*{\scriptstyle{1}}
,(2,0)*{\scriptstyle{2}}
\endxy}\{-1\}$ 
is in degree -1. By MOY relations (\ref{eq:moy2b}) and (\ref{eq:moy2d}) we have a commutative diagram
\[
\xymatrix{
{\xy
(-5,-10);(-5,10)**\crv{(-1,-6)&(0,-5)&(0,-2)&(0,2)&(0,5)&(-1,6)}
?>(.55)*\dir{>}?>(.1)*\dir{>}?>(.95)*\dir{>}
,(5,-10);(5,10)**\crv{(1,-6)&(0,-5)&(0,-2)&(0,2)&(0,5)&(1,6)}
?>(.1)*\dir{>}?>(.95)*\dir{>}
,(-7,-10)*{\scriptstyle{1}}
,(7,-10)*{\scriptstyle{1}}
,(-7,10)*{\scriptstyle{1}}
,(7,10)*{\scriptstyle{1}}
,(2,0)*{\scriptstyle{2}}
\endxy}\{-1\}
\ar[rr]^-{\left(
\begin{smallmatrix}
1\circ \chi_1\\ \\
\chi_0\circ 1
\end{smallmatrix}\right)} 
\ar@{=}[dd]
&&
{\xy
(-5,-10);(-5,10)**\dir{-}?>(.5)*\dir{>}
,(5,-10);(5,10)**\dir{-}?>(.5)*\dir{>}
,(-7,-10)*{\scriptstyle{1}}
,(7,-10)*{\scriptstyle{1}}
,(-7,10)*{\scriptstyle{1}}
,(7,10)*{\scriptstyle{1}}
\endxy}\oplus \essecircesse\ar[dd]_{1\oplus\varphi}^{\wr}
\ar[rr]^-{\left(
\begin{smallmatrix}
-\chi_0\circ 1 &&
1\circ \chi_1 
\end{smallmatrix}\right)}
&&
\quad{\xy
(-5,-10);(-5,10)**\crv{(-1,-6)&(0,-5)&(0,-2)&(0,2)&(0,5)&(-1,6)}
?>(.55)*\dir{>}?>(.1)*\dir{>}?>(.95)*\dir{>}
,(5,-10);(5,10)**\crv{(1,-6)&(0,-5)&(0,-2)&(0,2)&(0,5)&(1,6)}
?>(.1)*\dir{>}?>(.95)*\dir{>}
,(-7,-10)*{\scriptstyle{1}}
,(7,-10)*{\scriptstyle{1}}
,(-7,10)*{\scriptstyle{1}}
,(7,10)*{\scriptstyle{1}}
,(2,0)*{\scriptstyle{2}}
\endxy}\{1\}
\ar@{=}[dd]
\\ \\
{\xy
(-5,-10);(-5,10)**\crv{(-1,-6)&(0,-5)&(0,-2)&(0,2)&(0,5)&(-1,6)}
?>(.55)*\dir{>}?>(.1)*\dir{>}?>(.95)*\dir{>}
,(5,-10);(5,10)**\crv{(1,-6)&(0,-5)&(0,-2)&(0,2)&(0,5)&(1,6)}
?>(.1)*\dir{>}?>(.95)*\dir{>}
,(-7,-10)*{\scriptstyle{1}}
,(7,-10)*{\scriptstyle{1}}
,(-7,10)*{\scriptstyle{1}}
,(7,10)*{\scriptstyle{1}}
,(2,0)*{\scriptstyle{2}}
\endxy}\{-1\}
\ar[rr]^-{\left(
\begin{smallmatrix}
1\circ \chi_1\\ \\
\alpha\\
\\
1
\end{smallmatrix}\right)} 
&&
{\xy
(-5,-10);(-5,10)**\dir{-}?>(.5)*\dir{>}
,(5,-10);(5,10)**\dir{-}?>(.5)*\dir{>}
,(-7,-10)*{\scriptstyle{1}}
,(7,-10)*{\scriptstyle{1}}
,(-7,10)*{\scriptstyle{1}}
,(7,10)*{\scriptstyle{1}}
\endxy}\oplus \esse\{1\} \oplus \esse\{-1\}
\ar[rr]^-{\left(
\begin{smallmatrix}
-\chi_0\circ 1 &&
1 && 0 
\end{smallmatrix}\right)}
&&
\quad{\xy
(-5,-10);(-5,10)**\crv{(-1,-6)&(0,-5)&(0,-2)&(0,2)&(0,5)&(-1,6)}
?>(.55)*\dir{>}?>(.1)*\dir{>}?>(.95)*\dir{>}
,(5,-10);(5,10)**\crv{(1,-6)&(0,-5)&(0,-2)&(0,2)&(0,5)&(1,6)}
?>(.1)*\dir{>}?>(.95)*\dir{>}
,(-7,-10)*{\scriptstyle{1}}
,(7,-10)*{\scriptstyle{1}}
,(-7,10)*{\scriptstyle{1}}
,(7,10)*{\scriptstyle{1}}
,(2,0)*{\scriptstyle{2}}
\endxy}\{1\}
}
\]
and so
\[
Kh_n\left(\,\,
\xy
(3,-7);(3,7)**\crv{(3,-5)&(3,-4)&(-2,0)&(3,4)&(3,5)}?>(.99)*\dir{>}
,(-3,-7);(-0.3,-2)**\crv{(-3,-5)&(-3,-4)}
,(-3,7);(-0.3,2)**\crv{(-3,5)&(-3,4)}?>(0)*\dir{<}
,(.6,-1.3);(0.6,1.4)**\crv{(3,0)}
,(-4,-7.5)*{\scriptstyle{1}}
,(4,-7.5)*{\scriptstyle{1}}
\endxy\,\,\right)\quad =
\left[
\xymatrix{
{\xy
(-5,-10);(-5,10)**\crv{(-1,-6)&(0,-5)&(0,-2)&(0,2)&(0,5)&(-1,6)}
?>(.55)*\dir{>}?>(.1)*\dir{>}?>(.95)*\dir{>}
,(5,-10);(5,10)**\crv{(1,-6)&(0,-5)&(0,-2)&(0,2)&(0,5)&(1,6)}
?>(.1)*\dir{>}?>(.95)*\dir{>}
,(-7,-10)*{\scriptstyle{1}}
,(7,-10)*{\scriptstyle{1}}
,(-7,10)*{\scriptstyle{1}}
,(7,10)*{\scriptstyle{1}}
,(2,0)*{\scriptstyle{2}}
\endxy}\{-1\}
\ar[r]^-{\left(
\begin{smallmatrix}
1\circ \chi_1\\ \\
\alpha\\
\\
1
\end{smallmatrix}\right)} 
&
{\xy
(-5,-10);(-5,10)**\dir{-}?>(.5)*\dir{>}
,(5,-10);(5,10)**\dir{-}?>(.5)*\dir{>}
,(-7,-10)*{\scriptstyle{1}}
,(7,-10)*{\scriptstyle{1}}
,(-7,10)*{\scriptstyle{1}}
,(7,10)*{\scriptstyle{1}}
\endxy}\oplus \esse\{1\} \oplus \esse\{-1\}
\ar[rr]^-{\left(
\begin{smallmatrix}
-\chi_0\circ 1 &&
1 && 0 
\end{smallmatrix}\right)}
&&
{\xy
(-5,-10);(-5,10)**\crv{(-1,-6)&(0,-5)&(0,-2)&(0,2)&(0,5)&(-1,6)}
?>(.55)*\dir{>}?>(.1)*\dir{>}?>(.95)*\dir{>}
,(5,-10);(5,10)**\crv{(1,-6)&(0,-5)&(0,-2)&(0,2)&(0,5)&(1,6)}
?>(.1)*\dir{>}?>(.95)*\dir{>}
,(-7,-10)*{\scriptstyle{1}}
,(7,-10)*{\scriptstyle{1}}
,(-7,10)*{\scriptstyle{1}}
,(7,10)*{\scriptstyle{1}}
,(2,0)*{\scriptstyle{2}}
\endxy}\{1\}
}
\right]
\]
By the equivalence relation (eq3) we get
\[
Kh_n\left(\,\,
\xy
(3,-7);(3,7)**\crv{(3,-5)&(3,-4)&(-2,0)&(3,4)&(3,5)}?>(.99)*\dir{>}
,(-3,-7);(-0.3,-2)**\crv{(-3,-5)&(-3,-4)}
,(-3,7);(-0.3,2)**\crv{(-3,5)&(-3,4)}?>(0)*\dir{<}
,(.6,-1.3);(0.6,1.4)**\crv{(3,0)}
\endxy\,\,\right)\quad =
\left[
\xymatrix{
0
\ar[r]
&
{\xy
(-5,-10);(-5,10)**\dir{-}?>(.5)*\dir{>}
,(5,-10);(5,10)**\dir{-}?>(.5)*\dir{>}
,(-7,-10)*{\scriptstyle{1}}
,(7,-10)*{\scriptstyle{1}}
,(-7,10)*{\scriptstyle{1}}
,(7,10)*{\scriptstyle{1}}
\endxy}\oplus \esse\{1\}
\ar[rr]^-{\left(
\begin{smallmatrix}
-\chi_0\circ 1 &&
1 
\end{smallmatrix}\right)}
&&
{\xy
(-5,-10);(-5,10)**\crv{(-1,-6)&(0,-5)&(0,-2)&(0,2)&(0,5)&(-1,6)}
?>(.55)*\dir{>}?>(.1)*\dir{>}?>(.95)*\dir{>}
,(5,-10);(5,10)**\crv{(1,-6)&(0,-5)&(0,-2)&(0,2)&(0,5)&(1,6)}
?>(.1)*\dir{>}?>(.95)*\dir{>}
,(-7,-10)*{\scriptstyle{1}}
,(7,-10)*{\scriptstyle{1}}
,(-7,10)*{\scriptstyle{1}}
,(7,10)*{\scriptstyle{1}}
,(2,0)*{\scriptstyle{2}}
\endxy}\{1\}
}
\right]
\]
\[
\phantom{mmmmmmmm}=
\left[
\xymatrix{
0
\ar[r]
&
\esse\{1\}
\oplus
\idtwo\ar[rr]^-{\left(
\begin{smallmatrix}
1 &&
-\chi_0\circ 1
\end{smallmatrix}\right)}
&&
{\xy
(-5,-10);(-5,10)**\crv{(-1,-6)&(0,-5)&(0,-2)&(0,2)&(0,5)&(-1,6)}
?>(.55)*\dir{>}?>(.1)*\dir{>}?>(.95)*\dir{>}
,(5,-10);(5,10)**\crv{(1,-6)&(0,-5)&(0,-2)&(0,2)&(0,5)&(1,6)}
?>(.1)*\dir{>}?>(.95)*\dir{>}
,(-7,-10)*{\scriptstyle{1}}
,(7,-10)*{\scriptstyle{1}}
,(-7,10)*{\scriptstyle{1}}
,(7,10)*{\scriptstyle{1}}
,(2,0)*{\scriptstyle{2}}
\endxy}\{1\}
}
\right],
\]
where in the second equality we have almost trivially used (eq1). By the equivalence relation (eq3) again, we finally get
\[
Kh_n\left(\,\,
\xy
(3,-7);(3,7)**\crv{(3,-5)&(3,-4)&(-2,0)&(3,4)&(3,5)}?>(.99)*\dir{>}
,(-3,-7);(-0.3,-2)**\crv{(-3,-5)&(-3,-4)}
,(-3,7);(-0.3,2)**\crv{(-3,5)&(-3,4)}?>(0)*\dir{<}
,(.6,-1.3);(0.6,1.4)**\crv{(3,0)}
\endxy\,\,\right)\quad =
\left[
\xymatrix{
0
\ar[r]
&
\idtwo\ar[r]
&
0
}
\right]=
\left[
\,
\idtwo\,
\right]=Kh_n\left(\,\,
\xy
(-3,-7);(-3,7)**\dir{-}?>(1)*\dir{>}
,(3,-7);(3,7)**\dir{-}?>(1)*\dir{>}
\endxy\,\,
\right).
\]

\subsection{Invariance by Reidemeister 2 (second kind)}

We want to prove that
\[
Kh_n\left(\,\,\xy
(-5,10);(5,10)**\crv{(-5,2)&(0,-6)&(5,2)}
?>(0)*\dir{<},
?>(.71)*{\color{white}\bullet},
?>(.30)*{\color{white}\bullet}
,(-5,-10);(5,-10)**\crv{(-5,-2)&(0,6)&(5,-2)}
?>(1)*\dir{>}
\endxy\,\,\right)
=
Kh_n\left(\,\,\xy
(-5,-10);(5,-10)**\crv{(-5,-2)&(0,-2)&(5,-2)}
?>(1)*\dir{>}
,(-5,10);(5,10)**\crv{(-5,2)&(0,2)&(5,2)}
?>(0)*\dir{<}
\endxy\,\,\right)
\]
By definition of $Kh_n$, and omitting the identity morphisms from the notation when there is no ambiguity, we have
\[
Kh_n\left(\,\,\xy
(-5,10);(5,10)**\crv{(-5,2)&(0,-6)&(5,2)}
?>(0)*\dir{<},
?>(.71)*{\color{white}\bullet},
?>(.30)*{\color{white}\bullet}
,(-5,-10);(5,-10)**\crv{(-5,-2)&(0,6)&(5,-2)}
?>(1)*\dir{>}
\endxy\,\,\right)
=
[\mathrm{tot}\left(\raisebox{48pt}{
\xymatrix{
{\xy
(-10,-10);(-10,10)**\dir{-}?>(.5)*\dir{>}
,(-12,-11)*{\scriptstyle{1}}
,(-12,11)*{\scriptstyle{1}}
,(0,10);(0,-10)**\crv{(0,7)&(0,6)&(0,5)&(0,0)&(0,-5)&(0,-6)&(0,-7)}
?>(.2)*\dir{>}?>(.99)*\dir{>}?>(.55)*\dir{>}
,(0,3);(0,-3)**\crv{(0,4)&(0,5)&(-4,6)&(-6,0)&(-4,-6)&(0,-5)&(0,-4)}
?>(.55)*\dir{<}
,(2,-11)*{\scriptstyle{1}}
,(2,11)*{\scriptstyle{1}}
,(2,0)*{\scriptstyle{2}}
,(-6,0)*{\scriptstyle{1}}
\endxy}\{-1\}
\ar[r]^-{\chi_0\otimes 1}\ar[d]_{1\otimes \chi_1}&
{\xy
(-10,-10);(-10,10)**\crv{(-6,-6)&(-5,-5)&(-5,-2)&(-5,-2)&(-5,5)&(-6,6)}
?>(.05)*\dir{>}?>(.99)*\dir{>}?>(.55)*\dir{>}
,(10,-10);(10,10)**\crv{(6,-6)&(5,-5)&(5,-2)&(5,2)&(5,5)&(6,6)}
?>(.03)*\dir{<}?>(.92)*\dir{<}?>(.45)*\dir{<}
,(-5,0);(5,0)**\crv{(-5,-2)&(-5,-3)&(-4,-5)&(-2,-6)&(2,-6)&(4,-5)&(5,-3)&(5,-2)}?>(.5)*\dir{<}
,(-5,0);(5,0)**\crv{(-5,2)&(-5,3)&(-4,5)&(-2,6)&(2,6)&(4,5)&(5,3)&(5,2)}?>(.5)*\dir{>}
,(-11,-10)*{\scriptstyle{1}}
,(-11,10)*{\scriptstyle{1}}
,(11,-10)*{\scriptstyle{1}}
,(11,10)*{\scriptstyle{1}}
,(-7,0)*{\scriptstyle{2}}
,(0,-8)*{\scriptstyle{1}}
,(6.5,0)*{\scriptstyle{2}}
,(0,8)*{\scriptstyle{1}}
\endxy}\ar[d]^{1\otimes \chi_1}\\
{\xy
(-10,-10);(-10,10)**\dir{-}?>(.5)*\dir{>}
,(-12,-11)*{\scriptstyle{1}}
,(-12,11)*{\scriptstyle{1}}
,(-2,-5);(-2,-5)**\crv{(3,-5)&(3,5)&(-7,5)&(-7,-5)}
?>(.3)*\dir{<}
,(3.5,1)*{\scriptstyle{1}}
,(6,-10);(6,10)**\dir{-}?>(.5)*\dir{<}
,(8,-11)*{\scriptstyle{1}}
,(8,11)*{\scriptstyle{1}}
\endxy}
\ar[r]^-{\chi_0\otimes 1}&
{\xy
(0,-10);(0,10)**\crv{(0,-7)&(0,-6)&(0,-5)&(0,0)&(0,5)&(0,6)&(0,7)}
?>(.2)*\dir{>}?>(.99)*\dir{>}?>(.55)*\dir{>}
,(0,-3);(0,3)**\crv{(0,-4)&(0,-5)&(4,-6)&(6,0)&(4,6)&(0,5)&(0,4)}
?>(.55)*\dir{<}
,(2,-11)*{\scriptstyle{1}}
,(2,11)*{\scriptstyle{1}}
,(-2,0)*{\scriptstyle{2}}
,(7,0)*{\scriptstyle{1}}
,(10,-10);(10,10)**\dir{-}?>(.5)*\dir{<}
,(12,-11)*{\scriptstyle{1}}
,(12,11)*{\scriptstyle{1}}
\endxy}\{1\}
}}
\right)
]
\]
\[
=\left[
{\xy
(-10,-10);(-10,10)**\dir{-}?>(.5)*\dir{>}
,(-12,-11)*{\scriptstyle{1}}
,(-12,11)*{\scriptstyle{1}}
,(0,10);(0,-10)**\crv{(0,7)&(0,6)&(0,5)&(0,0)&(0,-5)&(0,-6)&(0,-7)}
?>(.2)*\dir{>}?>(.99)*\dir{>}?>(.55)*\dir{>}
,(0,3);(0,-3)**\crv{(0,4)&(0,5)&(-4,6)&(-6,0)&(-4,-6)&(0,-5)&(0,-4)}
?>(.55)*\dir{<}
,(2,-11)*{\scriptstyle{1}}
,(2,11)*{\scriptstyle{1}}
,(2,0)*{\scriptstyle{2}}
,(-6,0)*{\scriptstyle{1}}
\endxy}\{-1\}
\xrightarrow{{\left(
\begin{smallmatrix}
 \chi_0\otimes 1\\ \\
1\otimes \chi_1
\end{smallmatrix}\right)}}
{\xy
(-10,-10);(-10,10)**\crv{(-6,-6)&(-5,-5)&(-5,-2)&(-5,-2)&(-5,5)&(-6,6)}
?>(.05)*\dir{>}?>(.99)*\dir{>}?>(.55)*\dir{>}
,(10,-10);(10,10)**\crv{(6,-6)&(5,-5)&(5,-2)&(5,2)&(5,5)&(6,6)}
?>(.03)*\dir{<}?>(.92)*\dir{<}?>(.45)*\dir{<}
,(-5,0);(5,0)**\crv{(-5,-2)&(-5,-3)&(-4,-5)&(-2,-6)&(2,-6)&(4,-5)&(5,-3)&(5,-2)}?>(.5)*\dir{<}
,(-5,0);(5,0)**\crv{(-5,2)&(-5,3)&(-4,5)&(-2,6)&(2,6)&(4,5)&(5,3)&(5,2)}?>(.5)*\dir{>}
,(-11,-10)*{\scriptstyle{1}}
,(-11,10)*{\scriptstyle{1}}
,(11,-10)*{\scriptstyle{1}}
,(11,10)*{\scriptstyle{1}}
,(-7,0)*{\scriptstyle{2}}
,(0,-8)*{\scriptstyle{1}}
,(6.5,0)*{\scriptstyle{2}}
,(0,8)*{\scriptstyle{1}}
\endxy}\oplus
{\xy
(-10,-10);(-10,10)**\dir{-}?>(.5)*\dir{>}
,(-12,-11)*{\scriptstyle{1}}
,(-12,11)*{\scriptstyle{1}}
,(-2,-5);(-2,-5)**\crv{(3,-5)&(3,5)&(-7,5)&(-7,-5)}
?>(.3)*\dir{<}
,(3.5,1)*{\scriptstyle{1}}
,(6,-10);(6,10)**\dir{-}?>(.5)*\dir{<}
,(8,-11)*{\scriptstyle{1}}
,(8,11)*{\scriptstyle{1}}
\endxy}
\xrightarrow{{\left(
\begin{smallmatrix}
-1\otimes \chi_1\ &&
\chi_0\otimes 1 
\end{smallmatrix}\right)}}
{\xy
(0,-10);(0,10)**\crv{(0,-7)&(0,-6)&(0,-5)&(0,0)&(0,5)&(0,6)&(0,7)}
?>(.2)*\dir{>}?>(.99)*\dir{>}?>(.55)*\dir{>}
,(0,-3);(0,3)**\crv{(0,-4)&(0,-5)&(4,-6)&(6,0)&(4,6)&(0,5)&(0,4)}
?>(.55)*\dir{<}
,(2,-11)*{\scriptstyle{1}}
,(2,11)*{\scriptstyle{1}}
,(-2,0)*{\scriptstyle{2}}
,(7,0)*{\scriptstyle{1}}
,(10,-10);(10,10)**\dir{-}?>(.5)*\dir{<}
,(12,-11)*{\scriptstyle{1}}
,(12,11)*{\scriptstyle{1}}
\endxy}\{1\}
\right],
\]
where the leftmost term is in degree -1. By MOY relations  (\ref{eq:moychi0}-\ref{eq:moychi1}) and (\ref{eq:moynu1}-\ref{eq:moynu2}) we have a commutative diagram
\[
\scalebox{.75}{
\xymatrix{
{\xy
(-10,-10);(-10,10)**\dir{-}?>(.5)*\dir{>}
,(-12,-11)*{\scriptstyle{1}}
,(-12,11)*{\scriptstyle{1}}
,(0,10);(0,-10)**\crv{(0,7)&(0,6)&(0,5)&(0,0)&(0,-5)&(0,-6)&(0,-7)}
?>(.2)*\dir{>}?>(.99)*\dir{>}?>(.55)*\dir{>}
,(0,3);(0,-3)**\crv{(0,4)&(0,5)&(-4,6)&(-6,0)&(-4,-6)&(0,-5)&(0,-4)}
?>(.55)*\dir{<}
,(2,-11)*{\scriptstyle{1}}
,(2,11)*{\scriptstyle{1}}
,(2,0)*{\scriptstyle{2}}
,(-6,0)*{\scriptstyle{1}}
\endxy}\{-1\}
\ar[rr]^-{{\left(
\begin{smallmatrix}
 \chi_0\otimes 1\\ \\
1\otimes \chi_1
\end{smallmatrix}\right)}}
\ar[dd]_{1\otimes \mu}^{\wr}
&&
{\xy
(-10,-10);(-10,10)**\crv{(-6,-6)&(-5,-5)&(-5,-2)&(-5,-2)&(-5,5)&(-6,6)}
?>(.05)*\dir{>}?>(.99)*\dir{>}?>(.55)*\dir{>}
,(10,-10);(10,10)**\crv{(6,-6)&(5,-5)&(5,-2)&(5,2)&(5,5)&(6,6)}
?>(.03)*\dir{<}?>(.92)*\dir{<}?>(.45)*\dir{<}
,(-5,0);(5,0)**\crv{(-5,-2)&(-5,-3)&(-4,-5)&(-2,-6)&(2,-6)&(4,-5)&(5,-3)&(5,-2)}?>(.5)*\dir{<}
,(-5,0);(5,0)**\crv{(-5,2)&(-5,3)&(-4,5)&(-2,6)&(2,6)&(4,5)&(5,3)&(5,2)}?>(.5)*\dir{>}
,(-11,-10)*{\scriptstyle{1}}
,(-11,10)*{\scriptstyle{1}}
,(11,-10)*{\scriptstyle{1}}
,(11,10)*{\scriptstyle{1}}
,(-7,0)*{\scriptstyle{2}}
,(0,-8)*{\scriptstyle{1}}
,(6.5,0)*{\scriptstyle{2}}
,(0,8)*{\scriptstyle{1}}
\endxy}\oplus
{\xy
(-10,-10);(-10,10)**\dir{-}?>(.5)*\dir{>}
,(-12,-11)*{\scriptstyle{1}}
,(-12,11)*{\scriptstyle{1}}
,(-2,-5);(-2,-5)**\crv{(3,-5)&(3,5)&(-7,5)&(-7,-5)}
?>(.3)*\dir{<}
,(3.5,1)*{\scriptstyle{1}}
,(6,-10);(6,10)**\dir{-}?>(.5)*\dir{<}
,(8,-11)*{\scriptstyle{1}}
,(8,11)*{\scriptstyle{1}}
\endxy}
\ar[rr]^-{{\left(
\begin{smallmatrix}
-1\otimes \chi_1\ &&
\chi_0\otimes 1 
\end{smallmatrix}\right)}}
\ar[dd]_{\nu\oplus\lambda}^{\wr}
&&
{\xy
(0,-10);(0,10)**\crv{(0,-7)&(0,-6)&(0,-5)&(0,0)&(0,5)&(0,6)&(0,7)}
?>(.2)*\dir{>}?>(.99)*\dir{>}?>(.55)*\dir{>}
,(0,-3);(0,3)**\crv{(0,-4)&(0,-5)&(4,-6)&(6,0)&(4,6)&(0,5)&(0,4)}
?>(.55)*\dir{<}
,(2,-11)*{\scriptstyle{1}}
,(2,11)*{\scriptstyle{1}}
,(-2,0)*{\scriptstyle{2}}
,(7,0)*{\scriptstyle{1}}
,(10,-10);(10,10)**\dir{-}?>(.5)*\dir{<}
,(12,-11)*{\scriptstyle{1}}
,(12,11)*{\scriptstyle{1}}
\endxy}\{1\}
\ar[dd]_{\mu\otimes 1}^{\wr}
\\
\\
[n-1]{\xy
(-5,-10);(-5,10)**\dir{-}?>(.5)*\dir{>}
,(-7,-11)*{\scriptstyle{1}}
,(-7,11)*{\scriptstyle{1}}
,(5,-10);(5,10)**\dir{-}?>(.5)*\dir{<}
,(7,-11)*{\scriptstyle{1}}
,(7,11)*{\scriptstyle{1}}
\endxy}\{-1\}\ar@{=}[dd]
&&
\xy
(-5,-10);(5,-10)**\crv{(-5,-2)&(0,-2)&(5,-2)}
?>(1)*\dir{>},
(-6,-10)*{\scriptstyle{1}},(6,-10)*{\scriptstyle{1}}
,(-5,10);(5,10)**\crv{(-5,2)&(0,2)&(5,2)}
?>(0)*\dir{<},
(-6,10)*{\scriptstyle{1}},(6,10)*{\scriptstyle{1}}
\endxy
\oplus\, [n-2]
{\xy
(-5,-10);(-5,10)**\dir{-}?>(.5)*\dir{>}
,(5,-10);(5,10)**\dir{-}?>(.5)*\dir{<}
,(-7,-10)*{\scriptstyle{1}}
,(7,-10)*{\scriptstyle{1}}
,(-7,10)*{\scriptstyle{1}}
,(7,10)*{\scriptstyle{1}}
\endxy}
\oplus\, [n]
{\xy
(-5,-10);(-5,10)**\dir{-}?>(.5)*\dir{>}
,(5,-10);(5,10)**\dir{-}?>(.5)*\dir{<}
,(-7,-10)*{\scriptstyle{1}}
,(7,-10)*{\scriptstyle{1}}
,(-7,10)*{\scriptstyle{1}}
,(7,10)*{\scriptstyle{1}}
\endxy}\ar@{=}[dd]
&&
[n-1]{\xy
(-5,-10);(-5,10)**\dir{-}?>(.5)*\dir{>}
,(-7,-11)*{\scriptstyle{1}}
,(-7,11)*{\scriptstyle{1}}
,(5,-10);(5,10)**\dir{-}?>(.5)*\dir{<}
,(7,-11)*{\scriptstyle{1}}
,(7,11)*{\scriptstyle{1}}
\endxy}\{1\}\ar@{=}[dd]
\\ \\
(\mathbb{Q}\{1-n\}\oplus[n-2]){\xy
(-5,-10);(-5,10)**\dir{-}?>(.5)*\dir{>}
,(-7,-11)*{\scriptstyle{1}}
,(-7,11)*{\scriptstyle{1}}
,(5,-10);(5,10)**\dir{-}?>(.5)*\dir{<}
,(7,-11)*{\scriptstyle{1}}
,(7,11)*{\scriptstyle{1}}
\endxy}
\ar[rr]^-{\left(
\begin{smallmatrix}
0 & 0\\
1 & 0\\
0 &1 \\
0 & 1\\
0 & 0
\end{smallmatrix}\right)}
&&
\xy
(-5,-10);(5,-10)**\crv{(-5,-2)&(0,-2)&(5,-2)}
?>(1)*\dir{>},
(-6,-10)*{\scriptstyle{1}},(6,-10)*{\scriptstyle{1}}
,(-5,10);(5,10)**\crv{(-5,2)&(0,2)&(5,2)}
?>(0)*\dir{<},
(-6,10)*{\scriptstyle{1}},(6,10)*{\scriptstyle{1}}
\endxy
\oplus\,  (\mathbb{Q}\{1-n\}\oplus[n-2]^{\oplus 2}\oplus \mathbb{Q}\{n-1\})
{\xy
(-5,-10);(-5,10)**\dir{-}?>(.5)*\dir{>}
,(5,-10);(5,10)**\dir{-}?>(.5)*\dir{<}
,(-7,-10)*{\scriptstyle{1}}
,(7,-10)*{\scriptstyle{1}}
,(-7,10)*{\scriptstyle{1}}
,(7,10)*{\scriptstyle{1}}
\endxy}
\ar[rr]^-{\left(
\begin{smallmatrix}
0 & 0& -1 &1& 0\\
0 &  0 & 0 &0 &1\\
\end{smallmatrix}\right)}
&&
([n-2]\oplus\mathbb{Q}\{n-1\}){\xy
(-5,-10);(-5,10)**\dir{-}?>(.5)*\dir{>}
,(-7,-11)*{\scriptstyle{1}}
,(-7,11)*{\scriptstyle{1}}
,(5,-10);(5,10)**\dir{-}?>(.5)*\dir{<}
,(7,-11)*{\scriptstyle{1}}
,(7,11)*{\scriptstyle{1}}
\endxy}
}
}
\]
and so, by the equivalence relation (eq1),
\begin{align*}
Kh_n\left(\,\,{\xy
(-5,10);(5,10)**\crv{(-5,2)&(0,-6)&(5,2)}
?>(0)*\dir{<},
?>(.71)*{\color{white}\bullet},
?>(.30)*{\color{white}\bullet}
,(-5,-10);(5,-10)**\crv{(-5,-2)&(0,6)&(5,-2)}
?>(1)*\dir{>}
\endxy}\,\,\right)
&=
[
(\mathbb{Q}\{1-n\}\oplus[n-2]){\xy
(-5,-10);(-5,10)**\dir{-}?>(.5)*\dir{>}
,(-7,-11)*{\scriptstyle{1}}
,(-7,11)*{\scriptstyle{1}}
,(5,-10);(5,10)**\dir{-}?>(.5)*\dir{<}
,(7,-11)*{\scriptstyle{1}}
,(7,11)*{\scriptstyle{1}}
\endxy}
\xrightarrow{}\\
&\xrightarrow{\left(
\begin{smallmatrix}
0 & 0\\
1 & 0\\
0 &1 \\
0 & 1\\
0 & 0
\end{smallmatrix}\right)}
{\xy
(-5,-10);(5,-10)**\crv{(-5,-2)&(0,-2)&(5,-2)}
?>(1)*\dir{>},
(-6,-10)*{\scriptstyle{1}},(6,-10)*{\scriptstyle{1}}
,(-5,10);(5,10)**\crv{(-5,2)&(0,2)&(5,2)}
?>(0)*\dir{<},
(-6,10)*{\scriptstyle{1}},(6,10)*{\scriptstyle{1}}
\endxy}
\oplus\,  (\mathbb{Q}\{1-n\}\oplus[n-2]^{\oplus 2}\oplus \mathbb{Q}\{n-1\})
{\xy
(-5,-10);(-5,10)**\dir{-}?>(.5)*\dir{>}
,(5,-10);(5,10)**\dir{-}?>(.5)*\dir{<}
,(-7,-10)*{\scriptstyle{1}}
,(7,-10)*{\scriptstyle{1}}
,(-7,10)*{\scriptstyle{1}}
,(7,10)*{\scriptstyle{1}}
\endxy}\\
&
\xrightarrow{\left(
\begin{smallmatrix}
0 & 0& -1 &1& 0\\
0 &  0 & 0 &0 &1\\
\end{smallmatrix}\right)}
([n-2]\oplus\mathbb{Q}\{n-1\}){\xy
(-5,-10);(-5,10)**\dir{-}?>(.5)*\dir{>}
,(-7,-11)*{\scriptstyle{1}}
,(-7,11)*{\scriptstyle{1}}
,(5,-10);(5,10)**\dir{-}?>(.5)*\dir{<}
,(7,-11)*{\scriptstyle{1}}
,(7,11)*{\scriptstyle{1}}
\endxy}]\\
&=
[
(\mathbb{Q}\{1-n\}\oplus[n-2]){\xy
(-5,-10);(-5,10)**\dir{-}?>(.5)*\dir{>}
,(-7,-11)*{\scriptstyle{1}}
,(-7,11)*{\scriptstyle{1}}
,(5,-10);(5,10)**\dir{-}?>(.5)*\dir{<}
,(7,-11)*{\scriptstyle{1}}
,(7,11)*{\scriptstyle{1}}
\endxy}
\xrightarrow{}\\
&\xrightarrow{\left(
\begin{smallmatrix}
1 & 0\\
0 &1 \\
0 & 1\\
0 & 0\\
0 & 0
\end{smallmatrix}\right)}
 (\mathbb{Q}\{1-n\}\oplus[n-2]^{\oplus 2})
{\xy
(-5,-10);(-5,10)**\dir{-}?>(.5)*\dir{>}
,(5,-10);(5,10)**\dir{-}?>(.5)*\dir{<}
,(-7,-10)*{\scriptstyle{1}}
,(7,-10)*{\scriptstyle{1}}
,(-7,10)*{\scriptstyle{1}}
,(7,10)*{\scriptstyle{1}}
\endxy}\,\oplus\,
{\xy
(-5,-10);(5,-10)**\crv{(-5,-2)&(0,-2)&(5,-2)}
?>(1)*\dir{>},
(-6,-10)*{\scriptstyle{1}},(6,-10)*{\scriptstyle{1}}
,(-5,10);(5,10)**\crv{(-5,2)&(0,2)&(5,2)}
?>(0)*\dir{<},
(-6,10)*{\scriptstyle{1}},(6,10)*{\scriptstyle{1}}
\endxy}\,\oplus\, \mathbb{Q}\{n-1\}{\xy
(-5,-10);(-5,10)**\dir{-}?>(.5)*\dir{>}
,(5,-10);(5,10)**\dir{-}?>(.5)*\dir{<}
,(-7,-10)*{\scriptstyle{1}}
,(7,-10)*{\scriptstyle{1}}
,(-7,10)*{\scriptstyle{1}}
,(7,10)*{\scriptstyle{1}}
\endxy}
\\
&
\xrightarrow{\left(
\begin{smallmatrix}
0 & 1& -1 &0& 0\\
0 &  0 & 0 &0 &1\\
\end{smallmatrix}\right)}
([n-2]\oplus\mathbb{Q}\{n-1\}){\xy
(-5,-10);(-5,10)**\dir{-}?>(.5)*\dir{>}
,(-7,-11)*{\scriptstyle{1}}
,(-7,11)*{\scriptstyle{1}}
,(5,-10);(5,10)**\dir{-}?>(.5)*\dir{<}
,(7,-11)*{\scriptstyle{1}}
,(7,11)*{\scriptstyle{1}}
\endxy}]
\end{align*}
By the equivalence relations (eq2) and (eq3) we therefore find
\begin{align*}
Kh_n\left(\,\,{\xy
(-5,10);(5,10)**\crv{(-5,2)&(0,-6)&(5,2)}
?>(0)*\dir{<},
?>(.71)*{\color{white}\bullet},
?>(.30)*{\color{white}\bullet}
,(-5,-10);(5,-10)**\crv{(-5,-2)&(0,6)&(5,-2)}
?>(1)*\dir{>}
\endxy}\,\,\right)
&=\bigl[
[n-2]){\xy
(-5,-10);(-5,10)**\dir{-}?>(.5)*\dir{>}
,(-7,-11)*{\scriptstyle{1}}
,(-7,11)*{\scriptstyle{1}}
,(5,-10);(5,10)**\dir{-}?>(.5)*\dir{<}
,(7,-11)*{\scriptstyle{1}}
,(7,11)*{\scriptstyle{1}}
\endxy}
\xrightarrow{}\\
&\xrightarrow{\left(
\begin{smallmatrix}
1 \\
 1\\
 0\\
 0
\end{smallmatrix}\right)}
( [n-2]^{\oplus 2})
{\xy
(-5,-10);(-5,10)**\dir{-}?>(.5)*\dir{>}
,(5,-10);(5,10)**\dir{-}?>(.5)*\dir{<}
,(-7,-10)*{\scriptstyle{1}}
,(7,-10)*{\scriptstyle{1}}
,(-7,10)*{\scriptstyle{1}}
,(7,10)*{\scriptstyle{1}}
\endxy}\,\oplus\,
{\xy
(-5,-10);(5,-10)**\crv{(-5,-2)&(0,-2)&(5,-2)}
?>(1)*\dir{>},
(-6,-10)*{\scriptstyle{1}},(6,-10)*{\scriptstyle{1}}
,(-5,10);(5,10)**\crv{(-5,2)&(0,2)&(5,2)}
?>(0)*\dir{<},
(-6,10)*{\scriptstyle{1}},(6,10)*{\scriptstyle{1}}
\endxy}\,\oplus\, \mathbb{Q}\{n-1\}{\xy
(-5,-10);(-5,10)**\dir{-}?>(.5)*\dir{>}
,(5,-10);(5,10)**\dir{-}?>(.5)*\dir{<}
,(-7,-10)*{\scriptstyle{1}}
,(7,-10)*{\scriptstyle{1}}
,(-7,10)*{\scriptstyle{1}}
,(7,10)*{\scriptstyle{1}}
\endxy}
\\
&
\xrightarrow{\left(
\begin{smallmatrix}
 1& -1 &0& 0\\
  0 & 0 &0 &1\\
\end{smallmatrix}\right)}
([n-2]\oplus\mathbb{Q}\{n-1\}){\xy
(-5,-10);(-5,10)**\dir{-}?>(.5)*\dir{>}
,(-7,-11)*{\scriptstyle{1}}
,(-7,11)*{\scriptstyle{1}}
,(5,-10);(5,10)**\dir{-}?>(.5)*\dir{<}
,(7,-11)*{\scriptstyle{1}}
,(7,11)*{\scriptstyle{1}}
\endxy}\bigr]\\
&=\bigl[
[n-2]{\xy
(-5,-10);(-5,10)**\dir{-}?>(.5)*\dir{>}
,(-7,-11)*{\scriptstyle{1}}
,(-7,11)*{\scriptstyle{1}}
,(5,-10);(5,10)**\dir{-}?>(.5)*\dir{<}
,(7,-11)*{\scriptstyle{1}}
,(7,11)*{\scriptstyle{1}}
\endxy}
\xrightarrow{\left(
\begin{smallmatrix}
1 \\
 1\\
 0
 \end{smallmatrix}\right)}
( [n-2]^{\oplus 2})
{\xy
(-5,-10);(-5,10)**\dir{-}?>(.5)*\dir{>}
,(5,-10);(5,10)**\dir{-}?>(.5)*\dir{<}
,(-7,-10)*{\scriptstyle{1}}
,(7,-10)*{\scriptstyle{1}}
,(-7,10)*{\scriptstyle{1}}
,(7,10)*{\scriptstyle{1}}
\endxy}\,\oplus\,
{\xy
(-5,-10);(5,-10)**\crv{(-5,-2)&(0,-2)&(5,-2)}
?>(1)*\dir{>},
(-6,-10)*{\scriptstyle{1}},(6,-10)*{\scriptstyle{1}}
,(-5,10);(5,10)**\crv{(-5,2)&(0,2)&(5,2)}
?>(0)*\dir{<},
(-6,10)*{\scriptstyle{1}},(6,10)*{\scriptstyle{1}}
\endxy}
\xrightarrow{\left(
\begin{smallmatrix}
 1& -1 &0
\end{smallmatrix}\right)}
[n-2]{\xy
(-5,-10);(-5,10)**\dir{-}?>(.5)*\dir{>}
,(-7,-11)*{\scriptstyle{1}}
,(-7,11)*{\scriptstyle{1}}
,(5,-10);(5,10)**\dir{-}?>(.5)*\dir{<}
,(7,-11)*{\scriptstyle{1}}
,(7,11)*{\scriptstyle{1}}
\endxy}\bigr]
\end{align*}
As we are working over $\mathbb{Q}$, we have a commutative diagram
\[
\xymatrix{
[n-2]{\xy
(-5,-10);(-5,10)**\dir{-}?>(.5)*\dir{>}
,(-7,-11)*{\scriptstyle{1}}
,(-7,11)*{\scriptstyle{1}}
,(5,-10);(5,10)**\dir{-}?>(.5)*\dir{<}
,(7,-11)*{\scriptstyle{1}}
,(7,11)*{\scriptstyle{1}}
\endxy}
\ar@{=}[d]
\ar[r]^-{\left(\begin{smallmatrix}
1 \\
 1\\
 0
 \end{smallmatrix}\right)}
 &
( [n-2]^{\oplus 2})
{\xy
(-5,-10);(-5,10)**\dir{-}?>(.5)*\dir{>}
,(5,-10);(5,10)**\dir{-}?>(.5)*\dir{<}
,(-7,-10)*{\scriptstyle{1}}
,(7,-10)*{\scriptstyle{1}}
,(-7,10)*{\scriptstyle{1}}
,(7,10)*{\scriptstyle{1}}
\endxy}\,\oplus\,
{\xy
(-5,-10);(5,-10)**\crv{(-5,-2)&(0,-2)&(5,-2)}
?>(1)*\dir{>},
(-6,-10)*{\scriptstyle{1}},(6,-10)*{\scriptstyle{1}}
,(-5,10);(5,10)**\crv{(-5,2)&(0,2)&(5,2)}
?>(0)*\dir{<},
(-6,10)*{\scriptstyle{1}},(6,10)*{\scriptstyle{1}}
\endxy}
 \ar[d]_{\wr}^-{\left(\begin{smallmatrix}
1/2 & 1/2& 0\\
 1/2 & -1/2 &0\\
 0 & 0 & 1
 \end{smallmatrix}\right)}
\ar[rr]^-{\left(
\begin{smallmatrix}
 1& -1 &0
\end{smallmatrix}\right)}
&&
[n-2]{\xy
(-5,-10);(-5,10)**\dir{-}?>(.5)*\dir{>}
,(-7,-11)*{\scriptstyle{1}}
,(-7,11)*{\scriptstyle{1}}
,(5,-10);(5,10)**\dir{-}?>(.5)*\dir{<}
,(7,-11)*{\scriptstyle{1}}
,(7,11)*{\scriptstyle{1}}
\endxy}
\ar[d]^{1/2}_{\wr}\\
[n-2]{\xy
(-5,-10);(-5,10)**\dir{-}?>(.5)*\dir{>}
,(-7,-11)*{\scriptstyle{1}}
,(-7,11)*{\scriptstyle{1}}
,(5,-10);(5,10)**\dir{-}?>(.5)*\dir{<}
,(7,-11)*{\scriptstyle{1}}
,(7,11)*{\scriptstyle{1}}
\endxy}
\ar[r]^-{\left(\begin{smallmatrix}
1 \\
 0\\
 0
 \end{smallmatrix}\right)}
 &
( [n-2]^{\oplus 2})
{\xy
(-5,-10);(-5,10)**\dir{-}?>(.5)*\dir{>}
,(5,-10);(5,10)**\dir{-}?>(.5)*\dir{<}
,(-7,-10)*{\scriptstyle{1}}
,(7,-10)*{\scriptstyle{1}}
,(-7,10)*{\scriptstyle{1}}
,(7,10)*{\scriptstyle{1}}
\endxy}\,\oplus\,
{\xy
(-5,-10);(5,-10)**\crv{(-5,-2)&(0,-2)&(5,-2)}
?>(1)*\dir{>},
(-6,-10)*{\scriptstyle{1}},(6,-10)*{\scriptstyle{1}}
,(-5,10);(5,10)**\crv{(-5,2)&(0,2)&(5,2)}
?>(0)*\dir{<},
(-6,10)*{\scriptstyle{1}},(6,10)*{\scriptstyle{1}}
\endxy}
\ar[rr]^-{\left(
\begin{smallmatrix}
0& 1 &0
\end{smallmatrix}\right)}
&&
[n-2]{\xy
(-5,-10);(-5,10)**\dir{-}?>(.5)*\dir{>}
,(-7,-11)*{\scriptstyle{1}}
,(-7,11)*{\scriptstyle{1}}
,(5,-10);(5,10)**\dir{-}?>(.5)*\dir{<}
,(7,-11)*{\scriptstyle{1}}
,(7,11)*{\scriptstyle{1}}
\endxy}
}
\]
where the vertical arrows are isomorphisms. So, by the equivalence relations (eq1) and (eq2), 
\begin{align*}
Kh_n\left(\,\,{\xy
(-5,10);(5,10)**\crv{(-5,2)&(0,-6)&(5,2)}
?>(0)*\dir{<},
?>(.71)*{\color{white}\bullet},
?>(.30)*{\color{white}\bullet}
,(-5,-10);(5,-10)**\crv{(-5,-2)&(0,6)&(5,-2)}
?>(1)*\dir{>}
\endxy}\,\,\right)
&=\bigl[
0\to 
[n-2]
{\xy
(-5,-10);(-5,10)**\dir{-}?>(.5)*\dir{>}
,(5,-10);(5,10)**\dir{-}?>(.5)*\dir{<}
,(-7,-10)*{\scriptstyle{1}}
,(7,-10)*{\scriptstyle{1}}
,(-7,10)*{\scriptstyle{1}}
,(7,10)*{\scriptstyle{1}}
\endxy}\,\oplus\,
{\xy
(-5,-10);(5,-10)**\crv{(-5,-2)&(0,-2)&(5,-2)}
?>(1)*\dir{>},
(-6,-10)*{\scriptstyle{1}},(6,-10)*{\scriptstyle{1}}
,(-5,10);(5,10)**\crv{(-5,2)&(0,2)&(5,2)}
?>(0)*\dir{<},
(-6,10)*{\scriptstyle{1}},(6,10)*{\scriptstyle{1}}
\endxy}
\xrightarrow{\left(
\begin{smallmatrix}
1 &0
\end{smallmatrix}\right)}
[n-2]{\xy
(-5,-10);(-5,10)**\dir{-}?>(.5)*\dir{>}
,(-7,-11)*{\scriptstyle{1}}
,(-7,11)*{\scriptstyle{1}}
,(5,-10);(5,10)**\dir{-}?>(.5)*\dir{<}
,(7,-11)*{\scriptstyle{1}}
,(7,11)*{\scriptstyle{1}}
\endxy}
\bigr]
\\
&=\bigl[
0\to 
{\xy
(-5,-10);(5,-10)**\crv{(-5,-2)&(0,-2)&(5,-2)}
?>(1)*\dir{>},
(-6,-10)*{\scriptstyle{1}},(6,-10)*{\scriptstyle{1}}
,(-5,10);(5,10)**\crv{(-5,2)&(0,2)&(5,2)}
?>(0)*\dir{<},
(-6,10)*{\scriptstyle{1}},(6,10)*{\scriptstyle{1}}
\endxy}
\to
0
\bigr] =
Kh_n
\left(\,\,{\xy
(-5,-10);(5,-10)**\crv{(-5,-2)&(0,-2)&(5,-2)}
?>(1)*\dir{>}
,(-5,10);(5,10)**\crv{(-5,2)&(0,2)&(5,2)}
?>(0)*\dir{<}
\endxy}\,\,\right)
\end{align*}

\subsection{Invariance by Reidemeister 3}

We want to prove that
\[
Kh_n\left(\,\,
\xy
(-9,-7);(9,7)**\crv{(-9,-5)&(-9,-3)&(9,3)&(9,5)}?>(.99)*\dir{>}
,(9,-7);(1,-.3)**\crv{(9,-5)&(9,-3)}
,(-1,.3);(-5.3,2)**\crv{(-3,1)}
,(-6.7,2.8);(-9,7)**\crv{(-9,4)&(-9,5)}?>(.99)*\dir{>}
,(0,-7);(-5.3,-2.8)**\crv{(0,-5)&(0,-4)}
,(0,7);(-6.1,2.4)**\crv{(0,6)&(0,5)&(0,4)&(-6,2.8)}?>(0)*\dir{<}
,(-6.5,-2.2);(-6.1,2.4)**\crv{(-9,0)}
\endxy
\,\,\right)
=
Kh_n\left(\,\,
\xy
(-9,-7);(9,7)**\crv{(-9,-5)&(-9,-3)&(9,3)&(9,5)}?>(.99)*\dir{>}
,(5.6,-2);(1,-.3)**\crv{(3,-1)}
,(9,-7);(6.7,-2.8)**\crv{(9,-5)&(9,-4)}
,(-1,.3);(-9,7)**\crv{(-9,3)&(-9,5)}?>(.99)*\dir{>}
,(0,-7);(6,-2.6)**\crv{(0,-5)&(0,-4)&(6,-3)}
,(0,7);(5.3,2.8)**\crv{(0,6)&(0,5)&(0,4)}?>(0)*\dir{<}
,(6,-2.6);(6.5,2.2)**\crv{(9,0)}
\endxy
\,\,\right)
\]
The proof is not difficult, but a bit long.  
By definition of $Kh_n$ we have
\[
Kh_n\left(\,\,
\xy
(-9,-7);(9,7)**\crv{(-9,-5)&(-9,-3)&(9,3)&(9,5)}?>(.99)*\dir{>}
,(9,-7);(1,-.3)**\crv{(9,-5)&(9,-3)}
,(-1,.3);(-5.3,2)**\crv{(-3,1)}
,(-6.7,2.8);(-9,7)**\crv{(-9,4)&(-9,5)}?>(.99)*\dir{>}
,(0,-7);(-5.3,-2.8)**\crv{(0,-5)&(0,-4)}
,(0,7);(-6.1,2.4)**\crv{(0,6)&(0,5)&(0,4)&(-6,2.8)}?>(0)*\dir{<}
,(-6.5,-2.2);(-6.1,2.4)**\crv{(-9,0)}
\endxy
\,\,\right)
=\bigl[\mathrm{tot}
\left(\raisebox{80pt}{
\scalebox{.7}{
\xymatrix{
&\esse\,\,\idoner\,\{-2\}\ar[r]\ar[rd]&\zorro\{-1\}\ar[rd] \\
\idthree\{-3\} \ar[r]\ar[ru]\ar[rd]&\idone\,\,\esse\,\{-2\}\ar[ru]|(.5)\hole\ar[rd]
&\essecircesse\,\longidoner\,\{-1\}\ar[r]&
\coso\\
&\esse\,\,\idoner\,\{-2\}\ar[r]\ar[ru]|(.5)\hole
&\zorrob\,\{-1\}\ar[ru]
}}}
\right)
\bigr]\{3n\}
\]
where all the arrows are $\chi_0$'s. For instance, the leftmost arrows are, from top to bottom, $(\chi_0\circ1\circ 1)\otimes 1$, $1\otimes (1\circ\chi_0\circ 1)$, and $(1\circ1\circ\chi_0)\otimes 1$. However, to avoid cumbersome notation in what follows we will write them simply as $\chi_0\otimes 1$ , $1\otimes \chi_0$, and $\chi_0\otimes 1$. This brings in some ambiguity, as the first $\chi_0\otimes 1$ is not the same as the second $\chi_0\otimes 1$ in this list (the morphism $\chi_0$ ``happens'' at different positions in the graph), but we were forced into preferring readability over rigorous notation here. We hope this will not bring too much confusion in the proof that follows. 
The term on the right is
\[
\left[
\idthree\{-3\}
\xrightarrow{\left(
\begin{smallmatrix}
\chi_0\otimes 1\\ 1\otimes \chi_0 \\
\chi_0 \otimes 1
\end{smallmatrix}\right)}
\esse\,\,\idoner\{-2\}\oplus \idone\,\,\esse\{-2\}\oplus \esse\,\,\idoner\{-2\}\right.
\]
\[
\left.
\xrightarrow{\left(
\begin{smallmatrix}
-1\otimes \chi_0 &\chi_0\otimes 1 &0 \\ 
-\chi_0\otimes 1 &0 &\chi_0\otimes1 \\
0 &\chi_0\otimes 1 &-1\otimes \chi_0
\end{smallmatrix}\right)}
\zorro \{-1\}\oplus \essecircesse\,\longidoner\{-1\}\oplus \zorrob \{-1\}
\xrightarrow{\left(\begin{smallmatrix}
\chi_0\otimes 1 &-1\otimes \chi_0 &-\chi_0\otimes 1\end{smallmatrix}\right)}\coso
\right]
\]
shifted by $3n$.
By (\ref{eq:moyz}) this is
\[
\left[
\idthree\{-3\}
\xrightarrow{\left(
\begin{smallmatrix}
\chi_0\otimes 1\\ 1\otimes \chi_0 \\
\chi_0 \otimes 1
\end{smallmatrix}\right)}
\esse\,\,\idoner\{-2\}\oplus \idone\,\,\esse\{-2\}\oplus \esse\,\,\idoner\{-2\}\right.
\]
\[
\left.
\xrightarrow{\left(
\begin{smallmatrix}
0 &0  &0 \\ 
-\chi_0 &0 &\chi_0\otimes 1\\
0 &\chi_0\otimes 1 &-1\otimes\chi_0
\end{smallmatrix}\right)}
\zorro \{-1\}\oplus \essecircesse\,\longidoner\{-1\}\oplus \zorrob \{-1\}
\xrightarrow{\left(\begin{smallmatrix}
0 &-1\otimes \chi_0 &-\chi_0\otimes 1\end{smallmatrix}\right)}\coso
\right]
\]
shifted by $3n$. By the equivalence relation (eq1) this is in turn equal to
\[
\left[
\idthree\{-3\}
\xrightarrow{\left(
\begin{smallmatrix}
\chi_0\otimes 1\\ 1\otimes \chi_0 \\
\chi_0 \otimes 1
\end{smallmatrix}\right)}
\esse\,\,\idoner\{-2\}\oplus \idone\,\,\esse\{-2\}\oplus \esse\,\,\idoner\{-2\}\right.
\]
\[
\left.
\xrightarrow{\left(
\begin{smallmatrix}
\chi_0\otimes 1 &0 &-\chi_0\otimes 1\\
0 &0  &0 \\ 
0 &\chi_0\otimes 1 &-1\otimes\chi_0
\end{smallmatrix}\right)}
 \essecircesse\,\longidoner\{-1\}\oplus \zorro \{-1\}\oplus \zorrob \{-1\}
\xrightarrow{\left(\begin{smallmatrix}
1\otimes \chi_0&0 &-\chi_0\otimes 1\end{smallmatrix}\right)}\coso
\right]\{3n\}
\]
By (\ref{eq:moy2a}), (\ref{eq:moy2b}) and (\ref{moyeta}), this equals
\[
\left[
\idthree\{-3\}
\xrightarrow{\left(
\begin{smallmatrix}
\chi_0\otimes 1\\ 1\otimes \chi_0 \\
\chi_0 \otimes 1
\end{smallmatrix}\right)}
\esse\,\,\idoner\{-2\}\oplus \idone\,\,\esse\{-2\}\oplus \esse\,\,\idoner\{-2\}\right.
\]
\[
\xrightarrow{\left(
\begin{smallmatrix}
1 &0 &-1\\
0 &0 &-\alpha \\ 
0 &0  &0 \\ 
0 &\chi_0\otimes 1 &-1\otimes\chi_0
\end{smallmatrix}\right)}
 \esse\,\,\idoner\{-2\}\oplus \esse\,\,\idoner \oplus \zorro \{-1\}\oplus \zorrob \{-1\}
\]
\[
\left.
\xrightarrow{\left(\begin{smallmatrix}
0 & 1 & \chi_1 & \chi_1\\ 
0 & 0 & 0 & 0
\end{smallmatrix}\right)}\esse\,\, \idoner\,\,\oplus \ti
\right]\{3n\}
\]

\[
=\left[
\idthree\{-3\}
\xrightarrow{\left(
\begin{smallmatrix}
1\otimes \chi_0 \\
\chi_0 \otimes 1
\end{smallmatrix}\right)}
\idone\,\,\esse\{-2\}\oplus \esse\,\,\idoner\{-2\}\right.
\]
\[
\xrightarrow{\left(
\begin{smallmatrix}
0 &-\alpha \\ 
0  &0 \\ 
\chi_0\otimes 1 &-1\otimes\chi_0
\end{smallmatrix}\right)}
 \esse\,\,\idoner \oplus \zorro \{-1\}\oplus \zorrob \{-1\}
\]
\[
\left.
\xrightarrow{\left(\begin{smallmatrix}
 1 & \chi_1 & \chi_1\\ 
 0 & 0 & 0
\end{smallmatrix}\right)}\esse\,\, \idoner\,\,\oplus \ti
\right]\{3n\}
\]

\[
=\left[
\idthree\{-3\}
\xrightarrow{\left(
\begin{smallmatrix}
1\otimes \chi_0 \\
\chi_0 \otimes 1
\end{smallmatrix}\right)}
\idone\,\,\esse\{-2\}\oplus \esse\,\,\idoner\{-2\}\right.
\]
\[
\left.
\xrightarrow{\left(
\begin{smallmatrix}
0  &0 \\ 
\chi_0\otimes 1 &-1\otimes\chi_0
\end{smallmatrix}\right)}
 \zorro \{-1\}\oplus \zorrob \{-1\}
\xrightarrow{\left(\begin{smallmatrix}
 0 & 0
\end{smallmatrix}\right)}
 \ti
\right]\{3n\}
\]
Now let us turn to the computation of $Kh_n\left(\,\,
\xy
(-9,-7);(9,7)**\crv{(-9,-5)&(-9,-3)&(9,3)&(9,5)}?>(.99)*\dir{>}
,(5.6,-2);(1,-.3)**\crv{(3,-1)}
,(9,-7);(6.7,-2.8)**\crv{(9,-5)&(9,-4)}
,(-1,.3);(-9,7)**\crv{(-9,3)&(-9,5)}?>(.99)*\dir{>}
,(0,-7);(6,-2.6)**\crv{(0,-5)&(0,-4)&(6,-3)}
,(0,7);(5.3,2.8)**\crv{(0,6)&(0,5)&(0,4)}?>(0)*\dir{<}
,(6,-2.6);(6.5,2.2)**\crv{(9,0)}
\endxy
\,\,\right)$ which, as it is to be expected, is completely analogous. We have
\[
Kh_n\left(\,\,
\xy
(-9,-7);(9,7)**\crv{(-9,-5)&(-9,-3)&(9,3)&(9,5)}?>(.99)*\dir{>}
,(5.6,-2);(1,-.3)**\crv{(3,-1)}
,(9,-7);(6.7,-2.8)**\crv{(9,-5)&(9,-4)}
,(-1,.3);(-9,7)**\crv{(-9,3)&(-9,5)}?>(.99)*\dir{>}
,(0,-7);(6,-2.6)**\crv{(0,-5)&(0,-4)&(6,-3)}
,(0,7);(5.3,2.8)**\crv{(0,6)&(0,5)&(0,4)}?>(0)*\dir{<}
,(6,-2.6);(6.5,2.2)**\crv{(9,0)}
\endxy
\,\,\right)
=\bigl[\mathrm{tot}
\left(\raisebox{80pt}{
\scalebox{.7}{
\xymatrix{
&\idone\,\,\esse\,\{-2\}\ar[r]\ar[rd]&\zorrob\{-1\}\ar[rd] \\
\idthree\{-3\} \ar[r]\ar[ru]\ar[rd]&\esse\,\,\idoner\,\{-2\}\ar[ru]|(.5)\hole\ar[rd]
&\,\,\longidoner\,\,\essecircesse\,\{-1\}\ar[r]&
\cosob\\
&\idone\,\,\esse\,\{-2\}\ar[r]\ar[ru]|(.5)\hole
&\zorro\,\{-1\}\ar[ru]
}}}
\right)
\bigr]\{3n\}
\]
\[
=\left[
\idthree\{-3\}
\xrightarrow{\left(
\begin{smallmatrix}
 1\otimes \chi_0 \\
\chi_0 \otimes 1\\
 1\otimes \chi_0
\end{smallmatrix}\right)}
\idone\,\,\esse\{-2\}\oplus \esse\,\,\idoner\{-2\}\oplus \idone\,\,\esse\{-2\}\right.
\]
\[
\left.
\xrightarrow{\left(
\begin{smallmatrix}
-\chi_0\otimes 1 &1\otimes \chi_0 &0 \\ 
-1\otimes \chi_0 &0 &1\otimes \chi_0\\
0 &1\otimes \chi_0&-\chi_0\otimes 1
\end{smallmatrix}\right)}
\zorrob \{-1\}\oplus \longidoner\,\essecircesse\,\{-1\}\oplus \zorro \{-1\}
\xrightarrow{\left(\begin{smallmatrix}
1\otimes \chi_0 &-\chi_0\otimes 1&-1\otimes \chi_0\end{smallmatrix}\right)}\cosob
\right]\{3n\}
\]
\[
=\left[
\idthree\{-3\}
\xrightarrow{\left(
\begin{smallmatrix}
1\otimes \chi_0\\
\chi_0\otimes 1\\ 1\otimes \chi_0 \\
\end{smallmatrix}\right)}
 \idone\,\,\esse\{-2\}\oplus \esse\,\,\idoner\{-2\}\oplus  \idone\,\,\esse\{-2\}\right.
\]
\[
\left.
\xrightarrow{\left(
\begin{smallmatrix}
1\otimes \chi_0 &0 &-1\otimes \chi_0\\
-\chi_0\otimes 1 &1\otimes \chi_0 &0 \\ 
0 &1\otimes \chi_0&-\chi_0\otimes 1
\end{smallmatrix}\right)}
\longidoner\,  \essecircesse\,\{-1\}\oplus \zorrob \{-1\}\oplus \zorro \{-1\}
\xrightarrow{\left(\begin{smallmatrix}
 \chi_0\otimes 1 & -1\otimes \chi_0&1\otimes\chi_0\end{smallmatrix}\right)}\cosob
\right]\{3n\}
\]
\[
=\left[
\idthree\{-3\}
\xrightarrow{\left(
\begin{smallmatrix}
1\otimes \chi_0\\
\chi_0\otimes 1\\ 1\otimes \chi_0 
\end{smallmatrix}\right)}
\idone\,\,\esse\{-2\}\oplus \esse\,\,\idoner\{-2\}\oplus \idone\,\,\esse\{-2\}\right.
\]
\[
\xrightarrow{\left(
\begin{smallmatrix}
1 &0 &-1\\
0 &0& -\alpha  \\ 
0&0 &0   \\ 
0&-1\otimes\chi_0 & \chi_0\otimes 1 
\end{smallmatrix}\right)}
\idoner\,\, \esse\{-2\}\oplus \idoner\,\,\esse \oplus \zorrob \{-1\}\oplus \zorro \{-1\}
\]
\[
\left.
\xrightarrow{\left(\begin{smallmatrix}
0 & 1 & \chi_1 & \chi_1\\ 
0 & 0 & 0 & 0
\end{smallmatrix}\right)}\idoner\,\, \esse\,\,\oplus \ti
\right]\{3n\}
\]

\[
=\left[
\idthree\{-3\}
\xrightarrow{\left(
\begin{smallmatrix}
\chi_0 \otimes 1\\
1\otimes \chi_0 
\end{smallmatrix}\right)}
 \esse\,\,\idoner\{-2\}\oplus \idone\,\,\esse\{-2\}\right.
\]
\[
\xrightarrow{\left(
\begin{smallmatrix}
0&-\alpha  \\ 
0 &0   \\ 
-1\otimes\chi_0 & \chi_0\otimes 1 
\end{smallmatrix}\right)}
\idoner \,\, \esse\oplus \zorrob \{-1\}\oplus \zorro \{-1\}
\]
\[
\left.
\xrightarrow{\left(\begin{smallmatrix}
 1 & \chi_1 & \chi_1\\ 
 0 & 0 & 0
\end{smallmatrix}\right)} \idoner\,\,\esse\,\,\oplus \ti
\right]\{3n\}
\]

\[
=\left[
\idthree\{-3\}
\xrightarrow{\left(
\begin{smallmatrix}
\chi_0 \otimes 1\\
1\otimes \chi_0 
\end{smallmatrix}\right)}
 \esse\,\,\idoner\{-2\}\oplus \idone\,\,\esse\{-2\}\right.
\]
\[
\left.
\xrightarrow{\left(
\begin{smallmatrix}
0 &0   \\ 
-1\otimes\chi_0 & \chi_0\otimes 1 
\end{smallmatrix}\right)}
 \zorrob \{-1\}\oplus \zorro \{-1\}
\xrightarrow{\left(\begin{smallmatrix}
0 & 0
\end{smallmatrix}\right)}
 \ti
\right]\{3n\}
\]
\[
=\left[
\idthree\{-3\}
\xrightarrow{\left(
\begin{smallmatrix}
1\otimes \chi_0 \\
\chi_0 \otimes 1
\end{smallmatrix}\right)}
 \idone\,\,\esse\{-2\}\oplus  \esse\,\,\idoner\{-2\}\right.
\]
\[
\left.
\xrightarrow{\left(
\begin{smallmatrix}
0 &0 \\ 
\chi_0\otimes 1 &-1\otimes\chi_0
\end{smallmatrix}\right)}
 \zorro \{-1\}\oplus \zorrob \{-1\}
\xrightarrow{\left(\begin{smallmatrix}
0 & 0
\end{smallmatrix}\right)}
 \ti
\right]\{3n\}
\]
\[
=Kh_n\left(\,\,
\xy
(-9,-7);(9,7)**\crv{(-9,-5)&(-9,-3)&(9,3)&(9,5)}?>(.99)*\dir{>}
,(9,-7);(1,-.3)**\crv{(9,-5)&(9,-3)}
,(-1,.3);(-5.3,2)**\crv{(-3,1)}
,(-6.7,2.8);(-9,7)**\crv{(-9,4)&(-9,5)}?>(.99)*\dir{>}
,(0,-7);(-5.3,-2.8)**\crv{(0,-5)&(0,-4)}
,(0,7);(-6.1,2.4)**\crv{(0,6)&(0,5)&(0,4)&(-6,2.8)}?>(0)*\dir{<}
,(-6.5,-2.2);(-6.1,2.4)**\crv{(-9,0)}
\endxy
\,\,\right)
\]

\end{proof}

\section{From $\mathbf{KR}$ to $\mathbf{Jones}_n$}
The category $\mathbf{KR}$ is a refinement of $\mathbf{Jones}_n$. More precisely, there is a natural braided monoidal functor 
\[
\mathcal{P}\colon \mathbf{KR}\to \mathbf{Jones}_n
\]
factoring the universal morphism $Z\colon \mathbf{TD}\to \mathbf{Jones}_n$. The main difficulty in defining $\mathcal{P}$ resides in the fact that among the graphs appearing in $\mathbf{KR}$ we have the additional graph
\[
T=\ti
\]
not appearing in $\mathbf{MOY}\cong \mathbf{Jones}_n$. To overcome this problem we introduce an auxiliary braided monoidal category $\mathbf{MOY}^{(T)}$.
\begin{definition}
The category $\mathbf{TrPD}^{(T)}$ of ($\mathbb{Q}[q,q^{-1}]$-linear combinations of) oriented trivalent planar diagrams enhanced with $T$ is the rigid monoidal category enriched in $\mathbb{Q}[q,q^{-1}]$-modules defines exactly as  $\mathbf{TrPD}$, but where now the graphs are allowed to contain subgraphs of the form $T$ (possibly with the reversed orientation).
 \end{definition} 
\begin{definition}
For a fixed nonnegative integer $n\geq 2$, let $\mathbf{I}_{MOY}^{(T)}$ be the tensor ideal in $\mathbf{TrPD}^{(T)}$ generated by the following relations (and their duals):
\begin{enumerate}
\item $\mathrm{ev}\circ \mathrm{coev} = [n]_q$,
i.e.,
\begin{equation}\label{moy-1t}\tag{moy-1T}
\iunknot=[n]_q\, \emptyset
\end{equation}
\item $S\circ S=(q+q^{-1})S$,
i.e.,
\begin{equation}\label{moy-2t}\tag{moy-2T}
\essecircesse=(q+q^{-1})\esse
\end{equation}
\item $(\mathrm{ev}\otimes \mathrm{id})\circ (\mathrm{id}^\vee \otimes S)\circ (\mathrm{coev}\otimes \mathrm{id})=[n-1]\mathrm{id}$, 
i.e.,
\begin{equation}\label{moy-3t}\tag{moy-3T}
{\xy
(0,-10);(0,10)**\crv{(0,-7)&(0,-6)&(0,-5)&(0,0)&(0,5)&(0,6)&(0,7)}
?>(.2)*\dir{>}?>(.99)*\dir{>}?>(.55)*\dir{>}
,(0,-3);(0,3)**\crv{(0,-4)&(0,-5)&(-4,-6)&(-6,0)&(-4,6)&(0,5)&(0,4)}
?>(.55)*\dir{<}
,(2,-11)*{\scriptstyle{1}}
,(2,11)*{\scriptstyle{1}}
,(2,0)*{\scriptstyle{2}}
,(-6,0)*{\scriptstyle{1}}
\endxy}\,\,\,=
[ n-1]_q\,\,{\xy
(0,-10);(0,10)**\dir{-}
?>(.95)*\dir{>}
,(2,-11)*{\scriptstyle{1}}
,(2,11)*{\scriptstyle{1}}
\endxy}
\end{equation}

\item $(\mathrm{id}\otimes \mathrm{ev}\otimes\mathrm{id}^\vee)\circ (S\otimes S^\vee) \circ (\mathrm{id}\otimes \mathrm{coev}\otimes\mathrm{id}^\vee)=\mathrm{coev}\circ\mathrm{ev}+ [n-2]_q\, \mathrm{id}\otimes \mathrm{id}^\vee$,
i.e.,
\begin{equation}\label{moy-4t}\tag{moy-4T}
{\xy
(-10,-10);(-10,10)**\crv{(-6,-6)&(-5,-5)&(-5,-2)&(-5,-2)&(-5,5)&(-6,6)}
?>(.05)*\dir{>}?>(.99)*\dir{>}?>(.55)*\dir{>}
,(10,-10);(10,10)**\crv{(6,-6)&(5,-5)&(5,-2)&(5,2)&(5,5)&(6,6)}
?>(.03)*\dir{<}?>(.92)*\dir{<}?>(.45)*\dir{<}
,(-5,0);(5,0)**\crv{(-5,-2)&(-5,-3)&(-4,-5)&(-2,-6)&(2,-6)&(4,-5)&(5,-3)&(5,-2)}?>(.5)*\dir{<}
,(-5,0);(5,0)**\crv{(-5,2)&(-5,3)&(-4,5)&(-2,6)&(2,6)&(4,5)&(5,3)&(5,2)}?>(.5)*\dir{>}
,(-11,-10)*{\scriptstyle{1}}
,(-11,10)*{\scriptstyle{1}}
,(11,-10)*{\scriptstyle{1}}
,(11,10)*{\scriptstyle{1}}
,(-7,0)*{\scriptstyle{2}}
,(0,-8)*{\scriptstyle{1}}
,(6.5,0)*{\scriptstyle{2}}
,(0,8)*{\scriptstyle{1}}
\endxy}
\quad
=
\xy
(-5,-10);(5,-10)**\crv{(-5,-2)&(0,-2)&(5,-2)}
?>(1)*\dir{>},
(-6,-10)*{\scriptstyle{1}},(6,-10)*{\scriptstyle{1}}
,(-5,10);(5,10)**\crv{(-5,2)&(0,2)&(5,2)}
?>(0)*\dir{<},
(-6,10)*{\scriptstyle{1}},(6,10)*{\scriptstyle{1}}
\endxy
+\, [n-2]_q
{\xy
(-5,-10);(-5,10)**\dir{-}?>(.5)*\dir{>}
,(5,-10);(5,10)**\dir{-}?>(.5)*\dir{<}
,(-7,-10)*{\scriptstyle{1}}
,(7,-10)*{\scriptstyle{1}}
,(-7,10)*{\scriptstyle{1}}
,(7,10)*{\scriptstyle{1}}
\endxy}
\end{equation}

\item[5.a.] 
$(S\otimes\mathrm{id}_1)\circ (\mathrm{id}_1\otimes S) \circ (S\otimes\mathrm{id}_1)= (S\otimes \mathrm{id}_1) + T$,
i.e., 
\begin{equation}\label{moy-5ta}\tag{moy-5Ta}
\coso
\,=\,
\esse\,\,\idoner\qquad
+ \ti
\end{equation}
\end{enumerate}

\item[5.b.]
$(\mathrm{id}_1\otimes S) \circ (S\otimes\mathrm{id}_1)\circ (\mathrm{id}_1\otimes S) = (S\otimes \mathrm{id}_1) + T$,
i.e., 
\begin{equation}\label{moy-5tb}\tag{moy-5Tb}
\cosob
\,=\,
\,\idone\,\,\esse \qquad
+ \ti
\end{equation}

The category $\mathbf{MOY}^{(T)}$  the monoidal category enriched in $\mathbb{Q}[q,q^{-1}]$-modules defined by
\[
\mathbf{MOY}^{(T)}=\mathbf{TrPD^{(T)}}/{\mathbf{I}_{MOY}^{(T)}}.
\]
\end{definition}

\begin{proposition}
The inclusion $j\colon \mathbf{TrPD}\to \mathbf{TrPD^{(T)}}$ induces an isomorphism between $\mathbf{MOY}$ and $\mathbf{MOY}^{(T)}$. In particular 
$\mathbf{MOY}^{(T)}$ is an untwited braided category with braidings
\[
\sigma^+= q^{n-1}\left(\,\, \idtwo-q\esse\,\,\right),
\]
\[
\sigma^- = q^{1-n}\left(\,\, \idtwo-q^{-1}\esse\,\,\right).
\]
\end{proposition}
\begin{proof}
To see that $j$ induces a morphism between $\mathbf{MOY}$ and $\mathbf{MOY}^{(T)}$ we only need to show that $j$ maps the ideal $\mathbf{I}_{MOY}$ into the ideal $\mathbf{I}_{MOY}^{(T)}$. We can check this on the generators. This is trivial for relations (\ref{moy-1})-(\ref{moy-4}) as these exactly coincide with relations (\ref{moy-1t})-(\ref{moy-4t}). For relation (\ref{moy-5}) we have
\begin{align*}
\left(\coso\,\,+\,\,\idone\quad \esse\right) &- \left( \cosob\,\,+\,\, \esse\quad\idoner\right)\\
&\hskip -4 cm=\left(\coso\,\, - \esse\quad\idoner -\ti\right) -\left(\cosob\,\ -\,\,\idone\quad \esse -\ti\right)
\end{align*}
and so it is an element in $\mathbf{I}_{MOY}^{(T)}$. To show that $j$ is actually an isomorphism, we exibit the inverse functor $\pi\colon \mathbf{MOY}^{(T)}\to \mathbf{MOY}$. We start by defining $\pi\colon \mathbf{TrPD^{(T)}}\to \mathbf{TrPD}$ as the identity on the subcategory $\mathbf{TrPD}$ of $\mathbf{TrPD^{(T)}}$ and which on the additional generator $T$ acts as
\[
\pi\left(\ti\right)=\coso
\,-\,
\esse\,\,\idoner.
\]
The functor $\pi$ maps the ideal $\mathbf{I}_{MOY}^{(T)}$ into the ideal $\mathbf{I}_{MOY}$. This is trivial for the generators (\ref{moy-1t})-(\ref{moy-4t}), while for the generators (\ref{moy-5ta}) and (\ref{moy-5tb}) we have
\[
\pi\left(
\coso
\,-\,
\esse\,\,\idoner\qquad
- \ti\right)=\pi\left(
\coso
\,-\,
\esse\,\,\idoner\right)-\pi\left(
+ \ti\right)=0
\]
and
\begin{align*}
\pi\left(
\cosob
\,-\,
\,\idone\,\,\esse \qquad
- \ti\right)
&=
\pi\left(
\cosob
\,-\,
\,\idone\,\,\esse \right)
- \pi\left( \ti\right)\\
&\hskip-6 cm
=\left(\cosob
\,-\,
\,\idone\,\,\esse\right) - \left(
\coso
\,-\,
\esse\,\,\idoner
\right)\\
&\hskip-6cm
=
 \left( \cosob\,\,+\,\, \esse\quad\idoner\right) -\left(\coso\,\,+\,\,\idone\quad \esse\right)
 \end{align*}
which is in $\mathbf{I}_{MOY}$ as this is precisely (\ref{moy-5}). Checking that $j\colon \mathbf{MOY}\to \mathbf{MOY}^{(T)}$ and $\pi\colon \mathbf{MOY}^{(T)}\to \mathbf{MOY}$ are inverse each other is then straightforward.
\end{proof}
We can now provide a factorization of the universal morphism $Z\colon \mathbf{TD}\to \mathbf{Jones}_n$ through $\mathbf{KR}$ by using the isomorphism $\mathbf{Jones}_n\cong \mathbf{MOY}\cong \mathbf{MOY}^{(T)}$ and providing a braided monoidal functor 
\[
\mathcal{P}\colon \mathbf{KR}\to \mathbf{MOY}^{(T)}
\]
factoring $Z$. The strategy to define $\mathcal{P}$ will consist in first defining it on shifted MOY-type graphs, next in extending this definition to formal direct sums of these, and finally to formal complexes of (formal direct sums of shifted) MOY-type graphs.
\begin{definition}
The map
\[
\mathcal{P}_0\colon \{\text{Formal direct sums of shifted MOY-type graphs}\}\to  \mathbf{MOY}^{(T)}
\]
is defined by
\[
\mathcal{P}_0\left(\oplus_j \Gamma_j\{m_j\}\right) = \sum_{j}q^{m_j}\Gamma_j.
\]
The map 
\[
\mathcal{P}\colon \mathbf{Complexes}_{MOY}\to  \mathbf{MOY}^{(T)}
\]
is defined by
\[
\mathcal{P}\left( M_\bullet\right)= \sum_{i}(-1)^i \mathcal{P}_0(M_i).
\]
\end{definition}

\begin{theorem}
The map $\mathcal{P}$ induces a braided monoidal functor
\[
\mathcal{P}\colon \mathbf{KR}\to  \mathbf{MOY}^{(T)}
\]
factoring the Reshetikhin-Turaev morphism $Z\colon \mathbf{TD}\to \mathbf{MOY}^{(T)}$ defined by $Z(\uparrow)=\uparrow^1$.
\end{theorem}
\begin{proof}
To begin with, notice that, manifestly, when $f\colon \Gamma_1\xrightarrow{\sim} \Gamma_2$ is any of the generating isomorphisms $\lambda,\mu,\varphi,\psi,\nu,\eta$ from Definition \ref{def:generating}, then $\mathcal{P}_0(\Gamma_1)=\mathcal{P}_0(\Gamma_2)$ by the relations (\ref{moy-1t})-(\ref{moy-5ta}/\ref{moy-5tb}). This implies that if $f\colon M\xrightarrow{\sim} N$ is an isomorphism of formal direct sums of shifted MOY-type graphs, then $\mathcal{P}_0(M)=\mathcal{P}_0(N)$, and so if $f_\bullet\colon M_\bullet\xrightarrow{\sim} N_\bullet$ is an isomorphism of formal complexes of (formal direct sums of shifted) MOY-type graphs, then $\mathcal{P}(M_\bullet)=\mathcal{P}(N_\bullet)$. In other words, if two formal complexes $M_\bullet$ and $N_\bullet$ are equivalent with respect to the equivalence relation (eq1), then $\mathcal{P}(M_\bullet)=\mathcal{P}(N_\bullet)$.
\par
Next, let us consider the equivalence relation (eq2). We have 
\begin{align*}
\mathcal{P}&\left(
\cdots \to M_{i-1}\xrightarrow{\left(
\begin{matrix}
\alpha_{i-1}\\A_{i-1}
\end{matrix}
\right)}
N\oplus M_i\xrightarrow{\left(
\begin{matrix}
1&\alpha_i\\0&A_i
\end{matrix}
\right)}
N\oplus M_{i+1}\xrightarrow{\left(
\begin{matrix}
0&A_{i+1}
\end{matrix}
\right)}
 M_{i+2}\to\cdots\right)\\
 &=\sum_{k\leq i-1} (-1)^k \mathcal{P}_0(M_k)+(-1)^{i}\left((\mathcal{P}_0(M_i)+\mathcal{P}_0(N)\right)+(-1)^{i+1}\left((\mathcal{P}_0(M_i+1)+\mathcal{P}_0(N)\right)+\sum_{k\geq i+2} (-1)^k \mathcal{P}_0(M_k)\\
 &=\sum_{k} (-1)^k \mathcal{P}_0(M_k)\\
 &=\mathcal{P}\left(\cdots\to  M_{i-1}\ \xrightarrow{A_{i-1}}
M_i\xrightarrow{A_i} M_{i+1}\xrightarrow{A_{i+1}} M_{i+2}
\to \cdots\right)
\end{align*}
That is, if two formal complexes $M_\bullet$ and $N_\bullet$ are equivalent with respect to the equivalence relation (eq2), then $\mathcal{P}(M_\bullet)=\mathcal{P}(N_\bullet)$. Exactly the same argument shows that  if two formal complexes $M_\bullet$ and $N_\bullet$ are equivalent with respect to the equivalence relation (eq3), then $\mathcal{P}(M_\bullet)=\mathcal{P}(N_\bullet)$. Therefore we see that the map
is well defined. Concerning the functoriality of $\mathcal{P}$, the fact that $\mathcal{P}$ maps the identities of $\mathbf{KR}$ to identities of $\mathbf{MOY}^{(T)}$ is manifest while
 the identity 
\[
\mathcal{P}([M_\bullet]\circ [N_\bullet])=\mathcal{P}([M_\bullet])\circ \mathcal{P}([N_\bullet])
\]
follows by the usual properties of the Poincar\'e polynomial/Euler characteristic of chain complexes of abelian groups with respect to the tensor product of complexes. Namely, the map $\mathcal{P}$ has been defined in such a way to be an formal version of the Poincar\'e polynomial/Euler and in particular it enjoys the good behaviour with respect to tensor product-type constructions of its classical counterpart. By the same reason, $\mathcal{P}$ is a monoidal functor, i.e.,
 \[
\mathcal{P}([M_\bullet]\otimes [N_\bullet])=\mathcal{P}([M_\bullet])\circ \mathcal{P}([N_\bullet]).
\]
Moreover, $\mathcal{P}$ manifestly preserves the evaluation and coevaluation morphisms. Finally,
\[
\mathcal{P}\left(\left[
{\xy
(-5,-10);(-5,10)**\dir{-}?>(.5)*\dir{>}
,(5,-10);(5,10)**\dir{-}?>(.5)*\dir{>}
,(-7,-10)*{\scriptstyle{1}}
,(7,-10)*{\scriptstyle{1}}
,(-7,10)*{\scriptstyle{1}}
,(7,10)*{\scriptstyle{1}}
\endxy}\,\{n-1\}\,
\xrightarrow{\phantom{mm}\chi_0\phantom{mm}} \,
\esse\{n\}\,\right]
\right)=q^{n-1}\left(\,\idtwo - q \esse\,\right)
\]
and
\[
\mathcal{P}\left(
\left[
\esse\{-n\}\,
\xrightarrow{\phantom{mm}\chi_1\phantom{mm}} \,
{\xy
(-5,-10);(-5,10)**\dir{-}?>(.5)*\dir{>}
,(5,-10);(5,10)**\dir{-}?>(.5)*\dir{>}
,(-7,-10)*{\scriptstyle{1}}
,(7,-10)*{\scriptstyle{1}}
,(-7,10)*{\scriptstyle{1}}
,(7,10)*{\scriptstyle{1}}
\endxy}\,\{1-n\}\,\right]
\right)=
q^{1-n}\left(\,\, \idtwo-q^{-1}\esse\,\,\right),
\]
which shows that $\mathcal{P}$ is braided. Therefore the composition
\[
\mathcal{P}\circ Kh_n\colon \mathbf{TD}\to \mathbf{Jones}_n
\]
is a braided monoidal functor and $(\mathcal{P}\circ Kh_n)(\uparrow)=\uparrow^1$. By the universal property of $\mathbf{PT}$, we find that $(\mathcal{P}\circ Kh_n)$ is naturally isomorphic to (and actually coincides with) the functor $Z\colon \mathbf{PT}\colon  \mathbf{Jones}_n$ defining the level $n$ Jones polynomial. 
\end{proof}
\subsection{Comparison to Khovanov-Rozansky homology}\label{sec:comparison}
In \cite{Khovanov-Rozansky}, Khovanov and Rozansky consider a monoidal category which could be called the ``homotopy category of chain complexes in  the Landau-Ginzburg category'' and therefore denoted by the symbol $\mathcal{K}(\mathcal{LG})$. Objetcs of $\mathcal{K}(\mathcal{LG})$ are Landau-Ginzburg \emph{potentials}, i.e., polynomials $W(\vec{x})$ into variables $\vec{x}=(x_1,x_2,\dots)$ over the field $\mathbb{Q}$ (or an extension of it), such that the Jacobian ring $\mathbb{Q}[x_1,x_2,\dots]/(\partial_1W,\partial_2W,\dots)$ is finite dimensional. Morphisms between $W_1(\vec{x})$ and $W_2(\vec{y})$ are homotopy equivalence classes of complexes of matrix factorizations of $W_1(\vec{x})-W_2(\vec{y})$. The category $\mathcal{K}(\mathcal{LG})$ is rigid, and it is expected to be braided.\footnote{Unfortunately, we are not aware of any reference proving (or disproving) this fact. Our impression is that most experts are quite certain that $\mathcal{K}(\mathcal{LG})$ is indeed a braided monoidal category.}  Even without knowing for sure whether $\mathcal{K}(\mathcal{LG})$ is braided or not, Khovanov and Rozansky are nevertheless able to define a rigid monoidal functor\[
KH_n\colon\mathbf{TD}\to \mathcal{K}(\mathcal{LG})
\]
with $KH_n(\uparrow)=x^{n+1}$, thus realizing $W(x)=x^{n+1}$ as a \emph{braided object} in $\mathcal{K}(\mathcal{LG})$. For a link $\Gamma$ the complex of matrix factorizations $KH_n(\Gamma)$ is a complex of matrix factorizations of the potential 0, and so it is simply a bicomplex of graded $\mathbb{Q}$-vector spaces. Up to homotopy, this is identified by the cohomology of its total complex. This is a bigraded $\mathbb{Q}$-vector space, called the \emph{Khovanov-Rozansky cohomology} of the link $\Gamma$.
\par 
As one could expect, the construction of $KH$ is in two steps: first, Khovanov and Rozansky define a rigid monoidal functor $KH_n\colon \mathbf{PTD}\to \mathcal{K}(\mathcal{LG})$ with $KH_n(\uparrow)=x^{n+1}$ and then they show that $KH$ is invariant with respect to the Reidemeister moves. The proof of this fact is quite nontrivial and  uses subtle properties of the complexes of matrix factorizations involved with the definition of $KH_n$. What we did in Section \ref{sec:relations} has been precisely to extract from Khovanov and Rozansky's proof all the properties of these complexes of matrix factorizations actually used, as well as additional properties demonstrated by Rasmussen in \cite{Rasmussen}, and change them into defining properties. By this, one tautologically has that the monoidal functor $KH$ factors as
 \begin{equation}\label{eq:krr}
\xymatrix{
\mathbf{TD}\ar[r]^{Kh_n}\ar@/_2pc/[rr]_{KH_n} &\mathbf{KR} \ar[r]^{\mathrm{mf}} &\mathcal{K}(\mathcal{LG})
}
\end{equation}
for some monoidal functor $\mathrm{mf}\colon \mathbf{KR}\to \mathcal{K}(\mathcal{LG})$ with $\mathrm{mf}(\uparrow)=x^{n+1}$, where `mf' stands for `matrix factorzations'. Moreover $\mathrm{mf}$ is compatible with shifts and direct sums. This implies that it commutes with the tensor product by graded $\mathbb{Q}$-vector spaces (with a fixed basis).
\par
Unfortunately we are not able to prove whether $\mathrm{mf}$ is an embedding of categories. Clearly, it is injective on objects, but its faithfulness on morphisms is elusive: there could be nonequivalent formal complexes of MOY-type graphs leading to homotopy equivalent complexes of matrix factorizations. Actually, our opinion is that $\mathrm{mf}$ should not be expeted to be faithful: it is likely that in $\mathcal{K}(\mathcal{LG})$ there are many more relations than those involved in the proof of the Reidemeister invariance of $KH_n$ by Khovanov and Rozansky and those evidentiated by Rasmussen.\par
Yet, if for a link $\Gamma$ we have
\[
Kh_n(\Gamma)=[V_\bullet]\otimes \emptyset,
\]
for some complex of graded $\mathbb{Q}$-vector spaces with trivial differentials 
\[
V_\bullet=\left(\cdots\to V_{i-1}\xrightarrow{0} V_i\xrightarrow{0} V_{i+1}\cdots \right),
\]
then the commutativity of the diagram \ref{eq:krr} implies that we have
\[
KH_n(\Gamma)=\oplus_i V_i,
\]
with the graded $\mathbb{Q}$-vector space $V_i$ in ``horizontal'' degree $i$. That is, for those links such that $Kh_n(\Gamma)=[V_\bullet]\otimes \emptyset$, we have that $Kh_n$ computes the Khovanov-Rozansky homology. This fact will be extensively used in the following chapter.

\chapter{The $Kh_n$ invariant of 2-strand braid links}\label{chapter:computation}
We comclude this Thesis by computing the Khovanov homology of the 2-strand braid link with $k$ crosssings, i.e., of the link
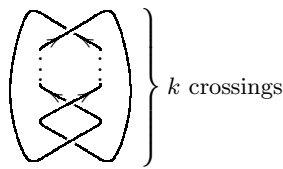
\[
\trsplusk{k}
\]
This is connected, i.e., it is a knot, for odd $k$, and has exactly two components for $k$ even. The formula we are going to find for the Poincar\'e polynomial of the Khovanov homology of the 2-strand braid with $k$ crosssings, in the odd case will extend to $n\leq 4$ the results by Rasmussen \cite{Rasmussen}, and will agree  with the prediction by Morozov and Anokhina \cite{Morozov1}. For an even number of crossings, the formula we will find will agree with the prediction by Morozov and Anokhina \cite{Morozov1}, and with the string theory inspired prediction by Gukov, Iqbal, Koz\c{c}az and Vafa for the particular case of two crossings (the Hopf link) \cite{Gukov}. For low values of $n$ and $k$, these formulas has been computer checked by Carqueville and Murfet in \cite{Carqueville-Murfet}. More precisely, the case of 2 crossings has been computer checked for $n\leq 11$,  the case of 3 crossings has been computer checked for $n\leq 11$, the case of 4 crossings for $n\leq 5$, the case of 5 crossings for $n\leq 6$, and the case of 6 crossings for $n\leq4$.

\section{Preliminary lemmas}
Here we provide a few preliminary results to be used in the computation of the Khovanov homology of the 2-strand braid links. In what follows, ``complex'' (resp.  ``bicomplex'') always mean ``formal complex  (resp., bicomplex) of formal sums of shifted MOY-type graphs''.
\begin{lemma}\label{lemma:one}
The bicomplex
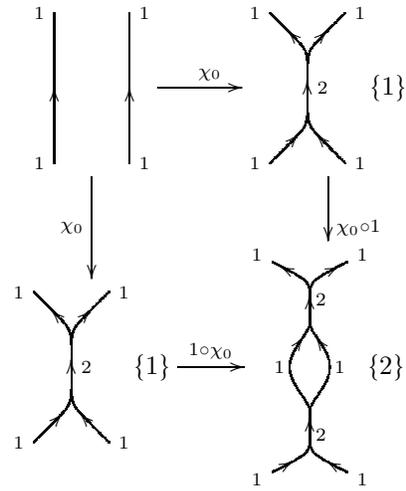
\[
\xymatrix{
\idtwo\ar[r]^{\chi_0}\ar[d]_{\chi_0}&\esse\{1\}\ar[d]^{\chi_0\circ 1}\\
\esse\{1\}\ar[r]^-{1\circ\chi_0}&\essecircesse\{2\}
}
\]
is equivalent to the bicomplex
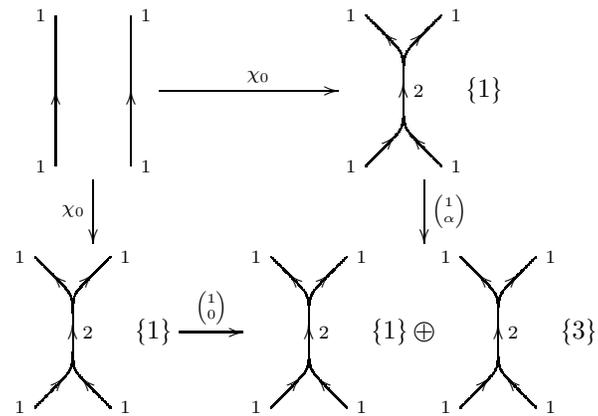
\[
\xymatrix{
\idtwo\ar[r]^{\chi_0}\ar[d]_{\chi_0}& \esse\{1\}\ar[d]^{\tiny{\left(\begin{matrix}1\\ \alpha\end{matrix}\right)}}\\
\esse\{1\}\ar[r]^-{\tiny{\left(\begin{matrix}1\\ 0\end{matrix}\right)}}&\esse\{1\}\oplus \esse\{3\}
}
\]
\end{lemma}
\begin{proof}
In the diagram
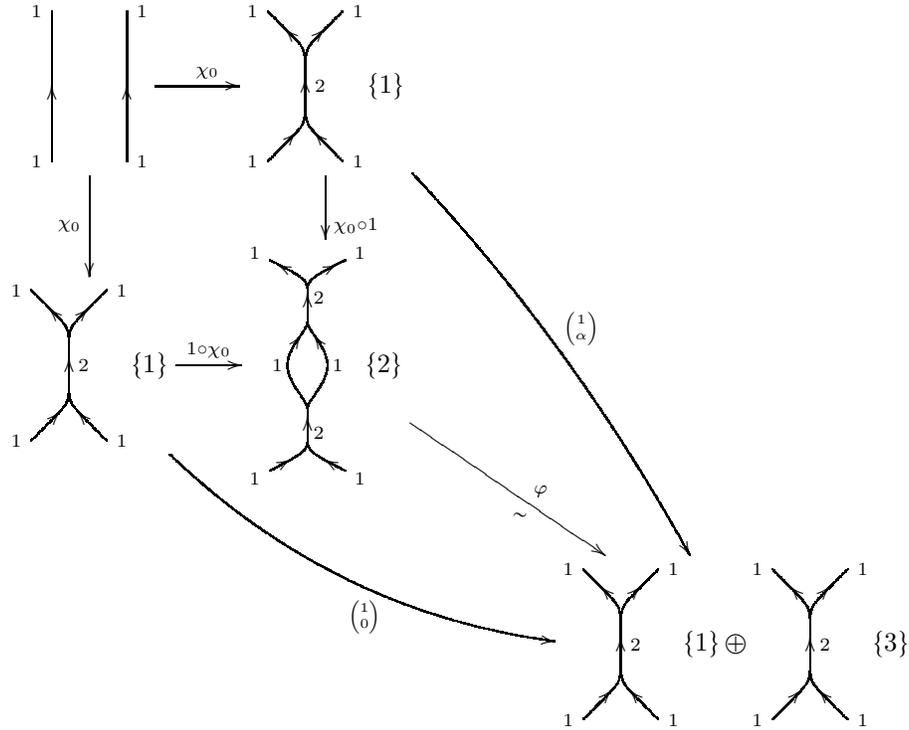
\[
\xymatrix{
\idtwo\ar[d]_{\chi_0}\ar[r]^{\chi_0}&\esse\{1\}\ar[d]^{\chi_0\circ 1}\ar@/^1pc/[rrdd]^{\tiny{\left(\begin{matrix}1\\ \alpha\end{matrix}\right)}}\\
\esse\{1\}\ar[r]^-{1\circ\chi_0}\ar@/_3pc/[rrrd]_{\tiny{\left(\begin{matrix}1\\ 0\end{matrix}\right)}}&\essecircesse\{2\}\ar[drr]^{\varphi}_{\sim}\\
&&&\esse\{1\} \oplus \esse\{3\}
}
\]
the morphism $\varphi$ is an isomorphism and all the subdiagrams commute, due to relations (\ref{eq:moy2a}) and (\ref{eq:moy2b}).
\end{proof}

\begin{lemma}\label{lemma:two}
The bicomplex
\[
\xymatrix{
\esse\ar[r]^{\alpha}\ar[d]_{\chi_0\circ 1}& \esse\{2\}\ar[d]^{\chi_0\circ 1}\\
\essecircesse\{1\}\ar[r]_-{1\circ \alpha}&\essecircesse\{3\}
}
\]
is equivalent to the bicomplex
\[
\xymatrix{
\esse\ar[r]^{\alpha}\ar[d]_{\tiny{\left(\begin{matrix}1\\ \alpha\end{matrix}\right)}}& \esse\{2\}\ar[d]^{\tiny{\left(\begin{matrix}1\\ \gamma\end{matrix}\right)}}\\
\esse\oplus \esse\{2\}\ar[r]_-{\tiny{\left(\begin{matrix}0  & 1\\ 0 & 0\end{matrix}\right)}}&\esse\{2\}\oplus \esse\{4\}
}
\]
\end{lemma}
\begin{proof}
In the diagram
\[
\xymatrix@C=.5pc{
&\esse\ar@/_1pc/[dddd]_{\tiny{\left(\begin{matrix}1\\ \alpha\end{matrix}\right)}}
\ar[ddr]^{\chi_0\circ 1}\ar[rrrr]^{\alpha}&&&&\esse\{2\}\ar[ddl]_{\chi_0\circ 1}\ar@/^1pc/[dddd]^{\tiny{\left(\begin{matrix}1\\ \gamma\end{matrix}\right)}}\\
\\
&&\essecircesse\{1\}\ar[ddl]_{\varphi}^{\sim}\ar[rr]^-{1\circ\alpha}&&\essecircesse\{3\}\ar[ddr]^{\psi}_{\sim}\\
\\
&\esse\oplus \esse\{2\}\ar[rrrr]_-{\tiny{\left(\begin{matrix}0  & 1\\ 0 & 0\end{matrix}\right)}}&&&&\esse\{2\} \oplus \esse\{4\}
}
\]
all the subdiagrams commute, thanks to relations (\ref{eq:moy2b}), (\ref{eq:moy2c}) and (\ref{eq:shift}). As $\varphi$ and $\psi$ are isomorphisms, the top commutative diagram is equivalent to the outer commutative diagram. 
\end{proof}

\begin{corollary}
We have $\gamma\alpha=0$.
\end{corollary}
\begin{proof}
By Lemma \ref{lemma:two} we have
\[
\left(\begin{matrix}
\alpha \\ \gamma\alpha
\end{matrix}
\right)=
\left(\begin{matrix}1\\ \gamma
\end{matrix}
\right)
\alpha=
\left(\begin{matrix}0 &1\\0&0
\end{matrix}
\right)
\left(\begin{matrix}1\\
\alpha
\end{matrix}
\right)=
\left(\begin{matrix}
\alpha\\0
\end{matrix}
\right).
\]
\end{proof}

\begin{lemma}
The total complex of the bicomplex 
\[
\xymatrix{
\idtwo\ar[r]^{\chi_0}\ar[d]_{\chi_0}& \esse\{1\}\ar[d]^{\tiny{\left(\begin{matrix}1\\ \alpha\end{matrix}\right)}}\ar[r]^{\alpha}\ar[d]& \esse\{3\}\ar[d]^{\tiny{\left(\begin{matrix}1\\ \gamma\end{matrix}\right)}}\\
\esse\{1\}\ar[r]_-{\tiny{\left(\begin{matrix}1\\ 0\end{matrix}\right)}}&\esse\{1\}\oplus \esse\{3\}\ar[r]_-{\tiny{\left(\begin{matrix}0  & 1\\ 0 & 0\end{matrix}\right)}}&\esse\{3\}\oplus \esse\{5\}
}
\]
is equivalent to
\[
\idtwo\xrightarrow{\chi_0} \esse\{1\}\xrightarrow{\alpha} \esse\{3\}\xrightarrow{\gamma} \esse\{5\}.
\]
\end{lemma}
\begin{proof}
The total complex of the bicomplex we are considering is
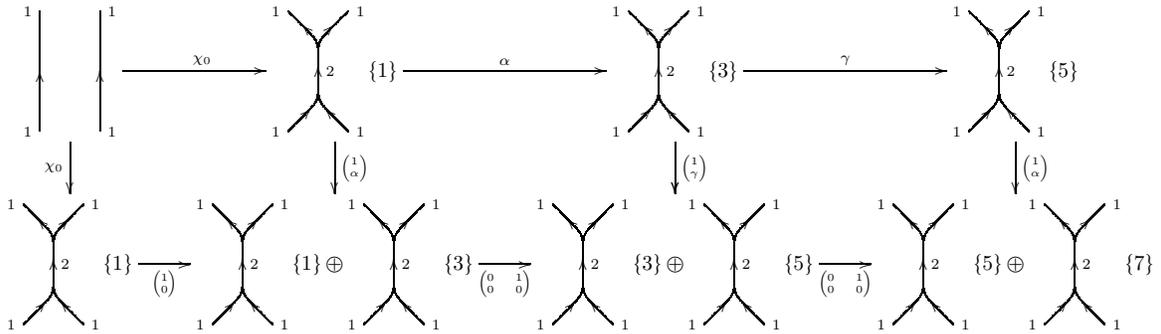
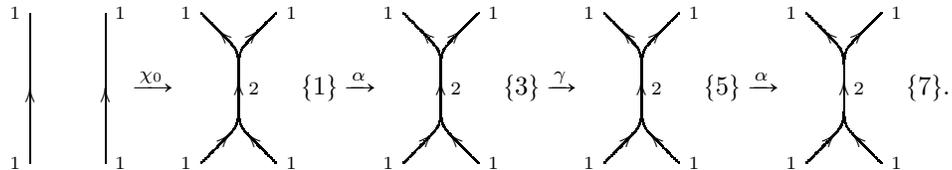
\begin{align*}
\idtwo&\xrightarrow{\tiny{\left(\begin{matrix}\chi_0\\ \chi_0\end{matrix}\right)}} \esse\{1\}\oplus \esse\{1\} \xrightarrow{\tiny{\left(\begin{matrix}1 &-1\\ 0& -\alpha \\ 0 & \alpha\end{matrix}\right)}} \esse\{1\}\oplus \esse\{3\}\oplus \esse\{3\}\\
&\xrightarrow{\tiny{\left(\begin{matrix}0  & 1 & 1\\ 0 & 0 & \gamma\end{matrix}\right)}} \esse\{3\}\oplus \esse\{5\}
\end{align*}
By (eq2) applied around the morphism $\left(\begin{matrix}1 &-1\\ 0& -\alpha \\ 0 & \alpha\end{matrix}\right)$, this complex is equivalent to the complex
\[
\idtwo\xrightarrow{\chi_0}  \esse\{1\} \xrightarrow{\tiny{\left(\begin{matrix} -\alpha \\  \alpha\end{matrix}\right)}}  \esse\{3\}\oplus \esse\{3\}
\xrightarrow{\tiny{\left(\begin{matrix}  1 & 1\\  0 & \gamma\end{matrix}\right)}} \esse\{3\}\oplus \esse\{5\}
\]
Next, we apply (eq2) applied around the morphism $\left(\begin{matrix}  1 & 1\\  0 & \gamma\end{matrix}\right)$ to see that this complex is 
equivalent to the complex
\[
\idtwo\xrightarrow{\chi_0}  \esse\{1\} \xrightarrow{\alpha}  \esse\{3\}
\xrightarrow{\gamma}  \esse\{5\}
\]

\end{proof}

\begin{lemma}\label{lemma:three}
The bicomplex
\[
\xymatrix{
\esse\ar[r]^{\gamma}\ar[d]_{\chi_0\circ 1}& \esse\{2\}\ar[d]^{\chi_0\circ 1}\\
\essecircesse\{1\}\ar[r]_-{1\circ \gamma}&\essecircesse\{3\}
}
\]
is equivalent to the bicomplex
\[
\xymatrix{
\esse\ar[r]^{\gamma}\ar[d]_{\tiny{\left(\begin{matrix}1\\ \gamma\end{matrix}\right)}}& \esse\{2\}\ar[d]^{\tiny{\left(\begin{matrix}1\\ \alpha\end{matrix}\right)}}\\
\esse\oplus \esse\{2\}\ar[r]_-{\tiny{\left(\begin{matrix}0  & 1\\ 0 & 0\end{matrix}\right)}}&\esse\{2\}\oplus \esse\{4\}
}
\]
\end{lemma}
\begin{proof}
In the diagram
\[
\xymatrix@C=.5pc{
&\esse\ar@/_1pc/[dddd]_{\tiny{\left(\begin{matrix}1\\ \gamma\end{matrix}\right)}}
\ar[ddr]^{\chi_0\circ 1}\ar[rrrr]^{\gamma}&&&&\esse\{2\}\ar[ddl]_{\chi_0\circ 1}\ar@/^1pc/[dddd]^{\tiny{\left(\begin{matrix}1\\ \alpha\end{matrix}\right)}}\\
\\
&&\essecircesse\{1\}\ar[ddl]_{\psi}^{\sim}\ar[rr]^-{1\circ\gamma}&&\essecircesse\{3\}\ar[ddr]^{\varphi}_{\sim}\\
\\
&\esse\oplus \esse\{2\}\ar[rrrr]_-{\tiny{\left(\begin{matrix}0  & 1\\ 0 & 0\end{matrix}\right)}}&&&&\esse\{2\} \oplus \esse\{4\}
}
\]
all the subdiagrams commute, thanks to relations (\ref{eq:moy2b}), (\ref{eq:moy2c}) and (\ref{eq:shift2}). As $\varphi$ and $\psi$ are isomorphisms, the top commutative diagram is equivalent to the outer commutative diagram. 
\end{proof}

\begin{corollary}
We have $\alpha\gamma=0$.
\end{corollary}
\begin{proof}
By Lemma \ref{lemma:three} we have
\[
\left(\begin{matrix}
\gamma \\ \alpha\gamma
\end{matrix}
\right)=
\left(\begin{matrix}1\\ \alpha
\end{matrix}
\right)
\gamma=
\left(\begin{matrix}0 &1\\0&0
\end{matrix}
\right)
\left(\begin{matrix}1\\
\gamma
\end{matrix}
\right)=
\left(\begin{matrix}
\gamma\\0
\end{matrix}
\right).
\]
\end{proof}

\begin{corollary}
The bicomplex
\[
\xymatrix{
\idtwo\ar[r]^{\chi_0}\ar[d]_-{\chi_0\circ 1}& \esse\{1\}\ar[d]^-{\chi_0\circ 1}\ar[r]^{\alpha}& \esse\{3\}\ar[d]^-{\chi_0\circ 1}
\ar[r]^{\gamma}& \esse\{5\}\ar[d]^-{\chi_0\circ 1}\\
\esse\{1\}\ar[r]^-{1\circ\chi_0}&\essecircesse\{2\}\ar[r]^-{1\circ \alpha}&\essecircesse\{4\}
\ar[r]^-{1\circ\gamma}&\essecircesse\{6\}
}
\]
is equivalent to the bicomplex
\[
\scalebox{.8}{
\xymatrix{
\idtwo\ar[r]^{\chi_0}\ar[d]_{\chi_0}& \esse\{1\}\ar[d]^{\tiny{\left(\begin{matrix}1\\ \alpha\end{matrix}\right)}}\ar[r]^{\alpha}& \esse\{3\}\ar[d]^{\tiny{\left(\begin{matrix}1\\ \gamma\end{matrix}\right)}}\ar[r]^{\gamma}\ar[d]& \esse\{5\}\ar[d]^{\tiny{\left(\begin{matrix}1\\ \alpha\end{matrix}\right)}}\\
\esse\{1\}\ar[r]_-{\tiny{\left(\begin{matrix}1\\ 0\end{matrix}\right)}}&\esse\{1\}\oplus \esse\{3\}\ar[r]_-{\tiny{\left(\begin{matrix}0  & 1\\ 0 & 0\end{matrix}\right)}}&\esse\{3\}\oplus \esse\{5\}\ar[r]_-{\tiny{\left(\begin{matrix}0  & 1\\ 0 & 0\end{matrix}\right)}}&\esse\{5\}\oplus \esse\{7\}
}
}
\]
\end{corollary}
\begin{proof}
Immediate from Lemma \ref{lemma:one}, Lemma \ref{lemma:two}, Lemma \ref{lemma:three}, and their proofs. 
\end{proof}

\begin{lemma}
The total complex of the bicomplex 
\[
\scalebox{.8}{
\xymatrix{
\idtwo\ar[r]^{\chi_0}\ar[d]_{\chi_0}& \esse\{1\}\ar[d]^{\tiny{\left(\begin{matrix}1\\ \alpha\end{matrix}\right)}}\ar[r]^{\alpha}& \esse\{3\}\ar[d]^{\tiny{\left(\begin{matrix}1\\ \gamma\end{matrix}\right)}}\ar[r]^{\gamma}\ar[d]& \esse\{5\}\ar[d]^{\tiny{\left(\begin{matrix}1\\ \alpha\end{matrix}\right)}}\\
\esse\{1\}\ar[r]_-{\tiny{\left(\begin{matrix}1\\ 0\end{matrix}\right)}}&\esse\{1\}\oplus \esse\{3\}\ar[r]_-{\tiny{\left(\begin{matrix}0  & 1\\ 0 & 0\end{matrix}\right)}}&\esse\{3\}\oplus \esse\{5\}\ar[r]_-{\tiny{\left(\begin{matrix}0  & 1\\ 0 & 0\end{matrix}\right)}}&\esse\{5\}\oplus \esse\{7\}
}
}
\]
is equivalent to
\[
\idtwo\xrightarrow{\chi_0} \esse\{1\}\xrightarrow{\alpha} \esse\{3\}\xrightarrow{\gamma} \esse\{5\}\xrightarrow{\alpha}\esse\{7\} .
\]
\end{lemma}

\begin{proof}
The total complex of the bicomplex we are considering is

\resizebox{.9\linewidth}{!}{
  \begin{minipage}{\linewidth}
\begin{align*}
\idtwo&\xrightarrow{\tiny{\left(\begin{matrix}\chi_0\\ \chi_0\end{matrix}\right)}} \esse\{1\}\oplus \esse\{1\} \xrightarrow{\tiny{\left(\begin{matrix}1 &-1\\ 0& -\alpha \\ 0 & \alpha\end{matrix}\right)}} \esse\{1\}\oplus \esse\{3\}\oplus \esse\{3\}\\
&\xrightarrow{\tiny{\left(\begin{matrix}0  & 1 & 1\\ 0 & 0 & \gamma\\ 0 & 0 &\gamma \end{matrix}\right)}} \esse\{3\}\oplus \esse\{5\}\oplus \esse\{5\}
\xrightarrow{\tiny{\left(\begin{matrix}  0 & 1 &-1\\  0 & 0 &-\alpha\end{matrix}\right)}} \esse\{5\}\oplus \esse\{7\}
\end{align*}
 \end{minipage}
 }
 
By (eq2) applied around the morphism $\left(\begin{matrix}1 &-1\\ 0& -\alpha \\ 0 & \alpha\end{matrix}\right)$, this complex is equivalent to the complex

\resizebox{.9\linewidth}{!}{
  \begin{minipage}{\linewidth}
\begin{align*}
\idtwo&\xrightarrow{\chi_0}  \esse\{1\} \xrightarrow{\tiny{\left(\begin{matrix} -\alpha \\  \alpha\end{matrix}\right)}}  \esse\{3\}\oplus \esse\{3\}
\xrightarrow{\tiny{\left(\begin{matrix}  1 & 1\\  0 & \gamma \\  0 & \gamma\end{matrix}\right)}} \esse\{3\}\oplus \esse\{5\}\oplus \esse\{5\}\\
&
\xrightarrow{\tiny{\left(\begin{matrix}  0 & 1 &-1\\  0 & 0 &-\alpha\end{matrix}\right)}}\esse\{5\}\oplus \esse\{7\}
\end{align*}
 \end{minipage}
 }

Next, we apply (eq2) applied around the morphism $\left(\begin{matrix}  1 & 1\\  0 & \gamma\\  0 & \gamma\end{matrix}\right)$ to see that this complex is 
equivalent to the complex
\begin{align*}
\idtwo&\xrightarrow{\chi_0}  \esse\{1\} \xrightarrow{\alpha}  \esse\{3\}
\xrightarrow{\tiny{\left(\begin{matrix}   \gamma\\   \gamma\end{matrix}\right)}}  \esse\{5\}\oplus \esse\{5\}\\
&\xrightarrow{\tiny{\left(\begin{matrix}   1 &-1\\   0 &-\alpha\end{matrix}\right)}}\esse\{5\}\oplus \esse\{7\}
\end{align*}
Finally, we apply (eq2) applied around the morphism $\left(\begin{matrix}  1 & -1\\  0 & -\alpha\end{matrix}\right)$ to see that this complex is 
equivalent to the complex
\[
\idtwo\xrightarrow{\chi_0}  \esse\{1\} \xrightarrow{\alpha}  \esse\{3\}
\xrightarrow{\gamma}  \esse\{5\}
\xrightarrow{-\alpha}\oplus \esse\{7\}
\]
This is in turn manifestly equivalent to
the complex
\[
\idtwo\xrightarrow{\chi_0}  \esse\{1\} \xrightarrow{\alpha}  \esse\{3\}
\xrightarrow{\gamma}  \esse\{5\}
\xrightarrow{\alpha} \esse\{7\}
\]
by (eq1) and the evident isomorphism of complexes
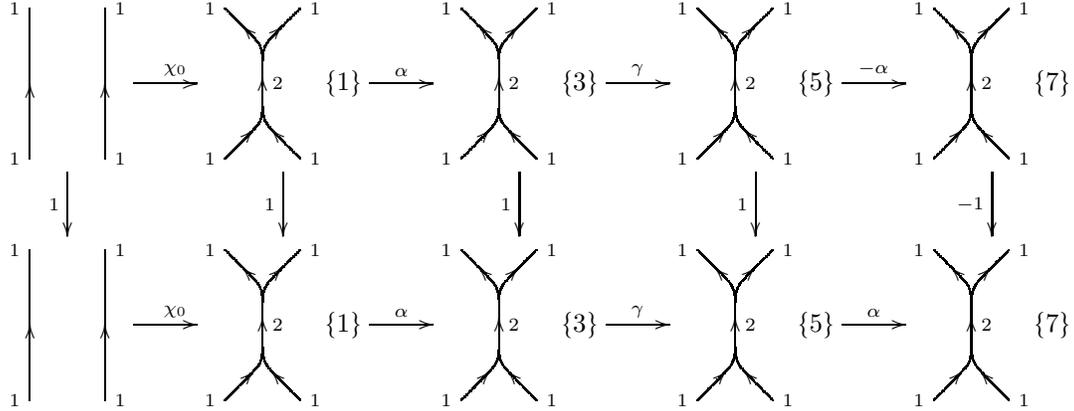
\[
\xymatrix{
\idtwo\ar[r]^{\chi_0}\ar[d]_{1}  & \esse\{1\}\ar[r]^{\alpha}\ar[d]_{1} &  \esse\{3\}\ar[r]^{\gamma}\ar[d]_{1} 
&  \esse\{5\}\ar[r]^{-\alpha}\ar[d]_{1} &
\esse\{7\}\ar[d]_{-1}\\
\idtwo\ar[r]^{\chi_0}  & \esse\{1\}\ar[r]^{\alpha}& \esse\{3\}\ar[r]^{\gamma} 
&  \esse\{5\}\ar[r]^{\alpha}&
\esse\{7\}
}
\]
\end{proof}

From this point on we can proceed inductively. This way one proves the following
\begin{proposition}\label{main-prop}
The total complex of the bicomplex
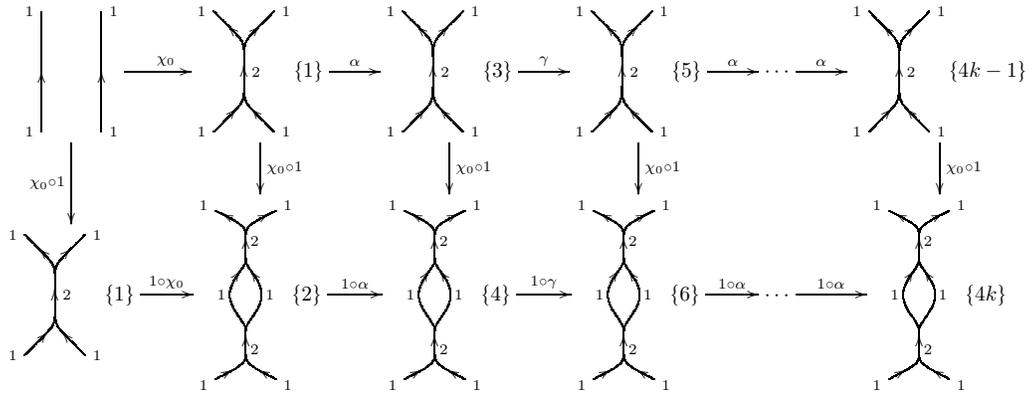
\[
\raisebox{44pt}{\scalebox{.8}{
\xymatrix{
\idtwo\ar[r]^{\chi_0}\ar[d]_-{\chi_0\circ 1}& \esse\{1\}\ar[d]^-{\chi_0\circ 1}\ar[r]^{\alpha}& \esse\{3\}\ar[d]^-{\chi_0\circ 1}
\ar[r]^{\gamma}& \esse\{5\}\ar[d]^-{\chi_0\circ 1}\ar[r]^-{\alpha}&\cdots\ar[r]^-{\alpha}& \esse\{4k-1\}\ar[d]^-{\chi_0\circ 1}\\
\esse\{1\}\ar[r]^-{1\circ\chi_0}&\essecircesse\{2\}\ar[r]^-{1\circ \alpha}&\essecircesse\{4\}\ar[r]^-{1\circ\gamma}&\essecircesse\{6\}
\ar[r]^-{1\circ\alpha}&\cdots\ar[r]^-{1\circ\alpha}& \essecircesse\{4k\}
}}}
\]
is equivalent to the complex
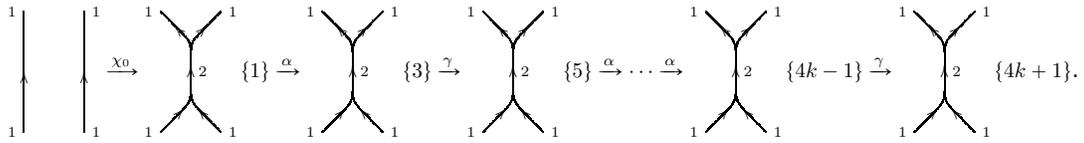
\[
\scalebox{.8}{
$\idtwo\xrightarrow{\chi_0} \esse\{1\}\xrightarrow{\alpha} \esse\{3\}\xrightarrow{\gamma} \esse\{5\}\xrightarrow{\alpha} \cdots
\xrightarrow{\alpha}\esse\{4k-1\}\xrightarrow{\gamma}\esse\{4k+1\}
$}.
\]
\vskip .4 cm

\noindent and the total complex of the bicomplex
\[
\raisebox{44pt}{\scalebox{.8}{
\xymatrix{
\idtwo\ar[r]^{\chi_0}\ar[d]_-{\chi_0\circ 1}& \esse\{1\}\ar[d]^-{\chi_0\circ 1}\ar[r]^{\alpha}& \esse\{3\}\ar[d]^-{\chi_0\circ 1}
\ar[r]^{\gamma}& \esse\{5\}\ar[d]^-{\chi_0\circ 1}\ar[r]^-{\alpha}&\cdots\ar[r]^-{\gamma}& \esse\{4k+1\}\ar[d]^-{\chi_0\circ 1}\\
\esse\{1\}\ar[r]^-{1\circ\chi_0}&\essecircesse\{2\}\ar[r]^-{1\circ \alpha}&\essecircesse\{4\}\ar[r]^-{1\circ\gamma}&\essecircesse\{6\}
\ar[r]^-{1\circ\alpha}&\cdots\ar[r]^-{1\circ\gamma}& \essecircesse\{4k+2\}
}}}
\]
is equivalent to the complex
\[
\scalebox{.8}{
$\idtwo\xrightarrow{\chi_0} \esse\{1\}\xrightarrow{\alpha} \esse\{3\}\xrightarrow{\gamma} \esse\{5\}\xrightarrow{\alpha} \cdots
\xrightarrow{\gamma}\esse\{4k+1\}\xrightarrow{\alpha}\esse\{4k+3\}
$}.
\]
\end{proposition}

\section{The computation}
Proposition \ref{main-prop} immediately leads us to our main result.

\begin{theorem}\label{main-thm}
We have, for any $k\geq 0$,
\begin{align*}
Kh_n\left(\,\splusk{2k+1}\, \right)&
=\left[\idtwo\xrightarrow{\chi_0} \esse\{1\}\xrightarrow{\alpha} \esse\{3\}\xrightarrow{\gamma} \esse\{5\}\right.
\\
&\left.\xrightarrow{\alpha} \cdots
\xrightarrow{\alpha}\esse\{4k-1\}\xrightarrow{\gamma}\esse\{4k+1\}\right]\{(n-1)(2k+1)\}
\end{align*}
and, for any $k\geq 1$,
\begin{align*}
Kh_n\left(\,\splusk{2k}\, \right)&
=\left[\idtwo\xrightarrow{\chi_0} \esse\{1\}\xrightarrow{\alpha} \esse\{3\}\xrightarrow{\gamma} \esse\{5\}\right.
\\
&\left.\xrightarrow{\alpha} \cdots
\xrightarrow{\gamma}\esse\{4k-3\}\xrightarrow{\alpha}\esse\{4k-1\}\right]\{(n-1)(2k)\}
\end{align*}

\end{theorem} 
\begin{proof}
We prove the statement inductively on the number $m$ of crossings. The base of the induction is $m=1$. In this case we need to show that
\[
Kh_n\left(\,\splus\, \right)
=\left[\idtwo\xrightarrow{\chi_0} \esse\{1\}\right]\{n-1\}
\]
and this is true by definition of $Kh$. Assume now the statement has been proven up to $m=m_0$ an let us prove it for $m=m_0+1$. If $m_0=2k-1$ then we know by the inductive hypothesis that
\begin{align*}
Kh_n\left(\,\splusk{2k-1}\, \right)&
=\left[\idtwo\xrightarrow{\chi_0} \esse\{1\}\xrightarrow{\alpha} \esse\{3\}\xrightarrow{\gamma} \esse\{5\}\right.
\\
&\left.\xrightarrow{\alpha} \cdots
\xrightarrow{\alpha}\esse\{4k-5\}\xrightarrow{\gamma}\esse\{4k-3\}\right]\{(n-1)(2k-1)\}
\end{align*}
As
\[
\splusk{2k}=\splusk{2k-1}\quad \circ \quad \splus\,,
\]
we have
\[Kh_n\left(\,\splusk{2k}\,\right)=Kh_n\left(\,\splusk{2k-1}\,\right) \circ Kh_n\left(\,\splus\,\right)=
\]
\[
=[\mathrm{tot}\left(\!\!\!\raisebox{44pt}{\scalebox{.8}{
\xymatrix{
\idtwo\ar[r]^{\chi_0}\ar[d]_-{\chi_0\circ 1}& \esse\{1\}\ar[d]^-{\chi_0\circ 1}\ar[r]^{\alpha}& \esse\{3\}\ar[d]^-{\chi_0\circ 1}
\ar[r]^{\gamma}& \esse\{5\}\ar[d]^-{\chi_0\circ 1}\ar[r]^-{\alpha}&\cdots\ar[r]^-{\gamma}& \esse\{4k-3\}\ar[d]^-{\chi_0\circ 1}\\
\esse\{1\}\ar[r]^-{1\circ\chi_0}&\essecircesse\{2\}\ar[r]^-{1\circ \alpha}&\essecircesse\{4\}\ar[r]^-{1\circ\gamma}&\essecircesse\{6\}
\ar[r]^-{1\circ\alpha}&\cdots\ar[r]^-{1\circ\gamma}& \essecircesse\{4k-2\}
}}}\!\!\!
\right)]\{(n-1)(2k)\}
\]
\[
=[\scalebox{.8}{
$\idtwo\xrightarrow{\chi_0} \esse\{1\}\xrightarrow{\alpha} \esse\{3\}\xrightarrow{\gamma} \esse\{5\}\xrightarrow{\alpha} \cdots
\xrightarrow{\gamma}\esse\{4k-3\}\xrightarrow{\alpha}\esse\{4k-1\}
$}]\{(n-1)(2k)\},
\]
by Proposition \ref{main-prop}. Therefore the statement is true for $m=m_0+1$ in this case. The proof for the case $m_0=2k$ is completely analogous.
\end{proof}
By ``closing the diagrams'', we immediately get the following.
\begin{corollary}
We have, for any $k\geq 0$,
\begin{align*}
Kh_n\left(\,\trsplusk{2k+1}\,\right)&
=\left[ [n] \{1-n\}\,
\xrightarrow{0} 0\xrightarrow{0} [n-1]\{4-n\} 
\xrightarrow{0} [n-1]\{4+n\} 
\xrightarrow{0} \cdots\right.\\
&\hskip -12 pt\left.\xrightarrow{0}[n-1]\{4k-n\}  
\xrightarrow{0} [n-1]\{4k+n\} \right]\{(n-1)(2k+1)\}
\end{align*}
and, for any $k\geq 1$,
\begin{align*}
Kh_n\left(\,\trsplusk{2k}\,\right)&
=\left[ [n] \{1-n\}\,
\xrightarrow{0} 0\xrightarrow{0} [n-1]\{4-n\} 
\xrightarrow{0} [n-1]\{4+n\} 
\xrightarrow{0} \cdots\right.\\
&\hskip -82 pt\left.\xrightarrow{0}[n-1]\{4(k-1)-n\}  \xrightarrow{0}[n-1]\{4(k-1)+n\}
\xrightarrow{0} [n][n-1]\{4k-1\} \right]\{(n-1)(2k)\}
\end{align*}
\end{corollary}
\begin{proof}
By Theorem \ref{main-thm} we have, for any $k\geq 0$,
\begin{align*}
Kh_n\left(\,\trsplusk{2k+1}\,\right)&
=\left[\twocircles
\xrightarrow{\chi_0} \tresse\{1\}\xrightarrow{\alpha} \tresse\{3\}\right.
\\
&\xrightarrow{\gamma} \tresse\{5\}\xrightarrow{\alpha} \cdots
\xrightarrow{\alpha}\tresse\{4k-1\}\\
&\left.\xrightarrow{\gamma}\tresse\{4k+1\}\right]\{(n-1)(2k+1)\}
\end{align*}
By equation (\ref{eq:moychi0}) we have a commutative diagram
\[
\xymatrix{
*++++++++++{\twocircles} \ar[rr]^-{\chi_0}\ar[d]_-{\lambda\otimes 1}^-{\wr} &&
*++++{{\tresse}\{1\}}\ar[d]^-{\mu\{1\}}_{\wr}\\
\unknot\,\, \{1-n\}\oplus [n-1]\phantom{mmm}\unknot\,\,\, \{1\}\,
\ar[rr]^-{\tiny{\left(\begin{matrix}0&1\end{matrix}\right)}}&&*+++{[n-1]\phantom{mmm}\unknot\,\,\, \{1\}}
}
\]
By equation (\ref{eq:moichi0a}) we have a commutative diagram
\[
\xymatrix{
*++++{{\tresse}\{-1\}}\ar[d]^-{\mu} \ar[r]^{\alpha} &*++++{{\tresse}\{1\}}\ar[d]^-{\mu}\\
*++++{[n-1]\phantom{mmm}\unknot\,\,\{-1\}}\ar[r]^{0} &*+++{[n-1]\phantom{mmm}\unknot\,\,\{1\}}
}
\]
By equation (\ref{eq:moichi0b}) we have a commutative diagram
\[
\xymatrix{
*++++{{\tresse}\{-1\}}\ar[d]^-{\mu} \ar[r]^{\gamma} &*++++{{\tresse}\{1\}}\ar[d]^-{\mu}\\
*+++{[n-1]\phantom{mmm}\unknot\,\,\{-1\}}\ar[r]^{\epsilon} &*+++{[n-1]\phantom{mmm}\unknot\,\,\{1\}}
}
\]
Therefore, by equation (eq1) we get
\begin{align*}
Kh_n\left(\,\trsplusk{2k+1}\,\right)&
=\left[\unknot\,\, \{1-n\}\oplus [n-1]\unknot\,\,\, \{1\}\,
\xrightarrow{\tiny{\left(\begin{matrix}0&1\end{matrix}\right)}} {[n-1]\unknot\,\,\{1\}}\right.\\
&\hskip -24 pt \xrightarrow{0} {[n-1]\unknot\,\,\{3\}}
\xrightarrow{\epsilon} {[n-1]\unknot\,\,\{5\}}\xrightarrow{0} \cdots
\xrightarrow{0}{[n-1]\unknot\,\,\{4k-1\}}\\
&\hskip -24 pt\left.\xrightarrow{\epsilon}{[n-1]\unknot\,\,\{4k+1\}}\right]\{(n-1)(2k+1)\}
\end{align*}
By the isomorphism
\[
\lambda\colon \unknot\xrightarrow{\sim} [n]
\]
 and by equation (\ref{eq:moyepsilon}) we obtain
\begin{align*}
Kh_n\left(\,\trsplusk{2k+1}\,\right)&
=\left[[n] \{1-n\}\oplus [n-1][n] \{1\}\,
\xrightarrow{\tiny{\left(\begin{matrix}0&1\end{matrix}\right)}} {[n-1][n]\{1\}}\right.\\
&\hskip -120 pt \xrightarrow{0} [n-1](\mathbb{Q}\{4-n\}\oplus [n-1]\{4\}) 
\xrightarrow{\tiny{\left(\begin{matrix}0&1\\0 & 0\end{matrix}\right)}} [n-1]([n-1]\{4\}\oplus\mathbb{Q}\{4+n\}) 
\xrightarrow{0} \cdots\\
&\hskip -120 pt\left.\xrightarrow{0}[n-1](\mathbb{Q}\{4k-n\}\oplus [n-1]\{4k\}) 
\xrightarrow{\tiny{\left(\begin{matrix}0&1\\0 & 0\end{matrix}\right)}} [n-1]([n-1]\{4k\}\oplus\mathbb{Q}\{4k+n\}) \right]\{(n-1)(2k+1)\}
\end{align*}
By changing the orders in the direct sums (and so by (eq1)) we can rewrite this as
\begin{align*}
Kh_n\left(\,\trsplusk{2k+1}\,\right)&
=\left[ [n-1][n] \{1\}\oplus[n] \{1-n\}\,
\xrightarrow{\tiny{\left(\begin{matrix}1&0\end{matrix}\right)}} {[n-1][n]\{1\}}\right.\\
&\hskip -120 pt \xrightarrow{0} [n-1]([n-1]\{4\}\oplus \mathbb{Q}\{4-n\}) 
\xrightarrow{\tiny{\left(\begin{matrix}1&0\\0 & 0\end{matrix}\right)}} [n-1]([n-1]\{4\}\oplus\mathbb{Q}\{4+n\}) 
\xrightarrow{0} \cdots\\
&\hskip -120 pt\left.\xrightarrow{0}[n-1]([n-1]\{4k\}\oplus \mathbb{Q}\{4k-n\} ) 
\xrightarrow{\tiny{\left(\begin{matrix}1&0\\0 & 0\end{matrix}\right)}} [n-1]([n-1]\{4k\}\oplus\mathbb{Q}\{4k+n\}) \right]\{(n-1)(2k+1)\}
\end{align*}
By repeated use of (eq2) we therefore obtain
\begin{align*}
Kh_n\left(\,\trsplusk{2k+1}\,\right)&
=\left[ [n] \{1-n\}\,
\xrightarrow{0} 0\xrightarrow{0} [n-1]\{4-n\} 
\xrightarrow{0} [n-1]\{n+4\} 
\xrightarrow{0} \cdots\right.\\
&\hskip -12 pt\left.\xrightarrow{0}[n-1]\{4k-n\}  
\xrightarrow{0} [n-1]\{n+4k\} \right]\{(n-1)(2k+1)\}
\end{align*}

The proof for the case of an even number of crossing is perfectly analogous.
\end{proof}

\begin{remark}\label{rem:poincare}
As the MOY complex associated with an $m$-crossing 2-stand braid link can be represented by a complex with only graded vector spaces (times $\emptyset$) as objects and all zero differentials, it is completely encoded into the \emph{Poincar\'e polynomial} of the graded dimensions, obtained by the rule
\[
\left[0\to 0\to \cdots\to 0\to  \hskip -15 pt\underbrace{\mathbb{Q}\{i\}}_{\text{\tiny{horizontal degree $j$}}}  \hskip -15 pt\to 0\to\cdots\to 0\right] \mapsto q^it^j
\]
This way we obtain
\begin{align*}
Kh_n\left(\,\trsplusk{2k+1}\,\right)&
=q^{(n-1)(2k+1)}\left(q^{1-n}[n]_q+\left(t^2q^{4-n}+t^3q^{4+n}+\right.\right.\\
&\qquad\qquad\left.\left. +t^4q^{8-n}+t^5q^{8+n}+\cdots+t^{2k}q^{4k-n} +t^{2k+1}q^{4k+n}\right)[n-1]_q\right)\\
&\\
&\hskip -90 pt=q^{(n-1)(2k+1)}\left(q^{1-n}[n]_q+t^2q^4(q^{-n}+tq^n)(1+t^2q^4+(t^2q^4)^2+\cdots+(t^2q^4)^{k-1})[n-1]_q \right)\\
&\\
&\hskip -90 pt=q^{(n-1)(2k+1)}\left(q^{1-n}[n]_q+(q^{-n}+tq^n)\frac{t^2q^4-t^{2k+2}q^{4k+4}}{1-t^2q^4}[n-1]_q \right)
\end{align*}
and
\begin{align*}
Kh_n\left(\,\trsplusk{2k}\,\right)&=q^{(n-1)(2k)}\left(q^{1-n}[n]_q+\left(t^2q^{4-n}+t^3q^{4+n}+\right.\right.\\
&\hskip -50 pt\left. \left.+t^4q^{8-n}+t^5q^{8+n}+\cdots+t^{2k-2}q^{4(k-1)-n} +t^{2k-1}q^{4(k-1)+n}\right)[n-1]_q+t^{2k}q^{4k-1}[n]_q[n-1]_q\right)\\
&\\
&\hskip -100 pt=q^{(n-1)(2k+1)}\left(q^{1-n}[n]_q+t^2q^4(q^{-n}+tq^n)(1+t^2q^4+\cdots+(t^2q^4)^{k-2})[n-1]_q +t^{2k}q^{4k-1}[n]_q[n-1]_q\right)\\
&\\
&\hskip -100 pt=q^{(n-1)(2k)}\left(q^{1-n}[n]_q+(q^{-n}+tq^n)\frac{t^2q^4-t^{2k}q^{4k}}{1-t^2q^4}[n-1]_q+t^{2k}q^{4k-1}[n]_q[n-1]_q \right)
\end{align*}
where
\[
[m]_q=\dim_q [m]=q^{1-m}+q^{3-m}+\cdots+q^{m-3}+q^{m-1}=\frac{q^m-q^{-m}}{q-q^{-1}}.
\]
\end{remark}
\begin{remark}
For an odd number $2k+1$ of crossings, the above formula extends to $n\leq 4$ the results by Rasmussen \cite{Rasmussen}, and agrees with the prediction by Morozov and Anokhina \cite{Morozov1}. For low values of $n$ and $k$, this formula has been computer checked by Carqueville and Murfet in \cite{Carqueville-Murfet}. More precisely, the case of 3 crossings has been computer checked for $n\leq 11$, and the case of 5 crossings for $n\leq 6$.
For an even number $2k$ of crossings, the above formula agrees with the prediction by Morozov and Anokhina \cite{Morozov1}, and with the string theory inspired prediction by Gukov, Iqbal, Koz\c{c}az and Vafa for the case of two crossings (the Hopf link) \cite{Gukov}. For low values of $n$ and $k$, this formula has been computer checked by Carqueville and Murfet in \cite{Carqueville-Murfet}. More precisely, the case of 2 crossings has been computer checked for $n\leq 11$, the case of 4 crossings for $n\leq 5$, and the case of 6 crossings for $n\leq4$.
\end{remark}

\begin{remark}
Evaluating at $t=-1$ the above Poincar\'e polynomials one recovers the known formulas for the level $n$ Jones polynomials of the $m$-crossings 2-strand braid link, see \cite{fw}:
\[
Jones_n\left(\,\trsplusk{2k+1}\,\right)=q^{(n-1)(2k+1)}\left(q^{1-n}[n]_q+(q^{-n}-q^n)\frac{q^4-q^{4k+4}}{1-q^4}[n-1]_q \right)
\]
\[
Jones_n\left(\,\trsplusk{2k}\,\right)
=q^{(n-1)(2k)}\left(q^{1-n}[n]_q+(q^{-n}-q^n)\frac{q^4-q^{4k}}{1-q^4}[n-1]_q+q^{4k-1}[n]_q[n-1]_q \right)
\]
\end{remark}

\begin{example}[The Hopf link] The 2-crossing 2-braid link is known as the \emph{Hopf link}, as collections of pairs of $S^1$ linked this way arise in the Hopf fibration $S^3\to S^2$. Specializing the formula from Remark \ref{rem:poincare} to the case $2k=2$, we find
\begin{align*}
Kh_n\left(\,\hopf\,\right)&=q^{2(n-1)}\left(q^{1-n}[n]_q+t^{2}q^{3}[n]_q[n-1]_q \right)\\
&=q^{2(n-1)}\left(q^{1-n}[n]_q+t^{2}q^{2}[n]_q([n]_q-q^{1-n}) \right)\\
&=q^{n-1}\left(\frac{q^n-q^{-n}}{q-q^{-1}}\right)+q^{2n}\left(\frac{q^n-q^{-n}}{q-q^{-1}}\right)^2t^2-q^{n-1}\left(\frac{q^n-q^{-n}}{q-q^{-1}}\right)t^2
\end{align*}
This way one sees how our general formula from Remark \ref{rem:poincare} specializes to the formula  in \cite{Carqueville-Murfet}.
\end{example}

\begin{example}[The trefoil knot] The 3-crossing 2-braid link is the \emph{trefoil knot}. Specializing the formula from Remark \ref{rem:poincare} to the case $2k+1=3$, we find
\begin{align*}
Kh_n\left(\,\trefoil\,\right)&=q^{3(n-1)}\left(q^{1-n}[n]_q+q^{-n}(q^{-n}+tq^n)t^2q^4[n-1]_q \right)\\
&=q^{2(n-1)}\left([n]_q+[n-1]_qq^{-1}(1+tq^{2n})t^2q^4 \right)\\
&=q^{2n-2}\left(\frac{q^n-q^{-n}}{q-q^{-1}}+\frac{q^{n-1}-q^{-n+1}}{q-q^{-1}}q^{-1}(1+tq^{2n})t^2q^4 \right)
\end{align*}
This way one sees how our general formula from Remark \ref{rem:poincare} specializes to the formula  in \cite{Carqueville-Murfet} (up to the change of variable $t\leftrightarrow t^{-1}$ witnessing the mirror refelction relating the trefoil knot considered here and the trefoil knot considered in \cite{Carqueville-Murfet}).
\end{example}

\eject

\noindent {\large\bf{Zusammenfassung}}
\vskip .5 cm
{\vbox{\hsize 12 cm

Im ersten Teil die Dissertation reformieren wir die Murakami-Ohtsuki-Yama\-da-Summen-Beschreibung 
des Levels $n$ Jonespolynomial einer orientierten link in
Bezug auf einegeeignete geflochtene monoidale Kategorie, deren Morphismen
$\mathbb{Q}[q,q^{-1}]$-lineare Kombinationen orientierte dreiwertige planare Graphen sind und geben
Sie eine entsprechende Beschreibung f\"ur das HOMFLY-PT-Polynom ein. 
\par
Im
zweiten Teil erweitern wir diese Konstruktion und bringen die Khovanov-
Rozansky-Homologie eines orientierten Links in Form einer kombinatorisch
definierten Kategorie zum Ausdruck, deren Morphismen \"Aquivalenzklassen
formaler Komplexe von (formalen direkten Summen verschobenen) orientierte planare Graphen sind. Kombinatorisches Arbeiten vermeidet viele der
rechnerischen Schwierigkeiten, die mit den Matrixfaktorisierungsberechnungen
der urspr\"unglichen Khovanov-Rozansky-Formulierung einhergehen: Sie
verwendet systematisch kombinatorische Beziehungen, die durch diese
Matrixfaktorisierungen befriedigt werden, um die Berechnung auf einer leicht zu
handhabenden Ebene zu erleichtern. Durch Verwendung dieser Technik k\"onnen
wir eine Berechnung des Levels $n$ Khovanov-Rozansky-Invariante der
zweisträngigen Geflechtverbindung mit $k$ Kreuzungen f\"ur beliebiges $n$ und $k$
bereitstellen, um vorherige Ergebnisse und mutma{\ss}liche Vorhersagen von
Anokhina-Morozov zu best\"atigen und zu erweitern. Carqueville-Murfet, Dolotin-
Morozov, Gukov-Iqbal-Kozcaz-Vafa und Rasmussen.
}}
\end{document}